%% file: hodge-grobner.tex
\documentclass[11pt,reqno]{amsart}

\usepackage{amsmath,amssymb,amsthm,xspace}
\usepackage[pagebackref]{hyperref}
\usepackage{mathtools}
\usepackage[margin=1.2in]{geometry}
\usepackage{xcolor}
\usepackage{tabmac}
\usepackage{float}
\usepackage{enumitem}
\usepackage{nicefrac}
\usepackage[parfill]{parskip}
\usepackage[section]{algorithm}
\usepackage{algpseudocode}
\usepackage{bm}
\usepackage{keytheorems}
\usepackage{tikz}

\input{shortcuts}

\input{preamble}

\usetikzlibrary{cd}

\title[Gröbner bases native to certain commutative algebras, with applications]{Gröbner bases native to term-ordered commutative algebras, with application to the Hodge algebra of minors}

\author{Joshua A. Grochow}
\address{(Grochow) University of Colorado Boulder, Boulder, CO 80309, U.S.A.}
\email{jgrochow@colorado.edu}

\author{Abhiram Natarajan}
\address{(Natarajan) University of Warwick, Coventry CV4 7HP, U.K.}
\email{abhiram.natarajan@warwick.ac.uk}

\begin{document}

\begin{abstract}
Motivated by better understanding the bideterminant (=product of minors) basis on the polynomial ring in $n \times m$ variables, we develop theory \& algorithms for Gröbner bases in not only algebras with straightening law (ASLs or Hodge algebras), but in any commutative algebra over a field that comes equipped with a notion of ``monomial'' (generalizing the standard monomials of ASLs) and a suitable term order. Rather than treating such an algebra $A$ as a quotient of a polynomial ring and then ``lifting'' ideals from $A$ to ideals in the polynomial ring, the theory we develop is entirely ``native'' to $A$ and its given notion of monomial.

When applied to the case of bideterminants, this enables us to package several standard results on bideterminants in a clean way that enables new results. In particular, once the theory is set up, it lets us give an almost-trivial proof of a universal Gröbner basis (in our sense) for the ideal of $t$-minors for any $t$. We note that here it was crucial that theory be native to $A$ and its given monomial structure, as in the standard monomial structure given by bideterminants each $t$-minor is a single variable rather than a sum of $t!$ many terms (in the ``ordinary monomial'' structure).
\end{abstract}

\maketitle

\setcounter{tocdepth}{3}

\tableofcontents

\input{introduction}

\input{prelim}

\input{results}

\input{theory}

\input{syzygies-modules}

\input{algorithms}

\input{applications}

\input{two-asl-structures}

\input{dimension}

\input{conclusion}

\bibliographystyle{alphaurl} 
\bibliography{hodge-grobner} 

\appendix
\input{appendix}

\end{document}

%% file: shortcuts.tex
\DeclareMathOperator{\Z}{\mathbb{Z}}
\DeclareMathOperator{\C}{\mathbb{C}}
\DeclareMathOperator{\N}{\mathbb{N}}

\DeclareMathOperator{\W}{\mathbb{W}}
\DeclareMathOperator{\F}{\mathbb{F}}

\newcommand{\suchthat}{:}
\newcommand*{\eqcomment}[2][\qquad]{#1 &&\text{[#2]}}
\newcommand*{\eqcommentnear}[1]{&&\text{[#1]}}
\newcommand{\perms}[1]{\mathfrak{S}\left(#1\right)}
\newcommand{\sgn}[1]{\mathrm{sgn}\left(#1\right)}

\newcommand{\ideal}[1]{\left\langle#1\right\rangle}
\newcommand{\bigslant}[2]{{\raisebox{.2em}{$#1$}\left/\raisebox{-.2em}{$#2$}\right.}}

\newcommand{\bitab}[2]{[#1\;|\;#2]}
\newcommand{\bigbitab}[2]{\left[\,\left. #1\;\right|\;#2\,\right]}
\newcommand{\bidet}[2]{(#1\;|\;#2)}
\newcommand{\bigbidet}[2]{\left(\,\left. #1\;\right|\;#2\,\right)}
\newcommand{\stdize}[1]{^{(s)}(#1)}

\newcommand{\ltfrac}[2]{\frac{\,#1\,}{\,#2\,}{}_{{\scriptscriptstyle lt}}}

\newcommand{\divides}{\mid}
\newcommand{\ltdivides}{\mid_{lt}}
\newcommand{\gendivides}{\mid_{gen}}

\DeclareMathOperator{\Ann}{Ann}

\newcommand{\gb}{Gr\"{o}bner }
\newcommand{\lt}{\text{LT}}
\newcommand{\lm}{\text{LM}}
\newcommand{\lmsm}{\text{LM}_{sm}}
\newcommand{\tail}{\text{tail}_{sm}}
\newcommand{\lcsm}{\text{LC}_{sm}}
\newcommand{\ltsm}{\text{LT}_{sm}}

\newcommand{\lcmlt}{LCM_{sm}^{lt}}
\newcommand{\lismlt}{LI_{sm}^{lt}}
\newcommand{\Ssm}{S_{sm}}
\newcommand{\lism}{LI_{sm}}

\newcommand{\piltd}{\pi_{lt}}

\newcommand{\Syz}{\text{Syz}}
\newcommand{\LAM}{\ensuremath{\text{LAM}}}
\newcommand{\LAB}{\ensuremath{\text{LAB}}}

\newcommand{\veq}{\rotatebox[origin=c]{-90}{$=$}}
\newcommand{\vlt}{\rotatebox[origin=c]{-90}{$>$}}

\newcommand{\shapebitab}{sh}
\newcommand{\shapetab}[1]{sh(#1)}
\newcommand{\sizetab}[1]{|#1|}

%% file: preamble.tex
\allowdisplaybreaks

\mathtoolsset{showonlyrefs=true} 

\newkeytheorem{example}[parent=section]
\newkeytheorem{theorem}[sibling=example]
\newkeytheorem{proposition}[sibling=example]
\newkeytheorem{lemma}[sibling=example]
\newkeytheorem{corollary}[sibling=example]
\newkeytheorem{claim}[sibling=example]
\newkeytheorem{observation}[sibling=example]
\newkeytheorem{conjecture}[sibling=example]

\theoremstyle{definition}
\newkeytheorem{definition}[sibling=example]
\newkeytheorem{remark}[sibling=example]
\newkeytheorem{open}[name={Open Question}, sibling=example]
\newkeytheorem{notation}[sibling=example]
\newkeytheorem{convention}[sibling=example]

\numberwithin{equation}{section}

\let\oldtocsection=\tocsection
\let\oldtocsubsection=\tocsubsection
\let\oldtocsubsubsection=\tocsubsubsection
\renewcommand{\tocsection}[2]{\hspace{0em}\oldtocsection{#1}{#2}}
\renewcommand{\tocsubsection}[2]{\hspace{1em}\oldtocsubsection{#1}{#2}}
\renewcommand{\tocsubsubsection}[2]{\hspace{2em}\oldtocsubsubsection{#1}{#2}}

%% file: introduction.tex

\section{Introduction}
Gröbner bases in polynomial rings both provide the core of most algorithms in such rings, and are also useful conceptually. Yet there are many rings in which one would like a theory of Gröbner bases---both for computation and for theoretical investigation---which go far beyond commutative polynomial rings. Indeed, Shirshov originally developed the theory in the context of Lie algebras \cite{shirshov}.\footnote{For more on the history see the surveys \cite{BFS, BokutChen}.} Gröbner--Shirshov basis theories \& algorithms have been developed in group algebras \cite{greenGroupBook} and monoid algebras \cite{reinert,moraMonoid} (and references therein), basic algebras \cite{FGKK, FFG}, Weyl algebras \cite{saito1998grobner}, more general solvable algebras (where elements commute modulo lower-order terms) \cite{KRW}, and algebras with multiplicative basis (where the product of basis elements is another basis element or 0) \cite{green}, among others.\footnote{In addition to the already cited references, see, e.g., the appendix to Becker \& Weispfenning \cite{BW} for a brief survey of others.} In this paper, we pursue Gröbner bases in arbitrary finitely generated commutative algebras over a field,
 and we develop such a theory whenever the algebra admits a suitable term order.

Of course, over a field $\F$, any finitely generated commutative $\F$-algebra $A$ is a quotient of a polynomial ring $\F[X_1, \ldots, X_n]$ by some ideal $J$, so one could simply lift from $A$ to $\F[X_1, \ldots, X_n]$, and lift an ideal $I \subseteq A$ to $I + J \subseteq \F[X_1, \ldots, X_n]$, and use ordinary Gröbner bases in polynomial rings. A particularly nice version of this approach is via so-called ``relative Gröbner bases'' \cite{spear,zacharias,siebert,LaScalaStillman,moraZacharias,HOS},\footnote{Not to be confused with another concept by the same name \cite{zhouWinkler}, which is about ordinary Gröbner bases relative to two simultaneous term orders.} which we discuss more in the 
`Related Work' section (Section \ref{sec:related}) below. However, there are many situations where one would want a theory of Gröbner bases that is ``native'' to $A$. In other words, one might want to compute and use \gb bases of ideals in $A$ without any reference to the ideal $J$. In addition to being interesting on its own, computationally this could be because there is a more efficient way of multiplying in $A$ directly, rather than lifting to $\F[X_1, \ldots, X_n]$ and performing multiplication modulo $J$ using standard Gröbner bases. Structurally, it could be because $A$ has additional properties that can be taken advantage of, for example, $A$ and/or $I$ could have symmetries that get obscured when lifted to $\F[X_1, \ldots, X_n]$.

\subsection{Motivation} 

This work arose out of the authors' attempts to compute by hand the algebraic de Rham cohomology of certain determinantal and tensor rank varieties using the algorithms of Oaku \& Takayama \cite{oakutakayama-derhamalgo-1999, oakutakayama-derham-2001}. Computing the cohomology led us to aim for \gb bases of ideals in the Weyl algebra. The Weyl algebra is a non-commutative algebra for which a theory of Gröbner bases is known \cite{saito1998grobner}. In our experience however, we found that using these algorithms to even compute by hand the cohomology of the $3 \times 3$ determinant hypersurface was intractable on account of the \gb basis computations. Indeed, while our ideals of interest carried symmetries akin to those in the ideals defining determinantal varieties---and involved a lot of minors---the \gb basis computations did not make use of this structure and were instead naively dealing with the generators of these ideals as just linear combinations of monomials in the Weyl algebra. Thus we were compelled to develop a theory of \gb bases which \emph{cooperates} with this determinantal structure (see also comments in Section \ref{sec:conclusion-weyl}). Applying our theory towards computing algebraic de Rham cohomology of the aforementioned varieties is the subject of a planned follow-up to this paper.

Specifically, one of our motivations thus became to develop a theory of \gb bases in the algebra of bideterminants (products of minors). Motivated by the role of minors in constructing polynomial invariants, bideterminants were studied and their properties were established in \cite{doubilet1974foundations, deconciniprocesi1976characteristic, DEPYoung, desarmenien1982invariant} (see also the ``Standard Monomial Theory'', e.\,g., \cite{SMT1, SMT2}). More generally, bideterminants are known to give the structure of an Algebra with Straightening Law (ASL, a.k.a. Hodge algebra) to the coordinate ring of $n \times m$ matrices $\F[X] = \F[X_{i,j} : i\in \{1,\dotsc,n\}, j \in \{1,\dotsc,m\}]$. The general framework of ASLs were introduced in \cite{deconcini1987hodge} (see also \cite{eisenbud-intro-1980}) as a generalization of several well-studied concrete examples of finitely generated algebras in which a product of monomials in the generators get rewritten or ``straightened'' according to a system of equations satisfying some nice properties. ASLs include Grassmannians, flag varieties, Schubert varieties, determinantal and Pfaffian varieties, varieties of minimal degree, and varieties of complexes.

One of the advantages of an ASL structure, already highlighted in the original papers \cite{deconcini1987hodge, eisenbud-intro-1980}, is that it facilitates computation with ideals whose generators are well-adapted to the generators of the ASL. For example, the bideterminant ASL structure on $\F[X]$ has as its generators the minors of the variable matrix $X$; ideals generated by various subsets of minors, which typically could take exponentially many monomials to write down in the usual monomial basis, instead become generated by single variables or monomials in the bideterminant basis! The price paid is that multiplication of monomials in such generators sometimes need to be rewritten to bring them back to a normal form. 

In developing such a theory, we then realized that in fact our theory applied beyond ASLs; all we really needed was a suitable notion of monomial and term order.

In the case of bideterminants, we indeed take advantage of the fact that in the bideterminant standard monomial basis, each minor is a single variable, from which it becomes almost trivial (once our theory of Gröbner bases is set up) to see that the minors of size $\geq t$ form a universal Gröbner basis. By contrast, in the case of ordinary monomials, universal Gröbner bases are only known for maximal minors \cite{sturmfels-maximal-1993,bernstein1993combinatorics} and minimal minors (i.e., size 2) \cite{sturmfels1996grobner}.

As another example of why one might want a theory of Gröbner bases native to a (pseudo-)ASL, in the specific case of bideterminants, from a computational perspective, more efficient algorithms are known for multiplying bideterminants \cite{ShaoLi, desarmenien1980algorithm, white} than one would get either by lifting to the ring generated by all minors (a polynomial ring in $2^n \times 2^m$ variables) and multiplying modulo the ideal of relations among minors, or by writing minors out as polynomials in the $nm$ variables $X_{i,j}$ (where a $k$-minor has $k!$ terms). Finally, from the perspective of symmetries, bideterminants provide a way of organizing and exhibiting symmetric structures in several contexts. For instance, bideterminants correspond to pairs of tableaux which in turn reflect symmetries coming from the action (representation theory) of $\mathrm{GL}_n$ \cite{DEPYoung}.

\subsection{Overview of our work}
In this paper, we develop a theory of \gb bases `native' to finitely generated commutative $\F$-algebras\footnote{We expect our results should readily generalize to coefficients being a commutative principal ideal domain such as $\Z$ or principal ideal ring such as $\bigslant{\Z}{p^k \Z}$.} with a choice of ``monomials'' and a choice of term order. The choice of monomials is encapsulated in what we call a pseudo-ASL structure (Definition \ref{defn:pseudo-ASL}): a choice of generating set and a particularly nice choice of products of those generators that form an $\F$-linear basis of the algebra, which then are called the standard monomials of the pseudo-ASL.\footnote{While these \emph{can} be the standard monomials relative to an ordinary Gröbner basis of the defining ideal of the algebra, our theory is more general than that particular construction.} (We call them pseudo-ASLs because they generalize ASLs, but in fact every finitely generated commutative $\F$-algebra admits a pseudo-ASL structure.)

To explain one of the key difficulties in developing such a theory, we must say a little more about pseudo-ASLs. In a pseudo-ASL $A=\bigslant{\F[X_1, \dotsc, X_n]}{S}$, because of the structure of $S$, products of standard monomials need not be standard monomials, but rather will be linear combinations of standard monomials. In attempting to define Gröbner bases in a pseudo-ASL, this issue complicates what it means for one standard monomial to divide another. 

Our resolution of this complication results in additional choices that offer tradeoffs in terms of both structure and algorithms. The high-level idea is that, in addition to a choice of term order (as in any theory of Gröbner bases), we also have a choice of ``algebra of leading terms'', which precisely governs which ``division-like'' relations among standard monomials we can use to cancel leading terms when it comes to Gröbner basis calculations. There are two natural possibilities that come to mind: (1) ordinary divisibility, that is, say that a standard monomial $a$ divides another standard monomial $b$ if there is a third standard monomial $c$ such that $ac=b$; or (2) say that $a$ divides $b$ if there is a $c$ such that the \emph{leading term} (according to a chosen term order) of $ac$ is precisely $b$. Each of these possibilities corresponds to a different choice of algebra of leading terms, and by reifying this choice in terms of an algebra, additional possibilities beyond those two offer various tradeoffs that we highlight throughout the paper.

In early iterations of this work, before developing the notion of algebra of leading terms, we had in fact given two separate and more difficult proofs for many results in both settings (1) and (2) separately. Working with the algebra of leading terms instead lets us follow Emmy Noether's path, simplifying, unifying, and streamlining those proofs across those two possibilities and more, essentially relying on the Noetherianity of the algebra of leading terms.

With this setup, we define the notion of pseudo-ASL \gb bases, and prove the following:
\begin{itemize}
\item existence of finite pseudo-ASL Gröbner bases (Theorem~\ref{thm:asl-gb-existence}),
\item uniqueness of reduced pseudo-ASL Gröbner bases (Theorem~\ref{thm:reduced}),
\item existence of finite universal pseudo-ASL \gb bases (Theorem~\ref{thm:universal}),
\item pseudo-ASL \gb bases for syzygies when $A$ is a domain (pseudo-ASL analogue of Schreyer's Theorem, Theorem~\ref{thm:schreyer}),
\item algorithms that halt in finite time for computing pseudo-ASL Gröbner bases (Corollary \ref{cor:AgenGB} and Theorem \ref{thm:asl-gb-alt}), and
\item calculating Krull dimension in pseudo-ASLs using pseudo-ASL \gb bases (Section~\ref{sec:dimension}).
\end{itemize}

Finally, we apply this theory to the case of the bideterminant ASL structure mentioned above. 
From a theoretical perspective, we show in Theorem~\ref{thm:t-minors-bdgb} that by developing a theory of Gröbner bases natively in the bideterminant ASL, one finds a very natural universal Gröbner basis (in our theory) for the ideal of $t$-minors for any $t$; in contrast, for classical Gröbner bases in the polynomial ring $\F[X]$, universal Gröbner bases are known only in the cases of maximal minors \cite{sturmfels-maximal-1993,bernstein1993combinatorics} and minimal minors (i.e., size 2) \cite{sturmfels1996grobner}. Moreover, the existing proofs in these cases are quite involved, whereas once our framework is in place, the corresponding proof becomes remarkably short.

\subsection{Comparison to related work} \label{sec:related}

In this section we gather the most relevant related works we are aware of, and provide a comparison with our work. 

\textbf{Related work on Gröbner bases in general.} 
There are several works on Gröbner bases in quotients of polynomial rings, e.g. \cite{spear,zacharias,siebert,LaScalaStillman,HOS}. If $P = \F[X_1,\dotsc,X_n]$ is a polynomial ring over a field, $A = \bigslant{P}{J}$, and $I \subseteq A$ is an ideal, these works take an approach that in essence amounts to taking a Gröbner basis in the polynomial ring $P$ of the preimage ideal $I + J$. In these approaches, the standard monomials are defined to be those in the polynomial ring outside the leading ideal of $J$; the leading monomial of $f \in I \subseteq A$ is defined as the corresponding standard monomial in $\F[X_1,\dotsc,X_n]$; and a Gröbner basis, called a relative Gröbner basis in \cite{HOS}, is (equivalently) defined as a subset $\{f_1,\dotsc,f_k\} \subseteq I$ such that the leading monomials of the $f_i$ generate an ideal in the polynomial ring that contains the leading monomials of all elements of $I$. This technique, of using a term order on the ``parent'' polynomial ring, has been greatly extended to more general coefficients rings $R$ and noncommutative $R$-algebras far beyond polynomial rings, e.g. \cite{moraZacharias, reinert, KRW, saito1998grobner}. 

In contrast to the above, our approach is intrinsic to the quotient ring itself, making no direct reference to the defining ideal $J$. We work directly with the algebraic structure of the quotient and thus require a term order native to it, that is, only on the standard monomials (that actually occur in $A$) and not on all monomials in the parent polynomial ring. To see how this comparison plays out in terms of syzygies, see Remark~\ref{rmk:compare-ann}. As a slightly more concrete class of examples, we note that if $A$ is a graded ASL, then there is a term order on $\F[X_1,\dotsc,X_n]$ such that the leading term of each straightening law is the unique non-standard monomial in it (so, straightening really moves terms downwards in the term order), but the latter does not hold in general for non-graded ASLs \cite{trung,GrabePauer}. In our approach, it is in principle possible to nonetheless have a theory of Gröbner bases on the given standard monomials, since we only require a suitable term order on the standard monomials themselves. 

Rather than having our leading ideals live in the parent polynomial ring---which contains all monomials, including ones that don't occur in $A$, i.\,e., non-standard monomials---we have our leading ideals live in a ring we denote $A_{lt}$,\footnote{We call these ``algebras of leading terms'' for $A$, and the notation $A_{lt}$ is a mnemonic for that.} that has as basis the same set of standard monomials as $A$ itself, but also has the useful property (like the polynomial ring) that a product of (standard) monomials is a (scalar multiple of a standard) monomial. (This is related to but not quite the same as the ``graded structures'' approach of Robbiano \cite{robbiano} and Mora \cite{mora, moraIntro}.) This allows us to define a notion of standard monomial ideals in $A_{lt}$ that behaves much like monomial ideals in the polynomial ring. In the approaches discussed above, the set of leading monomials of elements of $I$ is not an ideal of the polynomial ring, just a linear subspace, and one asks for the leading monomials of a Gröbner basis to generate an ideal that contains that subspace. However, in our situation, the leading ideal is indeed a standard monomial ideal in $A_{lt}$, and we can define a Gröbner basis as a generating set of the leading ideal, rather than as generating an ideal that contains the leading monomials of $I$, as in the aforementioned other approaches.  

The use of $A_{lt}$ provides new choices that both generalize previous results and prove useful in different settings.
The choice of $A_{lt}$ governs what we mean when we say that one standard monomial divides another, or can be used to cancel it out. When $A_{lt}$ is chosen to be what we call ``the discrete algebra of leading terms $A_{disc}$'', (Definition \ref{defn:agen-adisc}) then divisibility in $A_{lt}$ agrees with divisibility in the polynomial ring in that divisibility of two standard monomials can be defined in the usual way by inequality of the corresponding exponent vectors, similar to the aforementioned approaches. At the other extreme, we define a ``generic algebra of leading terms $A_{gen}$'', in which we say that $m | m'$ (standard monomials) if there is another term $t$ such that the leading monomial of $mt$ is $m'$. This notion of divisibility may be more complicated, but when $A$ is a domain, relative to $A_{gen}$ one now recovers the ordinary Gröbner basis theory in terms only of S-remainders, one does not have to deal with annihilating syzygies (see Corollary~\ref{cor:AgenGB}). $A_{gen}$ is the associated graded ring relative to the grading by the monoid of standard monomials, see Remark~\ref{rmk:semigroup}. In the case of non-commutative ``rings of solvable type'' (as in \cite{KRW}), and in particular Weyl algebras (as in \cite{saito1998grobner}), the associated graded ring is the commutative polynomial ring, but this need not be true in the more general settings of the other papers mentioned above \cite{moraZacharias,spear,zacharias,HOS}.

Our approach using $A_{gen}$ recovers and significantly generalizes the approach to Gröbner bases in the bracket algebra (of maximal minors) by Li, Shao, Huang, and Liu \cite{li-reduction-2014}. In addition to our approach generalizing from maximal minors to all minors (equivalently, from brackets to bideterminants), the additional flexibility of choosing different algebras of leading terms leads to further results. In particular, by using $A_{disc}$ instead of $A_{gen}$, we get a clean algorithm for computing Krull dimension in our framework (see Section~\ref{sec:dimension}). The instantiation of our general theory in the setting of bideterminants also leads to a clean proof of a universal pseudo-ASL Gröbner bases for the ideals generated by $t$-minors (for each value of $t$), see Theorem~\ref{thm:t-minors-bdgb}.

Another setting in which the difference between our approach and previous approaches becomes apparent is that of universal Gröbner bases. In our approach, we define a universal Gröbner basis as a set that is a Gröbner basis relative to every choice of term order on $A$ and every choice of algebra of leading terms $A_{lt}$. In contrast, for the approaches above using relative \gb bases, two natural choices suggest themself: defining a universal Gröbner basis in terms of either (a) all term orders on the polynomial ring with the same leading ideal $LI(J)$, or (b) all terms orders on the polynomial ring. The issue with (b) is that the set of standard monomials changes as $LI(J)$ changes, making it very difficult to prove that a set is a universal Gröbner basis, since even the set of standard monomials occurring in each element can change as the order changes.\footnote{Indeed, one natural approach to Gröbner bases for bideterminants is as follows. Since bideterminants are a graded ASL, by \cite{GrabePauer} there is a term order on the polynomial ring in $2^n \times 2^m$ variables---one variable for each minor of an $n \times m$ matrix---such that the leading term of a straightening rule is the unique non-standard bideterminant. Relative to such an order, standard monomials correspond exactly to standard bideterminants, as in our setup. However, if one attempted to prove an analogue of our universal bideterminant Gröbner basis (Theorem~\ref{thm:t-minors-bdgb}) here, the issue is that relative to some term order, there are individual minors that become non-standard.} Approach (a) is more similar to our approach, but we find ours more intrinsic to the given pseudo-ASL structure on $A$: since we define pseudo-ASL term orders only on the standard monomials given by the pseudo-ASL structure, the question of ``how $LI(J)$ changes with the term order'' simply does not arise in our approach.

A different line of work related to ours are those in which it is assumed that the given algebra has a multiplicative basis (a product of two basis elements is another basis element or zero). In the setting of monoid rings this is exemplified in Reinert's thesis \cite{reinert}, which takes an approach similar to the above, with a term order on the parent free monoid. Green \cite{green} develops a more ``intrinsic'' approach, under the assumption of a multiplicative basis, where the term order is defined directly on the multiplicative basis, similar to our approach, rather than on a parent ``free ring.'' On the one hand, Green's approach is developed for noncommutative rings, whereas we assume our algebras are commutative (though we believe our work can be extended to noncommutative settings, we leave that for future work). On the other hand, our work can be seen as generalizing Green's in a different direction, namely one in which the chosen ``basis of standard monomials'' need not be multiplicative, so long as it admits a suitable term order. 

In some sense, our work can be seen as combining the best of both approaches above: we work intrinsically, as in Green's approach, and we allow a product of standard monomials to be a nontrivial linear combination of standard monomials, as in the approach pursued by many authors mentioned above. The price paid is that we need to assume the existence of a suitable term order, which need not always exist (see Examples~\ref{ex:no-term-order} and \ref{ex:no-term-order-2}).

\textbf{Prior results on Gröbner bases for ideals of $t$-minors.} 
For any $t$, a \gb basis of the ideal of $t$-minors was proved implicitly in \cite{narasimhan, abhyankar}, and was obtained explicitly in \cite{caniglia, sturmfels-stanley-gb}, and later in \cite{ma-minors-1994}.

A \emph{universal} \gb basis of the ideal of maximal minors of a variable matrix was conjecturally obtained in \cite{sturmfels-maximal-1993}, and the conjecture was proved in \cite{bernstein1993combinatorics}. Proofs using different techniques are available in \cite{kalinin} and \cite{cdg-universal-2015}. A universal \gb basis of the ideal of minors of size $2$ of a variable matrix was shown in \cite{sturmfels1996grobner}, and is also a special case of a result in \cite[Proposition 10.1.11]{villarreal-monomial-2015}.

Aside from these works on maximal minors and minors of size 2, we are not aware of prior work that establishes (in any setting) an explicit\footnote{Since finite universal Gröbner bases are guaranteed to exist in the polynomial ring, the question is not whether one exists, but of being able to explicitly specify one.} universal Gröbner basis for the ideal of $t$-minors for other values of $t$. Even for maximal minors and minors of size 2, such results must grapple with the combinatorics of the many possibilities for the set of leading terms of the $t$-minors under all possible term orders. In contrast, by working with pseudo-ASL Gröbner bases defined ``natively'' in the ASL structure given by bideterminants, these ideals become generated by sets of variables. With several now-classical results about bideterminants packaged into the Gröbner basis theory in this way, the proof that minors of size $\geq t$ are a universal pseudo-ASL Gröbner basis (Theorem~\ref{thm:t-minors-bdgb}) becomes almost trivial. We also note that it is not even clear to us how to state such a result in the setting of \cite{li-reduction-2014}, since they focus on one particular term order on bideterminants.

\subsection{Acknowledgments} We would like to thank Dave Benson and Richard Stanley for making us aware of, and providing a reference for, Theorem \ref{thm:hilbert-series-facts}-\ref{thmhilberseriesfacts-pole}, and \href{https://mathoverflow.net/}{MathOverflow} for providing a forum in which we could ask and receive answers to such a question. AN was supported by EPSRC Grant EP/V003542/1. JAG was supported by NSF CAREER award CCF-2047756.

%% file: prelim.tex

\section{Preliminaries: defining monomials, algebras with pseudo-straightening law} 
\begin{notation}
Throughout, $R$ will be our ring of coefficients; our default assumption is that $R$ is a field, for simplicity, but most if not all of our results easily extend (with obvious minor modifications) to the case of $R$ being a PID such as the integers, and we expect (because of how we handle straightening rules) that with a little more work they should also extend to the case of $R$ being a principal ideal ring that need not be a domain. We will use the shorthand $[n] := \{1, \ldots, n\}$. We use $\W = \{0,1,2,3,\dotsc\}$ to denote the whole numbers. Given a ring $R$ and $f_1, \ldots, f_s \in R$, $\ideal{f_1, \ldots, f_s}$ will denote the ideal in $R$ generated by $f_1, \ldots, f_n$.
\end{notation}

In Gröbner basis theory it is apparent that two notions are crucial: that of a monomial, and that of a term order. In this section we define what we mean by standard monomials in an arbitrary finitely generated commutative algebra $A$. We use notation and concepts that are as consistent as possible with those for algebras with straightening law (a.k.a. Hodge algebras) \cite{deconcini1987hodge, eisenbud-intro-1980}, for reasons that we will make clear below.

All our algebras are unital and associative, and unless otherwise noted, are also commutative.

Given a finite set $H$, let $\W^H$ denote the set of functions from $H$ to $\W$. An \emph{(abstract) monomial} on $H$ is an element of $\W^{H}$. The product of two monomials $m$ and $m'$, denoted $mm'$, is the monomial defined by $(mm')(x) = m(x) + m'(x)$ for all $x \in H$ (pointwise addition). With this product we refer to $\W^H$ as the monoid of abstract monomials on $H$. A monomial $m'$ divides the monomial $m$, denoted $m' \divides m$, if there is another monomial $m''$ such that $m'm''=m$, or equivalently if $m'(x) \le m(x)$ for all $x \in H$. When $m' \divides m$, we define the quotient monomial of $m$ by $m'$, denoted $\nicefrac{m'}{m}$, as the monomial defined by $(\nicefrac{m'}{m})(x) = m'(x) - m(x)$ for all $x \in H$. 

If $A$ is a commutative ring and we have an injection of sets $\phi: H \hookrightarrow A$, then to every abstract monomial $m$ on $H$, there is an associated element $\hat{\phi}(m) := \prod_{x \in H} \phi(x)^{m(x)} \in A$. By abuse of notation we will sometimes refer to this element as $m$ itself. 

An ideal of (abstract) monomials is a subset $\Sigma \subseteq \W^{H}$ such that if $m \in \Sigma$, then for any $m' \in \W^{H}$, we also have that $mm' \in \Sigma$. Equivalently, if $R[H]$ is the polynomial ring on $H$---which is the same as the monoid ring of $\W^H$ over $R$---and we let $\phi \colon H \to R[H]$ be the identity function on $H$, then a subset $\Sigma \subseteq \W^H$ is an ideal if and only if $\hat{\phi}(\Sigma)$ is precisely the set of monomials (in the ordinary sense of a polynomial ring) in the monomial ideal $\langle \hat{\phi}(\Sigma) \rangle \subseteq R[H]$.

We now come to our key definition of what a ``monomial in $A$'' means:

\begin{definition}[Algebra with pseudo-straightening law, a.k.a. pseudo-ASL]
\label{defn:pseudo-ASL}
An \emph{algebra with pseudo-straightening law (pseudo-ASL)} over a (commutative) ring $R$ is an $R$-algebra $A$, together with a finite generating set $H$, and an ideal $\Sigma \subseteq \W^H$ such that $A$ is a free $R$-module with basis $\{m \in \W^H \backslash \Sigma\}$, and $H \cap \Sigma = \emptyset$. We refer to the elements of $\W^H \backslash \Sigma$ as \emph{standard monomials} of the pseudo-ASL $(A,H,\Sigma)$.
\end{definition}

\begin{remark}
We note that the condition in Definition \ref{defn:pseudo-ASL} that $H \cap \Sigma = \emptyset$ is simply an irredundancy condition that is essentially without loss of generality. For suppose we dropped that condition; then we still have that every pseudo-ASL is equivalent to one satisfying the condition, in the sense that if $(A, H, \Sigma)$ is a pseudo-ASL structure on $A$, then there is another pseudo-ASL structure $(A, H', \Sigma')$ with the same set of standard monomials as the first pseudo-ASL structure, but with $H' \cap \Sigma' = \emptyset$. Namely, define $H' = H \backslash \Sigma$, and $\Sigma' := \{m \in \W^{H \backslash \Sigma} : m \in \Sigma\}$. For if $h \in H \cap \Sigma$, then every monomial involving $h$ is non-standard, and thus we can simply eliminate $h$ entirely without changing the set of standard monomials.
\end{remark}

For brevity, we refer to these as pseudo-ASLs.\footnote{The typography of ``pseudo-ASL'' also makes clearer that the ``pseudo'' modifier applies to the entire term ``algebra with straightening law'' and not just to the term ``algebra''. That is, it is not a (pseudo-algebra) with straightening law, it is a pseudo-(algebra with straightening law). One could have reasonably called it an algebra with pseudo-straightening law, but pseudo-ASL and rolls off the tongue better.} Given an algebra $A$ and generating set $H$, there will in general be many different choices for $\Sigma$ (and of course, given just the algebra $A$ there are many different choices for the generating set $H$). A choice of pseudo-ASL structure on $A$ effectively amounts to a choice of which elements are to be considered monomials. We note that $1$ is always a standard monomial, for otherwise we would have $\vec{0} \in \Sigma$, and then $\Sigma = \W^H$ since $\Sigma$ is an ideal, but $\W^H \backslash \Sigma = \emptyset$ cannot be an $R$-basis for $A$. 

\begin{convention}
\label{conv:pseudo-asl}
In full generality, we will write $(A, R, H, \Sigma, \phi)$ to denote a pseudo-ASL. In some cases, we may not need to refer to one or more of the elements $R, H, \Sigma, \phi$; in those cases, we will just mention whatever is relevant. For instance, if we only want to talk about a pseudo-ASL and the standard monomials in it, we will denote our pseudo-ASL as $(A, H, \Sigma)$; if we only want to talk about the multiplication structure of an ASL, we will denote our pseudo-ASL as $(A, R)$, etc.
\end{convention}

We note that every finitely generated commutative algebra admits as pseudo-ASL structure:

\begin{observation}
\label{obs:fg-pseudo-asl}
If $A$ is a finitely generated commutative algebra over a field $\F$, then $A$ admits a pseudo-ASL structure.
\end{observation}

\begin{proof}
$A$ is a quotient of some polynomial ring $\F[X_1, \ldots, X_n]$, say $A \cong \bigslant{\F[X_1, \ldots, X_n]}{J}$. Let $H = \{X_1, \ldots, X_n\}$. Choose a term order on $\F[X_1, \ldots, X_n]$ (in the sense of ordinary Gröbner bases), and let $\Sigma$ be the set of monomials that occur in the leading ideal of $J$ with respect to the chosen term order. By construction $\Sigma$ is an ideal in $\W^H$. By Macaulay's Basis Theorem (e.\,g., \cite[Theorem~1.5.7]{kreuzer-robbiano-1}), the monomials outside of $\Sigma$ form an $\F$-linear basis for $A$.
\end{proof}

However, not all pseudo-ASL structures arise from the Gröbner basis construction in the proof Observation \ref{obs:fg-pseudo-asl}, as the next example shows:

\begin{example}[Pseudo-ASL structure not coming from Macaulay's Basis Theorem] \label{ex:pASL}
Let $A = \bigslant{R[x,y]}{\langle xy - (x^2 + y^2) \rangle}$ be given the pseudo-ASL structure with $H = \{x,y\}$ and $\Sigma = \langle xy \rangle$. Then the standard monomials in $A$ are of the form $x^i$ or $y^i$ for all $i \geq 0$. As the defining ideal $I = \langle xy - (x^2 + y^2) \rangle$ is principal, $\{xy-(x^2+y^2)\}$ is already an (ordinary) universal Gröbner basis for this ideal, that is, regardless of term order. To see that the pseudo-ASL structure $(A,H,\Sigma)$ does not come from the construction in the proof above, we show that $xy$ is not the leading term of $xy - (x^2 + y^2)$ in all possible term orders. If $x \prec y$, then we get $x^2 \prec xy \prec y^2$ (the first by multiplying $x \prec y$ by $x$, and the second inequality by multiplying $x \prec y$ by $y$), and thus the leading monomial is $y^2$. By symmetry, if instead $y \prec x$, then $x^2$ is the leading monomial.
\end{example}

For a given finitely generated commutative algebra $A$ over a field $\F$, once a generating set $H \subseteq A$ is chosen, while there can be multiple choices for $\Sigma$, by \cite[Lemma~5.1.2]{bruns2022determinants} there can be at most finitely many.

\subsection{Relation to ASLs.} 
The terminology ``pseudo-ASL'' is motivated by the definition of algebra with straightening law (ASL, or Hodge algebra) \cite{deconcini1987hodge, eisenbud-intro-1980}. To make the relation clearer, we use our definition of pseudo-ASL in recalling the definition of ASL.

\begin{definition}[Algebra with straightening law / ASL / Hodge algebra]
\label{defn:hodge-algebra}
An \emph{algebra with straightening law} over a commutative ring $R$ is an $R$-pseudo-ASL on $(H,\Sigma)$, together with a partial ordering $\leq$ on $H$ such that 
for each generator $m \in \Sigma$, if
\begin{equation}
\label{eqn:hodge-straightening}
m = \sum_i {r_{i} m_i}, \qquad \text{with }0 \neq r_i \in R,
\end{equation}
is the unique expression for $\hat{\phi}(m)$ (that is, when $m$ is viewed as an element of $A$) as an $R$-linear combination of distinct standard monomials, then 
\[
x \in H, x \divides m \qquad \Longrightarrow \qquad \forall i, \;\exists y_{i} \in H \text{ such that }y_i \divides m_i \text{ and }y_i < x.
\]
\end{definition}

The latter condition is why those relations are referred to as ``straightening:'' it is making a monomial ``straighter'' (in some loose sense of the English word) with respect to the partial order.

\begin{convention}
Similar to Convention \ref{conv:pseudo-asl}, depending on whether we want to be fully explicit or not, we will either write $(A, R, H, \le, \Sigma, \phi)$ or some subset of it to denote an $R$-ASL.
\end{convention}

In our definition of pseudo-ASL, the reason we called it a ``pseudo-straightening'' law is that our ``straightening'' relations in a pseudo-ASL need not satisfy any such ordering property as in an ASL. By slight abuse of terminology, we refer to the relations \eqref{eqn:hodge-straightening} as straightening (or rewriting) relations, even when we are only in a psuedo-ASL. 

We note that the straightening relations are uniquely determined by the pseudo-ASL structure, so the existence of an ASL structure on a given pseudo-ASL is merely a question of whether any partial order on $H$ satisfies the condition of an ASL. The pseudo-ASL structure of Example~\ref{ex:pASL} is readily seen to not be an ASL.

Just as every finitely generated commutative algebra admits a pseudo-ASL structure, we recall (just for context) that if a finitely generated commutative algebra $A$ over a field $\F$ is furthermore $\N$-graded, then a nontrivial theorem of Hibi shows that it admits an ASL structure \cite{hibi}. Hibi also showed \cite{hibiSemigroup} that any sub-ring of a polynomial ring that is generated by monomials admits an ASL structure.

%% file: results.tex

\section{Main Results and Roadmap}
\label{sec:roadmap}

Because of the size of this paper, we thought it useful to provide a roadmap to the key definitions and main results. This is not meant to collect all results in the paper in one place, but rather to give an overall high-level view of where we are going and how the pieces fit together. Items (theorems, definitions, etc.) in this section are numbered according to the number of where they appear in the rest of the paper. 

\subsection{Pseudo-ASL \gb bases}
Our main goal is to develop a theory of \gb bases in commutative rings equipped with a notion of monomials and term order. While the definition of pseudo-ASL (Definition~\ref{defn:pseudo-ASL}) gives us a definition of standard monomial, we next need a notion of term order:

\getkeytheorem{notlmsm}

\getkeytheorem{deftermorder}

While the above definition may seem natural enough, we remark that there are some important subtleties to it, that seem needed to get the theory to work out, see Remark~\ref{rem:ato-2-equivalent}.

Because the product of two standard monomials isn't necessarily another standard monomial, there is ambiguity in how one could define divisibility among standard monomials in a pseudo-ASL. To clarify this issue, to any pseudo-ASL $A$, we associate another pseudo-ASL $A_{lt}$ which will govern what it means for one standard monomial to ``divide'' another:

\getkeytheorem{defalt}

Thus, unlike ordinary Gröbner bases, where all that is required is a choice of term order, for pseudo-ASL Gröbner bases we use both a choice of term order and a choice of $A_{lt}$. For any $A$, there are at least two algebras of leading terms, denoted $A_{gen}$ and $A_{disc}$ (see Definition \ref{defn:agen-adisc} and Proposition~\ref{prop:agen-adisc}), which have different tradeoffs in their properties. (And these two will be distinct unless $A = \bigslant{R[H]}{\langle \Sigma \rangle} = A_{disc}$.) These tradeoffs are another advantage of allowing a choice of $A_{lt}$, and we will see advantages of different choices of $A_{lt}$ throughout the paper. A simple example we can highlight now is that computing pseudo-ASL Gröbner bases relative to $A_{gen}$ is nearly identical to the case of polynomial rings when $A$ is a domain (Corollary~\ref{cor:AgenGB}), whereas in $A_{disc}$ divisibility of monomials is a simpler concept---in particular, $A_{disc}$ is a quotient of $R[H]$ by a monomial ideal---and we can take advantage of this in using pseudo-ASL Gröbner bases relative to $A_{disc}$ to compute Krull dimension (Section~\ref{sec:dimension}).

Using algebras of leading terms also resolves another issue: since products of standard monomials need not be standard monomials in a pseudo-ASL $A$, it is unclear what should play the role analogous to monomial ideals in the ordinary theory of Gröbner bases. To properly develop this theory, we show that algebras of leading terms \emph{do} have the property that a product of standard monomials is a scalar multiple of a standard monomial or zero; we give such pseudo-ASLs a name:

\getkeytheorem{defn-monomial-asl}

At this point we remark that Proposition \ref{prop:well-order} shows that a pseudo-ASL term order is in fact a well-order; the proof involves transferring such an ordering $\preceq$ from $A$ to $A_{lt}$, in order to apply an analogue of the usual argument of well-ordering of term orders, which we show works in any monomial pseudo-ASL.

Lemma \ref{lem:monomial} and Lemma \ref{lem:unique-min} show that standard monomial ideals in pseudo-ASLs behave similarly to monomial ideals in the polynomial ring. Thus, we define our leading ideals to live in $A_{lt}$, as it is a monomial pseudo-ASL, so the leading ideal will be a standard monomial ideal:

\getkeytheorem{defleadingideal}

Using Definition \ref{def:leading-ideal}, we now define the notion of a pseudo-ASL \gb basis.

\getkeytheorem{defaslgb}

As in the case of ordinary Gröbner bases, finite pseudo-ASL Gröbner bases always exist:

\getkeytheorem{thmaslgbexistence}

As in the ordinary case, we also have a notion of a pseudo-ASL Gröbner basis being reduced (Definition~\ref{defn:reduced-gb}), and reduced pseudo-ASL Gröbner bases are unique (Theorem~\ref{thm:reduced}).

Similar to the case of polynomial rings, we have a characterization of the definition of pseudo-ASL \gb bases in terms of divisibility of leading terms in $A_{lt}$. Given $(A, A_{lt})$, for two standard monomials $m,m' \in A$, we use the notation $m \ltdivides m'$ to mean $\pi_{lt}(m) \divides \pi_{lt}(m')$ (the latter in $A_{lt}$).

\getkeytheorem{propaslgrob-divcharacterization}

The above characterization allows us to actually use pseudo-ASL \gb bases for many of the same things that \gb bases in the polynomial ring are used for. For example, given an ideal $I$ in a pseudo-ASL $A$, and any element $f \in A$, a pseudo-ASL \gb basis relative to any $A_{lt}$ can be used to ascertain if $f \in I$.

\getkeytheorem{thmmacaulay}

Given a pseudo-ASL $A$, any pseudo-ASL \gb basis of an ideal $I$ in $A$ is always relative to a fixed algebra of leading terms $A_{lt}$ that is accompanying $A$. This means that there given two standard monomials $m$ and $m'$, $m$ could be divisible by $m'$ relative to one particular $A_{lt}$, and not divisible relative to another $A_{lt}$. To account for this, we introduce the notion of compatibility of standard monomials.

\getkeytheorem{defncompatible}

Many of our subsequent definitions and statements of results are similar to the case of polynomial rings, but with the assumption of compatibility added where needed. Some of the classical proofs carry through with ``compatible'' added in the appropriate places, while for some of the results the use of compatibility requires more work.

We next give our definition of standard expression, which takes the choice of $A_{lt}$ and compatibility into account:

\getkeytheorem{defnstandardexpression}

The main difference between Definition \ref{defn:standard-expression} and the analogous definition in the polynomial ring is that all terms of all of the $h_i$'s are compatible with the leading terms of the correponding $g_i$'s. Within the notion of standard expressions, we get uniqueness of remainder, as shown in Proposition \ref{prop:unique-remainder}, whenever $\{g_1, \ldots, g_k\}$ is a pseudo-ASL \gb basis of an ideal.

\subsubsection{Universal pseudo-ASL \gb Bases}

Next, we introduce the notion of universal \gb bases, and also prove existence. Note that here, in addition to being universal with respect to choice of term order, we also have universality with respect to choice of algebra of leading terms:

\getkeytheorem{defuniversalgb}

We prove that universal pseudo-ASL Gröbner bases as above always exist. The proof is slightly more involved than the case of polynomial rings. In the case of polynomial rings, one shows that a given ideal $I$ has at most finitely many leading ideals over all choices of term order; but in our setting, the leading ideals of $I$ can live in different algebras $A_{lt}$. We thus first relate pseudo-ASL Gröbner bases relative to any $A_{lt}$ to pseudo-ASL Gröbner bases relative to $A_{disc}$ (while $A_{gen}$ depends on the choice of term order, $A_{disc}$ does not, and is a valid algebra of leading terms for all term orders). Then we use an argument similar to the case of polynomial rings in $A_{disc}$.

\getkeytheorem{thmuniversal}

\subsection{Syzygies and pseudo-ASL \gb bases for modules}
Extending our theory to free modules over pseudo-ASLs yields several results, notably the computation of syzygies, and, in case $A$ is a domain (or more generally when the leading terms of the Gröbner basis are non-zerodivisors), an analogue of Buchberger's criterion and a more efficient algorithm for computing pseudo-ASL Gröbner bases. We extend our theory to free modules over a pseudo-ASL in Section~\ref{sec:sygyzies}.

In Definition \ref{def:termorder-modules}, we begin by defining the notion of a pseudo-ASL term order for free modules, and in Proposition \ref{prop:ato-are-well}, we demonstrate that such term orders are well orders. We note that the definition of pseudo-ASL term order for a free module again has some additional subtleties, even beyond those for the definition of pseudo-ASL term order itself; in particular, compare \ref{def:termorder-modules:mult} and \ref{def:termorder-modules:restrict} in Definition \ref{def:termorder-modules}.

\getkeytheorem{defmlt}

Observation \ref{obs:mlt} justifies the terminology in Definition~\ref{def:MLT}, by completing the analogy with the definition of algebra of leading terms for $A$ (Definition~\ref{def:ALT}).

\getkeytheorem{obsmlt}

We now come to our development of syzygies. In addition to syzygies coming from S-remainders (as in the case of polynomial rings), when $A_{lt}$ has standard monomial zerodivisors (which it does unless $A$ is a domain and $A_{lt} \cong A_{gen}$), we additionally have what we call ``compatibility syzygies'', since they turn a pair of incompatible monomials into a compatible expression:

\getkeytheorem{defncompatsyzygy}

Just as S-remainder syzygies correspond conceptually to the notion of S-closure, compatibility syzygies also have a corresponding notion of closure, which will lead us to our analogue of Buchberger's criterion, and (a bit later) streamlined algorithms. In the following definition, $LAM(m)$ is the set of ``least annihilating monomials'' of $m$, that is, those standard monomials $m'$ such that $m' m = 0$ and such that $m'$ is division-minimal (in $A_{lt}$) among all such $m'$.

\getkeytheorem{defannclosed}

\begin{remark} 
\label{rmk:compare-ann}
Compatibility syzygies and Ann-closure are similar in spirit to the annihilating syzygies in the theory of relative Gröbner bases \cite{HOS} as well as the so-called A-polynomials in the theory of ordinary Gröbner bases over a principal ideal ring \cite[Theorem~4.6]{NS}. The difference with the latter is that in our case, monomials can annihilate other monomials, whereas in their case annihilation only happens in terms of the coefficient ring $R$. We do not treat the case of non-field coefficient rings, but believe, in part because of the similarity with \cite[Theorem~4.6]{NS}, that our theory should easily extend to the case where $R$ is a principal ideal ring. In comparison to relative Gröbner bases \cite{HOS}, a starting difference is that they assume a Gröbner basis for the defining ideal $J$ such that $A = \bigslant{R[H]}{J}$, and this allows them to define A-syzygies only as those between standard monomials and elements of the Gröbner basis of $J$. In our setting, we want to avoid having to refer to (or know much about) $J$ at all, and so our notion of Ann-closure is slightly more complicated to state, but remains native to $A$.
\end{remark}

Before coming to our pASL analogue of Buchberger's Criterion, we need a few more definitions. In a polynomial ring, given two monomials, there is just one unique least common multiple (LCM). This no longer holds in general in the case of standard monomials in a pseudo-ASL. Thus in Definition \ref{defn:lcms}, given two standard monomials $m,m'$, we define their set of least common multiples as the division-minimal elements of $\langle m \rangle \cap \langle m' \rangle$ (note that the latter takes place in $A_{lt}$, which is a monomial pseudo-ASL, so standard monomial ideals behave similar to the case of polynomial rings). Consequently, given $f, f' \in A$, for each LCM of their leading terms, we get an S-pair, and thus we get a set (of size potentially larger than 1) of S-pairs even for a single choice of $f,f'$, unlike the case of a polynomial ring (Definition \ref{defn:sset}).

We generalize the notion of S-closure from polynomial rings to pseudo-ASLs with choice of term order and $A_{lt}$:

\getkeytheorem{defnsclosed}

Equipped with all this, we get our analogue of Buchberger's Criterion in pseudo-ASLs:

\getkeytheorem{thmaslbuchberger2}

We remark that the notion of Ann-closure is non-trivial for all algebras of leading terms, except possibly for $A_{gen}$. In this regard, it is worth noting Corollary \ref{cor:buchberger-agen} which says that if $A$ is a domain (or more generally if the leading terms of the elements of $G$ are not zerodivisors) and we choose $A_{gen}$ as our algebra of leading terms, S-closedness alone is necessary and sufficient for a set $G$ to be a pseudo-ASL \gb basis.

\newcommand{\lvsm}{\text{LV}_{sm}}

Next, we prove a pseudo-ASL analogue of Schreyer's Theorem on syzygies. The basic idea is that the S-remainder syzygies and a set of compatibility syzygies as in the definition of Ann-closure form a pseudo-ASL Gröbner bases for the syzygy module. To make this work, given a homomorphism of free modules $\varphi \colon A^d \to A^e$, and a pseudo-ASL term order $\preceq$ on the target $A^e$, we need to define what it means for a term order on $A^d$ to fit well with $\varphi$ and $\preceq$. For lack of a more descriptive word, we call such a term order ``good'' for $(\prec, \varphi)$ (Definition~\ref{def:good}).

When $A$ is a domain---or more generally when the leading terms of $\varphi(e_i)$ are non-zerodivisors---we show that good terms orders always exist (Proposition~\ref{obs:good_domain}).

We now come state our analogue of Schreyer's Theorem on syzygies. In the following, $\tau_{ij}^m$ denotes the S-remainder syzygy between $f_i$ and $f_j$ with respect to $m$, one of the LCMs of $\ltsm(f_i), \ltsm(f_j)$; and $\beta_i^s$ denotes the compatibility syzygy for $f_i$ with respect to the monomial $s$ that is incompatible with $\ltsm(f_i)$.

\getkeytheorem{thmaslschreyer}

\subsection{Algorithms}

In Proposition \ref{prop:div-alg}, we prove that there is an algorithm that performs division in any $A_{lt}$, and in Corollary \ref{cor:std-exp-alg}, given a pseudo-ASL $A$ and an algebra of leading terms $A_{lt}$, we show that for any $f, g_1, \ldots, g_k \in A$, there is an algorithm that finds a standard expression (see Definition \ref{defn:standard-expression}) for $f$ with respect to $(g_1, \dotsc, g_k)$. Also, in Lemma \ref{lem:comp-LCM} we show that the set of LCMs of a pair of standard monomials is computable in any $A_{lt}$, and in Lemma \ref{lem:comp-Ssm} we show that the S-set of a pair of elements of $A$ is computable. Using this, we establish the following theorem.

\getkeytheorem{thmcomputablesclosure}

Algorithm \ref{alg:asl-gb} computes a pseudo-ASL GB of any ideal relative to any algebra of leading terms (see Remark~\ref{rmk:oracle}). The algorithm requires computing the S-closure (done in Algorithm \ref{alg:s-closure}) and Ann-closure (Algorithm~\ref{alg:ann-closure}, Theorem \ref{thm:alg-ann-closure}). In the latter algorithm, we assume an oracle for testing whether a set is Ann-closed---this provides a clean presentation of the algorithm but appears to leave a gap in its computability. We show that this gap can be filled, establishing that pASL Gröbner bases are computable in finite time, although our algorithm relies on the parent polynomial ring (rather than working intrinsically in $A$) to test for termination (see Remark~\ref{rmk:oracle} for the proof and discussion). In Corollary \ref{cor:AgenGB}, we show that if the algebra of leading terms is a domain (which happens if and only if $A$ is a domain and $A_{lt} \cong A_{gen}$), then Ann-closure vacuously holds, so Algorithm~\ref{alg:s-closure} is sufficient to compute pseudo-ASL \gb bases. In Remark~\ref{rmk:oracle} we also sketch how to compute Ann-closure (hence pASL \gb bases) whenever the ideal under consideration admits a finite free resolution.

\subsection{Application to bideterminants}

We apply our theory of \gb bases in pseudo-ASLs to the case of bideterminants in Section \ref{sec:bideterminants}.

Section \ref{sec:bidet-background} defines bitableaux, bideterminants, and related notions, following the classical references --- \cite{doubilet1974foundations, desarmenien1982invariant, grinberginvariant}. In Theorem \ref{thm:bd-hodge-algebra} we recall how a (pseudo-)ASL structure can be imposed on the polynomial ring in the basis of standard bideterminants.

Our notion of ordering of bitableaux/bideterminants differs slightly from what is found in classical works, primarily by: (1) first sorting by degree (unnecessary in the classical works on the straightening law because in any one application of the straightening law, all terms have the same degree), (2) the cosmetic change that our tableaux are the transpose of those found in the literature, and (3) within degree, reversing the order, so that what is called in the classical literature ``longer shape'' in our sense corresponds to ``lower-order terms.''

\getkeytheorem{defshapeorder}

\getkeytheorem{defnorderbitab}

Theorem \ref{thm:straightening} and Theorem \ref{thm:leading-bd} record known facts about products of standard bideterminants. Using these, and Definition \ref{defn:order-bitab}, which is what gives a pseudo-ASL term order, we show that there is a theory of (pseudo-ASL) \gb bases in the basis of standard bideterminants in the polynomial ring.

\getkeytheorem{thmbdgrobnerbasis}

The algorithms in Section \ref{sec:algorithms} are written for any pseudo-ASL, and specifically, Lemma~\ref{lem:comp-LCM} shows that in any pseudo-ASL, LCMs relative to any algebra of leading terms are computable in finite time. However, to take advantage of the structure in the case of bideterminants, we present Algorithm \ref{alg:lcsm} which is tailored to the specific case finding LCMs of standard bideterminants relative to $A_{gen}^{bd}$. Not only is this more efficient than the generic one of Lemma~\ref{lem:comp-LCM}, it also more illuminating of the combinatorial aspects of LCMs in the bideterminant setting.

\getkeytheorem{thmlcmalgo}

Finally, we consider classical determinantal ideals and obtain their universal bd-\gb bases.

\getkeytheorem{thmtminorsbdgb}

\getkeytheorem{thmmaximalminorsbdgb}

\subsection{Hilbert--Poincar\'e Series and Krull Dimension in pseudo-ASLs}

In Section \ref{sec:dimension}, we show that when our pseudo-ASL is $\N$-graded, we can use pseudo-ASL \gb bases, relative to any $A_{lt}$, to compute dimension of quotients of pseudo-ASLs using an algorithm (Algorithm \ref{alg:hilbert-clo}) similar to that in the ordinary case of polynomial rings. The crucial step here is to show that one can indeed work in any $A_{lt}$. If $S$ is a graded ring, we use $H_S$ to denote the Hilbert--Poincaré series of $S$.

\getkeytheorem{thmhilbertaslalt}

Finally, using Theorem \ref{thm:hilbert-asl-alt}, in Section \ref{sec:krull-rank-1} we use our machinery to compute the dimension of the affine variety of rank $1$ $n \times m$ matrices taken as elements of $\C^{n \times m}$. Of course the result is entirely standard; the point here is to exhibit how our techniques can be used in a nontrivial example calculation. The computation crucially operates in $A^{bd}_{disc}$, in order to reduce to a calculation about ordinary monomial ideals, showing again the value of developing a theory of \gb bases for pseudo-ASLs relative to any $A_{lt}$.

%% file: theory.tex
\newcommand{\qn}{0}

\section{Theory}
\label{sec:theory}

\subsection{Core definitions and results}

In a pseudo-ASL $A$, we say that a standard monomial $m$ \emph{appears} in $f \in A$ if $m$ occurs with nonzero coefficient in the unique expression of $f$ as a linear combination of standard monomials (like in expression \eqref{eqn:hodge-straightening}). 

\begin{notation}[store=notlmsm, note={Leading monomial}]
Whenever there is a total order on the standard monomials of $A$, let $\lmsm(f)$ denote the largest standard monomial that appears in $f$. Throughout the paper, we use the subscript $_{sm}$ to emphasize that we are working the setting of standard monomials in a pseudo-ASL.
\end{notation}

\subsubsection{Pseudo-ASL term orders}

\begin{definition}[store=deftermorder, note=Pseudo-ASL term order]
\label{def:termorder}
Given a pseudo-ASL $A$, a \emph{pseudo-ASL term order} on $A$ is total ordering $\preceq$ on the standard monomials of $A$ such that 

\begin{enumerate}[label=(\textbf{ATO-\arabic*}),ref=(\textbf{ATO-\arabic*})]
\item \label{def:termorder:positive} (Positivity) $1 \preceq m$ for all standard monomials $m \in A$.
\item \label{def:termorder:mult} (Multiplicativity in leading terms when nonzero) For all standard monomials $f,g,h,k \in A$, if $f \prec g$ and $h \preceq k$, if $fh \neq 0$ and $gk \neq 0$, then
\[
\lmsm(fh) \prec \lmsm(gk).
\]
\end{enumerate}
\end{definition}

\begin{remark}
\label{rem:ato-2-equivalent}
We make two remarks about \ref{def:termorder:mult}. The first is that it is phrased as it is so that it continues to work even when $A$ is not a domain, and especially continues to work in the algebras of leading terms that we will introduce below. In particular, it is designed so that Proposition~\ref{prop:mult} goes through; to understand some of the reason for the definition, it is instructive to see where the proof of Proposition~\ref{prop:mult} breaks down with alternative definitions. In a domain, \ref{def:termorder:mult} is equivalent to: 
\begin{equation} \label{eq:ATO2equiv}
\text{for all standard monomials } f,g,h, \quad f \prec g \Rightarrow \lmsm(fh) \prec \lmsm(gh). 
\end{equation}
We show how to get \ref{def:termorder:mult} from this: if $f \prec g$ and $h \preceq k$, then by \eqref{eq:ATO2equiv}, we get $\lmsm(fh) \prec \lmsm(gh)$. If $h=k$, then we are done; otherwise we have $h \prec k$, so we may apply \eqref{eq:ATO2equiv} with $h,k$ in the place of $f,g$, resp., and $g$ in the place of $h$, to get $\lmsm(hg) \prec \lmsm(kg)$. Combining with the previous inequality, we then get $\lmsm(fh) \prec \lmsm(gk)$, as desired. 

Note that in \ref{def:termorder:mult}, $f \prec g$ is a strict inequality, as is the inequality in the conclusion. This is actually relevant for avoiding certain difficulties in the theory. If one weakens the assumption to $f \preceq g$, then because of the case $f=g$ one must also weaken the conclusion to $\lmsm(fh) \preceq \lmsm(gk)$. But, when $f \prec g$ (strictly), if we still allow $\lmsm(fh) = \lmsm(gk)$, we ran into trouble.
\end{remark}

We will see in Proposition~\ref{prop:well-order} below that pseudo-ASL term orders are well-orders, but the proof involves somewhat more machinery than the usual case, because of the need to use $\lmsm$ in \ref{def:termorder:mult} (because after straightening, $fh$ and $gk$ need not be single terms any more).

Once a pseudo-ASL term order $\preceq$ is fixed, we define $\lcsm(f)$ to be the coefficient of $\lmsm(f)$ in the unique expression of $f$ as a linear combination of standard monomials, and we define the leading term $\ltsm(f) := \lcsm(f) \lmsm(f)$. Whenever we use $\ltsm, \lcsm, \lmsm$, it is understood that there is a fixed, if unspecified, pseudo-ASL term order $\preceq$ throughout.

\begin{remark}
When $A$ is in fact an ASL, it is natural to wonder in what way a pseudo-ASL term order is ``compatible'' with the ASL structure. Any such compatibility that we require with the straightening process is encapsulated in \ref{def:termorder:mult}. Intuitively, the idea is that ``straightening results in terms that are only lower in the order'';\footnote{Indeed, such an idea is similar to the approach taken via relative Gröbner bases \cite{HOS}. In a graded ASL $A = \bigslant{R[H]}{J}$, there is always an ordinary term order on $R[H]$ such that the left-hand-side of a straightening rule is the leading term, namely, the degree lexicographic order built from any total order on $H$ extending the partial order $<$ on $H$ that is part of the ASL structure; in non-graded ASLs there may fail to exist any such order \cite{GrabePauer}. 

We note that the existence of such a term order on $R[H]$ need not imply the existence of a pseudo-ASL term order, as in Examples~\ref{ex:no-term-order} and \ref{ex:no-term-order-2}. In the other direction, we do not know that a pseudo-ASL term order extends to a(ny) term order on $R[H]$, let alone one satisfying the above property. See the related Open Question~\ref{q:orders}.}
making this precise in our setting is tricky, however, as the pseudo-ASL term order is only defined on standard monomials, but the LHS of a straightening rule is (by definition) a non-standard monomial. Nonetheless, \ref{def:termorder:mult} has a similar effect.
\end{remark}

In the theory of Gröbner bases in polynomial rings, one has the convenient property that $\lt(fg) = \lt(f) \lt(g)$. The straightening law in a pseudo-ASL prevents this property from being true as-is, but when a pseudo-ASL admits a pseudo-ASL term order, we get a closely related property that will be good enough for our purposes:

\begin{proposition} \label{prop:mult} Let $A$ be a pseudo-ASL with a pseudo-ASL term order $\preceq$. For all $f, g \in A$, such that $\ltsm(f) \ltsm(g) \neq 0$, we have
\begin{equation}
\label{obs:mult-eqn1}
\ltsm(fg) = \ltsm(\ltsm(f) \ltsm(g)) = \ltsm(f \ltsm(g)).
\end{equation}
Also,
\begin{equation}
\label{obs:mult-eqn2}
\lmsm(fg) = \lmsm(\lmsm(f) \lmsm(g)) = \lmsm(f \lmsm(g)).
\end{equation}

Furthermore, for all $f,g,h \in A$ such that $\ltsm(f) \ltsm(g) \ltsm(h) \neq 0$, we have 
\begin{equation} \label{obs:mult-eqn3}
\ltsm(\ltsm(fg)\ltsm(h)) = \ltsm(\ltsm(f) \ltsm(g) \ltsm(h)).
\end{equation}
\end{proposition}

\begin{proof}
We begin by proving the first equality of \eqref{obs:mult-eqn1}; the second equality then follows immediately by setting $g = \ltsm(g)$ in the first equality and from the fact that $\ltsm(\ltsm(g)) = \ltsm(g)$.

Suppose that for $r_i, r'_i \in R$ nonzero, $m_i, m'_i \in A$ standard monomials, we have 
\[
f = \sum_{i \in [k]} r_i m_i \;\text{ and }\; g = \sum_{j \in [k']} r'_j m'_j, \qquad \text{with }m_1 \succ \dotsb \succ m_k \;\text{ and }\; m'_1 \succ \dotsb \succ m'_{k'}.
\] 
Then 
\[
fg = \sum_{i \in [k], j \in [k']} r_i r'_j m_i m'_j.
\] 
By assumption we have $r_1 r'_1 m_1 m_1' \neq 0$, so in particular $m_1 m'_1 \neq 0$. Similarly, for any $i \in [k]$ and $j \in [k']$ if $r_i r'_j m_i m_j \neq 0$, then $m_i m'_j \neq 0$. Without loss of generality, let $i > 1$, and then since $m_i \prec m_1$ (strict inequality) and $m'_j \preceq m'_1$, \ref{def:termorder:mult} gives us that $\ltsm(r_i r'_j m_i m'_j) \prec \ltsm(r_1 r'_1 m_1 m'_1)$.

Equation~\eqref{obs:mult-eqn2} then follows by replacing all nonzero coefficients by 1 in Equation \eqref{obs:mult-eqn1}.

To prove \eqref{obs:mult-eqn3}, let $f,g,h \in A$ be such that $\ltsm(f) \ltsm(g) \ltsm(h) \neq 0$.
Then we have
\begin{align*}
& \ltsm(\ltsm(fg) \ltsm(h)) \nonumber \\
& = \ltsm\left(\ltsm\left(\ltsm(f)\ltsm(g)\right) \ltsm(h)\right) \eqcomment{applying \eqref{obs:mult-eqn1} to $f,g$} \\
 & = \ltsm\left(\ltsm(f)\ltsm(g) \ltsm(h)\right),
\end{align*}
where the final equality here comes from applying \eqref{obs:mult-eqn1} with $\ltsm(f)\ltsm(g)$ in the preceding equation playing the role of $f$ in \eqref{obs:mult-eqn1}, and $\ltsm(h)$ playing the role of $g$.
\end{proof}

Next we observe that if a graded pseudo-ASL admits a term order, then it admits a graded term order. Let us now define carefully what all these terms mean and see why this should be true. This result is essentially one of convenience, it is not crucial for the development of the general theory. 

\begin{definition}[Graded pseudo-ASL, graded pseudo-ASL term order] \label{defn:graded-asl}
Given a commutative monoid $(\mathcal{M},+,0)$, we say that an \emph{algebra $A$ is graded by $\mathcal{M}$} if there are $R$-submodules $A_m$ (not necessarily nonzero) for each $m \in \mathcal{M}$ such that $A = \bigoplus_{m \in \mathcal{M}} A_m$, $A_0 \cong R$ contains the identity of $A$, and $A_m \cdot A_{m'} \subseteq A_{m+m'}$ for all $m,m' \in \mathcal{M}$. We call elements $x \in A_m$ \emph{homogeneous of degree $m$}, and write $\deg(x) = m$ for such elements. 

We say a \emph{pseudo-ASL $A$ is graded by $\mathcal{M}$} if, furthermore, all standard monomials are homogeneous. 

A well-ordered commutative monoid is a commutative monoid $\mathcal{M}$ together with a total ordering $<$ such that $0 \leq a$ for all $a$, and $a < b$ and $c \leq d$ implies $a+c < b+d$. If $\mathcal{M}$ is a (well)-ordered commutative monoid and $A$ is an $\mathcal{M}$-graded pseudo-ASL, then a \emph{graded pseudo-ASL term order} on $A$ is a pseudo-ASL term order $\preceq$ that refines the partial ordering by degree, that is, if $\deg(x) < \deg(y)$, then $x \prec y$. 
\end{definition}

\begin{observation} \label{obs:graded_order}
Suppose $A$ is a pseudo-ASL graded by a well-ordered commutative monoid, and admitting a pseudo-ASL term order $\preceq$. Then there is a graded pseudo-ASL term order on $A$.
\end{observation}

\begin{proof}
Define $\preceq'$ by $m \preceq' m'$ if $\deg(m) < \deg(m')$ or if $\deg(m) = \deg(m')$ and $m \preceq m'$. By construction, $\preceq'$ is a total order that refines the ordering by degree. We show that it is in fact a graded \emph{pseudo-ASL term} order.

By definition of graded pseudo-ASL, we have $\deg(1) = 0$, and $1$ is the only standard monomial of degree 0, so positivity is readily seen to hold, since $0 \leq m$ for all $m \in \mathcal{M}$. Now suppose $f,g,h,k$ are standard monomials with $f \prec' g$ and $h \preceq' k$ such that $fh \neq 0$ and $gk \neq 0$. If $\deg(f) < \deg(g)$, then \[\deg(fh) = \deg(f) + \deg(h) < \deg(g) + \deg(k) = \deg(gk),\] so we have $\lmsm(fh) \prec' \lmsm(gk)$ according to their degree. 

By similar analysis, if instead $\deg(h) < \deg(k)$, we'd also have that $\lmsm(fh) \prec' \lmsm(gk)$. 

On the other hand, if $\deg(f) = \deg(g)$ and $\deg(h) = \deg(k)$, then we have $\lmsm(fh) \prec' \lmsm(gk)$ since they have the same degree and $\lmsm(fh) \prec \lmsm(gk)$, since $\prec$ was a term order.
\end{proof}

\begin{remark}
We attempted to generalize Observation~\ref{obs:graded_order} to filtered pseudo-ASLs, rather than graded. However, the proof breaks down at the point where we have written $\deg(g) + \deg(k) = \deg(gk)$. In a filtered ring, we have $\deg(g) + \deg(k) \geq \deg(gk)$, but that is the wrong direction for use in that part of the argument.
\end{remark}

Not all pseudo-ASLs, nor even all ASLs, admit pseudo-ASL term orders, as the next two examples show. These ASLs are in fact both $\N$-graded (by the usual notion of degree), monomial ASLs in the sense of Section~\ref{sec:monomial}, and in both cases $R[x,y]$ admits term orders (e.g., deglex) relative to which the non-standard monomial is the leading term of the straightening laws (as discussed in \cite{GrabePauer}).

\begin{example}[ASL not admitting a pseudo-ASL term order] \label{ex:no-term-order}
Let $A = \bigslant{R[x,y]}{\langle xy - x^2 \rangle}$ with the ASL structure given by $H = \{x,y\}, \Sigma = \langle x^2 \rangle$ and the order on $H$ given by $y < x$. In this ASL structure, the standard monomials are $\{y^j : j \geq 0\} \cup \{x y^j : j \geq 0\}$. Note that in $A$, $x$ is annihilated by $y-x$, but neither $y$ nor $x$ annihilates $x$ on its own. Suppose $\preceq$ is a pseudo-ASL term order on $(A, H, \Sigma)$. Since $\preceq$ is total, we must have either $x \prec y$ or $y \prec x$. If $x \prec y$, then multiplying both sides by $x$, we get that since both $x^2$ and $xy$ are nonzero, \ref{def:termorder:mult} implies the contradiction
\[
xy = \lmsm(x^2) \prec \lmsm(xy) = xy,
\]
since in $A$ we have $x^2 = xy$ and $xy$ is standard while $x^2 \in \Sigma$ is not. Similarly if we instead have $y \prec x$. (This argument is later generalized in Observation~\ref{obs:no-binomial-ann}.)
\end{example}

\begin{example}[Another ASL not admitting a pseudo-ASL term order, but not because of the existence of proper binomial annihilators] \label{ex:no-term-order-2}
Let $A = \bigslant{R[x,y]}{\langle x^2 - y^2 \rangle}$, with the same $H,\Sigma$ as in the preceding example. Suppose to the contrary that $A$ admits a term order $\prec$. If $x \prec y$, then when we multiply both sides by $x$ we get $x^2 = y^2$ on the left and $xy$ on the right, so we must have $y^2 \prec xy$. But if we instead multiply both sides by $y$, we get $xy \prec y^2$, a contradiction. A similar contradiction arises in essentially the same way if instead $y \prec x$. (To see that Observation~\ref{obs:no-binomial-ann} is not the reason for the lack of term order here, one can check that all standard monomials are non-zerodivisors, so they have trivial annihilators. The standard monomials are $\{y^j : j \geq 0\} \cup \{x y^j : j \geq 0\}$, and the product on these standard monomials agrees with their product in $R[x,y]$ except in the case $x y^j \cdot x y^k = y^{j+k+2}$.)
\end{example}

Finally, we note that if $A$ admits a pseudo-ASL term order, then it is a domain iff all products of pairs of standard monomials are nonzero:

\begin{proposition} \label{prop:domain}
Suppose $A$ is a pseudo-ASL that admits a pseudo-ASL term order $\preceq$. Then $A$ is a domain if and only if for all standard monomials $m,m' \in A$, we have $mm' \neq 0$.
\end{proposition}

\begin{proof}
If $A$ is a domain, then certainly all products of standard monomials are nonzero.

Conversely, suppose that for all standard monomials $m,m'$, we have $mm' \neq 0$. We will show that $A$ is a domain. Let $f = \sum_i r_i m_i$ and $g = \sum_j r'_j m'_j$ where $m_1 \succ m_2 \succ \dotsb$ and $m'_1 \succ m'_2 \succ \dotsb$ are standard monomials, and all $r_i,r'_j$ are nonzero elements of $R$. We show that $fg \neq 0$. By assumption $m_1 m'_1 \neq 0$. Since all products of pairs of standard monomials are nonzero, it follows from \ref{def:termorder:mult} that $\lmsm(m_1 m'_1) \succ \lmsm(m_i m'_j)$ whenever $i > 1$ or $j > 1$ (or both). Thus, in the product $fg = \sum_{i,j} r_i r'_j m_i m_j$, no term can cancel $\lmsm(m_1 m'_1)$, so the product $fg$ is nonzero. Thus $A$ is a domain.
\end{proof}

\subsubsection{Monomial pseudo-ASLs and standard monomial ideals} \label{sec:monomial}

We define a \emph{standard term} to be an $R$-multiple of a standard monomial. 

\begin{definition}[store=defn-monomial-asl, note={Monomial pseudo-ASL}]
We say a pseudo-ASL $A'$ is a \emph{monomial} pseudo-ASL if the product of any two standard monomials in $A'$ is either a standard term or 0.
\end{definition}

\begin{remark}
In the case of bideterminants (see Section \ref{sec:bideterminants} for definitions), Desarmenien showed that for all standard monomials $m_1, m_2$, it is true that $\lcsm(m_1 m_2) = 1$  \cite{desarmenien1980algorithm}. However, this is not true in an arbitrary pseudo-ASL---for example, a slight modification of Example~\ref{ex:no-term-order} to $\bigslant{R[x,y]}{\langle 2xy - x^2 \rangle}$---so we use standard \emph{terms} in the above definition rather than standard monomials.
\end{remark}

\begin{convention}
Throughout Section~\ref{sec:monomial}, we use $A'$ to denote a (monomial) pseudo-ASL. This is to highlight the fact that when we return to our theory of pseudo-ASL Gröbner bases, the original pseudo-ASL $A$ will typically \emph{not} be monomial, but in the next section we will build an auxiliary monomial pseudo-ASL related to $A$; the latter will eventually be denoted $A_{lt}$, but the results of this section apply to arbitrary monomial pseudo-ASLs, not only the ones that occur in the aforementioned auxiliary construction.
\end{convention}

Monomial pseudo-ASLs are graded pseudo-ASLs, graded by the ``monoid of standard monomials'', as detailed in the next remark.

\begin{remark}[{Monomial pseudo-ASLs are twisted monoid algebras with zero; cf. \cite[Prop.~1.11]{ES}}] \label{rmk:semigroup}
Given a monomial pseudo-ASL $A'$, 
we may define a monoid $M$ whose elements are the standard monomials of $A'$ together with 0, and with multiplication in $M$ defined by $m_1 \cdot_M m_2 := \lmsm(m_1 m_2)$ (where multiplication on the right-hand side here occurs in $A'$, and we identify $0 \in A'$ with $0 \in M$). Then the fact that $A'$ is monomial is equivalent to saying that $A'$ is a twisted monoid $R$-algebra of $M$ in which the zero of $M$ is identified with $0 \in A'$. The ``twisting'' here is a function $\tau \colon M \times M \to R$, defined by $m_1 m_2 = \tau(m_1, m_2) \lmsm(m_1 m_2)$; that is, $\tau(m_1, m_2) = \lcsm(m_1 m_2)$. As usual, because $A'$,  hence $M$, is associative, $\tau$ is in fact a 2-cocycle of $M$ with coefficients in the multiplicative semigroup of $R$, that is, for all $m_1, m_2, m_3 \in M$,
\[
\tau(m_1, m_2) \tau(m_1 m_2, m_3) = \tau(m_1, m_2 m_3) \tau(m_2, m_3).
\]
Note that $M$ is generated by the set of generators $H$ of the pseudo-ASL, so in particular is finitely generated.
\end{remark}

In monomial pseudo-ASLs, monomial ideals behave similar to monomial ideals in polynomial rings, as encapsulated in the following lemma. 

\begin{lemma}[store=lemmonomial, note=Monomial ideals in monomial pseudo-ASLs] \label{lem:monomial}
Let $A'$ be a monomial pseudo-ASL. For an ideal $I$ in $A'$, the following are equivalent:
\begin{enumerate}[label=(\Roman*),ref=(\Roman*)]
\item \label{point:gen-by-standard} $I$ is generated by standard monomials.
\item \label{point:standard-monomial-closure} For any $f \in A'$, $f$ is in $I$ if and only if every standard monomial appearing in $f$ is in $I$.
\end{enumerate}

Furthermore, when one (equivalently, both) of these conditions hold, if $G$ is any generating set of $I$ consisting only of standard monomials, then a standard monomial $m$ is in $I$ if and only if it is divisible by some element of $G$.
\end{lemma}

The proof of this lemma follows the standard proof very closely, \emph{mutatis mutandis}; we put it in Appendix~\ref{sec:app:theory} for completeness. Given $(A',H)$, one might worry that different choices of pseudo-ASL structure $\Sigma$ could lead to different sets of standard monomial ideals; we will show in Observation~\ref{lem:standard-all-sigma} below that in fact this is not the case, and the notion of ``standard monomial ideal'' in a monomial pseudo-ASL $(A', H, \Sigma)$ only depends on $(A', H)$, but not on $\Sigma$.

\begin{definition}[Standard monomial ideal]
We call an ideal in a monomial pseudo-ASL satisfying either (hence both) of the conditions of Lemma~\ref{lem:monomial} a \emph{standard monomial ideal}.
\end{definition}

\begin{lemma} \label{lem:unique-min}
Suppose $A'$ is a monomial pseudo-ASL and $I \subseteq A'$ is a standard monomial ideal. Let $M$ be the set of standard monomials in $I$ that are minimal under divisibility, among the standard monomials in $I$. Then $M$ is finite, and any set $M'$ of standard monomials generates $I$ if and only if $M \subseteq M'$.
\end{lemma}

This result and its proof are quite similar to Green's \cite[Proposition 2.5]{green}. There are some differences, however, so we include the full proof here. The key difference is that Green assumes an admissible term order (in his terminology), which prevents the existence of an infinite descending chain of elements, each divisible by the next, whereas we use Noetherianity of $A'$ to get the latter property. Green also assumes a strictly multiplicative basis, in which the product of two basis elements is zero or a basis element, whereas we allow the product of two standard monomials to be a scalar multiple of a standard monomial, but this change has little effect on the proof.

\begin{proof}
Because $A'$ is Noetherian, there are no infinite descending chains under the divisibility ordering, and thus by Lemma~\ref{lem:monomial}, every standard monomial in $I$ is divisible by some element of $M$. Thus $M$ generates $I$, as does any superset of $M$.

Also by Noetherianity, $M$ is finite: because no element of $M$ divides any other, and by Lemma~\ref{lem:monomial}, each $m \in M$ is not in $\langle M \backslash \{m\} \rangle$. So if $M$ were infinite, then a sequence of larger and larger finite subsets of $M$ would generate a strictly increasing nested sequence of ideals. Thus $M$ is finite.

Finally, suppose that $M'$ is a set of standard monomials that generates $I$. Let $m \in M$. Since $m$ is in $I$, by Lemma~\ref{lem:monomial}, there must be some element $m' \in M'$ such that $m' \divides m$. But by minimality of $m$ under divisibility, we must have $m' = m$, and thus $m$ must be in $M'$. As this is true for all $m \in M$, we have $M \subseteq M'$. This completes the proof.
\end{proof}

If $A'$ is a monomial pseudo-ASL on the generating set $H$, then it is the quotient of the polynomial ring $R[H]$ by a binomial ideal, i.\,e., an ideal generated by binomials and monomials. Thus many of the results of, e.g., Eisenbud \& Sturmfels \cite{ES} on binomial ideals apply; indeed, this is one of the benefits of the auxiliary algebra $A_{lt}$ being a monomial pseudo-ASL. However, while many of their proofs often work \emph{mutatis mutandis} in our setting, we will see that many of the results can in fact be applied directly. The reason this is not immediately obvious is that the notion of ``monomial'' in the setting of binomial ideals---the usual notion of monomial in a (parent) polynomial ring---need not agree with the standard monomials in a given pseudo-ASL that is a quotient of a polynomial ring by a binomial ideal. If the choice of $\Sigma$ in the pseudo-ASL structure agrees with the leading ideal of an (ordinary) Gröbner basis of the defining binomial ideal relative to some term order on $R[H]$, then the two notions will agree, but not necessarily otherwise. However, the following observation lets us directly apply most of the results from their paper that we want, without having to dig into their proofs (not even \emph{mutatis mutandis}):

\begin{observation} \label{lem:standard-all-sigma}
Let $(A',H,\Sigma)$ be a monomial pseudo-ASL
and $m \in \W^H$ be a monomial such that $m$ is standard relative to $\Sigma$ (that is, $m \notin \Sigma$). For all $\Sigma'$ such that $(A', H, \Sigma')$ is a monomial pseudo-ASL,
there is a monomial $m' \in \W^H$ that is standard relative to $\Sigma'$, and $r \in R$ such that $m=rm'$ in $A'$.

In particular, if $I \subseteq A'$ is a standard monomial ideal relative to $\Sigma$, then it is a standard monomial ideal relative to $\Sigma'$.
\end{observation}

\begin{proof}
Since $H \cap \Sigma' = \emptyset$ by the definition of pseudo-ASL, and $m$ is standard relative to $\Sigma$, $m$ is a product of elements of $H$. In particular, this means that $m$ is a product of standard monomials relative to $\Sigma'$. Since $(A',H,\Sigma')$ is a monomial pseudo-ASL by assumption, this product of standard monomials is a standard term relative to $\Sigma'$, say $rm'$ with $r \in R$ and $m'$ a standard monomial relative to $\Sigma'$. Thus if $I$ is generated by monomials $m_1, \dotsc, m_k$ that are standard relative to $\Sigma$, each $m_i$ is equal in $A'$ to some standard-relative-to-$\Sigma'$ term $r_i m'_i$, and thus $m'_1, \dotsc, m'_k$ are standard monomials relative to $\Sigma'$ that generate $I$ as well.
\end{proof}

The following corollary will be useful in defining LCMs when we get to pseudo-ASL Gröbner bases.

\begin{corollary}[{cf. \cite[Cor~1.6(a)]{ES}}] \label{cor:intersection-monomial}
Let $A'$ be a monomial pseudo-ASL. The intersection of two standard monomial ideals is a standard monomial ideal.
\end{corollary}

We give two proofs of this result; one for simplicity that will extend \emph{mutatis mutandis} to the case of monomial submodules of $(A')^d$, and one that illustrates the use of Observation~\ref{lem:standard-all-sigma}.

\begin{proof}[Proof that generalizes to monomial submodules]
Let $I,I'$ be two standard monomial ideals. Let $f \in I \cap I'$. Then since $I$ is standard monomial, every standard monomial $m_i$ that occurs in $f$ with nonzero coefficient occurs in $I$; similarly, $m_i$ occurs in $I'$. Thus every standard monomial in $f$ is in $I \cap I'$, and therefore $I \cap I'$ is a standard monomial ideal.
\end{proof}

\begin{proof}[Proof using Observation~\ref{lem:standard-all-sigma}]
Let $(A',H,\Sigma)$ be the given pseudo-ASL structure, and $I,J \subseteq A'$ two standard monomial ideals. Let $K$ be the defining ideal of $A'$, that is $A' = \bigslant{R[H]}{K}$; since $A'$ is a monomial pseudo-ASL, $K$ is a binomial ideal. Let $\hat{I}, \hat{J}$ be, respectively, the preimages of $I$ (resp., $J$) in $R[H]$ (that is, $\hat{I}$ is ``essentially'' $I+K$). Cor 1.6(a) of Eisenbud \& Sturmfels \cite{ES} then implies that $\hat{I} \cap \hat{J}$ is generated by a set $M$ of monomials in $R[H]$ along with the ideal $K$. Choosing a term order on $R[H]$, we may let $\Sigma'$ be the leading ideal of $K$ relative to that term order; we may assume without loss of generality that none of the elements of $M$ are in $K$ (since if they were we could simply omit them from the generating set $M$), and hence none of them are in $\Sigma'$. Then $(A',H,\Sigma')$ gives another pseudo-ASL structure on $A'$, and we find that $I \cap J$ is a standard monomial ideal relative to $\Sigma'$. But then by Observation~\ref{lem:standard-all-sigma}, $I \cap J$ is also a standard monomial ideal relative to the originally given pseudo-ASL structure $(A',H,\Sigma)$.
\end{proof}

Our last two results in this section are about monomial pseudo-ASLs in the presence of a pseudo-ASL term order.

\begin{lemma} \label{lem:mon-div}
Let $A'$ be a monomial pseudo-ASL. If $t,t'$ are standard terms such that $t \mid t'$, then there is a standard term $t''$ such that $tt'' = t'$.

If, furthermore, $A'$ admits a pseudo-ASL term order $\preceq$, then the standard term as above is unique.
\end{lemma}

\begin{proof}
By assumption, there is some $f \in A'$ such that $tf=t'$. Write $f = \sum_i r_i m_i$, where each $r_i \in R$ and each $m_i$ is a standard monomial. Since a product of standard monomials is a standard term or 0, we have $ft = \sum_i r_i m_i t$, where each $m_i t$ is either a standard term or 0. But since $t'=ft$, the standard monomial underlying $t'$ must appear as a nonzero multiple of at least one element of the set $\{m_i t\}$. If $r t' = r_i m_i t$, then we have $t' = (\nicefrac{r_i}{r}) m_i t$, and thus the standard term $t'' = (\nicefrac{r_i}{r}) m_i$ is such that $tt'' = t'$.

For the furthermore, now suppose that $A'$ admits a pseudo-ASL term order. Suppose for the sake of contradiction that $t,t'$ are two distinct standard terms such that $mt = mt' = m'$. If $\lmsm(t) = \lmsm(t')$, then $(t-t')m = 0$ is a relation of the form $r\lmsm(t) m = 0$ with $r$ a nonzero element of $R$, which is impossible, because $mt = m' \neq 0$ is an $R$-scalar multiple of $\lmsm(t) m$. Thus $t,t'$ are terms on two distinct standard monomials. Without loss of generality, suppose $\lmsm(t) \prec \lmsm(t')$. Since $mt = mt' = m' \neq 0$, \ref{def:termorder:mult} gives us that $\lmsm(mt) \prec \lmsm(mt')$, but this is again a contradiction as both sides here are equal to $m'$. 
\end{proof}

\begin{lemma} \label{lem:div}
Suppose $A'$ is a monomial pseudo-ASL, and suppose $\preceq$ is a pseudo-ASL term order on $A'$. If $x,y \in A'$ are two distinct standard monomials such that $x \divides y$, then $x \preceq y$.
\end{lemma}

\begin{proof}
If $x = y$ we are done. Otherwise, let $r \in R$ and $m \neq 1$ be a standard monomial such that $xrm= y$, which exists by Lemma~\ref{lem:mon-div}. We have $1 \prec m$ by \ref{def:termorder:positive}. Further, we have $xm \neq 0$, since $xrm = y \neq 0$. Thus \ref{def:termorder:mult} applies, giving us 
\[
\lmsm(1 \cdot x) \prec \lmsm(m \cdot x). 
\]
Now, since $x$ is a standard monomial by assumption, we have $\lmsm(1x) = \lmsm(x) = x$. Since $xm$ is a standard term that is an $R$-multiple of $y$, we have $\lmsm(mx) = y$. Putting these together, we get $x \prec y$, as claimed.
\end{proof}

\subsubsection{Algebras of leading terms}
\label{sec:algebra-leading-terms}
The straightening law in a pseudo-ASL results in divisibility of standard monomials not always being as clear-cut as divisibility of ordinary monomials in a polynomial ring; put another way, because of the straightening law, many pseudo-ASLs (and ASLs) of interest are not monomial pseudo-ASLs. However, because of the importance of monomial ideals and divisibility of monomials in \gb basis theory, we find it useful to have an auxiliary monomial pseudo-ASL $A_{lt}$ associated to any pseudo-ASL $A$, in which the leading terms will live, and where divisibility in the auxiliary algebra will govern what we mean for one standard monomial to ``divide'' another. This approach is similar to the ``graded structures'' approach developed by Robbiano \cite{robbiano} and Mora \cite{mora, moraIntro} for \gb bases in other settings. We encapsulate our desired monomial pseudo-ASL in the following definition, and we give two key examples below, that exist relative to any pseudo-ASL term order.

\begin{definition}[store=defalt, note=Algebra of leading terms] \label{def:ALT}
Given an $R$-pseudo-ASL $(A, H, \Sigma)$ with pseudo-ASL term order $\preceq$, we say that an algebra $A_{lt}$ is an \emph{algebra of leading terms} for $(A, H, \Sigma, \preceq)$, if 
\begin{enumerate}
\item \label{point:a-lt-asl} $A_{lt}$ is also an $R$-pseudo-ASL governed by the same $(H,\Sigma)$, and 
\item \label{point:pi-let} if $\pi_{lt}$ denotes the unique $R$-linear bijection $A \to A_{lt}$ that restricts to the identity map on standard monomials, then for all standard monomials $m$, $m'$, 
\[
\pi_{lt}(m)\pi_{lt}(m') = 0,\qquad \text{or}\qquad  \pi_{lt}(m)\pi_{lt}(m') = \pi_{lt}(\ltsm(mm')).
\]
\end{enumerate}
\end{definition}

\begin{lemma} \label{lem:ALT}
If $A_{lt}$ is an algebra of leading terms for some $(A,\preceq)$, then $A_{lt}$ is a monomial pseudo-ASL.
\end{lemma}

\begin{proof}
Suppose $m_1, m_2 \in A_{lt}$ are standard monomials. If $m_1 m_2 = 0$ we are done. Otherwise, let $\overline{m}_i = \pi_{lt}^{-1}(m_i)$. Then by assumption $\pi_{lt}(\overline{m}_1) \pi_{lt}(\overline{m}_2) = m_1 m_2 \neq 0$. Thus, by Definition \ref{def:ALT}, \[\pi_{lt}(\ltsm(\overline{m}_1 \overline{m}_2)) = \pi_{lt}(\overline{m}_1) \pi_{lt}(\overline{m}_2) = m_1 m_2.\] Since the left-hand side here is $\pi_{lt}$ applied to a standard term, it is, itself, a standard term, and we are done.
\end{proof}

\begin{notation}
Henceforth when we write $(A, \preceq, A_{lt})$ it is assumed that $A$ is an $R$-pseudo-ASL, $\preceq$ is a pseudo-ASL term order, and $A_{lt}$ is an algebra of leading terms for $(A,\preceq)$.
\end{notation}

\begin{observation} \label{obs:ASLlt}
If $A$ is an ASL on $(H, \leq, \Sigma)$, and $\preceq$ is a pseudo-ASL term order on $A$, then any algebra of leading terms $A_{lt}$ relative to $(A,\preceq)$ is an ASL on the same $(H, \leq, \Sigma)$.
\end{observation}

\begin{proof}
Since the only additional structure on an ASL relative to a pseudo-ASL is the partial order $\leq$ on $H$, and the fact that the straightening laws satisfy the ASL condition, that is all we need to show. However, since the monomial occurring on the RHS of a (pseudo-)straightening law in $A_{lt}$ is among the monomials occurring on the RHS of the straightening law in $A$, the ASL condition is automatically satisfied.
\end{proof}

\newcommand{\Adisc}{A_{disc}}

\begin{definition}[$A_{gen}$ and $A_{disc}$ --- two \emph{extreme} algebras of leading terms]
\label{defn:agen-adisc}
\begin{enumerate}
\item (cf. discrete ASL \cite{deconcini1987hodge}) Given $(H, \Sigma)$, the \emph{discrete (pseudo-)ASL} is defined as $\bigslant{R[H]}{\langle \Sigma \rangle}$; we denote this by  $\Adisc$. 
\item Given $(A,H,\Sigma,\preceq)$, we define the \emph{generic algebra of leading terms}, denoted $A_{gen}$, as the $R$-algebra with $R$-basis the standard monomials $\W^H \backslash \Sigma$, and with multiplication defined as follows. Let $\pi_{lt} \colon A \to A_{gen}$ be the $R$-linear bijection that is the identity on standard monomials. Given two standard monomials $m,m' \in A$, we define multiplication in $A_{gen}$ by
\begin{equation} \label{eq:agen}
\pi_{lt}(m) \pi_{lt}(m') = \pi_{lt}(\ltsm(mm'))
\end{equation}
for all standard monomials $m,m' \in A$, and then extending by linearity.
\end{enumerate}
\end{definition}

The definition of $A_{gen}$ amounts to taking the straightening rules in $A$ and ``truncating'' them after the leading term on their right-hand sides. That is, if $m$ is an element of $\Sigma$, then the rewriting rule for $m$ in $A_{gen}$ is defined to be $m = \ltsm(m)$; note that in the definition of $A_{gen}$ we include \eqref{eq:agen} for \emph{all} $mm' \in \Sigma$, not just the generators of $\Sigma$. The fact that we have to impose \eqref{eq:agen} for all $mm' \in \Sigma$ and not just generators is closely related to the fact that, even in an ASL, the straightening rules only for the generators of $\Sigma$ need not give a full set of defining relations for the ASL \cite{trung,GrabePauer}, though they do  in the case of graded ASLs \cite[Prop.~1.1]{deconcini1987hodge}.

\begin{proposition} \label{prop:agen-adisc}
Let  $(A, H, \Sigma)$ be a pseudo-ASL with a pseudo-ASL term order $\preceq$. Then the discrete pseudo-ASL $\Adisc$ and the generic algebra of leading terms $A_{gen}$ are both algebras of leading terms for $(A, H, \Sigma, \preceq)$. 
\end{proposition}

It follows from Observation~\ref{obs:ASLlt} that if $A$ is an ASL governed by  $(H,\leq,\Sigma)$, then $\Adisc$ and $A_{gen}$ are ASLs governed by the same $(H,\leq,\Sigma)$ as well. Whereas $A_{disc}$ was already considered in \cite{deconcini1987hodge} as a deformation of an ASL to one that is a quotient of a polynomial ring by a monomial ideal, $A_{gen}$ provides an ASL that is similarly related to the original ASL $A$, but now is a quotient of $R[H]$ by a binomial ideal. (We also note that $\Adisc$ is always an ASL, even if $A$ is not an ASL; we do not know whether this holds for $A_{gen}$.)

\begin{proof}[Proof of Proposition~\ref{prop:agen-adisc}]
$[\Adisc{}]$: Since the standard monomials are precisely those not in $\Sigma$, and the monomials in $\Sigma$ are the only monomials occurring in the $R[H]$-ideal $\langle \Sigma \rangle$, the standard monomials certainly form an $R$-basis for $\Adisc$. It remains to show that $\Adisc$ satisfies the multiplicative property of Definition~\ref{def:ALT}. Towards that end, let $m,m' \in A$ be two standard monomials. If $\pi_{lt}(m) \pi_{lt}(m') = 0$, then we are done. If $\pi_{lt}(m) \pi_{lt}(m') \neq 0$, then by definition of $\Adisc$ as $\bigslant{R[H]}{\langle \Sigma \rangle}$, we must have that $\pi_{lt}(m) \pi_{lt}(m')$ does not lie in $\Sigma$, hence is a standard term. Thus, since the product $mm'$ in $R[H]$ is not in $\Sigma$, we have that in $A$, no straightening occurs in the product $mm'$; thus we have that $mm'$ is a standard term in $A$ as well, and in fact must be the same standard term in $A$ as it is in $\Adisc$. Since $mm'$ is a standard term, we have $mm' = \ltsm(mm')$, and since it is the same standard term as in $\Adisc$, we have $\pi_{lt}(mm') = \pi_{lt}(m) \pi_{lt}(m')$. Putting these together, we get $\pi_{lt}(m)\pi_{lt}(m') = \pi_{lt}(\ltsm(mm'))$, as desired.

$[A_{gen}]$: First, we show that $A_{gen}$ is in fact an associative algebra. For this, it suffices to prove associativity for standard monomials, as then full associativity follows by bilinearity of the product, since the standard monomials span (in fact are a basis for) $A_{gen}$ by construction. Towards this end, let $m_1, m_2, m_3 \in A$ be standard monomials. We have
\begin{align}
&(\pi_{lt}(m_1) \pi_{lt}(m_2)) \pi_{lt}(m_3) \nonumber \\
&\qquad \qquad \qquad = \pi_{lt}(\ltsm(m_1 m_2)) \pi_{lt}(m_3) \eqcommentnear{by definition of $A_{gen}$} \nonumber \\
 &\qquad \qquad \qquad  = \pi_{lt}(\ltsm(\ltsm(m_1 m_2) m_3)) \eqcommentnear{by definition of $A_{gen}$} \nonumber\\
 &\qquad \qquad \qquad  = \pi_{lt}(\ltsm(\ltsm(m_1 m_2) \ltsm(m_3))) \eqcommentnear{since $m_3$ is standard}. \label{eq:Agen1} 
\end{align}
Similarly, we get
\begin{equation}
\pi_{lt}(m_1) (\pi_{lt}(m_2) \pi_{lt}(m_3)) = \pi_{lt}(\ltsm(\ltsm(m_1) \ltsm(m_2 m_3))). \label{eq:Agen2}
\end{equation}

So it is necessary and sufficient to show that 
\[
\ltsm(\ltsm(m_1 m_2) \ltsm(m_3)) = \ltsm(\ltsm(m_1) \ltsm(m_2 m_3)). 
\]
If $\ltsm(m_1) \ltsm(m_2) \ltsm(m_3) \neq 0$, then by \eqref{obs:mult-eqn3}, both of the preceding are equal to $\ltsm(\ltsm(m_1) \ltsm(m_2) \ltsm(m_3))$, as desired.

On the other hand, if $\ltsm(m_1) \ltsm(m_2) \ltsm(m_3) = 0$, we use the following lemma:

\begin{lemma} \label{lem:triplezero}
Given $(A,\preceq)$, if $m_1,m_2.m_3$ are three standard monomials in $A$ then
\[
m_1 m_2 m_3 = 0 \Rightarrow \ltsm(m_1 m_2) m_3 = 0.
\]
\end{lemma}

\begin{proof}
If $m_1 m_2 = 0$, then $\ltsm(m_1 m_2) = 0$ and we are done. Otherwise, let $m_1 m_2 = \sum_{i=1}^k r_i m'_i$ for nonzero $r_i \in R$ and standard monomials $m'_i$, with 
\[
\ltsm(m_1 m_2) = m'_1 \succ m'_2 \succ \dotsb \succ m'_k.
\] 
By assumption we have 
\[
\sum r_i m'_i \ltsm(m_3) = \sum r_i m'_i m_3 = 0.
\]
But by \ref{def:termorder:mult}, if $m'_1 m_3 \neq 0$, then we have $\ltsm(m'_1 m_3) \succ \ltsm(m'_i m_3)$ for all $i > 1$ such that $m'_i m_3 \neq 0$, and thus no other terms could cancel out $\ltsm(m'_1 m_3)$. Thus the term $\ltsm(m'_1 m_3)$ must vanish since $\sum r_i m'_i m_3 = 0$, hence we have $m'_1 m_3 = 0$, as claimed.
\end{proof}

Returning to $A_{gen}$, it remains to handle the case that $\ltsm(m_1) \ltsm(m_2) \ltsm(m_3) = 0$. Applying Lemma~\ref{lem:triplezero} once as written and once with indices permuted, we then get
\[
\ltsm(m_1 m_2) \ltsm(m_3) = 0 \qquad \text{ and } \qquad \ltsm(m_1) \ltsm(m_2 m_3) = 0.
\]
Thus, in this case, both \eqref{eq:Agen1} and \eqref{eq:Agen2} are equal to 0, and therefore multiplication in $A_{gen}$ is associative.

So far, we have that $A_{gen}$ is an associative $R$-algebra with $R$-basis given by the standard monomials. By construction, if $\pi_{lt}(m) \pi_{lt}(m') \neq 0$, we have $\pi_{lt}(m) \pi_{lt}(m') = \pi_{lt}(\ltsm(mm'))$, so $A_{gen}$ satisfies the equation required in Definition~\ref{def:ALT}, and is thus an algebra of leading terms for $(A,\preceq)$.
\end{proof}

In other situations, such as the Weyl algebra \cite[Section~1.1]{saito1998grobner}, the role of the algebra of leading terms is played by the associated graded ring. In fact, although we will not need it in the sequel, we show that $A_{gen}$ is the associated graded algebra, when $A$ is graded in the natural way by the monoid of standard monomials:

\begin{proposition}
Every algebra of leading terms is a graded pseudo-ASL, graded by the monoid $\mathcal{M}$ of standard monomials (as in Remark~\ref{rmk:semigroup}). Furthermore, given $(A,\preceq)$, $A$ is filtered by $\mathcal{M}$, and $A_{gen}$ is the associated graded ring of this filtration.
\end{proposition}

\begin{proof}
Given $(A, \preceq)$, let $\mathcal{M}$ be the monoid of standard monomials associated to $A_{gen}$ as in Remark~\ref{rmk:semigroup}. Since the $m$-graded piece consists precisely of the $R$-multiples of $m$, and the standard monomials form an $R$-basis of $A$ because $A$ is a pseudo-ASL, we have that 
\[
A = \bigoplus_{m \in \mathcal{M}} Rm = \bigoplus_{m \in \mathcal{M}} A_m,
\]
where here $A_m$ denotes the $m$-graded piece (and not, say, the localization); and the same direct sum decomposition holds for any $A_{lt}$. By construction of the monoid of standard monomials, in $A_{lt}$ we have $(A_{lt})_m \cdot (A_{lt})_{m'} \subseteq (A_{lt})_{mm'}$, so $A_{lt}$ is $\mathcal{M}$-graded (here we write the subscripts of our grading multiplicatively rather than additively), where if $mm' = 0$ in $A$ (and hence also in $\mathcal{M}$), we understand that $(A_{lt})_{mm'}$ on the RHS of the preceding equation to simply mean $(A_{lt})_0 := 0$.

In $A$, since $\mathcal{M}$ was defined by taking the leading term of $mm'$, we have $A_m \cdot A_{m'} \subseteq A_{\preceq mm'}$, so $A$ is $\mathcal{M}$-filtered. Furthermore, we have $\bigslant{A_{\preceq mm'}}{A_{\prec mm'}} \cong A_{mm'}$, and it immediately follows that $A_{gen}$ is the associated graded of $A$ as an $\mathcal{M}$-filtered pseudo-ASL.
\end{proof}

\begin{remark}
In fact, $A_{gen}$ is in some sense the ``universal'' algebra of leading terms, in that any algebra of leading terms for $(A, \preceq)$ is obtained from $A_{gen}$ by further forcing certain products of standard monomials to be 0. Similarly, we see that $\Adisc$ is the ``opposite extreme,'' in the sense that starting from any $A_{lt}$, one can recover $\Adisc$ by forcing certain additional products of standard monomials to be 0. For a conjectural geometric version of this picture in terms of flat deformations, see Conjecture~\ref{conj:flatdeform}.
\end{remark}

\begin{proposition} \label{prop:orders}
If $\preceq$ is a pseudo-ASL term order on the pseudo-ASL $A$, and $A_{lt}$ is an algebra of leading terms for $(A,\preceq)$, let $\preceq_{lt}$ denote the order on the standard monomials of $A_{lt}$ that agrees with $\preceq$, i.e. 
\[
m \preceq_{lt} m' \Leftrightarrow \pi_{lt}^{-1}(m) \preceq \pi_{lt}^{-1}(m').
\]
Then $\preceq_{lt} $ is a pseudo-ASL term order on $A_{lt}$.
\end{proposition}

\begin{proof}
For any standard monomial $m \in A_{lt}$, we have
\begin{align}
&1_A \preceq \pi_{lt}^{-1}(m) \eqcomment{by Axiom \ref{def:termorder:positive} on $A$} \nonumber \\
\qquad \iff & \pi_{lt}^{-1}(1_{A_{lt}}) \preceq \pi_{lt}^{-1}(m) \eqcomment{by definition of $\pi_{lt}$} \nonumber \\
\qquad \iff & 1_{A_{lt}} \preceq_{lt} m \eqcomment{by definition of $\preceq_{lt}$}, \nonumber
\end{align}
thus establishing Axiom \ref{def:termorder:positive} for $(A_{lt}, \preceq_{lt})$.

To verify Axiom \ref{def:termorder:mult} for $\preceq_{lt}$, suppose $f,g,h,k$ are standard monomials in $A_{lt}$ with $f \prec_{lt} g$ and $h \preceq_{lt} k$, and $fh$ and $gk$ both nonzero. Let $f' = \pi_{lt}^{-1}(f)$, and similarly for $g',h',k'$. By the definition of $\preceq_{lt}$, we have $f' \prec g'$ and $h' \preceq k'$. 

Next, we claim that $f'h' \neq 0$ and $g'k' \neq 0$. We show this for $f'h'$, the argument for $g'k'$ being similar, \emph{mutatis mutandis}. By Definition~\ref{def:ALT}, since 
\[
0 \neq fh = \pi_{lt}(f')\pi_{lt}(h'),
\]
 we have that
\[
\pi_{lt}(\ltsm(f'h')) = \pi_{lt}(f') \pi_{lt}(h') = fh \neq 0. 
\]
And since $\pi_{lt}$ applied to the leading term of $f'h'$ is nonzero, $f'h'$ itself must be nonzero. 

Finally, since $\preceq$ is a pseudo-ASL term order, we have $\ltsm(f'h') \prec \ltsm(g'k')$. But now, by Definition~\ref{def:ALT}, applying $\pi_{lt}$ to both sides we get 
\begin{eqnarray*}
\pi_{lt}\left(\ltsm(f'h')\right) \preceq_{lt} \pi_{lt}\left(\ltsm(g'k')\right) & \Longleftrightarrow & \pi_{lt}(f') \pi_{lt}(h') \prec_{lt} \pi_{lt}(g') \pi_{lt}(k') \\
 & \Longleftrightarrow & fh \prec_{lt} gk,
\end{eqnarray*}
as desired. Thus $\preceq_{lt}$ is a pseudo-ASL term order on $A_{lt}$.
\end{proof}

\begin{corollary}[of Proposition \ref{prop:orders}]
\label{cor:divisor-less-than-dividend}
Given $(A, \preceq)$ and two distinct standard monomials $m,m'$, if there exists a standard term $t$ such that $\ltsm(mt) = m'$, then $m \prec m'$. 
\end{corollary}

\begin{proof}
Then existence of such a $t$ is equivalent to the divisibility $\pi_{lt}(m) \divides \pi_{lt}(m')$ in $A_{gen}$. As $A_{gen}$ is a monomial pseudo-ASL (Lemma~\ref{lem:ALT}), and $\prec$ induces a pseudo-ASL term order on $A_{gen}$ by Proposition~\ref{prop:orders}, Lemma~\ref{lem:div} applies, so we get that $\pi_{lt}(m) \prec \pi_{lt}(m')$. As the term order on $A$ and $A_{gen}$ give the same order on standard monomials, we get $m \prec m'$.
\end{proof}

The following result is key to ensuring that proofs and algorithms may proceed by downward inductions that reduce the leading term of a given element of $A$, and be guaranteed to reach a base case in finitely many steps. 

\begin{proposition} \label{prop:well-order}
Pseudo-ASL term orders are well-orders.
\end{proposition}

Interestingly, although the statement of this proposition does not mention monomial pseudo-ASLs nor algebras of leading terms, the proof uses them in a crucial way, in order to apply Lemma~\ref{lem:div} and an analogue of the ``usual'' Noetherianity-based proof (e.g., in polynomial rings).

\begin{proof}[Proof outline]
Given a pseudo-ASL $A$ and pseudo-ASL term order $\preceq$ on $A$, by Proposition~\ref{prop:agen-adisc}, there exists at least one algebra of leading terms $A_{lt}$ for $(A,\preceq)$. By Proposition~\ref{prop:orders}, $\preceq$ is isomorphic to a pseudo-ASL term order $\preceq_{lt}$ on the standard monomials in $A_{lt}$. So it suffices to show that $\preceq_{lt}$ is a well-ordering. The ``usual'' proof from polynomial rings then works in $A_{lt}$ because $A_{lt}$ is a monomial pseudo-ASL, and Noetherian.
\end{proof}

\begin{proof}
We follow the above outline, showing here that the usual proof works in $A_{lt}$. To show that $\preceq_{lt}$ is a well-ordering, for the sake of contradiction suppose that 
\[
x_1 \succ_{lt} x_2 \succ_{lt} x_3 \succ_{lt} \dotsb
\]
is an infinite chain of standard monomials in $A_{lt}$ that is strictly descending according to $\preceq_{lt}$. Let $I \subseteq A_{lt}$ be the ideal generated by all the $x_i$. Since $A_{lt}$ is Noetherian, there is a finite subset of the $x_i$, say $\{x_{i_1}, \dotsc, x_{i_k}\}$, that generates $I$. Since $I$ is generated by standard monomials, by Lemma~\ref{lem:monomial}, a standard monomial $m$ is in $I$ iff it is divisible by one of $\{x_{i_1}, x_{i_2}, \dotsc, x_{i_k}\}$. Thus for all $j$, each $x_j$ is divisible by at least one of $\{x_{i_1}, x_{i_2}, \dotsc, x_{i_k}\}$. Now consider $j > \max\{i_1, \dotsc, i_k\}$, and suppose without loss of generality, for simplicity of notation, that $x_j$ is divisible by $x_{i_1}$. By Lemma~\ref{lem:div}, we thus have $x_{i_1} \preceq x_j$. But since $j > i_1$, by our original assumption we also have $x_{i_1} \succ x_j$, contradicting the fact that $\prec$ is asymmetric.
\end{proof}

\subsubsection{Leading ideals and pseudo-ASL Gröbner bases}
Unlike the case of ordinary Gröbner bases, in the case of pseudo-ASL Gröbner bases, in order for the leading ideal to be a monomial ideal, we define the leading ideal to be an ideal of $A_{lt}$ rather than of $A$.

\begin{definition}[store=defleadingideal, note=Leading ideal]
\label{def:leading-ideal}
Given $(A,\preceq, A_{lt})$, for an ideal $I \subseteq A$, its \emph{leading ideal} is 
\[
\lismlt(I) := \langle \pi_{lt}(\ltsm(f)) \suchthat f \in I \rangle  \subseteq A_{lt}.
\] 
\end{definition}

It follows immediately from the definition that the leading ideal is a standard monomial ideal. The next proposition shows that the only standard monomials in $\lismlt(I)$ are the ``expected'' ones, namely, $\pi_{lt}(\ltsm(f))$ for $f \in I$:

\begin{proposition} \label{prop:LI}
Given an $R$-pseudo-ASL $A$, for any ideal $I \subseteq A$, \[Span_R\{\pi_{lt}(\ltsm(f)) \suchthat f \in I\} = \lismlt(I).\]
\end{proposition}

\begin{proof}
Let $S = \{\pi_{lt}(\ltsm(f)) \suchthat f \in I\}$. We first show that $Span_R(S)$ is an ideal. By definition $Span_R(S)$ is an $R$-module. All the remains to show is that $Span_R(S)$ is closed under products with arbitrary elements of $A_{lt}$. By $R$-bilinearity of the product, it suffices to show that $S \cup \{0\}$ is closed under products with standard monomials in $A_{lt}$. Towards that end, let $m \in A_{lt}$ be a standard monomial, and let $m’ \in S$. If $mm’ = 0$ we’re done. Otherwise, by Definition~\ref{def:ALT} we have $\pi_{lt}(\ltsm(\pi_{lt}^{-1}(m) \pi_{lt}^{-1}(m’))) = mm’$. Since $m’$ is in $S$, there is an $f \in I$ with $\pi_{lt}(\ltsm(f)) = m’$. Since $I$ is an ideal, $\pi_{lt}^{-1}(m) f$ is also in $I$. Thus by definition of $S$, $\pi_{lt}(\ltsm(\pi^{-1}(m)f))$ is also in $S$. Now we have
\begin{align}
S \ni \pi_{lt}(\ltsm(\pi_{lt}^{-1}(m) f) & = \pi_{lt}(\ltsm(\pi_{lt}^{-1}(m) \ltsm(f)) & \eqcomment{by Proposition~\ref{prop:mult}} \\ 
 & = \pi_{lt}(\ltsm(\pi_{lt}^{-1}(m) \pi_{lt}^{-1}(m’)) \\
 & = mm’. 
\end{align}
Thus $mm’$ is in $S$, so $S \cup \{0\}$ is closed under multiplication by arbitrary standard monomials, and thus $Span_R(S)$ is an ideal.
Finally, since $\lismlt(I)$ is generated by $Span_R(S)$, and $Span_R(S)$ is an ideal, the two are equal.
\end{proof}

We shall now define our key notion, i.e. that of pseudo-ASL \gb bases.

\begin{definition}[store=defaslgb,note={Pseudo-ASL \gb basis}]
\label{def:asl-gb}
Given $(A, \preceq, A_{lt})$ and an ideal $I \subseteq A$, a set $G \subseteq I$ is a \emph{pseudo-ASL Gröbner basis} for $I$ with respect to $(A,\preceq,A_{lt})$ if 
\[
\lismlt(I) = \langle \pi_{lt}(\ltsm(g)) \suchthat g \in G \rangle.
\]
\end{definition}

(The superscript $^{lt}$ is to emphasize that the leading ideal is contained in $A_{lt}$, and is not merely $\langle \ltsm(f) : f \in I \rangle \subseteq A$.)

\begin{theorem}[store=thmaslgbexistence, note=Existence of pseudo-ASL Gröbner bases] \label{thm:asl-gb-existence}
Given $(A, \preceq, A_{lt})$, every ideal $I \subseteq A$ has a finite pseudo-ASL \gb basis with respect to $(A, \preceq, A_{lt})$.
\end{theorem}

\begin{proof}
Let $I$ be an ideal in $A$ with leading ideal $\lismlt(I) \subseteq A_{lt}$. Define
\[
S = \{\pi_{lt}(\ltsm(f)) \suchthat f \in I\}.
\]
By definition, $\lismlt(I)$ is generated by $S$. Since $R$ is Noetherian, $A_{lt}$ is Noetherian, so $\lismlt(I)$ is generated by a finite subset of $S$, say $m_1, \dotsc, m_k$. By definition of $S$, there exist $g_i \in I$ such that $\pi_{lt}(\ltsm(g_i)) = m_i$ for all $i$. Thus $\{g_1,\dotsc,g_k\}$ is a finite pseudo-ASL Gröbner basis for $I$.
\end{proof}

Proposition \ref{prop:asl-grobner-divisibility-characterization} below is an equivalent characterization of the definition of pseudo-ASL \gb basis in terms of divisibility of leading terms in $A_{lt}$. This divisibility characterization allows one to answer basic questions about ideals in a pseudo-ASL in an effective manner. For instance, given a pseudo-ASL $A$ and an ideal $I$ in $A$, using a pseudo-ASL \gb basis of $I$ with respect to any $A_{lt}$, we can check if any $f \in A$ belongs to $I$ in a simple mechanical way just by using Proposition \ref{prop:asl-grobner-divisibility-characterization}.

\begin{definition}[Divisibility of leading terms] \label{def:ltdivides}
Let $A_{lt}$ be an algebra of leading terms for $(A,\preceq)$, and let $f, g \in A$. If $\pi_{lt}(f) \divides \pi_{lt}(g)$ (in $A_{lt}$), we write $f \ltdivides g$. If $m,m' \in A$ are standard terms such that $m \ltdivides m'$, we may further write $\ltfrac{m'}{m}$ for the unique (by Lemma~\ref{lem:mon-div}) standard term $t \in A_{lt}$ such that $\pi_{lt}(m)t = \pi_{lt}(m')$. 
\end{definition}

\begin{proposition}[store=propaslgrob-divcharacterization, note={Divisibility characterization of pseudo-ASL Gröbner bases}]
\label{prop:asl-grobner-divisibility-characterization}
Given $(A, \preceq, A_{lt})$ and an ideal $I \subseteq A$, a subset $G \subseteq I$ is a pseudo-ASL Gröbner basis if and only if for all $f \in I$, there exists $g \in G$ such that $\ltsm(g) \ltdivides \ltsm(f)$.
\end{proposition}

\begin{proof}
Suppose $G$ is a pseudo-ASL Gröbner basis and $f \in I$. Then $\pi_{lt}(\ltsm(f)) \in \lismlt(I)$. Since by definition $\lismlt(I)$ is a standard monomial ideal that is generated by $\{\pi_{lt}(\ltsm(g)) \suchthat g \in G \}$, by Lemma \ref{lem:monomial}, there is a $g \in G$ such that $\pi_{lt}(\ltsm(g))$ divides $\pi_{lt}(\ltsm(f))$. This is precisely the definition of $\ltsm(g) \ltdivides \ltsm(f)$.

Conversely, suppose every $f \in I$ has its leading term $\ltdivides$-divisible by the leading term of some $g \in G$. Let $m$ be a standard monomial in $\lismlt(I)$. There exists $f \in I$ such that $\pi_{lt}(\ltsm(f)) = m$. By assumption, there is a $g \in G$ such that $\ltsm(g) \ltdivides \ltsm(f)$. Thus $\pi_{lt}(\ltsm(f))$ is in $\langle \pi_{lt}(\ltsm(g)) \suchthat g \in G \rangle$. Since $m$ was an arbitrary standard monomial in $\lismlt(I)$ and $\lismlt(I)$ is a standard monomial ideal, it follows that 
\[
\lismlt(I) \subseteq \langle \pi_{lt}(\ltsm(g)) \suchthat g \in G \rangle.
\]
The opposite inclusion is immediate from the definition of $\lismlt(I)$, so $G$ is a pseudo-ASL Gröbner basis.
\end{proof}

\begin{theorem}[store=thmmacaulay, note=Pseudo-ASL Analogue of Macaulay's Basis Theorem] \label{thm:macaulay}
Given $(A, \preceq, A_{lt})$, let $I \subseteq A$ be an ideal, and $\rho \colon A \to \bigslant{A}{I}$ the natural projection homomorphism. Let
 \[
 B = \left\{\rho(m) \suchthat m \text{ is a standard monomial and } \pi_{lt}(m) \notin \lismlt(I)\right\}.
 \]
 Then $B$ is an $R$-linear basis of $\bigslant{A}{I}$.
\end{theorem}

For pseudo-ASLs that are both over a field and monomial, this result is mentioned already in Green \cite[p.~36]{green}. As the proof in our setting follows the proof for ordinary Gröbner bases very closely, we relegate it to Appendix~\ref{sec:app:theory}.

\subsubsection{Compatibility, Standard Expressions, and Remainders}

The key new concept we need in comparison with ordinary Gröbner bases is the following notion of ``compatibility'' of two standard monomials. Nearly all the subsequent definitions and results we use will rely in some way on the notion of compatibility.

\begin{definition}[store=defncompatible, note=Compatible] \label{def:compatible}
Given $(A,\preceq,A_{lt})$, we say two standard monomials $m,m' \in A$ are \emph{compatible (with respect to $A_{lt}$)}  if $\pi_{lt}(m) \pi_{lt}(m') \neq 0$.\footnote{We tried defining compatibility instead as ``if $mm \neq 0$ then $\pi_{lt}(m) \pi_{lt}(m') \neq 0$'', with the intuition that if $A$ is not a domain, we ought only to care whether $m$ and $m'$ are ``more compatible'' in $A$ than in $A_{lt}$. That definition led to a number of issues that were sidestepped with our current definition. While there might be a way to get this alternative definition to work, we did not find it.} For arbitrary $f,g \in A$, we say that $f$ and $g$ are compatible (with respect to $A_{lt}$) if every standard monomial occurring in $f$ is compatible with every standard monomial occurring in $g$.
\end{definition}

In terms of compatibility, we can rephrase the condition in the definition of algebra of leading terms (Definition~\ref{def:ALT}) as: if $m,m'$ are compatible (with respect to $A_{lt}$), then $\pi_{lt}(m)\pi_{lt}(m') = \pi_{lt}(\ltsm(mm'))$. Note that, by the defining condition in $A_{lt}$, if two standard monomials $m,m' \in A$ are compatible, then $mm' \neq 0$ in $A$ as well.

\begin{definition}[store=defnstandardexpression,note=Standard expression]
\label{defn:standard-expression}
Given $(A, \preceq, A_{lt})$ and a list of elements $g_1,\dotsc,g_k, f \in A$, a \emph{standard expression} for $f$ relative to $g_1,\dotsc,g_k$ is a list $r,h_1,\dotsc,h_k \in A$ such that:
\begin{enumerate}
\item We have
\begin{equation} \label{eqn:standard-expression} f = r + \sum_{i=1}^{k} h_i g_i, \end{equation}

\item for any standard term $t$ occurring in $r$, $\ltsm(g_i) \nmid_{lt} t$ for all $i \in [k]$,
\item $\lmsm(f) \succeq \lmsm(h_i g_i)$ for all $i$, and
\item for every $i=1,\dotsc,k$, $h_i$ is compatible with $\ltsm(g_i)$.
\end{enumerate}

When this happens, we say that $r$ is \emph{a remainder} of $f$ relative to $\{g_1,\dotsc,g_k\}$, or that \emph{$f$ reduces to $r$ relative to $\{g_1, \ldots, g_k\}$}. 
\end{definition}

In general, remainders are not unique; however, the next proposition shows that they are unique when the $g_i$ form a pseudo-ASL Gr\"{o}bner basis.

\begin{proposition}
\label{prop:unique-remainder}
Let $I \subseteq A$ be an ideal and let $G \subseteq I$ be an pseudo-ASL \gb basis of $I$ with respect to $(\preceq, A_{lt})$. Then for every $f \in A$, there exists remainder $r$ of $f$ relative to $\{g_1, \dotsc, g_k\}$, and this remainder is unique. Furthermore, the remainder is 0 if and only if $f$ is in $I$.
\end{proposition}

In the case where $A$ is a monomial pseudo-ASL over a field, and the product of standard monomials is a standard monomial (not just a standard term), this proposition follows directly from Green's \cite[Prop~2.7]{green}. In our setting, we must carefully use both $A$ and $A_{lt}$ (which is always a monomial pseudo-ASL), but the basic idea of the proof is otherwise similar.

\begin{proof}

[Existence]. We will show this by induction on $\lmsm(f)$ (which works because $\preceq$ is a well-order, by Proposition~\ref{prop:well-order}). The base case is $\lmsm(f) = 0$, which is trivial. For the inductive step, as long as there is a $g \in G$ such that $\ltsm(g) \ltdivides \ltsm(f)$, we proceed as follows. In this case, there exists a standard term $h \in A_{lt}$ such that $h\pi_{lt}(\ltsm(g)) = \pi_{lt}(\ltsm(f))$. Setting $h' = \pi_{lt}^{-1}(h)$, we have

\begin{equation}
\label{eqn:f-hg-lt}
\pi_{lt}(\ltsm(f)) = h\pi_{lt}(\ltsm(g)) = \pi_{lt}(h')\pi_{lt}(\ltsm(g)) = \pi_{lt}(\ltsm(h'g)),
\end{equation} 

where the third equality is due to Definition~\ref{def:ALT}, since we have that $h'$ and $\ltsm(g)$ are compatible, since $\pi_{lt}(h') \pi_{lt}(\ltsm(g)) = \pi_{lt}(\ltsm(f)) \neq 0$.

Now consider $f' = f - h'g$. By Equation \eqref{eqn:f-hg-lt}, $\pi_{lt}(\ltsm(f)) = \pi_{lt}(\ltsm(h'g))$, and since $\pi_{lt}$ is $R$-linear and is the identity on standard monomials, $\ltsm(f) = \ltsm(h'g)$, implying that $\ltsm(f') \prec \ltsm(f)$. Therefore, this procedure can be applied inductively till we are left with $r$ such that either $r=0$ (which we will show happens iff $f \in I$), or $\forall g \in G,\; \ltsm(g) \nmid_{lt} \ltsm(r)$ (which we will show happens iff $f \not\in I$). Consequently, we will have $f = r + \sum_{g \in G} h_gg$, where $h_g \in A$. 

If $r = 0$, the proposition is trivially true, so we only have to consider the case where $r \neq 0$. We can thus continue the procedure with lower terms of $r$ as needed, starting with the greatest standard term appearing in $r$, say $t$, such that for some $g \in G$, $\ltsm(g) \ltdivides t$. As long as we proceed with greatest-first in this fashion, the induction continues, thus producing for us an $r$ such that every standard term $t$ appearing in $r$ is not $\ltdivides$-divisible by $\ltsm(g)$ for any $g \in G$.

The $h_g$ are compatible with $\ltsm(g)$ by construction. Furthermore, every standard monomial $m$ we added to $h_g$ has the property that $\ltsm(mg)$ is either one of the terms of $f$, so we have $\ltsm(mg) \preceq \ltsm(f)$, or one of the terms of a previously-added $m'g'$, in which case we have $\ltsm(mg) \preceq \ltsm(m'g') \preceq \ltsm(f)$ (by induction ``upward towards $\ltsm(f)$''). Thus the above produces a standard expression for $f$ relative to $G$.

[Uniqueness]. To see uniqueness, suppose we have two such decompositions: 
\[
f = r + \sum_{g \in G} h_g g = r' + \sum_{g \in G} h'_g g.
\] 
Then we have that 
\[
r-r' = \sum_{g \in G} (h_g - h'_g) g \in I.
\]
Thus there must exist a $g \in G$ such that $\ltsm(g) \ltdivides \ltsm(r-r')$ by Proposition~\ref{prop:asl-grobner-divisibility-characterization}, since $G$ is a pseudo-ASL \gb basis for $I$. But $\lmsm(r-r')$ appears in either $r$ or $r'$, so by construction of $r$ and $r'$, there cannot exist a $g \in G$ such that $\ltsm(g) \ltdivides \lmsm(r-r')$. This in turn means that there cannot be any non-zero terms in $r-r'$, so we have $r-r'=0$, as desired.

[$r=0 \Leftrightarrow f \in I$] Each step of the above algorithm preserves the coset $f + I$, because we only ever add elements of $I$ (of the form $h'g$) to $f$. In particular, $r + I = f+I$. Thus if $r=0$, then $f$ is in $I$. Conversely, if $f$ is in $I$, then so is every intermediate $f'$ produced at every step of the above algorithm, since we have $f + I = f' + I$. By Proposition~\ref{prop:asl-grobner-divisibility-characterization}, there is thus some $g \in G$ such that $\ltsm(g) \ltdivides \ltsm(f')$, so the algorithm can continue to the next stage. The only way this terminates---which it must since the leading term is decreasing every time and $\preceq$ is a well-order---is when $r=0$.
\end{proof}

\subsection{Reduced pseudo-ASL Gröbner bases}
\label{sec:reduced}

\begin{definition}[store=defnreducedgb, note=Irredundant/Reduced \gb Basis]
\label{defn:reduced-gb}
Given $(A, \preceq, A_{lt})$, a pseudo-ASL Gröbner basis $G$ is \emph{irredundant} if there is no pair $g,g' \in G$ of distinct elements such that $\ltsm(g) \ltdivides \ltsm(g')$, and $G$ is \emph{reduced} if for every $g \in G$, for each standard monomial $m$ appearing in $g$, there is no distinct $g' \in G$ such that $\ltsm(g') \ltdivides m$, and furthermore, $\lcsm(g) = 1$ for all $g \in G$. 
\end{definition}

Being irredundant follows from being reduced, as the condition for irredundancy only depends on the leading monomials of the elements of $G$, while the condition for being reduced is similar but applies to all standard monomials appearing in any element of $G$. Nonetheless, conceptually it is often easier to split into two steps: first make a Gröbner basis irredundant, and then make it reduced.

The next result and its proof follow the standard treatment; we include the full details in Appendix~\ref{sec:app:theory} to ensure no subtlety is overlooked.

\begin{theorem}[store=thmreduced, note={Existence, uniqueness, and algorithm for reduced pseudo-ASL Gröbner bases}] 
\label{thm:reduced}
Given $(A, \preceq, A_{lt})$ and an ideal $I \subseteq A$, there is a unique reduced pseudo-ASL Gröbner basis for $I$ with respect to $(\preceq, A_{lt})$.

Furthermore, there is an algorithm that takes as input a pseudo-ASL Gröbner basis $G$ and produces as output a reduced pseudo-ASL Gröbner basis.
\end{theorem}

One might attempt to bound the runtime of this algorithm in terms of the total number of standard monomials appearing in every element of $G$, but we have found this tricky without knowing more details about the straightening law in the pseudo-ASL one is working with. The issue is that when the algorithm uses $t g_j$ to cancel out a monomial $m$ in $g_i$, the straightening law might result in $tg_j$ having more monomials than $g_j$ itself, thus increasing the total number of monomials in some steps.

Next we will show that the unique reduced pseudo-ASL Gröbner basis for $I$ relative to $(\preceq, \Adisc)$ actually only depends on the leading ideal $\lismlt(I)$, and not on any other information about the pseudo-ASL term order. This will be useful for proving existence of universal pseudo-ASL Gröbner bases in Theorem~\ref{thm:universal} below. 

\begin{lemma} \label{lem:reduced-Adisc}
Given an ideal $I \subseteq A$ in a pseudo-ASL, and two pseudo-ASL term orders $\preceq, \preceq'$, suppose that $(\lismlt)_{\preceq}(I) = (\lismlt)_{\preceq'}(I)$ in $\Adisc$. Then the reduced pseudo-ASL Gröbner bases for $I$ relative to $(\preceq, \Adisc)$ and $(\preceq',\Adisc)$ are equal.
\end{lemma}

\begin{proof}
Let $G$ be the unique reduced pseudo-ASL Gröbner basis for $I$ relative to $(\preceq, \Adisc)$, and similarly $G'$ relative to $(\preceq', \Adisc)$. It suffices to show that $G \subseteq G'$, for then by symmetry, applying the same argument with the roles reversed, we automatically get $G' \subseteq G$.

Let $g \in G$ and $m = \ltsm^{\preceq}(g)$; then $m$ is in $\lismlt(I)$.  Therefore, by Proposition~\ref{prop:asl-grobner-divisibility-characterization}, there is some element $g' \in G'$ such that $\ltsm^{\preceq'}(g')$ divides $m$ in $\Adisc$. We will show that $g' = g$. 

By Lemma~\ref{lem:unique-min}, $m$ is minimal under divisibility in $\lismlt(I) \subseteq \Adisc$, and thus we must have $\lmsm^{\preceq'}(g') = m$. Since $G'$ is reduced, we in fact have 
\[
\ltsm^{\preceq'}(g') = m = \ltsm^{\preceq}(g).
\] 

If $g \neq g'$, then there is at least one monomial in the support of $g-g'$, and it must be different than $\lmsm(g) = \lmsm(g')$. Let $m'$ be the $\preceq$-leading monomial of $g-g'$. If $m'$ occurs in $g$, then, because $m'$ is the $\preceq$-leading monomial of an element of $I$ (namely, $g-g'$), we have that there is some $g_2 \in G$ such that $\lmsm^{\preceq}(g_2)$ divides $m'$ in $\Adisc$, contradicting the fact that $G$ was reduced. On the other hand, if $m'$ occurs in $g'$, then since $(\lismlt)_{\preceq}(I) = (\lismlt)_{\preceq'}(I)$ in $\Adisc$, $m'$ also occurs as the $\preceq'$-leading monomial of some element of $I$ (not necessarily $g-g'$, since it need not be the case that $\lmsm^{\preceq}(g-g') = \lmsm^{\preceq'}(g-g')$). And thus there is some $g'_2 \in G'$ such that $\lmsm^{\preceq'}(g'_2)$ divides $m'$ in $\Adisc$, contradicting the fact that $G'$ was reduced. Thus, it must be the case that $g=g'$, completing the proof.
\end{proof}

\subsection{Relations between different algebras of leading terms}
In this subsection we prove several results clarifying the relationships between various algebras of leading terms and their associated pseudo-ASL Gröbner bases. In addition to being a natural question on its own, we will use the results of this section to show the existence of universal pseudo-ASL Gröbner bases in the next section.

Given $(A, \preceq)$, we say that two algebras of leading terms for $(A,\preceq)$, say $A_{lt}$ and $A'_{lt}$, are \emph{isomorphic as algebras of leading terms} for $(A,\preceq)$ if the bijection $\pi_{lt} \circ (\pi'_{lt})^{-1}$ (which is the identity on standard monomials) is an isomorphism of algebras $A'_{lt} \to A_{lt}$; we denote this by $A_{lt} \cong_{sm} A'_{lt}$.

\begin{proposition}[Characterization of $A_{gen}$ in terms of compatibility] \label{prop:Agen}
Given $(A,\preceq)$, all pairs of standard monomials $m,m' \in A$ such that $mm' \neq 0$ are compatible with respect to $A_{lt}$ if and only if \[A_{lt} \cong_{sm} A_{gen}.\]
\end{proposition}

\begin{proof}
Suppose $A_{lt}$ is isomorphic to $A_{gen}$ as in the statement. By the definition of $A_{gen}$, $\pi_{lt}(m)\pi_{lt}(m') = 0$ iff $mm' = 0$ in $A$, thus, all pairs $(m,m')$ such that $mm' \neq 0$ in $A$ are compatible relative to $A_{gen}$.

Conversely, suppose that all pairs of standard monomials whose products in $A$ are nonzero are compatible with respect to $A_{lt}$. For any two standard monomials $m,m' \in A$, if $mm' \neq 0$, then by the assumption of compatibility we have that $\pi_{lt}(m) \pi_{lt}(m') \neq 0$. But then by Definition~\ref{def:ALT}, we have $\pi_{lt}(m) \pi_{lt}(m') = \pi_{lt}(\ltsm(mm'))$. Since the latter holds for any pair of standard monomials $m,m' \in A$ such that $mm' \neq 0$, these are precisely the defining relations of $A_{gen}$, and thus $A_{lt}$ is isomorphic to $A_{gen}$ as an algebra of leading terms.
\end{proof}

\begin{proposition}[Relationship between pseudo-ASL Gröbner bases for different algebras of leading terms] \label{thm:diff}
Suppose that $A_{lt}$ and $A'_{lt}$ are two algebras of leading terms for the same $(A,\preceq)$, and suppose that, for all standard monomials $m,m'$, if $mm' \neq 0$ in $A_{lt}'$, then $mm' \neq 0$ in $A_{lt}$. 

Then for any ideal $I \subseteq A$, the reduced pseudo-ASL Gröbner basis for $I$ with respect to $(\preceq,A_{lt})$ is contained in the reduced pseudo-ASL Gröbner basis for $I$ with respect to $(\preceq,A'_{lt})$. 
\end{proposition}

The proof of this result is very similar to that of Lemma~\ref{lem:reduced-Adisc}. There, we fixed the algebra of leading terms to $\Adisc$ and allowed the term order to vary (as long as the leading ideal didn't change). Here, we allow the algebra of leading terms to vary, but we fix the term order. In the present setting, note that the leading ideal also essentially stays the same: the leading ideals in the two algebras of leading terms will be spanned by the same set of standard monomials, and sent to one another bijectively by $\pi'_{lt} \circ \pi_{lt}^{-1}$. (We note that our proof of existence of universal pseudo-ASL Gröbner bases below will use both Lemma~\ref{lem:reduced-Adisc} and Proposition~\ref{thm:diff}.)

\begin{proof}
Let $G$ be the reduced pseudo-ASL Gröbner basis for $I$ relative to $(\preceq, A_{lt})$ and let $G'$ be the reduced pseudo-ASL Gröbner basis for $I$ relative to $(\preceq, A'_{lt})$. Let $g \in G$; we will show that $g$ must be in $G'$ as well. 

Let $m = \ltsm(g)$; then $m$ is in $\lismlt(I)$ in any algebra of leading terms relative to $\preceq$.  In particular, the standard monomial $m$ must occur in $\lismlt(I)$ in $A'_{lt}$. Therefore, by Proposition~\ref{prop:asl-grobner-divisibility-characterization}, there is some element $g' \in G'$ such that $\ltsm(g')$ divides $m$ in $A'_{lt}$. We will show that $g' = g$. 

By Lemma~\ref{lem:unique-min}, $m$ is minimal under divisibility in $\lismlt(I) \subseteq A_{lt}$. Now, since $\lismlt(I) \subseteq A_{lt}'$ consists of the same set of standard monomials as it does in $A_{lt}$, and for any two standard monomials $n,n'$ that if $n$ divides $n'$ in $A_{lt}'$ then $n$ divides $n'$ in $A_{lt}$ (because the only difference between $A_{lt}$ and $A_{lt}'$ is that more products of standard monomials in $A_{lt}'$ are zero than in $A_{lt}$), we also have that $m$ is minimal under divisibility in $\lismlt(I)$ in $A_{lt}'$. Thus we must have $\lmsm(g') = m$. Since $G'$ is reduced, we in fact have 
\[
\ltsm(g') = m = \ltsm(g).
\] 

If $g \neq g'$, then $\lmsm(g-g') \prec \lmsm(g)$ and $\lmsm(g-g') \neq 0$. Let $m' = \lmsm(g-g')$; $m'$ must occur in either $g$ or $g'$. If $m'$ occurs in $g$, then, because $m'$ is the leading monomial of an element of $I$, we have that there is some $g_2 \in G$ such that $\lmsm(g_2)$ divides $m'$ in $A_{lt}$. But this contradicts the fact that $G$ was reduced. On the other hand, if $m'$ occurs in $g'$, then, by the same reasoning, we have that there is some $g'_2 \in G'$ such that $\lmsm(g'_2)$ divides $m'$ in $A'_{lt}$, contradicting the fact that $G'$ was reduced. Thus, it must be the case that $g=g'$, completing the proof.
\end{proof}

\begin{remark}
Given Proposition~\ref{thm:diff}, one might naturally wonder why bother with any algebras of leading terms other than $A_{gen}$, since any other algebras of leading terms can only yield larger reduced pseudo-ASL Gröbner bases. Although the reduced Gröbner basis relative to $A_{gen}$ will indeed be the smallest among Gröbner bases relative to all choices of $A_{lt}$, deciding divisibility and computing LCMs relative to other $A_{lt}$ may be significantly simpler than in $A_{gen}$. For example, relative to $A_{disc}$ the LCM, if it exists, is unique and is the one you would expect based on ordinary monomials (Observation~\ref{obs:lcm-unique-disc}); in contrast, relative to $A_{gen}$ LCM sets can be much more complicated, see, e.g., Theorem~\ref{thm:lcm-algo} for the case of bideterminants. We will also see that $A_{disc}$ has certain theoretical advantages, for example, when it comes to universal pseudo-ASL Gröbner bases (\S \ref{sec:universal}), $A$-modules (\S \ref{sec:sygyzies}) and computing Krull dimension (\S \ref{sec:dimension}).
\end{remark}

\subsection{Universal pseudo-ASL Gröbner bases} \label{sec:universal}

\begin{definition}[store=defuniversalgb, note={Universal \gb basis}]
\label{def:universal-gb}
Given a pseudo-ASL $A$ and an ideal $I \subseteq A$, a set $G \subseteq I$ is a \emph{universal pseudo-ASL Gröbner basis} for $I$ if $G$ is a pseudo-ASL Gröbner basis for $I$ relative to all pairs $(\preceq, A_{lt})$ consisting of a pseudo-ASL term order $\preceq$ on $A$ and an algebra of leading terms $A_{lt}$ for $(A,\preceq)$.
\end{definition}

The usual proof of existence of universal Gröbner bases relies on showing that there are only finitely many possible leading ideals for a given ideal. However, in our case, the leading ideals may reside in different algebras of leading terms, so some care must be taken. The following proof, however, does not rely on there being only finitely many algebras of leading terms, it instead uses the fact that $\Adisc$ is an algebra of leading terms for any order.

\begin{theorem}[store=thmuniversal, note={Universal \gb bases}] \label{thm:universal}
Every ideal $I \subseteq A$ in a pseudo-ASL has a finite universal pseudo-ASL Gröbner basis.
\end{theorem}

Green's Proposition~2.11 \cite{green} gives this result in the case of monomial pseudo-ASLs over a field, under the conditions that $A = A_{lt}$ (which is allowed when $A$ is a monomial pseudo-ASL) and subject to condition that the product of two standard monomials is a standard monomial (not merely a scalar multiple thereof).

We start by outlining the proof, then will give a key lemma, and give the full proof after the lemma.

\begin{proof}[Proof outline]
Note that $\Adisc$ is an algebra of leading terms for $A$ relative to \emph{all} pseudo-ASL term orders. Thus, for each fixed term order $\preceq$, by Theorem~\ref{thm:diff}, the unique reduced pseudo-ASL Gröbner basis relative to $(\preceq, \Adisc)$ is universal among pseudo-ASL Gröbner bases relative to algebras of leading terms for this particular order $\preceq$. This effectively reduces to the case of $\Adisc$, and in $\Adisc$, essentially the same proof as for ordinary Gröbner bases shows that each ideal has at most finitely many leading ideals in $\Adisc$.
\end{proof}

\begin{lemma}[store=lemuniversal] \label{lem:universal}
Given a pseudo-ASL $A$ and an ideal $I \subseteq A$, among all possible pseudo-ASL term orders, $I$ has only finitely many leading ideals in $\Adisc$.
\end{lemma}

Since the proof of the lemma is similar to the case of ordinary Gröbner bases, we relegate it to Appendix~\ref{sec:app:theory}. We note that in fact the same result holds with $A_{lt}$ in place of $A_{disc}$, as long as we restrict to those term orders for which $A_{lt}$ is in fact an algebra of leading terms; but we will only need to use the results for $A_{disc}$.

\begin{proof}[Proof of Theorem~\ref{thm:universal}]
We follow the above outline. Let $I \subseteq A$ be an ideal. By Lemma~\ref{lem:universal}, among all possible pseudo-ASL term orders, there are only finitely many leading ideals for $I$ in $\Adisc$, call these $I_1, \dotsc, I_k \subseteq \Adisc$. For each pseudo-ASL term order $\preceq$, let us say $\preceq$ has ``type $j$'' if the leading ideal of $I$ relative to $(\preceq, \Adisc)$ is $I_j$. By Lemma~\ref{lem:reduced-Adisc}, all term orders of type $j$ have the same reduced pseudo-ASL Gröbner bases; let us call this $G_j$. By Theorem~\ref{thm:diff}, if $\preceq$ is a term order of type $j$ and $A_{lt}$ is an algebra of leading terms for $(A,\preceq)$, then the reduced pseudo-ASL Gröbner basis for $I$ relative to $(\preceq, A_{lt})$ is contained in $G_j$. Thus all reduced pseudo-ASL Gröbner bases for $I$, relative to all choices of pseudo-ASL term order and algebra of leading terms, are contained in $\bigcup_{j=1}^k G_j$. As each $G_j$ is finite, $\bigcup G_j$ is thus a finite universal pseudo-ASL Gröbner basis for $I$.
\end{proof}

%% file: syzygies-modules.tex

\section{Syzygies and pseudo-ASL \gb bases for modules}
\label{sec:sygyzies}

\subsection{Definitions and setup}
Let $A$ be a pseudo-ASL, and $M$ a free $A$-module with basis $e_1, \dotsc, e_d$, so in particular $M \cong A^d$. We define a \emph{standard monomial} in $M$ to be any element of the form $m e_i$, where $m \in A$ is a standard monomial, and similarly a \emph{standard term} in $M$ is any element of the form $t e_i$ for a standard term $t \in A$. 

\begin{definition}[pseudo-ASL term order for free modules] 
\label{def:termorder-modules} 
Given a pseudo-ASL $A$ and a free module $A^d$ with basis $e_1, \dotsc, e_d$, a \emph{pseudo-ASL term order} on $A^d$ is total ordering $\preceq$ on the standard monomials of $A^d$ such that 

\begin{enumerate}[label=(\textbf{MATO-\arabic*}),ref=(\textbf{MATO-\arabic*})]
\item \label{def:termorder-modules:positive} (Positivity) $e_i \preceq m e_i$ for all $i$ and all standard monomials $m \in A$.

\item \label{def:termorder-modules:mult} (Multiplicativity in leading terms when nonzero) For all $i,j$ and all standard monomials $f,g,h \in A$, if $f e_i \prec g e_j$, then\footnote{Why not also include here $h \preceq k$ and conclude $\lmsm(fh e_i) \prec \lmsm(gk e_j)$? Because we do not have a term order on $A$ to use here! This leads us to \ref{def:termorder-modules:restrict}.}
\[
fh \neq 0 \text{ and } gh \neq 0 \Longrightarrow \lmsm(fh e_i) \prec \lmsm(gh e_j).
\]
\item \label{def:termorder-modules:restrict} (Restriction to each $Ae_i$ is a pseudo-ASL term order) For all $i$, and all standard monomials $f,g,h,k \in A$, if $f e_i \prec g e_i$ and $h e_i \preceq k e_i$, then 
\[
fh \neq 0 \text{ and } gk \neq 0 \Longrightarrow \lmsm(fh e_i) \prec \lmsm(gk e_i).
\]
\end{enumerate}
\end{definition}

\begin{remark}
As in, e.g., Green's approach to Gröbner bases for modules \cite{green}, the above definition can be gotten from Definition~\ref{def:termorder} by building a sort of ``semi-direct product algebra'' as follows. Suppose $(A,\preceq)$ is a pseudo-ASL with pseudo-ASL term order. Then we can define a pseudo-ASL structure on the Abelian group $A \oplus A^d$ as follows; we will denote the resulting algebra $A \ltimes A^d$. Let $e_1, \dotsc, e_d$ be a standard free basis for $A^d$ as an $A$-module. The copy of $A^d$ will be an ideal in $A \ltimes A^d$. Elements of the form $(a,0)$ form a sub-ring isomorphic to $A$: $(a,0)(a',0) = (aa',0)$. Elements of the form $(a,0)$ act on elements of the form $(0,*)$ as $A$ acts on the $A$-module $A^d$, viz. $(a,0) \cdot (0,me_i) = (0,ame_i)$. And finally, in the subring $A^d$, basis elements multiply according to $(0,me_i) (0, m' e_j) = (0, mm' \delta_{ij} e_i)$, where $\delta_{ij}$ is the Kronecker delta (=1 if $i=j$ and 0 otherwise). (In Green's approach, in contrast, one would have $(0,*)(0,*) = (0,0)$, but using the Kronecker $\delta$ here gives us the above definition---in particular, \ref{def:termorder-modules:restrict}---which is useful in our developments below.) 
The pseudo-ASL structure on $A \ltimes A^d$ will be as follows. Let $H,\Sigma$ be the given pseudo-ASL structure on $A$. For $i=0,\dotsc,d$, let $H_i$ be a copy of $H$, with $H_0$ generating the copy of $A$, while $H_i$ for $1 \leq i \leq d$ generates the subring $\{(0,f e_i) : f \in A\}$; let $\Sigma_i$ be an isomorphic copy of $\Sigma$ on $H_i$. Then we define the pseudo-ASL structure on $A \ltimes A^d$ by 
\[
H' := H_0 \sqcup \dotsb \sqcup H_d
\]
and
\[
\Sigma' := \left\langle \left(\bigcup_{i=0}^d \Sigma_i\right) \cup \left(\bigcup_{0 \leq i < j \leq d} (H_i \cup \{1_i\})(H_j \cup \{1_j\}) \right)\right\rangle,
\]
as an ideal in the semigroup $\W^{H'}$, where $1_i$ denotes $e_i$ if $i \geq 1$ and $1_i$ denotes the copy of $1 \in A$ if $i=0$ (including those in the above corresponds to, e.g., the straightening law $(1,0) \cdot (0,e_i) = (0,e_i)$ where the LHS is non-standard and the RHS is standard). It is readily verified that this gives a pseudo-ASL structure on $A \ltimes A^d$, in which the set of standard monomials is the disjoint union of the sets of standard monomials on each of the $d+1$ copies of $A$.\footnote{If $A$ is an ASL, this construction almost yields an ASL. The straightening rules for each $\Sigma_i$ are isomorphic to those for $\Sigma$ in $A$; the straightening rules on $H_i H_j$ are zero, and the straightening rules on $H_0 H_i$ simply encode the action of $A$ on $Ae_i$, viz. $a^{(0)} b^{(i)} = a^{(i)} b^{(i)}$. The reason this is not quite an ASL is that when we rewrite $a^{(0)} b^{(i)}$ as $a^{(i)} b^{(i)}$, we are not guaranteed that for every $x \in H_i$ there is some $y < x$ in $H_i$ occurring in $a^{(i)} b^{(i)}$. However, if we replace the strict $<$ with $\leq$, then this weakened form of the axiom is satisfied, but the weakened form of the axiom is likely to break some inductive arguments about ASLs.}
\end{remark}

Many of the basic Gröbner-basis-related properties of modules over a pseudo-ASL will follow from the observation that a pseudo-ASL term order on a free module $A^d$ restricts to an ordinary pseudo-ASL term order on each $1$-dimensional submodule $Ae_i$, under the identification $Ae_i \cong A$ (isomorphism as $A$-modules, but under this isomorphism, we may treat $Ae_i$ as a pseudo-ASL in order to apply, e.g., Definition~\ref{def:termorder} and others).

\begin{proposition}
\label{prop:ato-are-well}
Pseudo-ASL term orders on free modules $A^d$ are well-orders.
\end{proposition}

\begin{proof}
Suppose $\preceq$ is a pseudo-ASL term order on $A^d$. Let $e_1,\dotsc,e_d$ be the standard free basis of $A^d$. The restriction of $\preceq$ to each $Ae_i$ gives a pseudo-ASL term order on $A$, which we'll denote by $\preceq_i$. By Proposition~\ref{prop:well-order}, each $\preceq_i$ is a well-order on $Ae_i$. Now suppose, for the sake of contradiction, that $m_1 \succ m_2 \succ m_3 \succ \dotsb$ is a strictly decreasing chain of infinitely many standard monomials in $A^d$. Let $M_i = \{m_j \suchthat m_j \in Ae_i\}$. At least one of the $M_i$ must be infinite, say $M_i$. But then $M_i$ is an infinite strictly decreasing chain according to the pseudo-ASL term order $\preceq_i$ on $A$, contradicting Proposition~\ref{prop:well-order}.
\end{proof}

\begin{definition}[store=defmlt, note={Algebra of leading terms for a module; module of leading terms}] 
\label{def:MLT}
Given a pseudo-ASL $A$, free module $A^d$ with basis, a pseudo-ASL term order $\preceq$ on $A^d$, let $\preceq_i$ be the pseudo-ASL term order on $A$ induced by restricting $\preceq$ to $Ae_i$ and then pulling back along the unique $A$-module isomorphism $Ae_i \cong A$. We say that $A_{lt}$ is an \emph{algebra of leading terms for $(A^d, \preceq)$} if $A_{lt}$ is simultaneously an algebra of leading terms for all $(A, \preceq_i)$, $i=1,\dotsc,d$. In this case, we say that $A_{lt}^d$ is a \emph{module of leading terms for $(A^d, \preceq)$}.
\end{definition}

Since $\Adisc$ is an algebra of leading terms relative to all pseudo-ASL term orders on $A$, we know that for any free module $A^d$ that admits a pseudo-ASL term order, there is at least one module of leading terms, namely $\Adisc^d$.

\begin{convention}
Henceforth, whenever we write $(A^d, \preceq, A_{lt}^d)$, the notation is understood to mean that $A$ is a pseudo-ASL, $A^d$ is a free $A$-module with basis, and $A_{lt}^d$ is an module of leading terms for $(A^d, \preceq)$, and thereby that $A_{lt}$ is an algebra of leading terms for $(A^d, \preceq)$. Note that when $d=1$, this is consistent with our earlier convention about pseudo-ASLs (rather than modules over them). 
\end{convention}

The next observation justifies the terminology in Definition~\ref{def:MLT}, by completing the analogy with the definition of algebra of leading terms for $A$ (Definition~\ref{def:ALT}).

\begin{observation}[store=obsmlt, note={Module of leading terms}]
\label{obs:mlt}
Given $(A^d, \preceq, A_{lt}^d)$, by slight abuse of notation let $\piltd \colon A^d \to A_{lt}^d$ be the map that applies $\pi_{lt} \colon A \to A_{lt}$ coordinate-wise. Then $\piltd$ is an $R$-linear bijection that restricts to the identity map on standard monomials, and for all standard monomials $m \in A, m' e_i \in A^d$, we have
\[
\pi_{lt}(m) \piltd(m' e_i) \in \{0, \piltd(\ltsm(mm' e_i))\}.
\]
\end{observation}

\begin{proof}
All the parts except perhaps the last are immediate. For the final property, suppose $m \in A$ is a standard monomial and $m' e_i \in A^d$ is a standard monomial. Then $\pi_{lt}(m) \piltd(m' e_i) = \pi_{lt}(m)\pi_{lt}(m') e_i$, by the definition of $\piltd$ as the coordinate-wise application of $\pi_{lt}$. But by the definition of algebra of leading terms, we have $\pi_{lt}(m) \pi_{lt}(m') \in \{0, \ltsm(mm')\}$, and the result then follows from the fact that $\ltsm(mm' e_i) = \ltsm(mm') e_i$. 
\end{proof}

\subsection{Annihilators over monomial pseudo-ASLs}

We start with a lemma about syzygies of standard monomial submodules (the module analogue of standard monomial ideal) of a free module $M$ over a monomial pseudo-ASL $A'$. (This is a pseudo-ASL analogue of, e.g., \cite[Lemma~15.1]{eisenbud}.)
It is readily verified that Lemma~\ref{lem:monomial} extends to monomial submodules.  Because our algebras of leading terms won't in general be domains, we must introduce some new terminology and syzygies to take care of this. 

Given standard monomials $m_1, \dotsc, m_d \in (A')^e$, let $\varphi \colon (A')^d \to (A')^e$ be the module homomorphism that sends the standard basis elements $e_i$ to the monomials $m_i$. Then the syzygy module $\Syz(m_1, \dotsc, m_d)$ is by definition $\ker(\varphi)$. 
There are two types of syzygies that will play an important role for us. We introduce the first type here as it will be useful in the next section for our Buchberger-like criterion; the other type is introduced in Definition~\ref{def:koszul} in a later section, when we get to our pseudo-ASL analogue of Schreyer's Syzygy Theorem. 

\begin{definition}[Annihilator syzygies] \label{def:annsyz}
For any $s \in \Ann(m_i) \subseteq A'$, define 
\[
\alpha_i^s := s e_i.
\]
Let $\LAM(m_i)$ (for ``Least Annihilating Monomials'') be the set of standard monomials $m \in \Ann(m_i)$ that are minimal under divisibility. Say a binomial in $\Ann(m_i)$ is a \emph{proper binomial annihilator} if it is not the sum of two standard terms that are separately in $\Ann(m_i)$. Define $\LAB(m_i)$ (for ``Least Annihilating Binomials'') to be the set of normalized (=at least one coefficient is 1)\footnote{In the presence of a term ordering, it would make sense to say ``leading coefficient 1'' here; however, the way we have phrased it here allows these particular definitions and results to be independent of term ordering.} standard proper binomial annihilators in $\Ann(m_i)$ that are minimal under divisibility. 
\end{definition}

By the usual Noetherianity argument, both $\LAM(m_i)$ and $\LAB(m_i)$ are finite. If a monomial pseudo-ASL admits a pseudo-ASL term order, then by Observation~\ref{obs:no-binomial-ann}, we will see that $\LAB(m_i)$ is in fact empty; we include it here for completeness, even though it won't play any role in our theory of pseudo-ASL Gröbner bases.

\begin{lemma}[{Annihilators of monomials, cf. \cite{ES}}] \label{lem:monomial_ann}
Let $A'$ be a monomial pseudo-ASL, and let $m_1$ be a standard monomial in $(A')^e$. Then $\Ann(m_1)$
 is generated by 
\[
\left\{ \alpha_1^s : s \in \LAM(m_1) \cup \LAB(m_1) \right\}.
\]
\end{lemma}

(Here we are slightly abusing notation by identifying the ring $A'$ with the 1-dimensional free module $A'$, generated freely by an element we call $e_1$.)

\begin{proof}
By Observation~\ref{lem:standard-all-sigma}, this follows directly from Eisenbud \& Sturmfels \cite[Cor.~1.7]{ES}. 
\end{proof}

It follows from Lemma~\ref{lem:mon-div} (\emph{mutatis mutandis} for modules) that in the presence of a pseudo-ASL term order, there are no proper binomial annihilators. We record this as a separate observation for ease of reference: 

\begin{observation} \label{obs:no-binomial-ann}
Let $A'$ be a monomial pseudo-ASL such that $(A')^e$ admits a pseudo-ASL term order, and let $m$ be a standard monomial in $(A')^e$. Then there are no proper binomial annihilators for $m$. In particular, it follows from Lemma~\ref{lem:monomial_ann} that the annihilator of $m$ is generated by the annihilator syzygies for the least annihilating monomials alone.
\end{observation}

\subsection{A Buchberger-like criterion for pseudo-ASL Gröbner bases}
\label{sec:buchberger}
Given $(A^d, \preceq, A_{lt}^d)$ and a submodule $M \subseteq A^d$, we would like to find a pseudo-ASL \gb basis of $M$. In this section we will prove a Buchberger-like criterion for pseudo-ASL \gb bases, that will eventually lead to an algorithm in Section~\ref{sec:algorithms}.

In the polynomial ring, given two monomials, there is just one least common multiple. However, relative to an arbitrary algebra of leading terms associated to some pseudo-ASL, we need to define the notion of a set of least common multiples of two standard monomials in $A_{lt}$.

\begin{definition}[Least common standard multiples]
\label{defn:lcms}
Given $(A^d,\preceq, A_{lt}^d)$ and two standard monomials $m,m' \in A_{lt}^d$, the set of their least common standard multiples, $\lcmlt(m, m') \subseteq A_{lt}^d$, consists of the standard monomials that are common multiples of $m$ and $m'$ that are minimal under divisibility. If $m,m' \in A^d$ are standard monomials, we also write $\lcmlt(m,m')$ for $\lcmlt(\piltd(m), \piltd(m'))$.
\end{definition}

Equivalently, we can rephrase this in terms of standard monomial submodules as follows. For a monomial pseudo-ASL $A'$, call a submodule $M \subseteq (A')^d$ a \emph{standard monomial (sub)module} if for all $f \in (A')^d$, $f$ is in $M$ iff every standard monomial occurring in $f$ is in $M$. The basic results on standard monomial ideals in $A'$ extend to standard monomial submodules \emph{mutatis mutandis}. In particular, by the module analogue of Lemma~\ref{lem:unique-min}, $\lcmlt(m,m')$ is the unique minimum generating set of standard monomials for the module $\langle m \rangle \cap \langle m' \rangle$, which is standard monomial by the module analogue of Corollary~\ref{cor:intersection-monomial} (whose proof is the same as the first proof of that result, \emph{mutatis mutandis}). 

For a concrete example of a pseudo-ASL where there can be multiple LCMs for a given pair of standard monomials, see Section \ref{sec:bd-gb-algos} where in Theorem \ref{thm:lcm-algo}, given two standard bideterminants, we provide an algorithm to compute all of their least common standard multiples relative to $A_{gen}$.

\begin{observation} \label{obs:lcm-unique-disc}
Relative to $A_{disc}$, every pair of standard monomials has at most one LCM. Equivalently, if $me_i ,m'e_j \in A_{disc}^d$ are standard monomials, then $\langle me_i \rangle \cap \langle m'e_j \rangle$ is either 0 or generated by a single element.
\end{observation}

\begin{proof}
Recall that the defining ideal of $A_{disc}$ is just $\langle \Sigma \rangle \subseteq R[H]$, and hence $A_{disc}$ is the quotient of $R[H]$ by a monomial ideal. Thus the product of two standard monomials in $A_{disc}$ is either zero, or is equal to their product in $\W^H$, since no (nonzero) rewriting can occur. In particular, if $m,m'$ are standard monomials with corresponding exponent vectors $e \colon H \to \W, f \colon H \to \W$, then $m \divides m'$ iff $e(h) \leq f(h)$ for all $h \in H$. It follows that the LCM of two monomials $m$ and $m'$ in $A_{disc}$ is the same as their LCM in $R[H]$, namely the monomial whose exponent vector is $\ell(h) = \max\{e(h), f(h)\}$. If this LCM is 0 in $A_{disc}$, then the two monomials do not have any nonzero common multiples, and otherwise this LCM is the only one. For standard monomials in $A_{disc}^d$, say $m e_i$ and $m'e_j$, if $i \neq j$ then $\langle me_i \rangle \cap \langle m'e_j \rangle = 0$, and if $i = j$, then $\langle me_i \rangle \cap \langle m'e_i \rangle = \langle LCM(m,m') e_i \rangle$.
\end{proof}

Below we define a notion of S-pairs in a module over a pseudo-ASL similar to the notion of S-pairs in polynomial rings. However, since there could be multiple LCM's (Definition \ref{defn:lcms}) at our level of generality, we could in turn have multiple S-pairs for just one single pair of elements $f, g \in A$. We first generalize divisibility of leading terms to the module setting. If $m e_i ,m' e_i \in A_{lt}^d$ are standard monomials in $A_{lt}^d$ supported on the same basis element, we say that $me_i$ divides $m' e_i$, and write $me_i \divides m' e_i$, if there exists $t \in A_{lt}$ such that $tme_i = m' e_i$ (or, equivalently, $tm = m'$). 

\begin{definition}[Divisibility of leading terms in modules] \label{def:ltdivides-modules}
Given $(A^d, \preceq, A_{lt}^d)$ and $f,g \in A^d$. If $\piltd(f) \divides \piltd(g)$, we write $f \ltdivides g$. If $m,m' \in A^d$ are standard terms such that $m \ltdivides m'$, we may further write $\ltfrac{m'}{m}$ for the unique (by the module analogue of Lemma~\ref{lem:mon-div}) standard term $t \in A_{lt}$ such that $\piltd(m)t = \piltd(m')$. 
\end{definition}

\begin{definition}[S-sets over a pseudo-ASL]
\label{defn:sset}
Given $(A^d, \preceq, A_{lt}^d)$, $f, g \in A^d$, and $l$ a common multiple of their leading terms in $A_{lt}^d$, define their S-polynomial with respect to $l$:
\[
\Ssm^l(f,g) := \pi_{lt}^{-1}\left(\ltfrac{l}{\ltsm(f)}\right)\cdot f - \pi_{lt}^{-1}\left(\ltfrac{l}{\ltsm(g)}\right)\cdot g,
\]
Then define their S-set:
\[
\Ssm(f,g) := \left\{ \Ssm^l(f,g) \suchthat l \in \lcmlt(\ltsm(f), \ltsm(g))\right\} \subseteq A^d,
\]
\end{definition}

Observation \ref{obs:s-poly-reduces-lm} below records how, just like in the polynomial ring, the leading term of all the elements in the S-set of a pair of elements from $A$ will be strictly smaller (by order $\preceq$) than the corresponding LCM.

\begin{observation}
\label{obs:s-poly-reduces-lm}
For each each element of an S-set, by construction, we have, \[\ltsm\left(\pi_{lt}^{-1}\left(\ltfrac{l}{\ltsm(f)}\right)\cdot f \right) = \ltsm \left(\pi_{lt}^{-1}\left(\ltfrac{l}{\ltsm(g)}\right)\cdot g \right) = \pi_{lt}^{-1}(l),\]
thus 
\[
\ltsm\left(\Ssm^l(f,g)\right) \prec \pi_{lt}^{-1}(l).
\]
\end{observation}

\begin{proof}
Immediate consequence of Definition \ref{defn:sset}.
\end{proof}

We now prove an analogue of Buchberger's algorithm/criterion for Gr\"{o}bner bases. 

\begin{definition}[store=defnsclosed, note=S-closed] \label{defn:S-closed}
Given $(A^d, \preceq, A_{lt}^d)$, a subset $G \subseteq A^d$ is \emph{S-closed} if for all pairs $g,g' \in G$, every element of $\Ssm(g,g')$ reduces to 0 relative to $G$ (see Definition \ref{defn:standard-expression}).
\end{definition}

In $A^d$, there are lifts of the annihilator syzygies from $A_{lt}^d$ (Definition~\ref{def:annsyz}), defined as follows (we will show in Lemma~\ref{lem:syz_leading} the precise sense in which they are ``lifts''): 

\begin{definition}[store=defncompatsyzygy, note={Compatibility syzygy}]
\label{defn:compat-syzygy}
Given $(A^d, \preceq, A_{lt}^d)$ and $G \subseteq A^d$, for any $g_0 \in G$ and any standard monomial $m_0$ that is incompatible with $\ltsm(g_0)$, if $\sum_g \ell_g g$ is a standard expression for $m_0 g_0$, then
\[
m_0 e_{g_0} - \sum_{g \in G} \ell_g e_g
\]
is a \emph{compatibility syzygy (for $(m_0, g_0)$ relative to $G$)}.
\end{definition}

For the next definition, we need a preparatory observation:

\begin{observation} \label{obs:ann-std-mon}
Suppose $A'$ is a monomial pseudo-ASL. If $(A')^d$ admits a psuedo-ASL term order, then $(\Ann(t) : m)$ is a standard monomial ideal in $A'$.
\end{observation}

\begin{proof}
By definition, we have 
\[
(\Ann(t) : m) = \{f \in A' \suchthat fmt = 0\} = \Ann(mt).
\]
Since $A'$ is monomial, we have that $mt \in (A')^d$ is either 0 (in which case the colon ideal is the whole ring, which is technically a standard monomial ideal albeit not a proper one) or a standard term. In the latter case, $\Ann(mt)$ is a standard monomial ideal by Observation~\ref{obs:no-binomial-ann}, since we have assumed $(A')^d$ admits a pseudo-ASL term order.
\end{proof}

\begin{definition}[store=defannclosed,note={Ann-closed}] 
\label{def:ann-closed}
Given $(A^d, \preceq, A_{lt}^d)$, a subset $G \subseteq A^d$ is \emph{Ann-closed} if there exists a set $S$ of compatibility syzygies relative to $G$ such that, for all $g \in G$: 
\begin{enumerate}
\item \label{def:ann-closed1}for every $\pi_{lt}(s) \in \LAM(\piltd(\ltsm(g)))$, $S$ contains a compatibility syzygy for $(s,g)$; and

\item \label{def:ann-closed2}for every standard monomial $m$ that occurs in any $\ell_g$ in any syzygy in $S$, and for every standard monomial $t$ in the unique minimal generating set of the colon ideal $(\Ann(\piltd(\ltsm(g))) : \pi_{lt}(m))$, $S$ contains a compatibility syzygy for $(\lmsm(tm), g)$.
\end{enumerate}
\end{definition}
In the preceding definition, we have used Observation~\ref{obs:ann-std-mon} along with the unique minimal generating set from the module analogue of Lemma~\ref{lem:unique-min}.

\begin{theorem}[store=thmaslbuchberger2, note={Pseudo-ASL analogue of Buchberger's criterion}] 
\label{thm:buchberger2}
Given $(A^d, \preceq, A_{lt}^d)$, suppose that $G \subseteq A^d$ has the property that for all $g \in G$, $\ltsm(g)$ is a non-zerodivisor. Then $G$ is a pseudo-ASL Gröbner basis if and only if it is S-closed and Ann-closed.
\end{theorem}

In particular, the condition about the leading terms in $G$ being non-zerodivisors automatically holds when $A$ is a domain, which is a common case of application (see also Proposition~\ref{prop:domain}).

\begin{proof}
If $G$ is a pseudo-ASL Gröbner basis, then the fact that S-sets reduce to 0 and the existence of compatibility syzygies to satisfy the definition of Ann-closure are immediate from (the module analogue of) Proposition~\ref{prop:unique-remainder}, since it is clear that every element of $\Ssm(g,g')$ is in $\langle G \rangle$, as is every element of the form $mg$ (whether $m$ is compatible with $\ltsm(g)$ or not).

Conversely, suppose $G$ is S-closed and Ann-closed. We split this direction up into the following two lemmata, corresponding to the use of Ann-closure and S-closure, respectively.

\begin{lemma} \label{lem:buchberger2}
Given $(A^d, \preceq, A_{lt}^d)$, suppose that $G \subseteq A^d$ has the property that for all $g \in G$, $\ltsm(g)$ is a non-zerodivisor. If $G$ is Ann-closed, then for every standard monomial $m$ and every $g \in G$, $mg$ has a standard expression relative to $G$ with 0 remainder.
\end{lemma}

\begin{lemma} \label{lem:buchbergerS}
Given $(A^d, \preceq, A_{lt}^d)$, suppose that $G \subseteq A^d$ has the property that for all $g \in G$, $\ltsm(g)$ is a non-zerodivisor. If $G$ is S-closed and for every standard monomial $m \in A$ and every $g \in G$, $mg$ has a standard expression with 0 remainder relative to $G$, then so does every $f \in \langle G \rangle$.
\end{lemma}

Before proving these lemmata, let us see how the theorem follows. Let $f = \sum h_i g_i$ be a standard expression for $f$ relative to $G$. Since $\lmsm(h_i g_i) \preceq \lmsm(f)$ for all $i$, we must have some $i$ such that $\lmsm(h_i g_i) = \lmsm(f)$. Since $h_i$ is compatible with $\ltsm(g_i)$, by Proposition~\ref{prop:mult}, we have 
\[
\lmsm(f) = \lmsm(h_i g_i) = \lmsm(\lmsm(h_i) \lmsm(g_i)).
\]
Furthermore, since $h_i$ is compatible with $\ltsm(g_i)$, we have 
\[
\pi_{lt}(\lmsm(h_i)) \pi_{lt}(\lmsm(g_i)) = \pi_{lt}(\lmsm(f)).
\]
Thus $\pi_{lt}(\lmsm(g_i)) \divides \pi_{lt}(\lmsm(f))$, and we are done by the module analogue of Proposition \ref{prop:asl-grobner-divisibility-characterization}. Thus, all that remains to prove the theorem is to prove the two lemmata, which we do next.
\end{proof}

\begin{remark}
We note that although we have split the proof up into ``first use Ann-closure, then use S-closure'', even the S-closure proof (Lemma~\ref{lem:buchbergerS}) differs from the classical case somewhat significantly, in that it must handle lower-order terms that are identically zero in the setting of ordinary Gröbner bases in a polynomial ring.
\end{remark}

\begin{proof}[Proof of Lemma~\ref{lem:buchberger2}]
We induct over $\lmsm(mg)$, according to the given pseudo-ASL term order $\preceq$. Suppose inductively that $m' g'$ has a standard expression with 0 remainder for all standard monomials $m'$ and all $g' \in G$ such that $\lmsm(m' g') \prec \lmsm(m_0 g_0)$. We will show that $m_0 g_0$ has a standard expression relative to $G$ with 0 remainder.

We will proceed with our inner induction as follows. Let $S$ be a set of compatibility syzygies as in the definition of Ann-closure. Given a pair $(m,g)$ where $m$ is a standard monomial such that $\pi_{lt}(m) \pi_{lt}(\ltsm(g)) = 0$, we define its \emph{distance from compatibility relative to $S$} as
\[
d_S(m,g) := \min\left\{\ltfrac{m}{m'} \suchthat S \text{ has a compatibility syzygy for } (m', g)\right\}.
\]
(We will only apply this to pairs $(m,g)$ such that the set on the RHS is non-empty.) Note that $d_S(m,g)=1$ implies that there is a compatibility syzygy for $(m,g)$ in $S$, and hence by definition $mg$ is has a standard expression relative to $G$ with zero remainder.

In order to show that $m_0 g_0$ has a standard expression in terms of $G$ with 0 remainder, we proceed by induction on the distance from compatibility relative to $S$, starting with $d_S(m_0, g_0)$. Since $\pi_{lt}(m_0) \pi_{lt}(\ltsm(g_0)) = 0$, there is some $\pi_{lt}(s) \in \LAM(\pi_{lt}(\ltsm(g_0)))$ such that $\pi_{lt}(s) \divides \pi_{lt}(m_0)$, and thus $d_S(m_0,g_0)$ exists. Let $t = \pi_{lt}^{-1}(d_S(m_0, g_0))$, and let $m_1 = \pi_{lt}^{-1}\left(\ltfrac{m_0}{t}\right)$.
Then $S$ has a compatibility syzygy for $(m_1, g_0)$, say 
\[
m_1 g_0 = \sum_{g \in G} \ell_g g.
\]
Multiplying by $t$, we get
\[
t m_1 g_0 = t\left(\sum_{g \in G} \ell_g g \right).
\]
Now, as we only have $m_0 = \lmsm(t m_1)$, it may be the case that $tm_1$ has additional lower-order terms than $m_0$. However, any such term $m_2 \in t m_1$ with $m_2 \neq m_0$ then has $m_2 \prec m_0$. Thus, since we have assumed $\ltsm(g_0)$ is a non-zerodivisor, 
we have $\lmsm(m_2 g_0) = \lmsm(m_2 \lmsm(g_0))$ by Prop.~\ref{prop:mult}, and similarly $\lmsm(m_0 g_0) = \lmsm(m_0 \lmsm(g_0))$. Thus by \ref{def:termorder:mult} we get 
\[
\lmsm(m_2 g_0)  = \lmsm(m_2 \lmsm(g_0)) \prec \lmsm(m_0 \lmsm(g_0)) = \lmsm(m_0 g_0).
\] 
Therefore, by our (outer) inductive assumption, all the lower-order terms $m_2 \in tm_1$ are such that $m_2 g_0$ has a standard expression with 0 remainder. And since $\lmsm(m_2 g_0) \preceq \lmsm(m_0 g_0)$, those standard expressions thus have all of their terms $\preceq \lmsm(m_0 g_0)$, so they may legitimately be part of a standard expression for $m_0 g_0$ as well. Thus, $tm_1 g_0$ has a standard expression with 0 remainder if and only if $m_0 g_0$ does.

If every standard term occurring in $t\ell_g$ is compatible with $\ltsm(g)$ for all $g \in G$, then the expression $\sum_j (t \ell_j) g_j$ is already standard, and we are done. (While not needed for the proof, it may be helpful for intuition to see that this is automatically the case if $t=1$, which happens iff $S$ already contains a compatibility syzygy for $(m_0, g_0)$.)

Otherwise, there is some $g \in G$, standard monomial $m_2 \in \ell_g$ and some standard monomial $m_3 \in tm_2$ such that $m_3$ is incompatible with $\ltsm(g)$. We break the final part of the argument into cases.

Case 1 [$m_3 = \lmsm(tm_2)$]: In this case, we have that $\pi_{lt}(t) \pi_{lt}(m_2)$ is either 0 or $\pi_{lt}(m_3)$, and in either case, $\pi_{lt}(t) \pi_{lt}(m_2)$ is in $\Ann(\piltd(\ltsm(g)))$. Thus by Ann-closure, there is some $s_2$ in the minimal generating set of $(\Ann(\piltd(\ltsm(g))) \suchthat \pi_{lt}(m_2))$ such that $s_2 \mid_{lt} t$ and there is a compatibility syzygy in $S$ for $(\ltsm(s_2 m_2), g)$. Using the latter syzygy, since $s_2 \mid_{lt} t$ and $m_3 = \ltsm(tm_2)$, we find that 
\[
d_S(m_3, g) \preceq \ltfrac{m_3}{\ltsm(s_2 m_2)} = \ltfrac{\ltsm(t m_2)}{\ltsm(s_2 m_2)} = \ltfrac{t}{s_2}.
\]
We will now show that $d_S(m_3,g)$ is \emph{strictly} less than $t$. From the preceding equation, this is equivalent to showing that $s_2 \succ 1$. The latter follows since $m_2$ is in $\ell_g$, so is compatible with $\ltsm(g)$ by assumption, and therefore $1$ is not in $(\Ann(\piltd(\ltsm(g))) \suchthat \pi_{lt}(m_2))$, so $s_2 \neq 1$. Thus we have $d_S(m_3, g) \prec t$, which implies, by our inner inductive assumption, that $m_3 g$ has a standard expression with zero remainder, 
and we are done.

Case 2 [$m_3 \neq \lmsm(t m_2)$]: Then we have 
\begin{equation} \label{eq:Ann1}
m_3 \prec \ltsm(tm_2) \preceq \ltsm(t\ell_g).
\end{equation}
Multiplying the first and last parts of this inequality by $\ltsm(g)$, and using our assumption that $\ltsm(g)$ is a non-zerodivisor, \ref{def:termorder-modules:mult} implies that 
\begin{equation} \label{eq:Ann2}
\ltsm(m_3 \ltsm(g)) \prec \ltsm(\ltsm(t \ell_g) \ltsm(g)).
\end{equation} 
Now, since $\ltsm(g)$ is a non-zerodivisor, by Prop.~\ref{prop:mult} the LHS here is equal to $\ltsm(m_3 g)$ and the RHS is equal to $\ltsm(t\ell_g g)$, so we have
\[
\ltsm(m_3 g) \prec \ltsm(t \ell_g g).
\]
Finally, since $\sum_g \ell_g g$ was a standard expression for $m_1 g_0$ with $m_1 \ltdivides m_0$, we have $\ltsm(\ell_g g) \preceq \ltsm(m_1 g_0)$. Multiplying both sides of this equation by $t$, and since both $\ltsm(g)$ and $\ltsm(g_0)$ are non-zerodivisors, we finally get $\ltsm(t \ell_g g) \preceq \ltsm(m_0 g_0)$. Combining with the above equation, this gives us $\ltsm(m_3 g) \prec \ltsm(m_0 g_0)$, and thus by our (outer) inductive assumption, $m_3 g$ has a standard expression with 0 remainder. This completes the proof of Lemma~\ref{lem:buchberger2}.
\end{proof}

\begin{proof}[Proof of Lemma~\ref{lem:buchbergerS}]
Write $G = \{g_1, \dotsc, g_s\}$. Let $f \in \langle G \rangle$; say $f = \sum_{i \in [s]} h'_i g_i$, where $h'_i \in A$. By Lemma~\ref{lem:buchberger2}, if $m$ is a standard monomial occurring in $h'_i$, then $m g_i$ has a standard expression relative to $G$ with 0 remainder. 
Thus, term-by-term within each $h'_i$, we may replace the preceding expression with 
\begin{equation}\label{eqn:f-intermsof-G} 
f = \sum_{i\in [s]}h_ig_i,
\end{equation} 
where each $h_i$ is compatible with $\ltsm(g_i)$. Let $\delta = \max_{i \in [s]}\lmsm(h_ig_i)$. (Note that while the preceding replacement process did not increase any of the leading terms relative to their initial values in the expression $\sum h'_i g_i$, it is possible that some of those leading terms are in fact strictly larger than $\lmsm(f)$.) If $\delta = \lmsm(f)$, then we have $\lmsm(h_i g_i) \preceq \lmsm(f)$ for all $i$, so the preceding is in fact a standard expression with 0 remainder 
and we are done. 

Thus, all that remains is to handle the case in which $\delta$ is strictly greater than $\lmsm(f)$. We will handle this case by showing that $f$ has another expression $\sum_{i \in [s]} h''_i g_i$ with $h''_i$ compatible with $\ltsm(g_i)$ for all $i$, and such that $\max_{i \in [s]} \ltsm(h''_i g_i)$ is strictly less than $\delta$. By induction on $\delta$, this procedure eventually gets to the base case where $\delta = \lmsm(f)$, completing the proof.

Toward that end, let $Q \subseteq [s]$ be the set of all $q \in [s]$ such that $\lmsm(h_qg_q) = \delta$; $Q$ is non-empty by the definition of $\delta$. Then we can re-write Equation \eqref{eqn:f-intermsof-G} as \begin{align}
f &= \sum_{q\in Q} h_qg_q + \sum_{i \in [s]\setminus Q}h_ig_i \nonumber \\
&= \underbrace{\sum_{q \in Q} \ltsm(h_q)g_q}_{\mathcal{A}} + \underbrace{\sum_{q \in Q} \left(h_q - \ltsm(h_q)\right)g_q}_{\mathcal{B}} + \underbrace{\sum_{i \in [s]\setminus Q}h_ig_i}_{\mathcal{C}}. \label{eqn:f-isolate}
\end{align}

We handle the three collections of terms $\mathcal{A}, \mathcal{B}, \mathcal{C}$ in reverse order:
\begin{itemize}
\item By definition of $Q$, the leading standard monomials of all summands $h_i g_i$ in $\mathcal{C}$ are strictly smaller than $\delta$. 

\item The leading standard monomials of all summands $(h_q - \ltsm(h_q))g_q$ of $\mathcal{B}$ are also strictly smaller than $\delta$: for, since $h_q$ is compatible with $\lmsm(g_q)$, by Proposition~\ref{prop:mult} we have 
\[
\delta = \lmsm(h_q g_q) = \lmsm(h_q \lmsm(g_q)),
\]
 and then we may apply Observation \ref{claim:lmsm-less-than-kappa} stated below to $f_1=h_q$ and $f_2 = \lmsm(g_q)$.
\end{itemize}

\begin{observation}[store=claimlmsmlessthankappa]
\label{claim:lmsm-less-than-kappa}
Suppose $f_1 \in A$ is compatible with $f_2 \in A^d$. Let $\tail(f_1) = f_1 - \ltsm(f_1)$. Then 
\[
\lmsm\left(\tail(f_1) f_2\right) \prec \lmsm(f_1 f_2).
\]
\end{observation}
See Appendix \ref{sec:app:theory} for the proof. (While it seems intuitively obvious and the proof is straightforward, one must use compatibility carefully in order to be able to apply Proposition~\ref{prop:mult} and \ref{def:termorder-modules:mult}.)

The remainder of the proof is devoted to handling the terms in $\mathcal{A}$. 

\begin{claim}[store=claimrewriteAintermsofS]
\label{claim:rewrite-A-in-terms-of-S}
We can rewrite $\mathcal{A}$ as follows.
For any $q_0 \in Q$, let $g_{\qn} := g_{q_0}$, then we have:
\begin{equation}
\mathcal{A} = \sum_{q \in Q} \lcsm(h_qg_q) \cdot \Ssm^\delta \left(g_{q}, g_{\qn}\right). \label{eqn:A-as-selem-sum} 
\end{equation}
\end{claim}

\begin{proof}[Proof of claim]
Since $\lmsm(f) \prec \delta$, and all terms in $\mathcal{A}$ have leading term $\delta$, and all terms in $\mathcal{B}$ and $\mathcal{C}$ have leading term $\prec \delta$, the $\delta$ terms in $\mathcal{A}$ must all cancel, i.\,e., we deduce that 
\begin{equation} \label{eqn:coeff-sum-zero} \sum_{q \in Q} \lcsm(h_qg_q) = 0.\end{equation} 
(N.b.! Even if all $g_q$ have $\lcsm(g_q)=1$, the above equation need not be the same as $\sum_q \lcsm(h_q) = 0$, as happens in the case of ordinary Gröbner bases, because in a pseudo-ASL, even for two standard monomials $m,m'$, we may have $\lcsm(mm') \neq 1$.)

Now, let $q_0$ be any element of $Q$, $g_{\qn} := g_{q_0}$ and $h_{\qn} := h_{q_0}$ for brevity and distinction. We can re-write $\mathcal{A}$ as
\begin{align}
\mathcal{A} &= \sum_{q \in Q} \ltsm(h_q)g_q \nonumber \\
&= \left(\sum_{q \in Q} \ltsm(h_q)g_q\right) - \frac{\ltsm(h_{\qn})g_{\qn}}{\lcsm(h_{\qn}g_{\qn})} \cdot 0 \nonumber \\
&= \left(\sum_{q \in Q} \ltsm(h_q)g_q\right) - \left(\frac{\ltsm(h_{\qn})g_{\qn}}{\lcsm(h_{\qn}g_{\qn})} \right) \cdot \sum_{q \in Q}\lcsm(h_qg_q) \eqcomment{by \eqref{eqn:coeff-sum-zero}} \nonumber \\
&= \sum_{q \in Q} \lcsm(h_qg_q)\left(\frac{\ltsm(h_q)g_q}{\lcsm(h_qg_q)} - \frac{\ltsm(h_{\qn})g_{\qn}}{\lcsm(h_{\qn}g_{\qn})}\right). \label{eqn:almost-spoly} \\
\end{align}

We will now show that each summand in \eqref{eqn:almost-spoly} is an (the) S-polynomial of $\ltsm(h_q)g_q$ and $\ltsm(h_{\qn}) g_{\qn}$. 
For since $\lmsm(\ltsm(h_q) g_q)$ is the same for all $q \in Q$ (namely, $\delta$), they are all equal to their common LCM in $A_{lt}^d$, and that LCM is unique, namely $\piltd(\delta)$. Straightforward calculation shows that 
\[
\Ssm^\delta (\ltsm(h_q) g_q, \ltsm(h_{q_0}) g_{q_0}) = \frac{\ltsm(h_{q})g_{q}}{\lcsm(h_{q}g_q)} - \frac{\ltsm(h_{\qn})g_{\qn}}{\lcsm(h_{\qn}q_{\qn})}.
\]

Finally, we show that $\Ssm^{\delta}(\ltsm(h_q) g_q, \ltsm(h_{\qn}) g_{\qn})$ is in fact equal to $\Ssm^\delta(g_q, g_{\qn})$. We have:
\[
\Ssm^{\delta}(g_q, g_{\qn}) = \pi_{lt}^{-1}\left(\ltfrac{\delta}{\ltsm(g_q)} \right) g_q - \pi_{lt}^{-1}\left(\ltfrac{\delta}{\ltsm(g_{\qn})} \right) g_{\qn} 
\]
and
\[
\Ssm^{\delta}(\ltsm(h_q) g_q, \ltsm(h_{\qn}) g_{\qn}) = \frac{\ltsm(h_q) g_q}{\lcsm(h_q g_q)} - \frac{\ltsm(h_{\qn}) g_{\qn}}{\lcsm(h_{\qn} g_{\qn})}.
\]
In both cases, the coefficient of $g_q$ is the unique (by Obs.~\ref{obs:no-binomial-ann}) standard monomial $m$ such that $\ltsm(m g_q) = \pi_{lt}^{-1}(\delta)$, and similarly for $g_{\qn}$. Thus we must have
\[
\pi_{lt}^{-1}\left(\ltfrac{\delta}{\ltsm(g_q)} \right) = \frac{\ltsm(h_q)}{\lcsm(h_q g_q)}  \qquad \text{ and } \qquad \pi_{lt}^{-1}\left(\ltfrac{\delta}{\ltsm(g_{\qn})} \right)  = \frac{\ltsm(h_{\qn})}{\lcsm(h_{\qn} g_{\qn})},
\]
showing that $\Ssm^\delta(g_q, g_{\qn}) = \Ssm^{\delta}(\ltsm(h_q) g_q, \ltsm(h_{\qn}) g_{\qn})$. Plugging this into \eqref{eqn:almost-spoly} finishes the proof of the claim.
\end{proof}

The remainder of the proof of Lemma~\ref{lem:buchbergerS} (and, hence, Theorem~\ref{thm:buchberger2}) is then to show that the S-polynomials arising in Claim~\ref{claim:rewrite-A-in-terms-of-S} can be rewritten in terms of S-polynomials of the $g_q$ relative to their LCMs (rather than $\delta$ as above, which is a common multiple but not necessarily a \emph{least} common multiple) plus lower-order terms. (These lower-order terms are another place where this proof extends the usual proof; for ordinary Gröbner bases in a polynomial ring---or even in a monomial pseudo-ASL---the aforementioned lower-order terms that arise in our proof vanish identically and so need not be dealt with.)

At this point we already have that the leading term of each $\Ssm^\delta\left(g_{q}, g_{\qn}\right)$ is strictly less than $\delta$. What remains is to show that each of the $\Ssm^\delta\left(g_{q}, g_{\qn}\right)$ terms can be written as $\sum h''_i g_i$ in which each $h''_i$ is compatible with $\ltsm(g_i)$ and in which every summand $h''_i g_i$ also has leading monomial strictly less than $\delta$.

By definition, for all $q \in Q$, $\piltd(\delta)$ is in $\langle \piltd(\lmsm(g_q)) \rangle \cap \langle \piltd(\lmsm(g_{\qn})) \rangle$, thus there exists 
\[
l_q \in \lcmlt(\lmsm(g_q), \lmsm(g_{\qn}))
\]
such that $l_q \divides \piltd(\delta)$. For all $q \in Q$, let $\mu_q := \pi_{lt}^{-1}\left(\ltfrac{\delta}{l_q}\right) \in A$. We will attempt to rewrite each $\Ssm^\delta( g_q,  g_{\qn})$ in terms of $\mu_q \Ssm^{l_q}(g_q, g_{\qn})$, and see that the difference between the two only involves lower-order terms.

\begin{lemma} \label{lem:Ssm-prod}
Let $g,g' \in A^d$, $l$ a standard monomial in $\langle\piltd(\ltsm(g))\rangle \cap \langle \piltd(\ltsm(g'))\rangle$, and $\mu \in A$ a standard monomial that is compatible with $l$. Then there exist $h,h' \in A$ such that
\[
\mu \Ssm^{l}(g,g') = \Ssm^{\pi_{lt}(\mu) l}(g,g') + hg + h'g'
\]
with $\ltsm(hg)$ and $\ltsm(h'g')$ both strictly less than $\ltsm(\mu \pi_{lt}^{-1}(l))$.
\end{lemma}

\begin{proof}
By definition, we have
\begin{align*}
\mu \Ssm^{l}(g, g') & = \mu\left(\pi_{lt}^{-1}\left(\ltfrac{l}{\ltsm(g)}\right) g - \pi_{lt}^{-1}\left(\ltfrac{l}{\ltsm(g')}\right) g' \right) \\
 \end{align*}
 At this point, since we will use them frequently throughout this derivation, we introduce the shorter notation $\alpha := \pi_{lt}^{-1}\left(\ltfrac{l}{\ltsm(g)}\right)$ and $\beta := \pi_{lt}^{-1}\left(\ltfrac{l}{\ltsm(g')}\right)$. Continuing with this new notation in hand, we have
\begin{align}
 \mu \Ssm^{l}(g, g') & = \mu \alpha g - \mu \beta g' \nonumber \\
  & = \left(\ltsm(\mu \alpha) + \tail(\mu \alpha) \right) g - \left(\ltsm(\mu \beta) + \tail(\mu \beta) \right) g' \nonumber \\
  & = \left(\ltsm(\mu \alpha) g - \ltsm(\mu \beta) g'\right) + \left(\tail(\mu \alpha) g - \tail(\mu \beta) g'\right)  \label{eq:mussm}\\
\end{align}
We claim that the first summand here is $\Ssm^{\pi_{lt}(\mu) l}(g,g')$, and that we may take $h = \tail(\mu \alpha)$ and $h' = -\tail(\mu \beta)$ to satisfy the conclusion of the lemma. To prove both of these, we now analyze $\ltsm(\mu \alpha)$ and $\ltsm(\mu \beta)$.

We have
\begin{align*}
\pi_{lt}\left(\ltsm(\mu \alpha)\right) & = \pi_{lt}\left(\ltsm\left(\mu \pi_{lt}^{-1}\left(\ltfrac{l}{\ltsm(g)} \right) \right)\right) \\
 & = \pi_{lt}\left( \mu \right) \pi_{lt}\left( \pi_{lt}^{-1}\left(\ltfrac{l}{\ltsm(g)} \right)\right)  \eqcommentnear{by Obs.~\ref{obs:mlt}} \\
 & = \pi_{lt}(\mu) \cdot \ltfrac{l}{\ltsm(g)} \\
 & = \pi_{lt}(\mu) \frac{l}{\pi_{lt}(\ltsm(g))} \eqcomment{by Def.~\ref{def:ltdivides-modules}} \\
& = \frac{\pi_{lt}(\mu) l}{\pi_{lt}(\ltsm(g))} 
\end{align*}

It follows that $\ltsm(\ltsm(\mu \alpha) g) = \pi_{lt}(\mu) l$. A similar calculation yields that $\ltsm(\ltsm(\mu \beta) g') = \pi_{lt}(\mu) l$. It follows that
\[
\left(\ltsm(\mu \alpha) g - \ltsm(\mu \beta) g'\right) = \Ssm^{\pi_{lt}(\mu) l}(g,g')
\]
and $\ltsm(\tail(\mu \alpha) g)$ and $\ltsm(\tail(\mu \beta) g')$ are both strictly less than $\ltsm(\mu \pi_{lt}^{-1}(l))$. This completes the proof of the lemma.
\end{proof}

Returning to our proof of Lemma~\ref{lem:buchbergerS}, we apply Lemma~\ref{lem:Ssm-prod} with $g = g_q, g' = g_{\qn}, l = l_q, \mu = \mu_q$. As we have $\delta = \pi_{lt}(\mu_q) l_q$, the conclusion of Lemma~\ref{lem:Ssm-prod} gives us
\[
\mu_q \Ssm^{l_q}(g_q,g_{\qn}) = \Ssm^{\delta}(g_q,g_{\qn}) + hg_q + h'g_{\qn}
\]
where $\ltsm(h g_q)$ and $\ltsm(h' g_{\qn})$ are both strictly less than $\pi_{lt}^{-1}(\delta)$. By our assumption, for each standard monomial $m$ in $h$, the product $m g_q$ can be rewritten in terms of a standard expression with 0 remainder; similarly for the monomials $m'$ occurring in $h'$, and the product $m' g_{\qn}$. Adding such expressions together gives us an expression in which all terms multiplied by $g_i$ are compatible with $\ltsm(g_i)$, and all summands have leading terms strictly less than $\delta$.

Since $G$ is S-closed, there is a standard expression $\Ssm^{l_q}(g_q, g_{\qn}) = \sum_{i \in [s]} k_i g_i$. Let $k'_i$ be the sum of those standard terms $t$ in $k_i$ such that $\mu_q t \neq 0$; then $\sum k'_i g_i$ still has $k'_i$ compatible with $\ltsm(g_i)$ for all $i$, and we have
\begin{equation} \label{eq:muq-s-lq}
\mu_q \Ssm^{l_q}(g_q, g_{\qn}) = \sum_{i \in [s]} \mu_q k_i g_i = \sum_{i \in [s]} \mu_q k'_i g_i.
\end{equation}

\begin{claim} \label{claim:delta}
For all $i$, if $\mu_q k'_i g_i \neq 0$, then $\ltsm(\mu_q k'_i g_i) \prec \pi_{lt}^{-1}(\delta)$.
\end{claim}

\begin{proof}[Proof of claim]
Assume $\mu_q k'_i g_i \neq 0$. Since every standard term in $k'_i$ occurs in $k_i$ by construction, we get $\ltsm(k'_i g_i) \preceq \ltsm(k_i g_i) \preceq \ltsm(\Ssm^{l_q}(g_q, g_{\qn})) \prec l_q$ (the latter by Observation~\ref{obs:s-poly-reduces-lm}). We now multiply both ends of this inequality by $\mu_q$. 

Since $k'_i$ is compatible with $\ltsm(g_i)$, we have $\ltsm(k'_i) \ltsm(g_i) \neq 0$ hence $\ltsm(k'_i g_i) = \ltsm(\ltsm(k'_i) \ltsm(g_i)) \neq 0$ by \eqref{obs:mult-eqn1}. Multiplying both sides by $\mu_q$ and taking leading terms, we get
\begin{equation} \label{eq:im-so-tired-of-coming-up-with-temporary-names1}
\ltsm(\mu_q \ltsm(k'_i g_i)) = \ltsm(\mu_q \ltsm(\ltsm(k'_i) \ltsm(g_i))).
\end{equation}

Next, by Lemma~\ref{lem:triplezero}, we have that if $\mu_q \ltsm(k'_i) \ltsm(g_i) = 0$, then $\ltsm(\mu_q k'_i) \ltsm(g_i) = 0$. But as we have assumed $\ltsm(g_i)$ is not a zerodivisor, the latter does not occur, and thus $\mu_q \ltsm(k'_i) \ltsm(g_i) \neq 0$. Then by \eqref{obs:mult-eqn3}, we get
\[
\ltsm(\mu_q \ltsm(\ltsm(k'_i) \ltsm(g_i))) = \ltsm(\mu_q \ltsm(k'_i) \ltsm(g_i)).
\]
Combining the preceding with \eqref{eq:im-so-tired-of-coming-up-with-temporary-names1} gives
\begin{align*}
\ltsm(\mu_q \ltsm(k'_i g_i)) & = \ltsm(\mu_q \ltsm(k'_i) \ltsm(g_i)) \\
 & = \ltsm(\ltsm(\mu_q \ltsm(k'_i)) \ltsm(g_i)), 
\end{align*}
where the last equality follows by \eqref{obs:mult-eqn1} applied to $f=\mu_q \ltsm(k'_i)$ (which is nonzero by construction of $k'_i$) and $g = \ltsm(g_i)$. Continuing, by construction of $k'_i$ we have $\mu_q \ltsm(k'_i) \neq 0$, hence $\ltsm(\mu_q \ltsm(k'_i)) = \ltsm(\mu_q k'_i)$ by \eqref{obs:mult-eqn1}. Combining with the preceding, we then have
 \[
 \ltsm(\mu_q \ltsm(k'_i g_i)) = \ltsm(\ltsm(\mu_q k'_i) \ltsm(g_i)).
 \]
 Since $\ltsm(g_i)$ is assumed to not be a zerodivisor, we may apply \eqref{obs:mult-eqn1} with $f = \mu_q k'_i$ and $g = g_i$ to see that the preceding is equal to $\ltsm(\mu_q k'_i g_i)$. Putting together the previous several paragraphs, we then we have
\begin{align*}
\pi_{lt}^{-1}(\delta) = \ltsm(\mu_q l_q) & \succ \ltsm(\mu_q \ltsm(k'_i g_i)) \eqcomment{by \ref{def:termorder:mult}} \\
& = \ltsm(\mu_q k'_i g_i) 
\end{align*}
This completes the proof of Claim~\ref{claim:delta}. 
\end{proof}

Finally, putting everything all together, we have
\begin{align} \label{eq:smaller}
\Ssm^{\delta}(g_q,g_{\qn})  & = \mu_q \Ssm^{l_q}(g_q,g_{\qn}) - hg_q - h'g_{\qn}, 
\end{align}
where we have shown how to write the RHS as an $A$-linear combination of elements of $G$ such that every summand has all terms strictly less than $\delta$. Claim~\ref{claim:rewrite-A-in-terms-of-S} shows that $\mathcal{A}$ can be written as an $R$-linear combination of terms of the form $\Ssm^{\delta}(g_q, g_{\qn})$, and by \eqref{eq:smaller}, we can rewrite each such term in terms of $A$-linear combinations of $G$ in which each summand is strictly less than $\delta$. By induction, the $f$ we started with thus has a standard expression with zero remainder, completing the proof of Lemma~\ref{lem:buchbergerS} (thus completing the proof of Theorem~\ref{thm:buchberger2}).
\end{proof}

\begin{corollary}[of Theorem \ref{thm:buchberger2}]
\label{cor:buchberger-agen}
Given $(A^d, \preceq, A_{lt}^d)$, suppose that $A_{lt} \cong_{sm} A_{gen}^{\preceq_i}$ for all $i$, where $\preceq_i$ is the restriction of $\preceq$ to $A e_i \subseteq A^d$ and $A_{gen}^{\preceq_i}$ denotes the generic algebra of leading terms relative to $(A,\preceq_i)$. If $\ltsm(g)$ is a non-zerodivisor for all $g \in G$, then $G$ is a pseudo-ASL Gröbner basis if and only if it is S-closed.
\end{corollary}

\begin{proof}
By assumption, the leading terms of $g \in G$ are not zerodivisors in $A$, hence the same is true in $A_{gen}$. Thus the annihilating ideals of all standard monomials in $A_{gen}$ are 0, so $G$ is vacuously Ann-closed.
\end{proof}

By contrasting the preceding corollary with Observation~\ref{obs:lcm-unique-disc}, we see again a tradeoff in using different algebras of leading terms, viz.: relative to $A_{gen}$, S-closure suffices to be a Gröbner bases but LCMs need not be unique, while relative to $A_{disc}$ LCMs are unique, but one must have both S-closure and Ann-closure to be a Gröbner basis.

\subsubsection{Towards dropping the nonzerodivisor condition}
It is plausible to us that the condition in Theorem~\ref{thm:buchberger2} that the leading terms of $G$ be non-zerodivisors could be dropped with some additional work. Although we are not as yet able to prove this, here we develop some preparatory lemmata that are headed in this direction, and end this section with a remark sketching some ideas and obstacles towards such a more general result.

\begin{lemma} \label{lem:ann-monomial}
Given $(A^d, \preceq$), if $m \in A^d$ is a standard monomial, then $\Ann_A(m) \subseteq A$ is a ``standard monomial ideal'', in the sense that for all $f \in A$, we have $f \in \Ann_A(m)$ iff every standard monomial $m'$ occurring in $f$ is also in $\Ann_A(m)$.
\end{lemma}

\begin{proof}
Suppose $f \in \Ann_A(m)$. Write $f$ as a sum of standard terms: $f = \sum_i t_i$. Consider $\delta := \max_i \lmsm(t_i m) \in A^d$. If $\delta=0$ then we are done, since then each $t_i$ annihilates $m$. Otherwise, among all $i$ such that $\lmsm(t_i m)$, we must have that the sum of the coefficients of $\delta$ vanish, and thus there must be at least two $i$, say $i$ and $i'$, such that $\delta = \lmsm(t_i m) = \lmsm(t_{i'} m)$. But the latter contradicts Observation~\ref{obs:no-binomial-ann}: let $e_i$ be the vector supporting $m$ (and hence also $\delta$), let $\preceq_i$ be the term order restricted to $Ae_i$, and let $A_{gen}$ be the generic algebra of leading terms relative to $(A,\preceq_i)$. Write $m = m' e_i$ where $m' \in A$. Then in $A_{gen}$ we have $\pi_{gen}(t_i) \pi_{gen}(m') = \pi_{gen}(t_{i'}) \pi_{gen}(m')$, contradicting Observation~\ref{obs:no-binomial-ann}. Thus it must be the case that for each $i$, we have $t_i m = 0$, and thus $t_i \in \Ann_A(m)$.
\end{proof}

In particular, we note that in the proof of Lemma~\ref{lem:buchberger2} we only use that $\ltsm(g)$ times any standard monomial is nonzero; Lemma~\ref{lem:ann-monomial} shows that such a condition is equivalent to $\ltsm(g)$ being a non-zerodivisor.

\begin{lemma} \label{lem:prep}
Given $(A,\preceq)$ (a not-necessarily-monomial psuedo-ASL and term order), suppose that $I \subseteq A$ is a standard monomial ideal in the sense of Lemma~\ref{lem:ann-monomial}. Then there is a finite set $m_1, \dotsc, m_s$ of standard monomials in $I$ such that for all standard monomials $m$, we have $m \in I$ iff $\exists i \in \{1,\dotsc,s\}$ such that $m_i \divides m$ (n.\,b.: divides in $A$).
\end{lemma}

In particular, this shows that ``standard monomial ideals'' in the sense of Lemma~\ref{lem:ann-monomial} are also generated by standard monomials. The converse (if an ideal in a pseudo-ASL is generated by standard monomials, then it is a ``standard monomial ideal'' in the sense of Lemma~\ref{lem:ann-monomial}) is false in general, see Example~\ref{ex:singleton}.

\begin{proof}
We will build an auxiliary monomial pseudo-ASL $A'$ as follows (it will be similar to, but not quite the same as, an algebra of leading terms). Suppose the pseudo-ASL structure on $A$ is given by $(H, \Sigma)$. For each $\sigma \in \Sigma$, suppose $f \in R[H]$ is the straightening law $\sigma - \sum_i t_i$, where the $t_i$ are standard terms. If the RHS of the straightening law for $\sigma$ consists of a single standard monomial---that is, $f = \sigma - t$---then in $A'$ the same straightening law applies; if the straightening law for $\sigma$ consists of 0 or at least 2 standard monomials, then in $A'$ we mod out by $\sigma$. (This has the effect of allowing ``incompatible'' products of standard monomials so long as their product in $A$ is still a standard monomial, but making all other incompatible products equal to $0$.) 

By construction, it is clear that $A'$ is a monomial pseudo-ASL that is also governed by the same $(H, \Sigma)$. Let $\pi \colon A \to A'$ be the unique $R$-linear map that is the identity on standard monomials. 

We claim that $\pi(I)$ is an ideal in $A'$. Since $A'$ is a monomial pseudo-ASL, for this it suffices to show that for any standard monomial $m \in \pi(I)$ and any standard monomial $m' \in A'$, we have $mm' \in \pi(I)$. If $mm' = 0$, then certainly $mm'$ is in $\pi(I)$. Otherwise, let $n = \pi^{-1}(m), n' = \pi^{-1}(m')$; we have $n \in I$. Since $mm' \neq 0$, by construction of the straightening rules in $A'$ this is only possible if either $mm'$ was already standard in $\W^H$, in which case the product $nn'$ in $A$ also involves no straightening, or if $mm' \in \W^H$ is non-standard, and the straightening rule in $A$ (sic!) has $nn' = t$ for some standard term $t$. In either case, we have $nn'$ is a standard monomial in $A$ and $\pi(nn') = mm'$. Since $n$ is in $I$, $nn'$ is also in $I$, and thus $mm'$ is in $\pi(I)$.

Since an element of $A$ is in $I$ iff all its standard monomials are, the same is true of $\pi(I)$. Since $A'$ is a monomial pseudo-ASL by construction, we thus have that $\pi(I)$ is a standard monomial ideal in a monomial pseudo-ASL, that is, in our usual sense.

Finally, by Lemma~\ref{lem:unique-min}, let $m'_1, \dotsc, m'_s$ be the unique minimum generating set of $\pi(I)$. Let $m_i = \pi^{-1}(m'_i)$ for $i=1,\dotsc,s$. We show that $\{m_1,\dotsc,m_s\}$ satisfies the conclusion of the lemma. Suppose $m \in I$ is a standard monomial; then $\pi(m)$ is divisible by some $m'_i$ by construction; let $t$ be a standard term such that $m'_i \pi(t) = \pi(m)$, which exists by Lemma~\ref{lem:mon-div}. Then since $m'_i \pi(t) \neq 0$, by construction of the straightening rules for $A'$ we must have $m_i t = m$, and thus $m_i \divides m$ in $A$, as claimed.
\end{proof}

\begin{remark}[Towards dropping the non-zerodivisor condition from Thm.~\ref{thm:buchberger2}]
Lem.~\ref{lem:prep} suggests that one could fruitfully modify the first condition in the definition of Ann-closure to use the set of standard monomials from Lemma~\ref{lem:prep} in place of merely the minimal generating set of $\Ann(\pi_{lt}(\ltsm(g))) \subseteq A_{lt}$. One would then want to extend a similar idea to the colon ideals, using $(\Ann_A(\ltsm(g)) : m)$ in place of $(\Ann_{A_{lt}}(\pi_{lt}(\ltsm(g))) : \pi_{lt}(m))$; an obstacle here is that we have not been able to prove that the former colon ideal is standard monomial in the sense of Lemma~\ref{lem:ann-monomial}, and thus we do not know what should play the role of its ``minimum generating set of standard monomials.'' Even with all this in place, the induction in the proof of Lemma~\ref{lem:buchberger2} would also need to change (an instance of the the key issue is that, in the notation of that proof, $\lmsm(m_2 g_0)$ could be greater than $\lmsm(m_0 g_0)$, if for example $m_2$ annihilates fewer of the largest terms of $g_0$ than $m_0$ does; a similar issue arises near the end of the proof as well). We note that Green's approach \cite{green} gets around this issue by only considering ``uniform'' elements in Gröbner bases, that is, elements of the form $\sum t_i$ where the $t_i$ are standard terms such that for all standard monomials $m$, $mt_i = 0 \Leftrightarrow mt_j = 0$. But intuitively the idea is that if $\ltsm(g)$ is a zerodivisor, then to get a pseudo-ASL Gröbner basis ``all one should need'' is to multiply $g$ by a suitable set of monomials that annihilate $\ltsm(g)$, and ensure that each such product has a standard expression relative to $G$ (reduce the product to its remainder, and if the remainder is nonzero, add the remainder to $G$).
\end{remark}

\subsection{Syzygies of standard monomials}
Heading towards the pseudo-ASL analogue of Schreyer's Syzygy Theorem (Theorem~\ref{thm:schreyer} below), in this section we analyze the syzygies of standard monomial modules $\langle m_1,\dotsc,m_d \rangle \in (A')^e$ over a monomial pseudo-ASL $A'$. A syzygy $\sum_{i=1}^d h_i e_i$ is called \emph{homogeneous} (of grade $\delta$) if there is a standard monomial $\delta$ or $\delta=0$ such that for all $i=1,\dotsc,d$ and all standard monomials $m$ appearing in $h_i$, $mm_i$ is a scalar multiple of $\delta$. (This definition aligns with the grading on $A$ by the standard monomial monoid, see Remark~\ref{rmk:semigroup}.) We denote the $R$-submodule of syzygies that are homogeneous of grade $\delta$ by a subscript, e.g. $\Syz(m_1, \dotsc, m_d)_{\delta}$ or $\ker(\varphi)_{\delta}$.

\begin{definition}[Divided Koszul syzygies] \label{def:koszul}
For all $i,j \in \{1,\dotsc,d\}$, and for each $m \in LCM(m_i, m_j)$, let $t_i,t_j$ be standard terms such that $m_i t_i = m = m_j t_j$. Define
\[
\sigma^m_{ij} := t_i e_i - t_j e_j \in (A')^d.
\]
\end{definition}

By Observation~\ref{obs:no-binomial-ann}, in the presence of a pseudo-ASL term order, the choices for $t_i,t_j$ are unique, but for the moment we continue our development without that assumption.
In general, although there may be multiple choices for the $t_i,t_j$, any two such choices differ by a binomial annihilator of $m_i$ (resp., $m_j$), and we may choose any $t_i, t_j$ that satisfy the definition without affecting the subsequent development. This is proved more carefully within the proof of the following lemma. 

\begin{lemma}[Syzygies of monomial submodules] \label{lem:syz_monomial}
Let $A'$ be a monomial pseudo-ASL, and let $m_1, \dotsc, m_d$ be standard monomials in $(A')^e$. Then $\bigslant{\Syz(m_1, \dotsc, m_d)}{\bigoplus_i \Ann(m_i) e_i}$ is generated by the divided Koszul syzygies

\begin{align*}
& 
\left\{\sigma^m_{ij} \suchthat i,j \in [d], m \in LCM(m_i, m_j)\right\}. 
\end{align*}

Furthermore, if $(A')^e$ admits a pseudo-ASL term order, then 
\[
\left(\bigcup_i \LAM(m_i) \right) \cup \left\{\sigma^m_{ij} \suchthat i,j \in [d], m \in LCM(m_i, m_j)\right\}
\]
is a universal pseudo-ASL Gröbner basis for $\Syz(m_1, \dotsc, m_d)$.
\end{lemma}

The proof here follows the same general outline as \cite[Section~15.1]{eisenbud}, but we must additionally deal with the fact that in our setting LCMs need not be unique.

\begin{proof}
Let $\varphi \colon (A')^d \to (A')^e$ be the $A'$-module homomorphism that sends the standard basis element $e_i$ to $m_i$ for all $i=1,\dotsc,d$, so that $\Syz(m_1,\dotsc,m_d) = \ker(\varphi)$.

We claim that $\ker(\varphi) = \bigoplus_n \ker(\varphi)_n$, where the sum is over all standard monomials $n$ or $n=0$ such that $\ker(\varphi)_n \neq 0$. Indeed, suppose $\sum s_i e_i \in \ker(\varphi)$. For each $n$ that occurs as a standard monomial in any $s_i m_i$, define $s_{i,n}$ to be the sum of standard terms $t$ occurring in $s_i$ such that $tm_i$ is a scalar multiple of $n$. In particular, we have $s_i = \sum_n s_{i,n}$, and $s_{i,n} m_i$ is a scalar multiple of $n$ for all $i,n$. Since $\sum s_i e_i$ was a syzygy among the $m_i$, we have $\sum s_i m_i = 0$, and hence each of these scalar multiples of $n$ must in fact be zero, that is, we must have $\sum_i s_{i,n} m_i = 0$ for all $n$. In particular, this means that $\sum_i s_{i,n} e_i \in \ker(\varphi)$ as well. We have thus written our syzygy $\sum s_i e_i$ as a sum of elements of $\ker(\varphi)_n$ for various $n$, and it is clear by construction that this expression is unique. This completes the proof of the claim that $\ker(\varphi) = \bigoplus_n \ker(\varphi)_n$.

(Although it is not needed in the proof, it is worth realizing that 
\(
\Ann(m_i) e_i \cap \ker(\varphi)_n = \Ann(m_i)_n e_i.
\))

Thus, to show the first part of the lemma, it suffices to show that each $\bigslant{\ker(\varphi)_n}{\bigoplus_i \Ann(m_i) e_i}$ is generated by the $\sigma_{i,j}^m$. We proceed by induction on the number of standard terms in our syzygy $\sum s_i e_i \in \ker(\varphi)_n$. Since we need only show the result modulo $\bigoplus \Ann(m_i) e_i$, we may assume without loss of generality that none of the $s_i$ include any tuple of terms $t_1,\dotsc,t_k$ $(k \geq 1)$ such that $\sum_{i=1}^k t_i \in \Ann(m_i)$. Thus there must be at least two distinct indices $i,j$ with $s_i,s_j$ both nonzero. Let $t_i$ be a nonzero term of $s_i$ and $t_j$ a nonzero term of $s_j$. As our syzygy is in $\ker(\varphi)_n$, by definition there are nonzero constants $\alpha_i, \alpha_j \in R$ such that $t_i m_i = \alpha_i n$ and $t_j m_j = \alpha_j n$. 

What we would like is to find an $m \in LCM(m_i, m_j)$ and standard terms $r_i, r_j$ such that $r_i m_i = r_j m_j = m$ and $r_i \divides t_i$ and $r_j \divides t_j$, as then we could subtract off a multiple of $\sigma_{ij}^m$ to reduce the number of terms of our syzygy. However, because of the potential non-uniqueness of division of standard monomials, this may not happen. What we will show is that we can arrange for this to happen by adding certain elements from $\Ann(m_i)_n e_i \oplus \Ann(m_j)_n e_j$.

Towards that end, as $n$ is a common multiple of $m_i, m_j$, there is some $m \in LCM(m_i, m_j)$ such that $m \divides n$ (we may use any such $m$). Let $r_i, r_j$ be standard terms such that $\sigma_{ij}^m = r_i e_i - r_j e_j$ (in particular, $r_i m_i = r_j m_j = m$). Let $q$ be a standard term such that $qm = n$. Then $qr_i m_i = q r_j m_j = n$. Thus $\alpha_i qr_i m_i = t_i m_i = \alpha_i n$, so $(\alpha_i q r_i - t_i) e_i$ is an element of $\Ann(m_i)_n e_i$. Similarly $(\alpha_j q r_j - t_j) e_j \in \Ann(m_j)_n e_j$. By adding these two elements to our syzygy, we maintain the total number of terms exactly, but we have replaced the term $t_i e_i$ by $\alpha_i q r_i e_i$, and replaced the term $t_j e_j$ by $\alpha_j q r_j e_j$. If we now subtract $\alpha_i q \sigma_{ij}^m$ from our syzygy, then the two terms $\alpha_i q r_i e_i + \alpha_j q r_j e_j$ get replaced by the one term $(\alpha_j + \alpha_i) q r_j e_j$. Thus we have reduced the number of terms by 1, by adding a multiple of one of the divided Koszul syzygies and some multiples of annihilating syzygies. By induction on the number of terms, this completes the proof of generation.

To see the ``furthermore,'' suppose that $(A')^e$ admits a pseudo-ASL term order. We will show that the least annihilating monomials and divided Koszul syzygies together form a universal pseudo-ASL Gröbner basis. We will show that, relative to any term order, the leading term of any syzygy must be divisible by the leading term of either a least annihilating monomial or a divided Koszul syzygy; the result will then follow from the module analogue of Proposition~\ref{prop:asl-grobner-divisibility-characterization} (whose proof is the same, \emph{mutatis mutandis}). 

As above, suppose that $\sum s_i e_i \in \ker(\varphi)$, and suppose the leading monomial of $\sum s_i e_i$ is $t_i e_i$. Let $n = \lmsm(t_i m_i)$; then we have that $t_i e_i$ has degree $n$, since $A'$ is a monomial pseudo-ASL, that is, the projection of $\sum s_i e_i$ onto $\ker(\varphi)_n$ includes $t_i e_i$ as one of its terms. If $t_i \in \Ann(m_i)$, then by Lemma~\ref{lem:monomial_ann}, $t_i$ is divisible by some least annihilating $s \in \LAM(m_i)$, hence $t_i e_i$ is divisible by $\alpha_i^s = s e_i$. If $t_i \notin \Ann(m_i)$, then, by Observation~\ref{obs:no-binomial-ann}, we must be in the case above where there is another term $t_j e_j$ with $\lmsm(t_j m_j) = n$ and $j \neq i$. Let $m, r_i, r_j, \alpha_i, \alpha_j, q$ be as above. Above we showed that 
\[
(\alpha_i q r_i - t_i) e_i \in \Ann(m_i), \text{ and } (\alpha_j q r_j - t_j) e_j \in \Ann(m_j).
\]
But again by Observation~\ref{obs:no-binomial-ann}, we must have that the latter two are identically zero (they cannot be proper binomial annihilators, and if we have $\alpha_j q r_i - t_i$ being a scalar multiple of $t_i$, then we would have $t_i \in \Ann(m_i)$, contrary to our assumption), and hence $\alpha_i q r_i = t_i$ and $\alpha_j q r_j = t_j$. We now claim that $\ltsm(\sigma_{ij}^m) = r_i e_i$. For if $r_i e_i \prec r_j e_j$, then multiplying by $q$, we know that $q r_i e_i$ is a scalar multiple of $t_i$, so it is not zero, and similarly neither is $q r_j e_j$. Thus we have $q r_i e_i \prec q r_j e_j$. But since the latter two differ from $t_i e_i$, resp. $t_j e_j$, by scalar multiples, we would have $t_i e_i \prec t_j e_j$, contradicting our asusmption that $t_i e_i$ was the leading term of our syzygy. Thus $\ltsm(\sum s_i e_i) = t_i e_i$ is divisible by $\ltsm(\sigma_{ij}^m) = r_i e_i$, and we are done.
\end{proof}

\subsection{Syzygies over general pseudo-ASLs} 

In the following definition, we will need the concept of ``leading vector'': given $(A^e, \preceq)$ and $f \in A^e$, we define the \emph{leading vector} of $f$, denoted $\lvsm(f)$, to be the basis element $e_i$ of $A^e$ such that $\lmsm(f)$ is an $A$-multiple of $e_i$. 
\begin{definition} 
\label{def:good}
Given $(A^e, \prec_0)$, a free $A$-module $A^d$ with basis $e_1, \dotsc, e_d$, and an $A$-module homomorphism $\varphi \colon A^d \to A^e$, let $f_i = \varphi(e_i)$ for $i=1,\dotsc,d$. We say that a pseudo-ASL term order $\prec$ on $A^d$ is \emph{good} for $(\prec_0, \varphi)$ if, for all $i,j$ (not necessarily distinct), and for all standard monomials $m,m' \in A$, 
\begin{enumerate}
\item \label{def:good:one} if $mf_i \neq 0$ and $m'f_j \neq 0$, then
\[
\lmsm(mf_i) \prec_0 \lmsm(m' f_j) \Rightarrow m e_i \prec m' e_j.
\]

\item \label{def:good:two} if $i=j$, $m f_i \neq 0$, $m' f_i \neq 0$, and $\lmsm(m f_i) = \lmsm(m' f_i)$, then 
\[
m \lvsm(f_i) \prec_0 m' \lvsm(f_i) \Leftrightarrow m e_i \prec m' e_i.
\]
\end{enumerate}
\end{definition}

If $A$ is furthermore a domain, we can show that a good term order exists, analogous to the usual case of polynomial rings:

\begin{proposition}[store=obsgooddomain] 
\label{obs:good_domain}
Given $(A^e, \prec_0)$, a free $A$-module $A^d$ with free basis $e_1, \dotsc, e_d$, and an $A$-module homomorphism $\varphi \colon A^d \to A^e$, if $\varphi(e_i)$ is a non-zerodivisor for all $i=1,\dotsc,d$, then there exists a pseudo-ASL term order on $A^d$ that is good for $(\prec_0, \varphi)$.
\end{proposition}

We do not yet know whether a good term order exists for all $A$, though we have checked at least one non-domain case and seen that it indeed still exists.

\begin{proof}
Let $f_i = \varphi(e_i)$ for $i=1,\dotsc,d$. Since each $\ltsm(f_i)$ is not a zerodivisor, 
the definition of ``good'' almost completely specifies the order, except in the case of ties, that is, when $\lmsm(m f_i) = \lmsm(m' f_j)$ but $i \neq j$. 
In such cases, we may choose to break ties by defining $me_i \prec m'e_j$ if $\lmsm(m f_i) = \lmsm(m f_j)$ and $i < j$. This order is good for $(\prec_0, \varphi)$ by construction. We will show that it is indeed a pseudo-ASL term order. 

First, let $m$ be any standard monomial. Then we have $\lmsm(f_i) \prec_0 \lmsm(m f_i)$, as $m\lmsm(f_i)$ nonzero since $\lmsm(f_i)$ is not a zerodivisor, so we have $\lmsm(m f_i) = m \lmsm(f_i)$. Thus, by Def.~\ref{def:good}(\ref{def:good:one}), we have $e_i \prec m e_i$, showing positivity \ref{def:termorder-modules:positive}.

Next, we will show that \ref{def:termorder-modules:mult} holds. Let $f,g,h \in A$ be standard monomials and suppose that $f e_i \prec g e_j$. If $fh \neq 0$ and $gh \neq 0$, then we need to show that $\lmsm(fh e_i) \prec \lmsm(gh e_j)$. We split into three cases based on the reason that $f e_i \prec g e_j$. 

[Case 1]: Suppose the reason that $f e_i \prec g e_j$ is because of Def.~\ref{def:good}(\ref{def:good:one}). That is, we have $\lmsm(f f_i) \prec_0 \lmsm(g f_j)$. Then, since $fh \neq 0$ by assumption and $\ltsm(f_i)$ is not a zerodivisor, we have $fh\ltsm(f_i) \neq 0$, and thus by \eqref{obs:mult-eqn3}, we have $\lmsm(f h f_i) = \lmsm(h \lmsm(ff_i))$. Similarly we have $\lmsm(g h f_j) = \lmsm( h \lmsm(g f_j))$. Since we have $\lmsm(ff_i) \prec_0 \lmsm(gf_j)$, by the definition of pseudo-ASL term order for $\prec_0$, we have 
\[
\lmsm(fhf_i) = \lmsm(h \lmsm(ff_i)) \prec_0 \lmsm(h \lmsm(g f_j)) = \lmsm(g h f_j),
\]
and therefore by Def.~\ref{def:good}(\ref{def:good:one}) we have $\lmsm(fhe_i) \prec \lmsm(gh e_j)$, as desired.

[Case 2]: Suppose $f e_i \prec g e_j$ because of Def.~\ref{def:good}(\ref{def:good:two}). We will show that in fact this cannot occur. In this case we have $i=j$ and $\lmsm(m f_i) = \lmsm(m' f_i)$. Since $\ltsm(f_i)$ is a non-zerodivisor, we have $\lmsm(m f_i) = \lmsm(m \lmsm(f_i))$, and similarly $\lmsm(m' f_i) = \lmsm(m' \lmsm(f_i))$. As $\lmsm(f_i)$ lives in a copy of $A$ inside $A^e$, we may work entirely in that one copy of $A$. Let $\preceq'$ denote the restriction of $\preceq$ to that copy of $A$, and let $A_{gen}$ be the generic algebra of leading terms relative to $\preceq'$. Then in $A_{gen}$, we get the binomial annihilator 
\[
(\pi_{gen}(m) - \pi_{gen}(m')) \pi_{gen}(\lmsm(f_i)) = 0,
\]
contradicting Obesrvation~\ref{obs:no-binomial-ann}. Thus this case cannot occur.

(Although not needed for the proof, it is instructive to think about other assumptions under which this case could possibly occur: it can only occur when we have, e.g., $m' \lmsm(f_i) = 0$ even though $m' f_i \neq 0$.)

[Case 3]: Suppose neither case 1 nor 2 holds; then the reason $fe_i \prec g e_j$ must be that $\lmsm(ff_i) = \lmsm(g f_j)$ and $i < j$. Then by the same argument as before we get $\lmsm (f h f_i) = \lmsm(g h f_j)$, but we still have $i < j$, and thus $\lmsm(fh e_i) \prec \lmsm( gh e_j)$ by construction, as desired.

Finally, we show \ref{def:termorder-modules:restrict}. Let $f,g,h,k$ be standard monomials and suppose that $f e_i \prec g e_i$ and $h e_i \preceq k e_i$, and $fh \neq 0$ and $gk \neq 0$. We need to show that $\lmsm(f h e_i) \prec \lmsm(g k e_i)$. By the argument above we have $\lmsm(f h e_i) \prec \lmsm (g h e_i)$ (NB: $gh$ not $gk$). (We note that here we are again crucially using the fact that $\lmsm(f_i)$ is not a zerodivisor, since nothing else in our assumptions would tell us that $gh f_i \neq 0$.) If $h=k$ then we are done. Otherwise we have $h e_i \prec k e_i$ (strict inequality). By applying the argument above, but this time multiplying $h e_i \preceq k e_i$ by $g$ on both sides, we get that $\lmsm(hg e_i) \prec \lmsm( kg e_i)$. Putting these two together, we get $\lmsm( f h e_i) \prec \lmsm(g k e_i)$. 
\end{proof}

The following observation about domains may be useful in light of Proposition~\ref{obs:good_domain}.

\begin{observation} \label{obs:alt-domain}
Given $(A,\preceq, A_{lt})$, $A_{lt}$ is a domain if and only if $A$ is a domain and $A_{lt} \cong A_{gen}$.
\end{observation}

\begin{proof}
($A_{gen}$ is a domain if and only if $A$ is.) If $A$ is a domain, then since a product of two standard monomials in $A$ is always nonzero, the same is true in $A_{gen}$. Thus by Prop.~\ref{prop:domain}, $A_{gen}$ is a domain. Conversely, if $A$ is not a domain, then by Prop~\ref{prop:domain} there is a pair of standard monomials whose product in $A$ is zero. By construction of $A_{gen}$, their product in $A_{gen}$ is also zero, so $A_{gen}$ is not a domain.

(Other $A_{lt}$'s are not domains.) Note that the way $A_{lt}$ is constructed, if $A_{lt}$ is not isomorphic to $A_{gen}$, then there is some pair of standard monomials whose product becomes $0$ in $A_{lt}$, since, by the definition of algebra of leading terms, a product of standard monomials is either $\pi_{lt}(\lmsm(mm'))$ or $0$. Thus, $A_{lt}$ is not a domain.
\end{proof}

To set up the pseudo-ASL analogue of Schreyer's Theorem, we will use the compatibility syzygies and the S-remainder syzygies, which we define as follows. For all $i,j \in [d]$ such that $\lmsm(f_i)$ and $\lmsm(f_j)$ involve the same basis element of $A^d$, and for all $m \in \lcmlt(\lmsm(f_i), \lmsm(f_j))$, choose a standard expression (Def.~\ref{defn:standard-expression}) for 
\[
\pi_{lt}^{-1}\left(\ltfrac{m}{\pi_{lt}(\ltsm(f_i))} \right) f_i - \pi_{lt}^{-1}\left(\ltfrac{m}{\pi_{lt}(\ltsm(f_j))}\right) f_j,
\]
say as $\sum_{k=1}^d h_k f_k$, and define:
\[
\tau_{ij}^m := \pi_{lt}^{-1}\left( \ltfrac{m}{ \pi_{lt}(\ltsm(f_i))} \right) e_i - \pi_{lt}^{-1}\left(\ltfrac{m}{\pi_{lt}(\ltsm(f_j))}\right) e_j - \sum_{k=1}^d h_k e_k. 
\]

Recall the divided Koszul syzygies $\sigma_{ij}^m$ (Def.~\ref{def:koszul}) and annihilating syzygies $\alpha_i^m$ in a monomial pseudo-ASL (Def.~\ref{def:annsyz}); these correspond, respectively, to the S-remainder syzygies and compatibility syzygies in a general pseudo-ASL, in a sense made precise in the following lemma.

\begin{lemma} \label{lem:syz_leading}
Given $(A^e, \prec_0, A_{lt}^e)$, suppose that $f_1, \dotsc, f_d$ is a pseudo-ASL Gröbner basis of the submodule $M \subseteq A^e$. Let $\varphi \colon A^d \to A^e$ be the $A$-module map defined by $\varphi(e_i) = f_i$ for $i=1,\dotsc,d$. Let $\prec$ be any pseudo-ASL term order on $A^d$ that is good for $(\prec_0, \varphi)$ (Def.~\ref{def:good}). Let 
\begin{itemize}
\item $\alpha_i^s$ be the annihilating syzygies of $\piltd(\lmsm(f_i))$, 
\item $\beta_i^s$ be a compatibility syzygy for $(s,f_i)$ in $A^d$, 
\item $\sigma_{ij}^m$ be the divided Koszul syzygies of $\piltd(\lmsm(f_i)), \piltd(\lmsm(f_j))$ in $A_{lt}^d$, and
\item $\tau_{ij}^m$ be the S-remainder syzygies of $f_i, f_j \in A^d$.
\end{itemize}
Then
\begin{enumerate}
\item 
$\piltd(\ltsm(\beta_i^s)) = \ltsm(\alpha_i^s),$ and

\item 
$\piltd(\ltsm(\tau_{ij}^m)) = \ltsm(\sigma_{ij}^m)$.
\end{enumerate}
\end{lemma}

\begin{proof}
(1) If $s f_i = 0$, then $\piltd(\beta_i^s) = \alpha_i^s$, so they certainly have the same leading term. 

Otherwise, suppose $\beta_i^s = s e_i - \sum_k h_k e_k$. Then since $\sum h_k f_k$ is a standard expression for $sf_i$, we have that $\lmsm(h_k f_k) \preceq_0 \lmsm(s f_i)$ for all $k=1,\dotsc,d$. Since $\prec$ is assumed to be good for $(\prec_0, \varphi)$, this implies that $\lmsm(h_k e_k) \preceq \lmsm(s e_i)$ for all $k=1,\dotsc,d$. Now, for $k \neq i$, the two cannot be equal since they involve different basis elements, so we necessarily have the strict inequality $\lmsm(h_k e_k) \prec \lmsm(s e_i)$. For $k=i$, if $\lmsm(h_i f_i) \prec_0 \lmsm(s f_i)$, then we again have $\lmsm(h_i e_i) \prec \lmsm(s e_i)$, as desired. 

Now instead suppose that $\lmsm(h_i f_i) = \lmsm(s f_i)$. Since $h_i$ is compatible with $\lmsm(f_i)$, we have $\lmsm(h_i) \lmsm(f_i) \neq 0$. Next, we will show that we must have $s \lmsm(f_i) = 0$ in $A^e$ (since $\beta_i^s$ is a compatibility syzygy, by definition we have $\pi_{lt}(s) \piltd(\lmsm(f_i)) = 0$ in $A_{lt}^e$; our claim here is that the same happens in $A^e$). For if not, then by Proposition~\ref{prop:mult} we have 
\[
\lmsm(\lmsm(h_i) \lmsm(f_i)) = \lmsm(h_i f_i) = \lmsm(s f_i) = \lmsm(s \lmsm(f_i)).
\]

Now, since $\lmsm(h_i f_i) = \lmsm(s f_i)$ occurs in some copy of $A \subseteq A^e$, say $Av$, we may focus our attention only on those terms in $f_i$ in the coordinate $v$. Let $\preceq'$ denote the ordering on $A$ that agrees with the restriction of $\preceq_0$ to $Av \subseteq A^e$. 

Let $A_{gen}$ be the generic algebra of leading terms relative to $\preceq'$. Then applying $\pi_{gen}$ to the equation above, we get 
\[
\pi_{gen}(\lmsm(h_i)) \pi_{gen}(\lmsm(f_i)) = \pi_{gen}(s) \pi_{gen}(\lmsm(f_i)).
\]
As $h_i$ is compatible with $\lmsm(f_i)$ in $A^e$ (relative to $A_{lt}^e$)---since $\sum h_i f_i$ is a standard expression---and $s$ is not, we must have $\lmsm(h_i) \neq s$, and thus the preceding centered equation gives us a nontrivial binomial annihilator of $\pi_{gen}(\lmsm(f_i))$ in $A_{gen}$, contradicting Observation~\ref{obs:no-binomial-ann}. Thus we have that $s \lmsm(f_i) = 0$ in $A^e$.

We claim, finally, that $h_i \prec' s$. For suppose not; then $\lmsm(h_i) \succeq' s$. But now let $m$ be the largest monomial in $f_i$ such that $s m \neq 0$. By the preceding argument, $m$ is not the leading monomial of $f_i$, so in particular we have $m \prec' \lmsm(f_i)$. But then, since $\prec'$ is a pseudo-ASL term order on $A$, and we have $s \preceq' \lmsm(h_i)$, and $m \prec' \lmsm(f_i)$, we get
$\lmsm(s m) \prec' \lmsm(\lmsm(h_i) \lmsm(f_i))$, since both sides here are nonzero. However, the LHS is (by the definition of $m$) equal to $\lmsm(s f_i)$ and the RHS is equal to $\lmsm(h_i f_i)$. The strict inequality between the two sides contradicts our assumption that $\lmsm(h_i f_i) = \lmsm(s f_i)$, and thus we have $h_i \prec' s$. 

Since $\prec'$ was by definition the restriction of $\prec$ to standard monomial multiples of $\lvsm(f_i)$, $h_i \prec' s$ is the same as $h_i \lvsm(f_i) \prec s \lvsm(f_i)$. Thus, by Definition~\ref{def:good}(\ref{def:good:two}), this implies that $h_i e_i \prec s e_i$.

Thus $\ltsm(\beta_i^s) = \ltsm( s e_i)$, so 
\[
\piltd(\ltsm(\beta_i^s)) = \ltsm(\pi_{lt}(s) e_i) = \ltsm(\alpha_i^s).
\] 

(2)  We will first show that one of $t_i e_i$ or $t_j e_j$ is the leading term of $\tau_{ij}^m$. Let $t_i=\pi_{lt}^{-1}\left(\ltfrac{m}{\ltsm(f_i))}\right)$, and analogously for $t_j$. By construction we thus have 
\(
\ltsm(t_i f_i) = \ltsm(t_j f_j) = m,
\)
and by Observation~\ref{obs:s-poly-reduces-lm} we have 
\(
m \succ_0 \lmsm(\Ssm^m(f_i, f_j)).
\)
 Next, by the definition of standard expression and the definition of $\tau_{ij}^m$, we have 
 \(
 \lmsm(\Ssm^m(f_i, f_j)) \succeq_0 \lmsm(h_k f_k)
 \)
 for all $k=1,\dotsc,d$. Combining these together we have 
 \[
 m = \lmsm(t_i f_i) = \lmsm(t_j f_j) \succ_0 \lmsm(h_k f_k)
 \]
 for all $k=1,\dotsc,d$. 
 
Now, as $t_i$ is compatible with $\lmsm(f_i)$ by construction, and every standard monomial in $h_k$ is compatible with $\ltsm(f_k)$ by the definition of standard expression, we have that $\lmsm(h_k) e_k \prec t_i e_i$, since $\prec$ is good for $(\prec_0, \varphi)$ (similar to the proof of (1) above). The same holds for $t_j e_j$. Thus one of $t_i e_i, t_j e_j$ must be the leading term of $\tau_{ij}^m$. 

As $\sigma_{ij}^m = t_i e_i - t_j e_j$, one of $t_i e_i, t_j e_j$ is the leading term of $\sigma_{ij}^m$, and since the ordering on standard monomials is the same in $A^d$ and $A_{lt}^d$, we have $\pi_{lt}(\ltsm(\tau_{ij}^m)) = \ltsm(\sigma_{ij}^m)$, as claimed.
\end{proof}

\nocite{eisenbud}
\begin{theorem}[store=thmaslschreyer, note={Pseudo-ASL analogue of Schreyer's Syzygy Theorem, cf. \cite[Theorem~15.10]{eisenbud}}] \label{thm:schreyer}
Given $(A^e, \prec_0, A_{lt}^e)$, suppose that $f_1, \dotsc, f_d$ is a pseudo-ASL Gröbner basis such that $\ltsm(f_i)$ is a non-zerodivisor for all $i$. Let $\varphi \colon A^d \to A^e$ be the $A$-module homomorphism defined by $\varphi(e_i) = f_i$ for $i=1,\dotsc,d$. Let $\prec$ be any pseudo-ASL term order on $A^d$ that is good for $(\prec_0, \varphi)$ and suppose that $A_{lt}$ is an algebra of leading terms for $(A^d,\prec)$ as well. 

Then the following S-remainder syzygies and compatibility syzygies are a pseudo-ASL Gröbner basis for the syzygies on the $f_i$ relative to $(\prec, A_{lt}^d)$:
\[
\left\{\tau_{ij}^m \suchthat i,j \in [d], m \in \lcmlt(\pi_{lt}(f_i), \pi_{lt}(f_j))\right\} \cup \left\{\beta_i^s \suchthat (s,f_i) \text{ as in Definition~\ref{def:ann-closed}}\right\}.
\]
\end{theorem}

\begin{proof}
Let $\sigma = \sum_{k=1}^d \ell_k e_k$ be a syzygy, that is, $\sum \ell_k f_k = 0$. We will show that $\pi_{lt}(\ltsm(\sigma))$ is divisible by the leading term of one of the S-remainder syzygies or compatibility syzygies in the statement of the theorem.  
For $k=1,\dotsc,d$, let $t_k e_k = \ltsm(\ell_k e_k)$. Since the terms $t_k e_k$ for distinct values of $k$ involve different basis elements, they cannot cancel one another, so there is some $i$ such that $\ltsm(\sigma) = t_i e_i$. 

If $\pi_{lt}(t_i) \piltd(\ltsm(f_i)) = 0$, then $\pi_{lt}(t_i) e_i$ is divisible by an annihilating syzygy of the form $s e_i$ for some $s \in \LAM(\pi_{lt}(\ltsm(f_i)))$. By Lemma~\ref{lem:syz_leading}, we then have that the corresponding compatibility syzygy has its leading term lt-dividing $t_i e_i = \ltsm(\sigma)$, and we are done.

On the other hand, if \[
\pi_{lt}(t_i) \piltd(\ltsm(f_i)) = \piltd(m) \neq 0,
\]
then we first proceed to eliminate other incompatible terms. For any $k$ such that 
\[
\pi_{lt}(t_k) \pi_{lt}(\ltsm(f_k)) = 0,
\]
we use the compatibility syzygies as in Lemma~\ref{lem:buchberger2} to replace those terms by other terms consisting of standard monomials $m_k$ times a basis vector $e_k$ such that $m_k$ is compatible with $\ltsm(f_k)$. (We note that, as in Lemma~\ref{lem:buchberger2}, our current proof of this uses the non-zerodivisor assumption. This is the only new place in this proof, beyond its use in previous lemmata, that we use that assumption.) Because the compatibility syzygies involved standard expressions, this does not increase the terms $t_k e_k$, and since none of those were the leading term $t_i e_i$, we get another syzygy where $t_i e_i$ is still the leading term, and now every $k$ satisfies $\pi_{lt}(t_k) \pi_{lt}(\ltsm(f_k)) \neq 0$.

Next, let 
\[
I = \left\{j \in [d] : \lmsm(\pi_{lt}(t_j) \piltd(\ltsm(f_j))) = \piltd(m) \right\}.
\] 
By Definition~\ref{def:good}(\ref{def:good:one}), we have that $m \succeq_0 \lmsm(t_j f_j)$ for all $j$, and hence the terms $\{\ltsm(t_j \ltsm(f_j)) : j \in I \}$ must cancel one another. Thus we have that 
\[
\sum_{j \in I} \pi_{lt}(t_j) \pi_{lt}(\ltsm(f_j)) = 0,
\]
so $\sigma' := \sum_{j \in I} \pi_{lt}(t_i) e_j$ is a syzygy among the $\pi_{lt}(\ltsm(f_j))$, over $A_{lt}$. By Lemma~\ref{lem:syz_monomial}, either a divided Koszul syzygy or an annihilating syzygy has the property that its leading term divides $\ltsm(\sigma')$. Finally, by Lemma~\ref{lem:syz_leading}, if a divided Koszul syzygy's leading term divides $\ltsm(\sigma')$, then the leading term of the corresponding S-remainder syzygy lt-divides $m= \lmsm(\sigma)$, and similarly, if an annihilating syzygy divides $\ltsm(\sigma')$, then the corresponding compatibility syzygy's leading term lt-divides $\lmsm(\sigma)$. And thus we conclude. 
\end{proof}

%% file: algorithms.tex

\section{Algorithms}
\label{sec:algorithms}

In this section we show that pseudo-ASL Gröbner bases can be computed by algorithms that are only slightly more complicated than the algorithms for computing ordinary Gröbner bases in a polynomial ring.

We assume that our pseudo-ASL is specified to the algorithms by giving: 
\begin{itemize}
\item $R$, a field in which the basic field operations are computable
\item the finite set $H$
\item A finite, minimal set of generators of the semigroup ideal $\Sigma \subseteq \W^H$
\item An algorithm for multiplying standard monomials in $A$. 
\item An algorithm for deciding the term order $\prec$, that is, the algorithm, given two standard monomials $m,m' \in A$, returns true or false according to whether $m \prec m'$ or not.
\item $A_{lt}$ is specified analogously to $A$ above, using the same $H,\Sigma$ as for $A$, with the further stipulation that we are given generators for the kernel of the map $R[H] \to A_{lt}$, such that the generators all have at most two monomials.
(For $A_{lt}$ we in fact need binomial generators of the defining ideal, and not merely a black box for multiplication in $A_{lt}$.)
\end{itemize}

Monomials, that is, elements of $\W^H$ (which may be thought of as living in one or more of $\W^H$, $R[H]$, $A$, or $A_{lt}$), are specified in the usual way by a list of their exponents. Polynomials, i.e. elements of $R[H]$, are given in their usual sparse representation as a list of monomials together with their coefficients. Elements of $A$ are also given as polynomials, the same as though they were in $R[H]$, but with the caveat that the only monomials allowed are the standard monomials. We assume a weak sort of type system, in which the data structure for a polynomial specifies whether it lives in $A$, $A_{lt}$, or $R[H]$ (so that, for example, if $f,g \in A$ are two polynomials, and we write $fg$, then the data structure implies that that means the product in $A$, rather than the product in $R[H]$ or $A_{lt}$; similarly if $f,g \in A_{lt}$ or $f,g \in R[H]$, and it is considered a type error to multiply a polynomial in $A$ by a polynomial in $A_{lt}$, etc.).

Throughout this section we present our algorithms for $A, A_{lt}$ for notational and linguistic simplicity (we never have to mention ``the vector supporting the leading term'' and so on), but the results and algorithms extend to $A^d, A_{lt}^d$ without any difficulty.

\subsection{Computing in pseudo-ASLs}
\begin{proposition} \label{prop:div-alg}
Given $(A, \preceq, A_{lt})$. There is an algorithm that halts in finite time such that, given two standard terms $t,t' \in A_{lt}$ as input, either outputs the unique (by Observation~\ref{obs:no-binomial-ann}) standard term $t/t'$, or a special symbol $\bot$ if $t'$ does not divide $t$.
\end{proposition}

The proof in fact does not use $(A,\prec)$ at all, though it does use Observation~\ref{obs:no-binomial-ann}, which depends on the existence of a term order on $A_{lt}$.

\begin{proof}
We represent monomials in this algorithm by their exponent vectors. Let $e,e'$ be the exponent vectors corresponding to $t,t'$, respectively.

Suppose the defining ideal of $A_{lt}$ as a quotient of $R[H]$ is given by a generating set consisting of $B \cup M$, where $B$ consists of binomials and $M$ consists of monomials, that is, 
\[
A_{lt} = \bigslant{R[H]}{\langle B \cup M \rangle}.
\] 
For each binomial $b \in B$, let $e_b$ be the difference of the exponent vectors of the two monomials in $b$. For $i=1,\dotsc,|H|$, let $e_i$ denote the $i$-th standard basis vector, with a 1 in the $i$-th position and 0s elsewhere; equivalently, the exponent vector of the variable $x_i \in H$. Then the exponent vectors in $\W^H$ that yield (not necessarily standard) monomials $m$ such that $t'm$ and $t$ are supported on the same monomial are precisely the solution to the following system of integer linear (in)equations:
\[
e_{t'} + \underbrace{\sum_{i=1}^{|H|} m_i e_i}_{\substack{\text{multiply by an} \\ \text{arbitrary monomial}}} + \underbrace{\sum_{b \in B} c_b e_b}_{\substack{\text{work} \\ \text{modulo } B}}  = e_t \qquad m_i \in \Z_{\geq 0}, c_b \in \Z \\
\]
(Equivalently, this could be written as $\sum_{b \in B} c_b e_b \geq e_t - e_{t'}$ with $c_b \in \Z$, but the $m_i$ above are the exponents of the monomial we want to solve for.) 
This is an instance of the well-known Integer Linear Programming problem, which is in the complexity class $\mathsf{NP}$, hence is solvable in exponential time. This gives us some monomial $m \in \W^H$ such that $t'm$ and $t$ are supported on the same standard monomial modulo $\langle B \rangle$. Thus, there is a coefficient $r \in R$ such that $t' rm = t$ modulo $B$. That is, $t'rm \in t + \langle B \rangle$.

However, there are two remaining issues: (1) we have not yet accounted for the monomials $M$ that become zero in $A_{lt}$, and (2) $m$ need not be standard.  For the first issue, suppose that our (not necessarily standard) monomial $m \in \W^H$ in fact becomes zero in $A_{lt}$. This happens if and only if $m$ is in $\langle B \cup M \rangle$. But then our containment $t'rm \in t + \langle B \rangle$ tells us that in fact $t$ must be in $\langle B \rangle + \langle B \cup M \rangle = \langle B \cup M \rangle$, contradicting the fact that $t$ is a nonzero standard monomial in $A_{lt}$. Thus, $m \neq 0$ in $A_{lt}$ (that is, even taking into account all the defining relations of $A_{lt}$, not just the binomial ones).

Finally, suppose $m$ is not standard. By Observation~\ref{obs:no-binomial-ann}, if $t' \divides t$, then there is a unique standard monomial $m'$ such that $t' m'$ is supported on the same monomial as $t$. Using our assumptions, we can algorithmically rewrite $m$ as a linear combination of standard monomials: since we are given (by assumption) both $H$ and an algorithm for multiplication in $A_{lt}$, suppose $m = h_1^{e_1} h_2^{e_2} \dotsb h_k^{e_k}$ where $H = \{h_1,\dotsc,h_k\}$, then using the assumed algorithm we compute the product $h_1^{e_1} h_2^{e_2} \dotsb h_k^{e_k}$. By assumption, the output of that algorithm is then the standard monomial representation of $m$. Finally, at least one of those standard monomials occurring in the expression for $m$, say $m'$, must have the property that $t'm'$ is supported on the same monomial as $t$, so we simply check each one in turn until one is found. 
\end{proof}

\begin{corollary} \label{cor:std-exp-alg}
There is an algorithm that halts in finite time for computing \hyperref[defn:standard-expression]{standard expressions}. More specifically, given $(A,\preceq, A_{lt})$, as well as $f, g_1, \dotsc, g_k \in A$, the algorithm outputs $r,h_1, \dotsc, h_k \in A$ such that $(r,h_1,\dotsc,h_k)$ is a standard expression for $f$ with respect to $(g_1, \dotsc, g_k)$. 
\end{corollary}

\begin{proof}
Initially, set 
\[
h_1 = \dotsb = h_k = 0.
\]
Repeatedly test $\pi_{lt}(\lmsm(f))$ for divisibility by $\pi_{lt}(\lmsm(g_i))$ for each $i$, using Proposition~\ref{prop:div-alg}. If it is found to be divisible, say $\pi_{lt}(\ltsm(f)) = m \pi_{lt}(\ltsm(g_i))$ for some standard term $m$, then replace $f$ by $f - \pi_{lt}^{-1}(m) g_i$ and continue, and add $\pi_{lt}^{-1}(m)$ to $h_i$. At the end, we are left with a remainder $r$; for the remainder, we must do the same procedure but with all terms of $r$, not only its leading term. 
\end{proof}

\subsection{Computing S-closures}

Fix $(A, \preceq, A_{lt})$. An \emph{S-closure} of a set $G=\{g_1, \dotsc, g_k\} \subseteq A$ is a set $G'$ such that $G \subseteq G' \subseteq \langle G \rangle$ such that $G'$ is S-closed (Definition~\ref{defn:S-closed}). Computing S-closures is of course quite similar to Buchberger's original algorithm in polynomial rings; the main difference in our setting is that LCMs of monomials need not be unique, and we need to show that we can compute them.

\begin{algorithm}[h]
\caption{Computes S-closure of a finite set with respect to $(A, \preceq, A_{lt})$}
\label{alg:s-closure}
\begin{algorithmic}[1]
\Require{$(A, \preceq, A_{lt})$}
\medskip
\Procedure{S-closure}{$f_1, \ldots, f_s$ (elements of $A$)}
    \State $G \leftarrow \{f_1, \ldots, f_s\}$
    \Repeat
    	\State $G_{prev} \leftarrow G$
         \For{each $p, q \in G_{prev}$}
         		\State $R \leftarrow \{\textsc{Remainder}(s, G') \suchthat s \in \Ssm(p, q)\} \backslash \{0\}$ 
		\State $G \leftarrow G \cup R$
         \EndFor
    \Until{$G = G_{prev}$}
    \State \Return $G$
\EndProcedure
\end{algorithmic}
\end{algorithm}

\begin{theorem}[store=thmcomputablesclosure] 
\label{thm:computable-Sclosure}
Given $(A, \preceq, A_{lt})$, there is an algorithm (Algorithm \ref{alg:s-closure}) that halts in finite time and returns an S-closure of the given set $\{f_1, \dotsc, f_k\}$ of elements of $A$.
\end{theorem}

\begin{lemma} \label{lem:comp-LCM}
In $A_{lt}$, the function $(M,N) \mapsto \lcmlt(M,N)$ (see Definition~\ref{defn:lcms}) is computable.
\end{lemma}

\begin{proof}
Let $\pi \colon R[H] \to A_{lt}$ be the defining map, with kernel $K$. Then, in $R[H]$, compute generators of the ideal intersection 
\[
I := (K + \ideal{M}) \cap (K + \ideal{N}).
\]
Since $A_{lt}$ is a monomial pseudo-ASL, it follows that $K$ is a binomial ideal and $\pi(I)$ is a standard monomial ideal (Cor.~\ref{cor:intersection-monomial}). Map the generators of $I$ into $A_{lt}$ using $\pi$, then compute the unique minimum generating set of $\pi(I)$ as follows. By Lemma~\ref{lem:unique-min}, this is the subset of the generators of $I$ that are $\ltdivides$-division-minimal. We can test division-minimality among the given generating set, after mapping it into $A_{lt}$ using $\pi$, using the algorithm of Proposition~\ref{prop:div-alg} or using standard Gröbner bases in $R[H]$ as in \cite{ES}. By Lemma~\ref{lem:unique-min} the minimum generating set is unique, and this is precisely $\lcmlt(M,N)$.
\end{proof}

\begin{remark}
The procedure in Lemma \ref{lem:comp-LCM} and other algorithms in this section are written for arbitrary pseudo-ASLs. Although the above algorithm does work in $R[H]$, rather than intrinsically in $A$, we note that it only works with binomial ideals in $R[H]$. However, if one knew more about the particular pseudo-ASL we are working in, we can expect that the algorithms can be streamlined to take advantage of the structure of the pseudo-ASL. For instance, in Section \ref{sec:bd-gb-algos}, we give a specifically tailored algorithm to compute the set of LCMs in the ASL structure imposed on the polynomial ring by the basis of standard bideterminants, and that algorithm is entirely using combinatorics of bitableaux, not appealing to calculations in $R[H]$ at all.
\end{remark}

\begin{lemma} \label{lem:comp-Ssm}
The function $(p,q) \mapsto \Ssm(p,q)$ is computable.
\end{lemma}

\begin{proof}
To compute $\Ssm(p,q)$, in looking at the definition we see it suffices to be able to:
\begin{enumerate}
\item Compute $\lcmlt(p,q)$
\item Divide an element of $A$ by its leading coefficient
\item Multiply and subtract two elements of $A$
\item Compute $a / b$ for two standard monomials $a,b \in A_{lt}$
\end{enumerate}

We note that, while technically $\pi_{lt}$ and $\pi_{lt}^{-1}$ are used in the definition, computationally these come ``for free'', as we may assume our elements of $A$ and $A_{lt}$ are given as lists consisting of standard monomials and their coefficients. Thus, in terms of the underlying data structure, a standard monomial $m$ in $A$ is represented by the \emph{same} data structure as $\pi_{lt}(m)$, and vice versa.

Now, (1) is computable by Lemma~\ref{lem:comp-LCM}. 

(2) is computable as we have assumed that arithmetic in $R$ is computable.

(3) is computable by assumption. Subtraction is nearly trivial, as elements of $A$ are represented by lists of standard monomials and their coefficients.

Finally, (4) is computable by Proposition~\ref{prop:div-alg}.
\end{proof}

\begin{proof}[Proof of Theorem~\ref{thm:computable-Sclosure}]
Since Algorithm~\ref{alg:s-closure} only ever adds elements to $G$, the output of the algorithm contains its input as a subset. Furthermore, if the algorithm halts, then we have $G = G_{prev}$ (following notation in Algorithm~\ref{alg:s-closure}), whence we must have had $R = \emptyset$. But the latter is only possible if all S-remainders reduced to 0, i.e., if $G_{prev}=G$ was S-closed. Therefore, if the algorithm halts, then it outputs an S-closure of its input.

For the algorithm to be computable, it is necessary and sufficient that each individual step be computable, and that the loop halts after finitely many iterations. Most of the steps are computable in standard fashion, with the possible exceptions of \textsc{Remainder} and $\Ssm(p,q)$, which are computable by Corollary~\ref{cor:std-exp-alg} and Lemma~\ref{lem:comp-Ssm}, respectively.

Finally, it remains to show that the main loop exits after finitely many iterations. The key idea here is to use Noetherianity in $A_{lt}$. Suppose $r \in R$ is a nonzero remainder to get added to $G$ in Algorithm~\ref{alg:s-closure}. As $r$ is a remainder (of something) relative to $G_{prev}$, it follows that for all $g \in G_{prev}$, $g \nmid_{lt} r$. As $I := \langle \pi_{lt}(\ltsm(g)) \suchthat g \in G_{prev} \rangle$ is a standard monomial ideal in $A_{lt}$, this means that $\pi_{lt}(\ltsm(r))$ is not in $I$. Thus, by adding $r$ to $G$, we have added new elements to $\langle \pi_{lt}(\ltsm(g)) \suchthat g \in G \rangle$. As $A_{lt}$ is Noetherian, this can only happen finitely many times. When it fails to happen, we will necessarily have $R = \emptyset$, hence $G=G_{prev}$ and the algorithm will terminate.
\end{proof}

\begin{remark} The computational complexity of the algorithm (such as runtime, space, completeness for various complexity classes, etc.) will of course depend on that of the algorithm to compute $\lcmlt(M,N)$ and the algorithm to compute remainders, as well as the structure of the pseudo-ASL $A$ and the complexity of multiplication---and in particular, the complexity of straightening---in $A$. Beyond those, we expect the complexity of this algorithm to be similar to that of Buchberger's algorithm for ordinary Gröbner bases, and we expect our algorithm is also amenable to improvements analogous to those made to algorithms for ordinary Gröbner bases over the last 40 years, e.\,g., \cite{GBsurvey, GB50, F4, F5}.
\end{remark}

In the nicest cases---which are indeed common in applications of (pseudo-)ASLs---already the ``ordinary'' Buchberger's algorithm for S-closure as above solves the problem of computing pseudo-ASL Gröbner bases, by Corollary~\ref{cor:buchberger-agen}: 

\begin{corollary} \label{cor:AgenGB}
If $A_{lt}$ is a domain (equivalently, by Obs.~\ref{obs:alt-domain}, $A$ is a domain and $A_{lt} \cong A_{gen}$), then Algorithm~\ref{alg:s-closure} correctly computes pseudo-ASL Gröbner bases in $A$.
\end{corollary}

\subsection{Computing Ann-closures and pseudo-ASL \gb bases}

The ``only'' remaining difficulty in computing pseudo-ASL Gröbner bases in general is to be able to compute Ann-closure. As with S-closure, we say that \emph{$G'$ is an Ann-closure of $G$} if $G \subseteq G' \subseteq \langle G \rangle$ and $G'$ is Ann-closed.

\begin{theorem} \label{thm:alg-ann-closure}
Given $(A, \preceq, A_{lt})$ and an oracle that decides if a set is Ann-closed, there is an algorithm (Algorithm~\ref{alg:ann-closure}) that halts in finite time that, given a finite set $G$, computes an Ann-closure of $G$.
\end{theorem}

For discussion about the oracle, see Remark~\ref{rmk:oracle} below.

\begin{algorithm}[h]
\caption{Computes Ann-closure of a finite set with respect to $(A, \preceq, A_{lt})$}
\label{alg:ann-closure}
\begin{algorithmic}[1]
\Require{$(A, \preceq, A_{lt})$, Oracle \textsc{Is-Ann-Closed}}
\medskip
\Procedure{Ann-closure}{$f_1, \ldots, f_s$ (elements of $A$)}
    \State $G \gets \{f_1, \ldots, f_s\}$
    \State $Q \gets \{(s,g) \suchthat s \in LAM(\pi_{lt}(\ltsm(g))), g \in G\}$ \label{alg:ann-closure:LAM}
    \State $Done \gets \{\}$
    \Repeat
    	\State Pick and remove $(s,g)$ from $Q$, and add it to $Done$
	\State Compute a standard expression $\pi_{lt}^{-1}(s) g = r + \sum_{\gamma \in G} h_\gamma \gamma$ \label{alg:ann-closure:std}
 	  \If{$r \neq 0$}
	    \State $G \gets G \cup \{r\}$ \label{alg:ann-closure:add-remainder} \label{alg:ann-closure:add-to-G}
	    \State $Q \gets Q \cup \{(s',r) : s' \in LAM(\pi_{lt}(\ltsm(r)))\}$
	  \EndIf
	  \State $Q \gets Q \cup \{(\lmsm(tm),\gamma) \suchthat m \in h_{\gamma}, t \in LAM(\pi_{lt}(m\ltsm(\gamma))), \gamma \in G\} \backslash Done$ \label{alg:ann-closure:Qupdate}
    \Until{$Q$ is empty or \textsc{Is-Ann-Closed}$(G)$} \label{alg:ann-closure-termination}
    \State \Return $G$
\EndProcedure
\end{algorithmic}
\end{algorithm}

\begin{proof}[Proof of Theorem \ref{thm:alg-ann-closure}]
To compute an Ann-closure of $G \subseteq A$, we follow the outline suggested by Definition~\ref{def:ann-closed} and Lemma~\ref{lem:buchberger2}. The idea is to build the set $S$ of compatibility syzygies from Definition~\ref{def:ann-closed} as follows. We maintain a set $Q$ (we might call it a ``queue'', but it doesn't matter whether it is first-in, first-out or not) of pairs $(s,g)$ where $s$ is a standard monomial and $g \in G$ such that $s$ is incompatible with $\ltsm(g)$, and such that we want to have a compatibility syzygy for $(s,g)$ in $S$. Initially, $Q$ is comes from the least annihilating monomials of each $\pi_{lt}(\ltsm(g))$ ($g \in G$) (line \ref{alg:ann-closure:LAM} of Algorithm~\ref{alg:ann-closure}). On subsequent rounds, $Q$ gets appended by pairs coming from $(\Ann_{A_{lt}}(\pi_{lt}(\ltsm(g))) \suchthat m)$ (line~\ref{alg:ann-closure:Qupdate}), where $m$ is a standard monomial occurring in some compatibility syzygy in $S$. 

When attempting to compute a compatibility syzygy for $(s,g)$, it may be that the standard expression for $\pi_{lt}^{-1}(s)g$ relative to the current $G$ has some non-zero remainder $r$ (line~\ref{alg:ann-closure:std}; in this case, $G$ is not a pASL Gröbner basis, and so is certainly not yet Ann-closed). In that case, we add $r$ to $G$ (line~\ref{alg:ann-closure:add-remainder}) so that we then have a compatibility syzygy for $(s,g)$.

Having overviewed the algorithm, we now show that each of its steps is computable.
The only points whose computability requires a little work are computing the updates to $Q$ (lines~\ref{alg:ann-closure:LAM} and \ref{alg:ann-closure:Qupdate}) and computing a standard expression (line~\ref{alg:ann-closure:std}). The latter is the content of Corollary~\ref{cor:std-exp-alg}. Computing a minimal generating set of $\Ann(\pi_{lt}(\ltsm(g)))$ can be done using standard Gröbner basis methods in $A_{lt}$, since we have its defining ideal by assumption and the latter is binomial; indeed, if $K$ is the defining ideal of $A_{lt}$, then we get $\Ann(m)$ in $A_{lt}$ by taking the colon ideal $(K : m)$ in $R[H]$, and taking those generators of $(K : m)$ that are not in $K$; these are necessarily monomials that are standard relative to a leading ideal of $K$ \cite{ES}, and by Observation~\ref{lem:standard-all-sigma}, we can then convert those to standard monomials (relative to the given pASL structure) generating $\Ann(m)$ in $A_{lt}$. Instantiating $m$ as $\pi_{lt}(\ltsm(g))$ gives us an algorithm for line~\ref{alg:ann-closure:LAM}, and instantiating it as $\pi_{lt}(m\ltsm(g))$ (in the notation of Algorithm~\ref{alg:ann-closure}) gives us an algorithm for line~\ref{alg:ann-closure:Qupdate}.

\textit{Termination.} If $G$ has stabilized (gotten to the point where line~\ref{alg:ann-closure:add-to-G} will never be executed again), then we claim that $G$ is Ann-closed. For the construction of $Q$ on lines~\ref{alg:ann-closure:LAM} and \ref{alg:ann-closure:Qupdate} is readily seen to be verifying the conditions of Ann-closure. If $G$ has stabilized, then it means that  line~\ref{alg:ann-closure:add-to-G} will never be executed again, which in turn means that on every syzygy computed on line~\ref{alg:ann-closure:std} going forward, the result has 0 remainder, and therefore that $G$ is Ann-closed. Now, to see that this must eventually happen, note that when $G$ is added to on line~\ref{alg:ann-closure:add-to-G}, it is necessarily the case that the ideal $\langle \pi_{lt}(g) : g \in G \rangle \subseteq A_{lt}$ is strictly increased; by Noetherianity, the latter can happen at most finitely many times. And thus eventually either the $Q$ is empty, or $G$ is Ann-closed, in which case the oracle correctly reports that this is the case and the algorithm correctly terminates.
\end{proof}

Several remarks are in order about the use of the oracle in Algorithm~\ref{alg:ann-closure}. 

\begin{remark}[Regarding the oracle assumption] \label{rmk:oracle}
First, we remark on the difficulty that led us to use the oracle in the first place. As written, it is conceivable that the queue $Q$ in Algorithm~\ref{alg:ann-closure} never empties; we conjecture that this cannot happen, but have been unable to either prove it or find a counterexample. 

Another natural approach to termination is to store the syzygies computed so far in a set $S$, and only add to $S$ and $Q$ the new syzygy from line~\ref{alg:ann-closure:std} does not reduce to $0$ relative to $S$ (this requires something like a term order on the ambient module in which $S$ lives, but that is not a significant issue, as for this purpose it seems any term order will do). Then one might hope to argue via Noetherianity that, once $G$ is not further expanded, $S$ can only be expanded finitely many times, thus ensuring termination. The issue here is that we are no longer assured that what we have at the end is Ann-closed, because the algorithm is no longer directly checking the definition of Ann-closure. Just because the new syzygy reduces to $0$ relative to $S$ does not necessarily mean that the second part of the Ann-closure definition is handled for all the coefficients of the new syzygy. 

Second, we remark that there is a slight modification where the oracle can be replaced by an actual algorithm. Namely, if we don't separate out the Ann-closure and S-closure computations, and instead try to compute a pseudo-ASL Gröbner basis all at once by, for example, computing S-closure every time $G$ is increased by the Ann-closure algorithm (line~\ref{alg:ann-closure:add-to-G} of Algorithm~\ref{alg:ann-closure}), then instead of an oracle to test Ann-closure we need an oracle to test whether a given set is a pASL Gröbner basis. But such an oracle can be replaced by an actual algorithm: by reverting to the parent polynomial ring, one can compute the (affine) Hilbert series of $I$ and of $\langle \pi_{lt}(g) : g \in G \rangle$; $G$ is a Gröbner basis if and only if the two are equal. On the one hand, this shows that such an oracle is not strictly necessary, and hence that computing pseudo-ASL Gröbner basis is computable in finite time; on the other hand, it is very much anathema to the theme of our development here, which is to compute intrinsically in $A$ (and perhaps $A_{lt}$) as much as possible. (We note that our other excursions into $R[H]$ have been limited to binomial ideals.)

Finally, we mention another possible approach suggested by the first paragraph of this remark. Another issue with the procedure described above is that even if the new syzyg is in $\langle S \rangle$, that does not guarantee that it reduces to $0$ modulo $S$. To ensure this reduction to $0$, one would then want to compute a pASL Gröbner bases for the syzygy module itself. This then leads to an iterative procedure computing syzygies, second syzygies, and so on, essentially computing a resolution ``all at once.'' We believe such an approach, analogous to that in La Scala \& Stillman \cite{LaScalaStillman}, can work ``intrinsically in $A$.'' This is quite a bit more complicated, however, and also would only terminate if the ideal being resolved has a resolution of finite length (which is no longer guaranteed when working with $A$-modules instead of modules over the polynomial ring, as in \cite{LaScalaStillman}). 
\end{remark}

\begin{algorithm}[h]
\caption{Computes a pseudo-ASL \gb of an ideal with respect to any $(A, \preceq, A_{lt})$}
\label{alg:asl-gb}
\begin{algorithmic}[1]
\Require{$(A, \preceq, A_{lt})$}
\medskip
\Procedure{pseudo-ASL-Gröbner-Basis}{$f_1, \ldots, f_s$}
\State $G \gets \{f_1, \ldots, f_s\}$
\Repeat
\State $G_{prev} \gets G$
\State $G \gets \textsc{S-closure}(G)$
\State $G \gets \textsc{Ann-closure}(G)$
\Until $G = G_{prev}$
\State \Return $G$
\EndProcedure
\end{algorithmic}
\end{algorithm}

\begin{theorem} \label{thm:asl-gb-alt}
Given $(A,\preceq, A_{lt})$ and a finite set $G \subseteq A$, there is an algorithm (Algorithm~\ref{alg:asl-gb}) that computes a set $G'$ that is both an Ann-closure and an S-closure of $G$ relative to $(\preceq, A_{lt})$. In particular, when $A$ is a domain, $G'$ is a pseudo-ASL Gröbner basis for $\langle G \rangle$.
\end{theorem}

\begin{proof}
The correctness of $\textsc{S-closure}(\cdot)$ is the content of Theorem~\ref{thm:computable-Sclosure}, and the correctness of $\textsc{Ann-closure}(\cdot)$ is the content of Theorem~\ref{thm:alg-ann-closure}. Thus the set $G$ that is returned by Algorithm \ref{alg:asl-gb}, whenever the algorithm terminates, is obviously both an S-closure and an Ann-closure of $G$, hence is a pASL Gröbner basis if $A$ is a domain.

Denote the sequence of intermediate $G$'s that are computed at each iteration of the loop in Algorithm \ref{alg:asl-gb} as $G_1, G_2, \ldots$. Notice that 
\[\ideal{\pi_{lt}(\ltsm(g)) \suchthat g \in G_1} \subseteq \ideal{\pi_{lt}(\ltsm(g)) \suchthat g \in G_2} \subseteq \ldots,\] is an increasing sequence of ideals in $A_{lt}$, and this is guaranteed to eventually stablize due to Noetherianity of $A_{lt}$. This in turn means that the loop terminates in finite time.
\end{proof}

\begin{example}
If $A$ is a domain, then $\{ f\}$ is a pseudo-ASL Gröbner basis for $\langle f \rangle$ relative to $A_{gen}$, since in this case by Corollary~\ref{cor:AgenGB} a set is a pseudo-ASL Gröbner basis if and only if it is S-closed, and a singleton is always S-closed.

More generally, even if $A$ is not a domain, if $m$ is a standard monomial, then $\{ m \}$ is a pseudo-ASL Gröbner basis for $\langle m \rangle$ relative to $A_{gen}$, since beyond S-closure (which holds vacuously because $G$ is a singleton), we only need to check standard monomial annihilators of $m$ in $A_{gen}$. But any such annihilator must also send $m$ to zero in $A$, hence $\{m\}$ is also Ann-closed.
\end{example}

However, relative to other $A_{lt}$, even singletons need not be pseudo-ASL Gröbner bases:

\begin{example} \label{ex:singleton}
We use as an example the bideterminant ASL structure on $2 \times 2$ matrices; this is isomorphic to the polynomial ring $R[x_{11}, x_{12}, x_{21}, x_{22}]$. While the larger explanation for this ASL structure is covered in the next section, here it is small enough we can write it out explicitly. We have $H = \{x_{11}, x_{12}, x_{21}, x_{22}, d\}$, with $\Sigma = \langle x_{21}x_{21} \rangle$, and the straightening law $x_{12} x_{21} = x_{11} x_{22} - d$, let 
\[
A = \bigslant{R[H]}{\langle x_{12} x_{21} - (x_{11} x_{22}-d) \rangle}.
\]
Consider the ideal $I \subseteq A$ generated by $x_{12}$. Let us compute a pseudo-ASL Gröbner basis for $I$ relative to $\Adisc$. As a singleton, it is already S-closed. The annihilator of $x_{12}$ in $A_{disc}$ is $\langle x_{21} \rangle$, and when we multiply $x_{12}$ by $x_{21}$ in $A$, we get $x_{11} x_{22}-d$, and this cannot be further reduced by a standard expression in terms of $x_{12}$, so we must add $x_{11} x_{22} - d$ to our Gröbner basis.

Depending on the term order, when we take the S-polynomial of $x_{12}$ and $x_{11} x_{22} - d$, we either get $x_{12} x_{11} x_{22}$ or $x_{12} d$, and either of these is $\ltdivides$-divisible by $x_{12}$, so the S-remainder is zero. Similarly, since both terms of $x_{11} x_{22} - d$ are nonzerodivisors in $A_{disc}$, our set is now Ann-closed as well, and thus the pseudo-ASL Gröbner basis of $\langle x_{12} \rangle$ relative to $A_{disc}$ is $\{x_{12}, x_{11} x_{22} - d\}$. In fact, since we were working with $A_{lt} = \Adisc$ and our argument worked regardless of which term order was used, this is the a universal pseudo-ASL Gröbner basis for $\langle x_{12} \rangle$, which in this case is reduced (universally, i.e., relative to all choices of term order and algebra of leading terms), hence is the unique reduced universal pASL Gröbner basis.
\end{example}

%% file: applications.tex

\section{Application to bideterminants and Grassmannians}
\label{sec:bideterminants}

In this section we shall demonstrate an application of our theory of ASL \gb bases. Specifically we will show that there is a theory of \gb bases in the polynomial ring in the basis of what are called standard bideterminants. 

\begin{remark}
Although we will not elaborate further on this point, it is standard that the same machinery will then apply to the coordinate rings of Grassmannians (see, e.g., \cite{DEPYoung, sturmfels}), which are generated by standard products of maximal minors.
\end{remark}

\subsection{Background on the bideterminant basis for the polynomial ring}
\label{sec:bidet-background}

We begin by providing some background on the bideterminant basis for the polynomial ring. Our notation and definitions will --- with some important modifications --- mostly follow \cite{doubilet1974foundations, desarmenien1982invariant, grinberginvariant}, though other important developments are in \cite{DEPYoung, brunsconcagrobner, brunsconcakrs}.

\subsubsection{Basic definitions}
\label{subsec:basic-bidet-definitions-poly-ring}

Let $X$ be an $n \times m$ matrix, each of whose entries is a different indeterminate $x_{i,j}$. We refer to $X$ as a ``generic matrix.'' A \emph{bideterminant} (of $X$) is a minor of $X$ (of any size) or a product thereof. To keep track of these, Young (bi)tableaux are a convenient and meaningful notation, which we build up to and recall here.

We will indicate the minor in rows $r_1, \dotsc, r_{k}$ and columns $c_1, \dotsc, c_k$ by the notation
\[
\bigbidet{r_1, \ldots, r_{k}}{c_1, \ldots, c_{k}} := \mathrm{det}\left(\begin{bmatrix}
x_{r_1, c_1} & \ldots & x_{r_1, c_k} \\
\vdots & \ddots & \vdots \\
x_{r_k, c_1} & \ldots & x_{r_k, c_k}
\end{bmatrix}\right).
\]
When a bideterminant is a product of minors, in order to more clearly see the relationship between its minor factors, we use a column-wise notation as follows (note that this is the transpose of what some authors use). A single minor as above will be denoted by the pair of columns of natural numbers:
\[
\bitab{R}{C} = \bigbitab{\tableau{r_1  \\ \vdots \\ r_{k}}}{\tableau{c_1 \\ \vdots \\ c_{k}}},
\]
When we wish to distinguish between the pair of columns as a combinatorial object versus the minor (a polynomial), we use square brackets $\bitab{R}{C}$ for the combinatorial object and parentheses $\bidet{R}{C}$ for the corresponding bideterminant, viz. with $R,C$ as above we have $\bidet{R}{C} = \bigbidet{r_1, \ldots, r_{k}}{c_1, \ldots, c_{k}}$.

For bideterminants that are products of minors, we extend the notation as follows. For a product of minors of sizes $\lambda_1 \geq \lambda_2 \geq \dotsc \geq \lambda_p$, we have
\begin{equation}
\label{eqn:generic-bitab}
\bitab{R}{C} = \bigbitab{\tableau{r_1^{(1)} & r_1^{(2)} & \bl & r_1^{(p)} \\ \vdots & \vdots & \bl & \vdots \\ \vdots & \vdots & \bl{\ldots} & r_{\lambda_p}^{(p)}\\ \vdots & r_{\lambda_2}^{(2)} \\ r_{\lambda_1}^{(1)}}}{\tableau{c_1^{(1)} & c_1^{(2)} & \bl & c_1^{(p)} \\ \vdots & \vdots & \bl & \vdots \\ \vdots & \vdots & \bl{\ldots} & c_{\lambda_p}^{(p)}\\ \vdots & c_{\lambda_2}^{(2)} \\ c_{\lambda_1}^{(1)}}},
\end{equation}
as the combinatorial datum defining the bideterminant:
\[
\bigbidet{R}{C} := \prod_{i=1}^{p} \left(\left.r_1^{(i)}, \ldots, r_{\lambda_i}^{(i)}\,\right|\,c_1^{(i)}, \ldots, c_{\lambda_i}^{(i)}\right).
\]
We refer to the combinatorial datum $\bitab{R}{C}$ as a \emph{bitableaux} (whose columns we assume are always put in non-increasing order by height), corresponding to the bideterminant $\bidet{R}{C}$. 

(One may think of the $i$-th column of $R$ as inextricably linked to the $i$-th column of $C$ in this notation. In which case one might prefer to think of $R$ and $C$ as stacked ``in front/behind'' of one another, coming out of the page.)

Given a bideterminant $f$, with the above conventions the bitableaux giving rise to $f$ is almost unique: the only possible choices are to re-order multiple columns of the same length. 

Young tableaux formalize what kind of combinatorial object $R$ and $C$ are (hence the name ``bitableaux'' for the pair $\bitab{R}{C}$). As is standard, we build up to Young tableaux starting from partitions of integers. For any $n \in \Z_{\ge 0}$, $\lambda = (\lambda_1, \ldots, \lambda_p)$ is defined as a \emph{partition} of $n$ if 
\[ \sum_{i=1}^p \lambda_i = n, \qquad \text{and}\qquad \lambda_1 \ge \ldots \ge \lambda_p > 0.\] 
We will use $\mu = (4, 2, 2, 1)$ as a running example; $\mu$ is a partition of $9$.

For a partition of $n$, $\lambda = (\lambda_1, \ldots, \lambda_p)$, we define its \emph{shape}, denoted $\shapetab{\lambda}$ (or sometimes just $\lambda$ when no confusion can arise) as the set of integer points $(i, -j)$ in the plane, with $1 \le i \le p$ and $1 \le j \le \lambda_i$. For example, the shape of $\mu$, i.e. $\shapetab{\mu}$, is 
\[ \tableau[s]{& & & \\ & & \\ \\ \\},\] with empty boxes in place of points just for easier visualization. We will say that the \emph{size of the shape/partition} $\lambda$, denoted as $\sizetab{\lambda}$, is $n$. For example, the size of $\mu = (4, 2, 2, 1)$, i.e. $\sizetab{\mu}=9$. 

A \emph{Young tableau} of shape $\lambda$ with values in a totally ordered set $E$ is an assignment of an element of $E$ to each point in the shape $\lambda$. For example, $Y_1$ and $Y_2$ are two Young tableaux of shape $\mu$ with values in $2\N$ and $\N$, respectively:
\[ Y_1 = \tableau[s]{2 & 2 & 4 & 8 \\ 4 & 4 & 10 \\6 \\ 10\\}, \qquad Y_2 = \tableau[s]{3 & 6 & 7 & 9\\4 & 9 & 8\\5 \\ 6\\}.\] 

It should now be apparent that the combinatorial objects $R,C$ arising in our indexing scheme for bideterminants are indeed Young tableaux, and a bitableaux $\bitab{R}{C}$ is any pair of Young tableaux of the same shape, which we also refer to as the shape of the bitableaux, $\shapebitab(\bitab{R}{C}) = \shapetab{R} = \shapetab{C}$. Note that, for a bideterminant $f$, if $f$ has multiple minor factors of the same size, there may be multiple distinct bitableaux that give rise to it, but they all have the same shape. Thus, if $f$ is a bideterminant, we may speak of its shape $\shapetab{f}$ unambiguously as the shape of any bitableaux giving rise to $f$.

For $\lambda = \shapebitab(\bitab{R}{C})$, we have that the length of the largest column of $\lambda$ is the degree of the largest minor factor of $\bidet{R}{C}$. Also, $|\lambda| = \deg(\bidet{R}{C})$, i.e. the size of $\lambda$ is the degree of $\bidet{R}{C}$ taken as a polynomial in the variables $x_{i,j}$. The number of parts of the partition $\lambda$ is the number of minor factors of $\bidet{R}{C}$, which is well-defined since minors are irreducible.

The convenience of Young tableaux as an indexing scheme for bideterminants will become more apparent as we recall their additional combinatorial properties. A Young tableau shall be called \hypertarget{def:normal}{\emph{normal}} if the entries in each column are strictly increasing from top to bottom. A Young tableau shall be called \emph{semi-standard}\footnote{\cite{desarmenien1982invariant} define the notion of a ``standard'' tableau, which is the transpose of what we are calling semi-standard. We choose to use the term semi-standard because ``standard'' refers to something else in modern combinatorics parlance (as in \cite{grinberginvariant}).} if it is normal and the entries in each row are non-decreasing left to right. For example, both $Y_1$ and $Y_2$ above are normal; $Y_1$ is semi-standard while $Y_2$ is not, because of the $9,8$ that are in decreasing order in the second row.

We shall call a bitableau \emph{standard} if both the row and column tableaux are semi-standard, and the bideterminant corresponding to a standard bitableau will be called a \emph{standard bideterminant}. These will end up being the standard monomials in the bideterminant ASL structure, so the nomenclature ``standard'' is consistent. 
Also, by convention, we define $\bidet{}{} = 1$. Correspondingly, we will say $\bitab{}{}$ is a standard bitableau, and thus $1$ is a standard bideterminant.

Using basic facts about the determinant, observe that if either the column $(r_1^{(i)}, \ldots, r_{\lambda_i}^{(i)})$ or $(c_1^{(i)}, \ldots, c_{\lambda_i}^{(i)})$ contains repeated entries, then $\bigbidet{r_1^{(i)}, \ldots, r_{\lambda_i}^{(i)}}{c_1^{(i)}, \ldots, c_{\lambda_i}^{(i)}} = 0$. Also, for permutations $\pi_r, \pi_c \in \mathfrak{S}_{\lambda_i}$, we have that \[\bigbidet{r_{\pi_r(1)}^{(i)}, \ldots, r_{\pi_r(\lambda_i)}^{(i)}}{c_{\pi_c(1)}^{(i)}, \ldots, c_{\pi_c(\lambda_i)}^{(i)}} = \sgn{\pi_r}\cdot \sgn{\pi_c}\cdot\bidet{r_1^{(i)}, \ldots, r_{\lambda_i}^{(i)}}{c_1^{(i)}, \ldots, c_{\lambda_i}^{(i)}}.\]

\begin{remark}
\label{rem:bidet-bitab-correspondence}
Although the correspondence between bitableaux and bideterminants is surjective but not injective (as remarked above where we defined $\shapetab{f}$), the correspondence between standard bitableaux and standard bideterminants is bijective.
\end{remark}

The \emph{content of a minor} $\bigbidet{r_1^{(i)}, \ldots, r_{\lambda_i}^{(i)}}{c_1^{(i)}, \ldots, c_{\lambda_i}^{(i)}}$ is a pair of vectors $(\alpha, \beta)$, $\alpha \in \Z_{\ge 0}^m$, and $\beta \in \Z_{\ge 0}^n$, where $\alpha_s = 1$ if $s \in \left\{r_1^{(i)}, \ldots, r_{\lambda_i}^{(i)}\right\}$, and $\alpha_s = 0$ otherwise, and $\beta_t = 1$ if $t \in \left\{c_1^{(i)}, \ldots, c_{\lambda_i}^{(i)}\right\}$, and $\beta_t = 0$ otherwise. Note that $\sum_{i=1}^m \alpha_i = \sum_{j=1}^n \beta_j$ and these sums are equal to the degree of the minor $\bigbidet{r_1^{(i)}, \ldots, r_{\lambda_i}^{(i)}}{c_1^{(i)}, \ldots, c_{\lambda_i}^{(i)}}$. The \emph{content of a bideterminant} is the (vector) sum of the contents of the minors that make up the bideterminant (counting multiplicity, if the same minor occurs more than once). For a bideterminant $f$, we will use the notation $\chi(f)$ to denote its content. 
We can also define the \emph{content of a bitableau} as the content of the corresponding bideterminant, noting that this is well-defined despite what is noted in Remark \ref{rem:bidet-bitab-correspondence}. Although we will not make use of this fact, it may help some readers to know that the content of a bideterminant is the same as its weight when the polynomial ring $\F[X]$ is viewed as a $\text{GL}_n \times \text{GL}_m$-representation under the action by left- and right-multiplication.

\subsubsection{Bideterminant ASL structure on the polynomial ring}

N.B: While the machinery we use in this section is not new \cite{desarmenien1982invariant,desarmenien1980algorithm}, our exact notion of ordering differs in subtle but crucial ways from the ordering used there, in order to get not just an ASL structure (which goes back to \cite{deconcini1987hodge,eisenbud-intro-1980}), but also an ASL term order.

The straightening law says that every bideterminant, and in turn, every element of $\F[X]$, has a unique presentation as a $\F$-linear combination of standard bideterminants. Additionally, the expression of a non-standard bideterminant as a $\F$-linear combination of standard bideterminants possesses some additional structure that makes it an ASL with a viable ASL term order, and in turn admits a theory of \gb bases in the basis of standard bideterminants. To make a precise statement (Theorem \ref{thm:straightening}), we need to introduce a few more definitions.

\begin{definition}[Partial ordering on minors]
\label{defn:po-minors}
We define a partial order $(\le)$ on minors as follows. Given two minors $\bidet{a_1, \ldots, a_m}{b_1, \ldots, b_m}$ and $\bidet{c_1, \ldots, c_n}{d_1, \ldots, d_n}$,
\[\bidet{a_1, \ldots, a_m}{b_1, \ldots, b_m} \le \bidet{c_1, \ldots, c_n}{d_1, \ldots, d_n} \iff m \ge n \text{ and } \forall i \in [n], a_i \le c_i \text{ \& } b_i \le d_i.\]
\end{definition}

There is an equivalent definition of standard bideterminant in terms of this order. Namely, a bidetermintant $f = \prod_{i=1}^t f_i$, where all $f_i$ are minors, is a standard bideterminant if all the $f_i$ are pairwise comparable in the above partial order. Equivalently, $f$ is a standard bideterminant if there is some $\sigma \in \mathfrak{S}_t$ such that \[f_{\sigma(1)} \le \ldots \le f_\sigma(t).\] In fact, this equivalent definition makes Remark \ref{rem:bidet-bitab-correspondence} even more clear because there is only one order of minors which will satisfy the requirement, which in turn indicates that there is a unique standard bitableau corresponding to this standard bideterminant.

The now-standard Theorem \ref{thm:bd-hodge-algebra} establishes that there is an ASL structure in the polynomial ring in the basis of standard bideterminants, relative to the ordering of minors of Definition~\ref{defn:po-minors}:

\begin{theorem}[\cite{deconcini1987hodge, eisenbud-intro-1980}] 
\label{thm:bd-hodge-algebra}
Let $H \subseteq \F[X]$ be the set of minors of $X$, and let $\phi: H \hookrightarrow \F[X]$ be the identity map. Let $\Sigma \subseteq \W^H$ be the set of non-standard bideterminants. Finally, let $\le$ be the order from Definition \ref{defn:po-minors}. Then $A^{bd} := (\F[X], H, \le, \Sigma, \phi)$ is a $\F$-ASL.
\end{theorem}

We will use the ordering on minors (or rather, its reverse) to build an ASL term order on bideterminants. As our orderings are slight modifications of known orderings, we define them here; we will show they give an ASL term order in Theorem~\ref{thm:bd-grobner-basis} below. 

\begin{definition}[store=defshapeorder, note={Total order on shapes ($\prec$)}]
\label{def:shape-order}
Given two shapes $\lambda$ and $\mu$, we say $\lambda$ is \emph{smaller than} $\mu$ if $\lambda$ is of smaller size than $\mu$, or if $\lambda$ and $\mu$ have the same size, and $\lambda$ is lexicographically larger than $\mu$ (when viewed as partitions), equivalently in symbols we define:
\[
\lambda \prec \mu \iff [\sizetab{\lambda} < \sizetab{\mu}] \text{ or } [\sizetab{\lambda} = \sizetab{\mu} \text{ and } \lambda >_{lex} \mu].\]
\end{definition}

\begin{definition}[store=defnorderbitab, note={Total order on standard bitableaux/bideterminants $(\prec)$\footnote{Abusing notation by reusing $\prec$.}}]
\label{defn:order-bitab}
For a standard bitableau $\bitab{R}{C}$, let $\theta(\bitab{R}{C})$ denote the sequence obtained by first reading off $R$ row by row, from left to right within each row, followed by the same for $C$. Given two standard bitableaux or standard bideterminants $f$ and $g$, we define
\[f \prec g \iff \left[\shapetab{f} \prec \shapetab{g}\right] \text{ or } \left[\shapetab{f} = \shapetab{g} \text{ and } \theta(f) >_{lex} \theta(g)\right].\]
\end{definition}

\begin{remark}
Compared to the ordering on shapes defined in \cite{desarmenien1982invariant}, in order to get an ASL term order where ``greater'' agrees with the desired notion in the standard theory of Gröbner bases, our ordering first orders by size, and among shapes of the same size is the reverse of the order on shapes from \cite{desarmenien1982invariant}. In contrast, they do not first order by size, although \cite{grinberginvariant} states a caveat about comparing two shapes where one of them is a prefix of another. Nonetheless, because the straightening algorithm of \cite[Theorem~1]{desarmenien1982invariant} always produces tableaux of the same content, they necessarily have the same size, and in that context our ordering is simply the reverse of theirs, so we get that their straightening algorithm only produces bitableaux that are lower (in our ordering).

Also, the order on standard bitableaux of the same shape, i.e. the order on the $\theta$ sequences of bitableaux, is the `transpose' of the order suggested in the proof of \cite[Theorem~1]{desarmenien1982invariant} (simply because we use the transpose convention for bitableaux everywhere), and is the reverse of the order suggested in \cite{desarmenien1980algorithm} (again, in order to make ours an ASL term order).
\end{remark}

\begin{example}
By Definition \ref{defn:po-minors}, $\bidet{1, 3, 7}{1, 2, 9} < \bidet{2, 3}{3, 4}$, whereas $\bidet{1, 3, 7}{1, 2, 9}$ and $\bidet{1 2}{3 4}$ are incomparable. Note, however, that although $\bidet{1, 3, 7}{1, 2, 9} < \bidet{2, 3}{3, 4}$ as minors, as 1-column bideterminants we have the opposite ordering relative to $\prec$ of Definition~\ref{defn:order-bitab} (this is similar to the reversal that happens with the grevlex ordering in ordering Gröbner bases).
By Definition \ref{def:shape-order}, $(5, 3) \prec (5, 2, 2, 2, 1)$, whereas $(5, 3) \succ (5, 4, 4, 2, 1)$.

For more complicated examples of Definition \ref{defn:order-bitab}, consider the bitableaux 
\[
T_1 = \bigbitab{\tableau[s]{2 & 2 & 4 & 8 \\ 4 & 4 & 10 \\6 \\ 10\\}}{\tableau[s]{2 & 2 & 4 & 8 \\ 4 & 4 & 10 \\6 \\ 10\\}}, \; T_2 = \bigbitab{\tableau[s]{2 & 2 \\ 4 & 4 \\ 6 & 7 \\ 10\\}}{\tableau[s]{2 & 2 \\ 4 & 4 \\ 6 & 7 \\ 10\\}},
\]  
and
\[
T_3 = \bigbitab{\tableau[s]{2 & 2 & 3 & 9 \\ 4 & 4 & 10 \\6 \\ 10\\}}{\tableau[s]{2 & 2 & 4 & 8 \\ 4 & 4 & 10 \\6 \\ 10\\}}. 
\]
$T_1$ has shape $\shapetab{\lambda_1} = (4, 2, 2, 1)$, and $T_2$ has shape $\shapetab{\lambda_2} = (4, 3)$, so since $\shapetab{\lambda_1} <_{lex} \shapetab{\lambda_2}$, $\shapetab{\lambda_1} \succ \shapetab{\lambda_2}$, and consequently $T_1 \succ T_2$. For the same reason $T_3 \succ T_2$. Finally, $\theta(T_1) = (2, 2, 4, 8, 4, \ldots, 6, 10)$, and $\theta(T_3) = (2, 2, 3, 9, 4, \ldots, 10)$, and since $\theta(T_1) >_{lex} \theta(T_3)$, $T_1 \prec T_3$.
\end{example}

\begin{theorem}[Straightening law \cite{hodge, doubilet1974foundations, desarmenien1982invariant}, see also \cite{DEPYoung, ABW, brunsconcagrobner, brunsconcakrs}]
\label{thm:straightening}
Let $\F$ be any arbitrary commutative ring, and let $\F[X] := \F[\{X_{ij} : i \in [n], j \in [m]\}]$.
\begin{enumerate}[label=(\Roman*),ref=(\Roman*),itemindent=-10pt]
\item \label{point:straightening} Suppose for two minors $f, g \in \F[X]$, the product $fg$ is not a standard bideterminant, then there is a unique representation of $fg$ as the following finite sum
\[fg = \sum_i {r_i} \eta_{i, 1} \eta_{i, 2},\]
where for all $i$, $r_i \in \F$ (in fact they are integers), $\eta_{i, 1} \eta_{i, 2}$ is a standard bitableau, and for some $\sigma \in \mathfrak{S}_2$, $\eta_{i, \sigma(1)} < f$, $\eta_{i, \sigma(1)} < g$, $f < \eta_{i, \sigma(2)}$, and $g < \eta_{i, \sigma(2)}$.
\item \label{point:bideterminant-in-sbd-basis} By successive application of \ref{point:straightening}, any arbitrary bideterminant $\bidet{R}{C}$ can be uniquely expressed as follows:
\begin{equation} \label{eqn:bitab-straighten} \bidet{R}{C} = \sum a_{R_i, C_i} \bidet{R_i}{C_i}, \end{equation} where $\bidet{R_i}{C_i}$ are standard, and only finitely many $a_{R_i, C_i} \in \F$ are non-zero. 
\item \label{point:straightening-orders} Equation \eqref{eqn:bitab-straighten} satisfies the additional property that, for all $i$ with $a_{R_i, C_i} \neq 0$, \[\shapebitab{\bidet{R_i}{C_i}} \le \shapebitab{\bidet{R}{C}}, \qquad \text{and} \qquad \chi(\bidet{R_i}{C_i}) = \chi(\bidet{R}{C}).\]
\item \label{point:polynomial-in-bd-basis} Since every monomial is trivially a bideterminant corresponding to a bitableau where each tableau has columns all of length one, and number of columns equal to the degree of the monomial, as a consequence of \ref{point:bideterminant-in-sbd-basis}, any polynomial $f \in \F[X]$ may be written uniquely as follows: 
\begin{equation} \label{eqn:poly-bidet-basis} f = \sum a_{R, C} \bidet{R}{C}, \end{equation}
where the bideterminants $\bidet{R}{C}$ are standard, and only finitely many $a_{R,C} \in \F$ are nonzero. Thus the standard bideterminants form a $\F$-linear basis of the polynomial ring.
\end{enumerate}
\end{theorem}

See Appendix \ref{sec:app:bidet-straightening} for the algorithmic details of Theorem \ref{thm:straightening}\ref{point:straightening}. For any $f \in \F[X]$, we say that a standard bideterminant $\bidet{R}{C}$ \emph{occurs} or \emph{appears} in $f$ if it is accompanied with a non-zero coefficient in the expression of $f$ (see Equation \eqref{eqn:poly-bidet-basis}) in the basis of standard bideterminants. We are careful to only apply this term to standard bideterminants.

Given two bitableaux $\bitab{R}{C}$ and $\bitab{R'}{C'}$, we define $\bitab{R + R'}{C + C'}$ to be any bitableau corresponding to the bideterminant $\bidet{R}{C} \cdot \bidet{R'}{C'}$. This corresponds to $R+R'$ having its set of columns be the union of the sets of columns of $R$ and $R'$, and similarly for $C$ and $C'$ (taking care to maintain the association between columns of $R$ with columns of $C$, and columns of $R'$ with their corresponding columns in $C'$). We may sometimes refer to this as ``merging'' the columns of $R$ and $R'$, and those of $C$ and $C'$, though the operation must be done simultaneously on both $R,R'$ and $C,C'$.

\begin{remark}
\begin{enumerate}
\item $\bitab{R + R'}{C + C'}$ is not well-defined usually because, if there are columns of the same height, all permutations of those columns are acceptable. For example, if
\[\bitab{R}{C} = \bigbitab{\tableau[s]{1 & 2 \\ 4 & 5\\ 5  \\ 7 \\ }}{\tableau[s]{1 & 1 \\ 2 & 6 \\ 4 \\ 5 \\ }}, \qquad \text{and} \qquad \bitab{R'}{C'} = \bigbitab{\tableau[s]{ \tf{2} & \tf{2} \\  \tf{4} \\  \tf{5} \\  \tf{9} \\ }}{\tableau[s]{ \tf{1} &  \tf{2} \\  \tf{3} \\  \tf{5} \\  \tf{9} \\ }},\]
then any of
\[\bigbitab{\tableau[s]{\tf{2} & 1 & 2 & \tf{2} \\ \tf{4} & 4 & 5 \\ \tf{5} & 5 & \bl \\ \tf{9} & 7 \\ }}{\tableau[s]{ \tf{1} & 1 & 1 & \tf{2} \\ \tf{3} & 2 & 6 \\ \tf{5} & 4  & \bl \\ \tf{9} & 5 \\ }} \qquad \text{and} \qquad \bigbitab{\tableau[s]{1 & \tf{2} & 2 & \tf{2} \\ 4 & \tf{4} & 5 \\ 5 & \tf{5} & \bl \\ 7 & \tf{9} \\ }}{\tableau[s]{ 1 & \tf{1} & 1 & \tf{2} \\ 2 & \tf{3} & 6 \\ 4  & \tf{5} & \bl \\ 5 & \tf{9} \\ }}\]
are valid representations of $\bitab{R+R'}{C+C'}$.
\item If it is possible to form a standard bitableau by merging, then there is only one such. In the above example,
\[\bigbitab{\tableau[s]{1 & \tf{2} & 2 & \tf{2} \\ 4 & \tf{4} & 5 \\ 5 & \tf{5} & \bl \\ 7 & \tf{9} \\ }}{\tableau[s]{ 1 & \tf{1} & 1 & \tf{2} \\ 2 & \tf{3} & 6 \\ 4  & \tf{5} & \bl \\ 5 & \tf{9} \\ }}\]
is the only standard bitableau, while if
\[\bitab{R}{C} = \bigbitab{\tableau[s]{1}}{\tableau[s]{2}}, \qquad \text{and} \qquad \bitab{R'}{C'} = \bigbitab{\tableau[s]{\tf{2}}}{\tableau[s]{\tf{1}}},\]
then there is no candidate for $\bitab{R+R'}{C+C'}$ which is a standard bitableau.
\item Although $\bitab{R + R'}{C + C'}$ is not uniquely defined, $\bidet{R+R'}{C+C'}$ is, and satisfies
\[\bidet{R+R'}{C+C'} = \bidet{R}{C} \cdot \bidet{R'}{C'}.\]
\end{enumerate}
\end{remark}

The following result gives the leading standard bitableau (or equivalently standard bideterminant) in the expression \eqref{eqn:bitab-straighten}. For a \hyperlink{def:normal}{normal} $\bitab{R}{C}$, define $\bitab{\stdize{R}}{\stdize{C}}$ to be the bitableau obtained from $\bitab{R}{C}$ by writing the entries of each row of each of the tableaux in ascending order. 

\begin{remark}
For those who are used to the transpose convention of ours, note that obtaining $\bitab{\stdize{R}}{\stdize{C}}$ does \emph{not} merely correspond to rearranging the order of indices in each minor, and thereby correspond just to a sign change of the bideterminant; it is really (potentially) changing which minors are occurring. Also, the sorting operation is done separately to $R$ and $C$, which again can change the bideterminant.
\end{remark}

\begin{theorem}[\cite{desarmenien1980algorithm}]
\label{thm:leading-bd}
\begin{enumerate}[label=(\Roman*),ref=(\Roman*),itemindent=-10pt]
\item Let $\bitab{R}{C}$ be a normal bitableau. Then $\bitab{\stdize{R}}{\stdize{C}}$ is the largest standard bitableau (relative to our ordering $\prec$ of Definition~\ref{defn:order-bitab}) whose bideterminant appears in the decomposition of $\bidet{R}{C}$ as in expression \eqref{eqn:bitab-straighten}, and has coefficient 1 in that expression.
\item Given two standard bitableaux $\bitab{R_1}{C_1}$ and $\bitab{R_2}{C_2}$, $\bitab{\stdize{R_1 + R_2}}{\stdize{C_1 + C_2}}$ is a standard bitableau, and \[LT_{\prec}(\bidet{R_1}{C_1}\bidet{R_2}{C_2}) = \bidet{\stdize{R_1 + R_2}}{\stdize{C_1 + C_2}}.\] 
\end{enumerate}
\end{theorem}

\begin{example}[Merging and sorting]
For example, \leavevmode{
\begin{table}[hbt!]
\centering
\begin{tabular}{cccc}
$\bitab{R}{C}$ & $\bitab{R'}{C'}$ & $\bitab{R + R'}{C + C'}$ & $\bigbitab{\stdize{R + R'}}{\stdize{C + C'}}$\\
& & & \\
\parbox{2.5cm}{\centering $\bigbitab{\tableau[s]{1 & 1 \\ 4 & 5\\ 5  & 6 \\ 7 \\ }}{\tableau[s]{1 & 1 \\ 2 & 2 \\ 4  & 4 \\ 5 \\ }}$} & \parbox{2.5cm}{\centering $\bigbitab{\tableau[s]{ \tf{2} & \tf{2} \\  \tf{3} & \tf{3} \\  \tf{5} \\  \tf{9} \\ }}{\tableau[s]{ \tf{1} &  \tf{2} \\  \tf{3} & \tf{4} \\  \tf{5} \\  \tf{9} \\ }}$} & \parbox{4.2cm}{\centering $\bigbitab{\tableau[s]{1 & \tf{2} & 1 & \tf{2} \\ 4 & \tf{3} & 5 & \tf{3} \\ 5 & \tf{5} & 6 & \bl \\ 7 & \tf{9} \\ }}{\tableau[s]{ 1 & \tf{1} & 1 & \tf{2} \\ 2 & \tf{3} & 2 & \tf{4} \\ 4  & \tf{5} & 4 & \bl \\ 5 & \tf{9} \\ }}$} & \parbox{4.2cm}{\centering $\bigbitab{\tableau[s]{1 & 1 & \tf{2} & \tf{2} \\ \tf{3} & \tf{3} & 4 & 5 \\ 5 & \tf{5} & 6 & \bl \\ 7 & \tf{9} \\ }}{\tableau[s]{ 1 & \tf{1} & 1 & \tf{2} \\ 2 & 2 & \tf{3} & \tf{4} \\ 4  & 4 & \tf{5} & \bl \\ 5 & \tf{9} \\ }}$}\\
\end{tabular}.
\end{table}}
\end{example}

\subsection{bd-\gb bases in the polynomial ring}

Theorem \ref{thm:bd-grobner-basis} below is the main theorem of this section, which establishes that there is a theory of ASL Gröbner bases in the basis of standard bideterminants in the polynomial ring.

\begin{theorem}[store=thmbdgrobnerbasis]
\label{thm:bd-grobner-basis} 
Let $A^{bd} = (\F[X], H, \le, \Sigma, \phi)$ be as defined in Theorem \ref{thm:bd-hodge-algebra}. Then the order $\preceq$ in Definition \ref{defn:order-bitab} is an ASL term order on $A^{bd}$. Thus $(A^{bd}, \preceq)$ admits a theory of ASL Gröbner bases.
\end{theorem}

\begin{proof}
Theorem \ref{thm:bd-hodge-algebra} gives us that $A^{bd}$ is indeed an ASL. All we have to do is to prove that $\preceq$ is a valid ASL term order. Once this is in place, all the results of the prior sections may be applied in the setting of the standard bideterminant basis.

First observe that Axiom \ref{def:termorder:positive} is easily proven by noting that $1$ corresponds to the bitableau $\bitab{}{}$, which in turn is the only bitableau that has size $0$; Definition \ref{defn:order-bitab} first orders by size, which ensures that $1 \preceq m$ for all standard bideterminants $m$.

To establish \ref{def:termorder:mult}, first notice that the product of two standard bideterminants is never $0$. Since $A^{bd}$ is a domain, by Remark~\ref{rem:ato-2-equivalent}, it is necessary and sufficient to show that for standard bideterminants $f, g, h$, if $f \prec g$, then $\lmsm(fh) \prec \lmsm(gh)$.

So we now show $\lmsm(fh) \prec \lmsm(gh)$ when $f \prec g$. We first prove this for the case where $h$ is a single minor, and then use this to derive the general case below. So suppose $h$ is a single minor of size $l_h$. 

We need the following simple claim which is proved in Appendix \ref{sec:app:bideterminants}.

\begin{claim}[store=claimlexicographicunion]
\label{claim:lexicographic-union}
Given a multiset $A$ with elements from $\N$, let $\mathcal{D}(A)$ denote the list of elements of $A$ in descending order, and let $\mathcal{U}(A)$ denote the list of elements of A in ascending order.
\begin{enumerate}[label=(\Roman*),ref=(\Roman*),itemindent=-10pt]
\item \label{point:descending} Suppose $\mathcal{D}(A) <_{lex} \mathcal{D}(B)$. Then for any $\gamma \in \N$, $\mathcal{D}(A \cup \{\gamma\}) <_{lex} \mathcal{D}(B \cup \gamma)$.
\item \label{point:ascending} Suppose $\mathcal{U}(A) <_{lex} \mathcal{U}(B)$. Then for any $\gamma \in \N$, $\mathcal{U}(A \cup \{\gamma\}) <_{lex} \mathcal{U}(B \cup \gamma)$.
\end{enumerate}

\end{claim}

[Case 1 $\shapebitab(f) \neq \shapebitab(g)$]: Then since $f \prec g$, $\shapebitab(f) \prec \shapebitab(g)$. From Theorem \ref{thm:leading-bd}, we have that 
\[ \sizetab{\shapebitab(\lmsm(fh))} = \sizetab{\shapebitab(f)} + l_h, \text{ and }\sizetab{\shapebitab(\lmsm(gh))} = \sizetab{\shapebitab(g)} + l_h.\] Thus, if $\sizetab{\shapebitab(f)} < \sizetab{\shapebitab(g)}$, then $\sizetab{\lmsm(fh)} < \sizetab{\lmsm(gh)}$, proving that $\lmsm(fh) \prec \lmsm(gh)$. 

It remains to consider the case where $\sizetab{\shapebitab(f)} = \sizetab{\shapebitab(g)}$. Since $f \prec g$, it must be true that $\shapebitab(g) <_{lex} \shapebitab(f)$. Now, by Theorem \ref{thm:leading-bd}, we have that the shape of the leading standard monomial that appears in $fh$ is $\shapebitab(\lmsm(fh)) = \mathcal{D}(\shapebitab(f) \cup \{l_h\})$,\footnote{We abuse notation here by treating $\shapebitab(f)$ as a set while it is supposed to be a list.} where $l_h$ is the height of the single column of the bitableau corresponding to $h$. Also, the shape $\shapebitab(\lmsm(gh)) = \mathcal{D}(\shapebitab(g) \cup \{l_h\})$. Thus by Claim \ref{claim:lexicographic-union}-\ref{point:descending}, we have that $\shapebitab(\lmsm(gh)) <_{lex} \shapebitab(\lmsm(fh))$, which implies that $\shapebitab(\lmsm(fh)) \prec \shapebitab(\lmsm(gh))$, and consequently, $\lmsm(fh) \prec \lmsm(gh)$.

[Case 2 $\shapebitab(f) = \shapebitab(g)$]: In this case, we must have $\theta(g) <_{lex} \theta(f)$. We shall first prove that $\lmsm(fh) \prec \lmsm(gh)$ when $f$ (and of course $g$ too because $\shapebitab(f) = \shapebitab(g)$) has exactly one row. Later, we will claim that this is without loss of generality, thus establishing $\lmsm(fh) \prec \lmsm(gh)$ for any $f, g, h$. Let the bitableaux be 
\[ 
f = \bigbitab{\tableau{f^{(r)}_1 & \ldots & f^{(r)}_s}}{\tableau{f^{(c)}_1 & \ldots & f^{(c)}_s}}, g = \bigbitab{\tableau{g^{(r)}_1 & \ldots & g^{(r)}_s}}{\tableau{g^{(c)}_1 & \ldots & g^{(c)}_s}},
\] 
and 
\[ 
h = \bigbitab{\tableau{h^{(r)}_1 \\ \vdots \\ h^{(r)}_t}}{\tableau{h^{(c)}_1 \\ \vdots \\ h^{(c)}_t}}.
\] 
By Theorem \ref{thm:leading-bd}, we have that 
\[
\theta(\lmsm(fh)) = \left(\underbrace{\mathcal{U}(\{f^{(r)}_1, \ldots, f^{(r)}_s, h^{(r)}_1\})}_{\mathcal{F}_1}, h^{(r)}_2, \ldots, h^{(r)}_t, \underbrace{\mathcal{U}(\{f^{(c)}_1, \ldots, f^{(c)}_s, h^{(c)}_1\})}_{\mathcal{F}_2}, h^{(c)}_2, \ldots, h^{(c)}_t\right),
\] 
and 
\[
\theta(\lmsm(gh)) = \left(\overbrace{\mathcal{U}(\{g^{(r)}_1, \ldots, g^{(r)}_s, h^{(r)}_1\})}^{\mathcal{G}_1}, h^{(r)}_2, \ldots, h^{(r)}_t, \overbrace{\mathcal{U}(\{g^{(c)}_1, \ldots, g^{(c)}_s, h^{(c)}_1\})}^{\mathcal{G}_2}, h^{(c)}_2, \ldots, h^{(c)}_t\right).
\] 
By definition, we have 
\[
\theta(f)=(f^{(r)}_1, \ldots, f^{(r)}_s, f^{(c)}_1, \ldots, f^{(c)}_s), \qquad \text{and} \qquad \theta(g)=(g^{(r)}_1, \ldots, g^{(r)}_s, g^{(c)}_1, \ldots, g^{(c)}_s).
\] Since $\theta(g) <_{lex} \theta(f)$, we can consider where the first mismatch in these two sequences occurs. It must either occur in the entries coming from the row tableaux or the column tableaux. 

If the first mismatch comes from the row tableaux, that is, in $(f^{(r)}_1, \ldots, f^{(r)}_s)$ and $(g^{(r)}_1, \ldots, g^{(r)}_s)$, then we are certain to have a mismatch between $\mathcal{F}_1$ and $\mathcal{G}_1$. By Claim \ref{claim:lexicographic-union}-\ref{point:ascending}, we have that 
\[
\mathcal{U}(\{g^{(r)}_1, \ldots, g^{(r)}_s, h^{(r)}_1\}) <_{lex} \mathcal{U}(\{f^{(r)}_1, \ldots, f^{(r)}_s, h^{(r)}_1\}),
\]
 proving that $\theta(\lmsm(gh)) <_{lex} \theta(\lmsm(fh))$, and in turn, $\lmsm(fh) \prec \lmsm(gh)$. 

On the other hand, suppose that the first mismatch between $\theta(f)$ and $\theta(g)$ occurs in the column tableaux, that is, between $(f^{(c)}_1, \ldots, f^{(c)}_s)$ and $(g^{(c)}_1, \ldots, g^{(c)}_s)$. Similar to before, comparing $\theta(\lmsm(fh))$ and $\theta(\lmsm(gh))$ is equivalent to comparing $\mathcal{F}_2$ and $\mathcal{G}_2$, and in this case, Claim \ref{claim:lexicographic-union}-\ref{point:ascending} gives us that $\theta(\lmsm(gh)) <_{lex} \theta(\lmsm(fh))$, implying $\lmsm(fh) \prec \lmsm(gh)$.

We have proved the claim for the case of single row bitableaux $f$ and $g$, and it remains to be argued that this is without loss of generality. Since $\shapebitab(f) = \shapebitab(g)$, the first mismatch between $\theta(f)$ and $\theta(g)$ will occur in some row $\kappa$ of both bitableaux. Even after multiplication by $h$, Theorem \ref{thm:leading-bd} implies that the mismatch between $\theta(\lmsm(fh))$ and $\theta(\lmsm(gh))$ will be continue to be in row $\kappa$ of $\lmsm(fh)$ and $(\lmsm(gh))$, and thus the same argument above, applied to just the row $\kappa$, gives us $\lmsm(fh) \prec \lmsm(gh)$ in full generality.
\end{proof}

\newcommand{\mset}[1]{\{\!\!\{#1\}\!\!\}}

\subsection{Algorithmic aspects of finding bd-\gb bases}
\label{sec:bd-gb-algos}

In this section we present and prove Algorithm \ref{alg:lcsm}, which takes two standard bideterminants as input and returns the set of least common standard multiples (Definition \ref{defn:lcms}) of these two standard bideterminants, relative to $A^{bd}_{gen}$. Note that it is completely straightforward to return least common multiples relative to $A_{disc}^{bd}$; it is exactly equivalent to taking LCMs by considering standard monomials as elements of $R[H]$ (see Obs.~\ref{obs:lcm-unique-disc}). While Lemma~\ref{lem:comp-LCM} shows that LCMs relative to any algebra of leading terms are computable in finite time, Algorithm \ref{alg:lcsm} is tailored to the specific case of bideterminants relative to $A_{gen}^{bd}$ and is not only more efficient than the generic one of Lemma~\ref{lem:comp-LCM}, but also yields more insight into the nature of LCMs in the bideterminant setting.

\begin{convention} \label{conv:multiset}
In this section, we will be dealing with multisets. Just to clarify, a multiset $A$ containing elements from $\N$ is naturally equivalent to a function $m_A: \N \rightarrow \W$. Given two multisets $A$ and $B$, $A \cup B$ will denote $m_{A \cup B}: \N \rightarrow \W$ that takes $x \mapsto \max\left\{m_A(x), m_B(x)\right\}$. $A \setminus B$ will denote $m_{A \setminus B} : \N \rightarrow \W$ that takes $x \mapsto \max\left\{m_A(x) - m_B(x), 0\right\}$. $|A|$ will denote $\sum_{x \in \N} m_A(x)$, and we will say $A \subseteq B$ if $m_A(x) \le m_B(x)$ for all $x \in \N$. Finally, instead of describing $A$ by $m_A$, we will describe $A$ by listing all the elements $x \in \N$ for which $m_A(x) > 0$, and in our description, each $x$ will be listed $m_A(x)$ number of times. For example,if $m_A$ is such that $m_A(1) = 1, m_A(2) = 2, m_A(3) = 1$, and $m_A(x) = 0$ for all $x > 3$, then we will write $A = \mset{1, 2, 2, 3}$.
\end{convention}

\begin{notation}
Given a tableau $T$, let $\mathcal{R}(T, i)$ denote the multiset of elements of the $i^{\text{th}}$ row of $T$ (which will be empty if $T$ has no $i^{\text{th}}$ row). 
\end{notation}

Also, we will take for granted the following two combinatorial procedures:

\newcommand{\getadsh}{\textsc{get-admissible-shapes}\xspace}
\newcommand{\getadsbd}{\textsc{get-admissible-standard-bd}\xspace}
\newcommand{\getvalidalt}{\textsc{get-LCMs-valid-in-$A_{lt}$}\xspace}

\begin{enumerate}
\item \getadsh$\bm{\left(\Lambda, (r_i)_{i \in [N]}, (c_i)_{i \in [N]}\right)}:$ The first parameter $\lambda$ is one single shape, and the second and third parameters, i.e. $(r_i)_{i \in [N]}$ and $(c_i)_{i \in [N]}$, are lists of sizes. This routine returns all valid shapes $\lambda'$ such that
\begin{enumerate}
\item number of columns in $\lambda'$ exactly equals $\max(r_1, \ldots, r_N, c_1, \ldots, c_N)$,
\item number of columns in the $i^{\text{th}}$ row of $\lambda'$ is at least $\max(r_i, c_i)$, and
\item treating $\lambda'$ and $\lambda$ as multisets, it is true that $\lambda \subseteq \lambda'$.
\end{enumerate}

\item \getadsbd$\bm{\left(\Lambda, (R_i)_{i \in [N]}, (C_i)_{i \in [N]}\right)}:$ The first parameter $\Lambda$ is a set of shapes. The second and third parameters are both a list of sets of elements; each $R_i$ contains a set of elements that must be present in the $i^{\text{th}}$ row of the row tableau of any bitableau that is returned, and each $C_i$ contains a list of elements that must be present in the $i^{\text{th}}$ row of the column tableau of any bitableau that is returned. This routine returns the set of all standard bitableaux whose shape $\lambda$ is in $\Lambda$, and which can be formed while ensuring that all elements of $R_i$ are present in the $i^{\text{th}}$ row of the row tableau and that all elements of $C_i$ are present in the $i^{\text{th}}$ row of the column bitableau.
\end{enumerate}

\begin{remark}
Note that the content of the $i$-th row will not be exactly $R_i$ (it will, by construction, be a super-multiset, but may properly contain $R_i$). There will in general be LCMs that realize this possibility, though there will also be at least one row of the row tableau whose content is exactly $R_i$ or one row of the column tableau whose content is exactly $C_i$.
\end{remark}

Theorem \ref{thm:lcm-algo} below proves the correctness of Algorithm \ref{alg:lcsm}.

\newcommand{\lambdaf}{\lambda^{(f)}}
\newcommand{\lambdag}{\lambda^{(g)}}
\newcommand{\lambdalcm}{\Lambda^{(lcm)}}

\begin{algorithm}
\caption{Algorithm to compute $\lcmlt(\cdot)$ for a specific $A_{lt}$}
\label{alg:lcsm}
\begin{algorithmic}[1]
\Require{Bitableau $f = \bitab{R_f}{C_f}$, $g = \bitab{R_g}{C_g}$, with $\shapetab{f} = (\lambdaf_1, \ldots, \lambdaf_r)$, $\shapetab{g} = (\lambdag_1, \ldots, \lambdag_s)$}
\medskip
\State $N = \max(\lambdaf_1, \lambdag_1)$ \label{line:N}
\For{$i \gets 1$ to $N$}
	\State $re_i \gets \mathcal{R}(R_f, i) \cup \mathcal{R}(R_g, i)$ \Comment{Multi-set union, see Convention~\ref{conv:multiset}} \label{line:re}
	\State $ce_i \gets \mathcal{R}(C_f, i) \cup \mathcal{R}(C_g, i)$ \Comment{Multi-set union} \label{line:ce}
	\State $rd_i \gets |re_i|$ \label{line:rd}
	\State $cd_i \gets |ce_i|$ \label{line:cd}
\EndFor
\medskip
\State $\lambdalcm \gets \getadsh\left(\shapetab{f} \cup \shapetab{g}, (rd_i)_{i \in [N]}, (cd_i)_{i \in [N]}\right)$ \label{line:admissible-shapes}
\medskip
\State $\mathfrak{L} \gets \getadsbd\left(\lambdalcm, (re_i)_{i \in [N]}, (ce_i)_{i \in [N]}\right)$ \Comment{LCMs relative to $A_{gen}^{bd}$} \label{line:admissible-sbd}
\State \Return $\{\bidet{R}{C} \in \mathfrak{L} \suchthat f \ltdivides \bidet{R}{C} \text{ and } g \ltdivides \bidet{R}{C}\}$. \label{line:check-divides}
\end{algorithmic}
\end{algorithm}

\begin{theorem}[store=thmlcmalgo]
\label{thm:lcm-algo}
Given two standard bideterminants $f = \bitab{R_f}{C_f}$ and $g = \bitab{R_g}{C_g}$ such that $\pi_{lt}(f)\pi_{lt}(g) \neq 0$, Algorithm~\ref{alg:lcsm} halts in finite time and returns the set of least common standard multiples of $f$ and $g$ relative to $A_{lt}^{bd}$.
\end{theorem}

Understanding how divisibility works in $A^{bd}_{gen}$ is the key piece required in the proof of Theorem \ref{thm:lcm-algo}. To get our more combinatorial understanding of divisibility in $A^{bd}_{gen}$, we use the following definition.

\begin{definition}[$\bitab{R' \setminus R}{C' \setminus C}$]
Let $\bitab{R'}{C'}$ and $\bitab{R}{C}$ be standard bitableaux of shapes $\lambda' = (\lambda'_1, \ldots, \lambda'_s)$ and $\lambda = (\lambda_1, \ldots, \lambda_t)$ respectively, such that $\lambda  \subseteq \lambda'$ as multisets, and the multiset of elements of every row of $R$ (resp., $C$) is a sub-multiset of the elements of the same row in $R'$ (resp., $C'$).
Then $\bitab{R' \setminus R}{C' \setminus C}$ is defined to be the standard bitableau of shape $\shapetab{\lambda'} \setminus \shapetab{\lambda}$ and \[\mathcal{R}(R' \setminus R, i) = \mathcal{R}(R', i) \setminus \mathcal{R}(R, i), \qquad \text{and} \qquad \mathcal{R}(C' \setminus C, i) = \mathcal{R}(C', i) \setminus \mathcal{R}(C, i).\] 
\end{definition}

\begin{example}[Formation of $\bitab{R' \setminus R}{C' \setminus C}$]
Given standard bitableaux $\bitab{R}{C}$ and $\bitab{R'}{C'}$, it might be useful to view the formation of $\bitab{R' \setminus R}{C' \setminus C}$ as a three step process: (Step 1) shape removal, (Step 2) element deletion, and (Step 3) compactification. Suppose \[\bitab{R'}{C'} = \bigbitab{\tableau[s]{1 & 2 & 2 \\ 3 & 3 & 5 \\ 5 & 6 & \bl \\ 9 \\ }}{\tableau[s]{ 1 & 1 & 2 \\ 2 & 3 & 4 \\ 4  & 5 \\ 9 \\ }}, \qquad \text{and} \qquad \bitab{R}{C} = \bigbitab{\tableau[s]{1 \\ 5\\ 6 \\ 9}}{\tableau[s]{1 \\ 4 \\ 5 \\ 9}},\qquad \] then computing $\bitab{R' \setminus R}{C' \setminus C}$ can be done as follows:
\leavevmode{\begin{table}[H]
\centering
\begin{tabular}{ccc}
{1. $\bitab{R'}{C'}$} & {2. remove shape of $\bidet{R}{C}$} &  {3. delete elements of $\bidet{R}{C}$}\\
$\downarrow$ & $\downarrow$ & $\downarrow$ \\
\parbox{3.3cm}{\centering
$\bigbitab{\tableau[s]{1 & 2 & 2 \\ 3 & 3 & 5 \\ 5 & 6 & \bl \\ 9 \\ }}{\tableau[s]{ 1 & 1 & 2 \\ 2 & 3 & 4 \\ 4  & 5 \\ 9 \\ }}$} & \parbox{3.3cm}{\centering
$\bigbitab{\tableau[s]{\bl{1} & 2 & 2 \\ \bl{3} & 3 & 5 \\ \bl{5} & 6 \\ \bl{9}}}{\tableau[s]{ \bl{1} & 1 & 2 \\ \bl{2} & 3 & 4 \\ \bl{4}  & 5 \\ \bl{9}}}$} &  \parbox{3.3cm}{\centering
$\bigbitab{\tableau[s]{\bl{} & 2 & 2 \\ \bl{3} & 3 &  \\ \bl{5} &  \\ \bl{}}}{\tableau[s]{ \bl{} & 1 & 2 \\ \bl{2} & 3 &  \\ \bl{4}  &  \\ \bl{}}}$}\\ \\
& {4. compactify to get $\bitab{R' \setminus R}{C' \setminus C}$}\\
& $\downarrow$ & \\ 
 & \parbox{2.5cm}{\centering
$\bigbitab{\tableau[s]{2 & 2 \\ 3 & 3 \\ 5}}{\tableau[s]{1 & 2 \\ 2 & 3 \\ 4}}$.}\\
\end{tabular}.
\end{table}}
\end{example}

\begin{lemma} \label{lem:lt-div-bd}
Consider $(A^{bd}, A^{bd}_{lt})$. Suppose $f = \bidet{R_f}{C_f}$ and $g = \bidet{R_g}{C_g}$ are both standard. If $f \ltdivides g$, then $\bitab{R_g \setminus R_f}{C_g \setminus C_f}$ is defined and satisfies \[\bidet{R_g \setminus R_f}{C_g \setminus C_f} = \ltfrac{g}{f}.\]
\end{lemma}

\begin{proof}
Immediate consequence of Theorem \ref{thm:leading-bd}.
\end{proof}

\begin{proof}[Proof of Theorem \ref{thm:lcm-algo}]
Since we are working in the ring $\F[\{X_{i,j}\}_{i \in [m], j \in [n]}]$, the columns of our bitableaux can be at most $\min(m, n)$ long, because any minor corresponding to a taller column is zero since it has a repeated row or repeated column. Thus both procedures \getadsh and \getadsbd halt in finite time, as they are exploring finite sets for bitableaux satisfying certainly easily recognizable conditions. By Prop.~\ref{prop:div-alg}, the division checks in $A_{lt}$ on line~\ref{line:check-divides} are also computable in finite time. These together imply that Algorithm \ref{alg:lcsm} indeed halts in finite time. It remains to establish that, relative to $A_{gen}^{bd}$, the algorithm returns all least common multiples (completeness), and returns nothing other than least common multiples (soundness). Once we show that it works correctly relative to $A^{bd}_{gen}$, we note that line~\ref{line:check-divides} then ensures it works correctly relative to $A^{bd}_{lt}$: For if $m,m'$ are two standard monomials in any $A_{lt}$, then the presence of term ordering guarantees that if $m \ltdivides m'$, then there is a unique standard term $m' / m$ (Lemma~\ref{lem:mon-div}); thus if we can compute $t = m' / m$ in $A_{gen}$, then we can simply check in $A_{lt}$ whether $tm = m'$ (in which case $m'$ is indeed a multiple of $m$ in $A_{lt}$) or $tg = 0$ (in which case it is not). 

(Soundness.) In Line \ref{line:admissible-sbd}, Algorithm \ref{alg:lcsm} computes $\mathfrak{L}$ which we claim is exactly equal to the set of LCMs of $f$ and $g$ relative to $A_{gen}^{bd}$. By the description of $\getadsbd$, we can see that in every standard bideterminant returned in Line \ref{line:admissible-sbd}, the $i^{\text{th}}$ row of the row tableau contains both $\mathcal{R}(R_f, i)$ and $\mathcal{R}(R_g, i)$ as subsets, and the $i^{\text{th}}$ row of the column tableau contains both $\mathcal{R}(C_f, i)$ and $\mathcal{R}(C_g, i)$ as sub-multisets. Also, by the description of $\getadsh$, the call in Line \ref{line:admissible-shapes} ensures that the shapes of the standard bitableaux that are eventually returned by Algorithm \ref{alg:lcsm} are such that they contain both $\shapetab{f}$ and $\shapetab{g}$ as subshapes. Thus, by Lemma \ref{lem:lt-div-bd}, we can see that every standard bideterminant that is returned by Algorithm \ref{alg:lcsm} is certainly a common multiple of $f$ and $g$.

To see that the common multiples returned are least, we have to show that for every standard bideterminant $l$ that is returned by the algorithm, there is no $l' \neq l$ that is a common multiple of $f,g \in A^{bd}_{gen}$ satisfying $l' \gendivides l$. Observe in the description of $\getadsh$ that all shapes returned have length exactly equal to $\max(r_1, \ldots, r_N, c_1, \ldots, c_N)$. By Lemma \ref{lem:lt-div-bd}, for any such $l'$ 
to satisfy $l' \gendivides l$, we would need the length of the shape of $l'$ to be strictly smaller than the length of the shape of $l$ (the length of the shape of $R \backslash R'$ is strictly less than that of $R$ unless $R'$ is the empty tableaux). This is impossible because it would mean that the elements in some row of the row tableau of $l'$ is a strict sub-multiset of the corresponding row in $R_f$ or $R_g$, or the elements in some row of the column tableau of $l'$ is a strict sub-multiset of the corresponding row in $C_f$ or $C_g$. This in turn would mean that $l'$ cannot possibly be a common multiple of $f$ and $g$, again by Lemma \ref{lem:lt-div-bd}. This proves the soundness of the algorithm. 

(Completeness.) To establish completeness, we first argue that the fact that the length of shapes returned by Algorithm \ref{alg:lcsm} is equal to $\max(r_1, \ldots, r_N, c_1, \ldots, c_N)$ is a necessary condition for being a least common multiple. We have already argued that it is not possible for a least common multiple, say $l$, to have length strictly smaller than $\max(r_1, \ldots, r_N, c_1, \ldots, c_N)$. It is also not possible for $l$ to have length strictly larger than $\max(r_1, \ldots, r_N, c_1, \ldots, c_N)$ because in that case, we would be able to remove a column from $l$, and after replacing some elements if necessary (to ensure that $f \gendivides l'$ and $g \gendivides l'$), we can form $l'$ which is also a common multiple of $f$ and $g$; since it is true by construction that $l' \gendivides l$, this would mean that $l$ was not a least common multiple to begin with. Thus it is clear that all $l$ have to necessarily have length equal to $\max(r_1, \ldots, r_N, c_1, \ldots, c_N)$. It is also a necessary condition that any shape of a least common multiple must the contain the shapes of $f$ and $g$ as subshapes; this is guaranteed from the description of $\getadsh$. Hence, we can conclude that the shapes returned by $\getadsh$ exactly equals the set of shapes in $LCM_{sm}^{gen}(f, g)$.

By Lemma \ref{lem:lt-div-bd}, it is clear that it is a necessary condition that the $i^{\text{th}}$ row of the row tableau of any least common multiple $l$ must contain both the elements of the $i^{\text{th}}$ row of $R_f$ and the $i^{\text{th}}$ row of $R_g$, and it is also a necessary condition that the $i^{\text{th}}$ row of the column tableau of $l$ must contain both the elements of the $i^{\text{th}}$ row of $C_f$ and the $i^{\text{th}}$ row of $C_g$. Looking at the description of $\getadsbd$, it is clear that these conditions are satisfied. Finally, since we enumerate all possible standard bideterminants which satisfy the above necessary conditions, it is clear that the algorithm returns all least common multiples, thus establishing the completeness of the algorithm.
\end{proof}

\begin{example}[Example run of Algorithm \ref{alg:lcsm}] Let the ambient ring be $\F[\{X_{i,j}\}_{i \in [3], j \in [3]}]$.
Let the input to Algorithm \ref{alg:lcsm} be \[f = \bigbitab{R_f}{C_f} = \bigbitab{\tableau[s]{1 & 1 \\ 2 & 2 \\ 3 }}{\tableau[s]{1 & 1 \\ 2 & 2 \\ 3}}, \qquad \text{and} \qquad g = \bigbitab{R_g}{C_g} = \bigbitab{\tableau[s]{1 \\ 3}}{\tableau[s]{1 \\ 2}},\] where $\shapetab{f} = (3, 2)$ and $\shapetab{g} = (2)$. Also suppose that we wish to compute the set of LCMs of $f$ and $g$ relative to $A_{gen}^{bd}$.

Since $N = \max(3, 2) = 3$ (Line \ref{line:N}), Lines~\ref{line:re}--\ref{line:ce} compute
\[re_1 = \{\!\!\{1, 1\}\!\!\}, re_2 = \{\!\!\{2, 2, 3\}\!\!\}, re_3 = \{\!\!\{3\}\!\!\}, \qquad ce_1 = \{\!\!\{1, 1\}\!\!\}, ce_2 = \{\!\!\{2, 2\}\!\!\}, ce_3 = \{\!\!\{3\}\!\!\},\]
and Lines~\ref{line:rd}--\ref{line:cd} compute 
\[rd_1 = 2, rd_2 = 3, rd_3 = 1, \qquad cd_1 = 2, cd_2 = 2, cd_3 = 1.\]
In Line \ref{line:admissible-shapes}, we look for shapes that satisfy the conditions that the number of columns of the shape is equal to $\max(2, 3, 1, 2, 2, 1) = 3$, the shape contains $\shapetab{f} \cup \shapetab{g} = (3,2)$ as a sub-multiset, and that the $i^{\text{th}}$ row has at least $\max(rd_i, cd_i)$ boxes. We find
\[\getadsh\left((3,2), (2,3,1), (2,2,1)\right) \qquad \text{returns} \qquad \lambdalcm = \{(3,3,2), (3,2,2)\}.\]
Finally, for each $\lambda \in \{(3,3,2), (3,2,2)\}$, in Line \ref{line:admissible-sbd}, the call 
\[
\getadsbd\left(\{(3,3,2), (3,2,2)\}, (\{\!\!\{1, 1\}\!\!\}, \{\!\!\{2, 2, 3\}\!\!\}, \{\!\!\{3\}\!\!\}), (\{\!\!\{1, 1\}\!\!\}, \{\!\!\{2, 2\}\!\!\}, \{\!\!\{3\}\!\!\})\right)
\] 
looks for all standard bitableaux of shape $\lambda$ such that the $i^{\text{th}}$ row of the row tableau contains $re_i$ and the $i^{\text{th}}$ row of the column tableau contains $ce_i$. For $\lambda = (3,3,2)$, we find that 
\[
\bigbitab{\tableau[s]{\tf{1} & \tf{1} & 1 \\ \tf{2} & \tf{2} & \tf{3} \\ \tf{3} & 3}}{\tableau[s]{\tf{1} & \tf{1} & 1 \\ \tf{2} & \tf{2} & 2 \\ \tf{3} & 3}}, \bigbitab{\tableau[s]{\tf{1} & \tf{1} & 1 \\ \tf{2} & \tf{2} & \tf{3} \\ \tf{3} & 3}}{\tableau[s]{\tf{1} & \tf{1} & 1 \\ \tf{2} & \tf{2} & 3 \\ \tf{3} & 3}}, \bigbitab{\tableau[s]{\tf{1} & \tf{1} & 1 \\ \tf{2} & \tf{2} & \tf{3} \\ \tf{3} & 3}}{\tableau[s]{\tf{1} & \tf{1} & 2 \\ \tf{2} & \tf{2} & 3 \\ \tf{3} & 3}},
\]
and
\[
\bigbitab{\tableau[s]{\tf{1} & \tf{1} & 2 \\ \tf{2} & \tf{2} & \tf{3} \\ \tf{3} & 3}}{\tableau[s]{\tf{1} & \tf{1} & 1 \\ \tf{2} & \tf{2} & 2 \\ \tf{3} & 3}}, \bigbitab{\tableau[s]{\tf{1} & \tf{1} & 2 \\ \tf{2} & \tf{2} & \tf{3} \\ \tf{3} & 3}}{\tableau[s]{\tf{1} & \tf{1} & 1 \\ \tf{2} & \tf{2} & 3 \\ \tf{3} & 3}}, \bigbitab{\tableau[s]{\tf{1} & \tf{1} & 2 \\ \tf{2} & \tf{2} & \tf{3} \\ \tf{3} & 3}}{\tableau[s]{\tf{1} & \tf{1} & 2 \\ \tf{2} & \tf{2} & 3 \\ \tf{3} & 3}},
\]
are valid standard bitableaux. For $\lambda = (3,2,2)$, we find that \[\bigbitab{\tableau[s]{\tf{1} & \tf{1} & 1 \\ \tf{2} & \tf{2} & \tf{3} \\ \tf{3}}}{\tableau[s]{\tf{1} & \tf{1} & 1 \\ \tf{2} & \tf{2} & 2 \\ \tf{3}}}, \bigbitab{\tableau[s]{\tf{1} & \tf{1} & 1 \\ \tf{2} & \tf{2} & \tf{3} \\ \tf{3}}}{\tableau[s]{\tf{1} & \tf{1} & 1 \\ \tf{2} & \tf{2} & 3 \\ \tf{3}}}, \bigbitab{\tableau[s]{\tf{1} & \tf{1} & 1 \\ \tf{2} & \tf{2} & \tf{3} \\ \tf{3}}}{\tableau[s]{\tf{1} & \tf{1} & 2 \\ \tf{2} & \tf{2} & 3 \\ \tf{3}}},\]
and
\[\bigbitab{\tableau[s]{\tf{1} & \tf{1} & 2 \\ \tf{2} & \tf{2} & \tf{3} \\ \tf{3}}}{\tableau[s]{\tf{1} & \tf{1} & 1 \\ \tf{2} & \tf{2} & 2 \\ \tf{3}}}, \bigbitab{\tableau[s]{\tf{1} & \tf{1} & 2 \\ \tf{2} & \tf{2} & \tf{3} \\ \tf{3}}}{\tableau[s]{\tf{1} & \tf{1} & 1 \\ \tf{2} & \tf{2} & 3 \\ \tf{3}}}, \bigbitab{\tableau[s]{\tf{1} & \tf{1} & 2 \\ \tf{2} & \tf{2} & \tf{3} \\ \tf{3}}}{\tableau[s]{\tf{1} & \tf{1} & 2 \\ \tf{2} & \tf{2} & 3 \\ \tf{3}}}\]
are valid standard bitableaux (the boxes with thicker border contains elements of $re_*$ or $ce_*$, and the remaining boxes can be freely set whilst ensuring that the resulting bitableaux is standard). All of the above standard bitableaux comprise $\lcmlt(f, g)$.
\end{example}

\subsection{Universal bd-Gröbner bases for determinantal ideals}

For $S \subseteq [m], T \subseteq [n]$, write $M_{S,T}$ for the minor of the $m \times n$ matrix $X$, whose rows are indexed by $S$ and columns are indexed by $T$.
\begin{theorem}[store=thmtminorsbdgb]
\label{thm:t-minors-bdgb}
Let $m > n \geq r$, and consider the determinantal ideal $I_r \subseteq \F[\{X_{i, j}\}_{i \in [m], j \in [n]}] = A^{bd}$ generated by the size-$r$ minors of $X$. Then
\[
G_r := \left\{M_{S, T} : S \subseteq [m], T \subseteq[n], n \geq |S|=|T| \geq r\right\}
\]
is a universal bd-Gröbner basis for $I_r$ (relative to any $A_{lt}^{bd}$).
\end{theorem}

In some sense, this proof is just a restatement of the standard fact (which follows from the bideterminant straightening law) that $I_r$ consists precisely of the polynomials each of whose standard monomials is a standard bideterminant of length $\geq r$. The details in the following proof are just to verify that all the ASL pieces fit together properly. 

\begin{proof} 
By the proof of Theorem~\ref{thm:universal}, it suffices to show that $G_r$ is a universal ASL Gröbner basis relative to arbitrary ASL term orders for $A^{bd}$ and with algebra of leading terms $A^{bd}_{disc}$. We use the characterization of ASL Gröbner bases in terms of divisibility (Proposition~\ref{prop:asl-grobner-divisibility-characterization}).

Since the elements of $G_r$ are themselves standard monomials, they are equal to their own leading terms regardless of choice of term order. Next, by the straightening law (Theorem \ref{thm:straightening} and the algorithm described in Appendix \ref{sec:app:bidet-straightening}), a polynomial $f$ is in $I_r$ if and only if for every standard monomial $m$ occurring in $f$ there is a minor of size $\geq r$ that divides $m$ (divides in $A^{bd}$). The latter set of minors is precisely $G_r$. Suppose $f \in I_r$. Regardless of term order, its leading term is thus divisible (in $A^{bd}$) by some element of $G_r$; say the leading term is $m$ and let $M_{S,T} \in G_r$ be an element of $G_r$ that divides $m$ in $A^{bd}$. All that remains is to show that $\pi_{disc}(M_{S,T}) \divides \pi_{disc}(m)$ in $A^{bd}_{disc}$. But this is almost immediate: for $m / M_{S,T}$ is another standard monomial $m'$, and since the product $m' M_{S,T}$ is standard in $A^{bd}$, in any algebra of leading terms we must have $\pi_{lt}(m') \pi_{lt}(M_{S,T}) = \pi_{lt}(m)$. Thus $G_r$ is an ASL Gröbner basis relative to any term order and any algebra of leading terms.
\end{proof}

\begin{remark}
In the case of Theorem \ref{thm:t-minors-bdgb}, we see that a bd-\gb basis can be much larger than an ordinary \gb basis. In fact, all standard bideterminants with at least one of the minors being one of $G_{r}$ (as defined in Theorem \ref{thm:t-minors-bdgb}) are a $\F$-basis of $I_r$ \cite[Corollary 3.4.2]{bruns2022determinants}. On the other hand, the proof of the theorem is extremely short and elementary.
\end{remark}

\begin{corollary}[store=thmmaximalminorsbdgb,note={of Theorem \ref{thm:t-minors-bdgb}}]
\label{cor:maximal-minors-bdgb}
The set maximal minors of a generic matrix is a universal bd-Gröbner basis for the ideal they generate.
\end{corollary}

\begin{proof}
Immediate by setting $r = n$ in Theorem \ref{thm:t-minors-bdgb}.
\end{proof}

\begin{remark}
It is worth noting how easy proofs of Corollary \ref{cor:maximal-minors-bdgb} and Theorem \ref{thm:t-minors-bdgb} were compared to analogous proofs using ordinary Gröbner bases. Indeed, since we are working in an ASL where each minor is a single variable, there was no choice for what the leading standard monomial was, whereas in the ordinary polynomial ring each such minor has $n!$ terms and one must contend with the various possibilities for what the leading terms are. 
\end{remark}

%% file: two-asl-structures.tex

\subsection{Relation to ordinary Gröbner bases} \label{sec:ordinary}
In this section we show that every bd-Gröbner basis relative to $\Adisc^{bd}$ is also an ordinary Gröbner basis relative to the diagonal term order, but the converse does not hold.

Let $<_{diag}$ denote the so-called ``diagonal ordering'' on monomials in the polynomial ring $\F[\{x_{ij}\}_{i \in [n], j \in [m]}]$, that is, $<_{diag}$ is the lexicographic ordering on monomials induced by the following order on variables: 
\[
x_{11} >_{diag} \dotsb >_{diag} x_{1m} >_{diag} x_{21} >_{diag} \dotsb >_{diag} x_{2m} >_{diag} \dotsb >_{diag} x_{n1} >_{diag} \dotsb >_{diag} x_{nm}.
\]
We use $\lm_{diag}$, $\lt_{diag}$, etc. to denote the leading monomial, leading term, etc. with respect to the diagonal ordering.

The following lemma is the key to showing that bd-Gröbner bases relative to $\Adisc^{bd}$ are also ordinary Gröbner basis relative to the diagonal term order.

\begin{lemma} \label{lem:ordinary}
\begin{enumerate}
\item If $f,g$ are standard bideterminants with $f \leq g$, then $\lm_{diag}(f) \leq_{diag} \lm_{diag}(g)$.

\item For all polynomials $f$, $\lm_{diag}(f) = \lm_{diag}(\lmsm^{bd}(f))$.

\item \label{lem:ordinary:div} Suppose $f,g$ are two standard bideterminants with $g \ltdivides f$ in $\Adisc^{bd}$. Then $\lm_{diag}(g)$ divides $\lm_{diag}(f)$.
\end{enumerate}
\end{lemma}

\begin{proof}
First, note that the leading monomial of a minor is precisely the product of its diagonal elements. (This can be proved by induction: if the rows involved are $r_1 < r_2 < \dotsb < r_k$ and the columns are $c_1 < c_2 < \dotsb < c_k$, then the largest variable involved is $x_{r_1,c_1}$. Now induct on the remaining $(k-1) \times (k-1)$ minor.)

\begin{enumerate}
\item Let $f,g$ be standard bideterminants with $f \leq g$. Suppose $f = \bitab{R}{C}$ and $g = \bitab{R'}{C'}$. If the bitableau first differ at $r = R_{ij} > R'_{i'j'} = r'$, then we have $\lm_{diag}(g)$ has at least one more variable from row $r'$ than $\lm_{diag}(f)$, and thus $\lm_{diag}(f) < \lm_{diag}(g)$. If the first difference is in the column tableaux, the same proof works \emph{mutatis mutandis}.

\item Since $\lmsm^{bd}(f)$ is greater than all other standard bideterminants $b$ appearing in $f$, we have $\lm_{diag}(\lmsm^{bd}(f)) > \lm_{diag}(b)$, by part (1). Thus $\lm_{diag}(\lmsm^{bd}(f)) = \lm_{diag}(f)$.

\item Let $f,g$ be two standard bideterminants with $g$ dividing $f$ in $\Adisc^{bd}$. Then the columns of the bitableaux for $g$ are a subset of the columns of the bitableaux for $f$. Since leading monomials multiply (for ordinary monomials in the polynomial ring), the leading monomial for $g$ in the diagonal ordering is thus a sub-product of the leading monomial for $f$. \qedhere
\end{enumerate}
\end{proof}

\begin{theorem} \label{thm:gb_diag}
If $G$ is a bd-Gr\"{o}bner basis for an ideal $I$ relative to $\Adisc^{bd}$, then it is also a Gr\"{o}bner basis with respect to the diagonal ordering.
\end{theorem}

\begin{proof}
Let $f \in I$. Then $\pi_{disc}(\lmsm^{bd}(f))$ is divisible by $\pi_{disc}(\lmsm^{bd}(g))$ in $\Adisc^{bd}$ for some $g \in G$, so by Lemma~\ref{lem:ordinary}(\ref{lem:ordinary:div}), $\lm_{diag}(f)$ is divisible by $\lm_{diag}(g)$. Thus $\{\lm_{diag}(g) : g \in G\}$ generates $\lm_{diag}(I)$.
\end{proof}

We remark that the divisibility-theoretic characterization of Gr\"{o}bner and bd-Gr\"{o}bner bases was very helpful here; if we instead used Buchberger's criterion, it was not at all clear---at least to us!---that all the S-polynomials with respect the diagonal ordering would reduce to 0.

Finally, we close this section by showing that the converse does not hold; that is, the concept of bd-Gr\"{o}bner basis is not equivalent to Gr\"{o}bner basis in the diagonal order. In fact, Example~\ref{ex:singleton} already does this, as it gives a principal ideal in the polynomial ring but whose smallest bd-Gröbner basis has size 2. Here we give a related example where the ordinary Gröbner basis is not a singleton, just for some variety:

\begin{example}[A Gr\"{o}bner basis in the diagonal ordering that is not a bd-Gr\"{o}bner basis]
Let $n=m=2$, and consider $I = \langle x_{11}, x_{12} \rangle$. Since $I$ is a monomial ideal already, it is easy to see that $\Gamma = \{x_{11}, x_{12}\}$ is the unique universal reduced (ordinary) Gr\"{o}bner basis for $I$. However, $\Gamma$ is not a bd-Gr\"{o}bner basis relative to $\Adisc^{bd}$. For $I$ contains the minor $x_{11} x_{22} - x_{12} x_{21}$, but in the bideterminant ASL structure the latter is the single variable $X_{12|12}$, yet no element of of $\Gamma$ divides the variable $X_{12|12}$ in $\Adisc^{bd}$.

The failure here is due to the fact that $\Gamma$ is not Ann-closed. For the annihilator of $x_{12}$ in $A^{bd}_{disc}$ is $\langle x_{21} \rangle$.
Since we have 
\[
x_{21}x_{12} = \bigbitab{\tableau[s]{1 & 2}}{\tableau[s]{2 & 1}} = \bigbitab{\tableau[s]{1 \\ 2}}{\tableau[s]{1 \\ 2}} - \bigbitab{\tableau[s]{1 & 2}}{\tableau[s]{1 & 2}},
\]
if we add $X_{12|12}$ to our generating set, we get an Ann-closed and S-closed generating set (as in Example~\ref{ex:singleton}), and thus a bd-Gröbner basis $\{X_{1|1}, X_{1|2}, X_{12|12}\}$.
\end{example}

%% file: dimension.tex

\section{Hilbert--Poincar\'e series and Krull dimension in pseudo-ASLs}
\label{sec:dimension}

In this section we show that in the $\N$-graded case we can use pseudo-ASL Gröbner bases---in particular through the monomial pseudo-ASLs of leading terms $A_{lt}$---to compute dimension of quotients of pseudo-ASLs using an algorithm similar to that in the ordinary case of polynomial rings.

\subsection{Background and Preliminaries}
\label{sec:hilbert-background}
We begin by recalling some standard definitions and results regarding Hilbert--Poincar\'e series. Let $A$ be an $\N$-graded ring over a commutative ring $R$, that is, $A = \bigoplus_{s \in \N} A_s$, where $A_i \cdot A_j \subseteq A_{i+j}$, and further suppose that each $A_s$ is a free $R$-module (if $R$ is a field this is automatic). If $a \in A$ lies entirely in $A_s$ for some $s$ then we say that $a$ is \emph{homogeneous of degree $\deg(a)=s$}. An ideal in $I$ is homogeneous relative to this grading if $I = \bigoplus_s (A_s \cap I)$, and if $I$ is a homogeneous ideal then the quotient ring $\bigslant{A}{I}$ inherits the grading from $A$ in the natural way $\left(\bigslant{A}{I}\right)_s := \bigslant{A_s}{(I \cap A_s)}$. The \emph{Hilbert--Poincar\'e function} of $A$ is $h_A(s) := \dim_R A_s$ and its \emph{Hilbert--Poincar\'e series} is
\[
H_A(t) := \sum_{s=0}^{\infty} h_A(s) t^s \in \Z[\![t]\!].
\]

As we will be particularly interested in the situation where $A$ is an $\N$-graded pseudo-ASL---meaning (Def.~\ref{defn:graded-asl}) that $A$ is $\N$-graded and all standard monomials are homogeneous---we focus on that situation; in this preliminary section we first focus on the case of an ordinary polynomial ring. 

Given a grading on the polynomial ring $R[X_1,\dotsc,X_n]$, if all monomials are homogeneous, then in particular the generating variables $X_i$ are homogeneous, so there is a function $d\colon [n] \to \N$ where $d(i) = \deg(X_i)$, and the degree of a monomial in $R[X_1,\dotsc,X_n]$ is then 
\[
\mathrm{deg}_d(x_1^{c_1}\ldots x_n^{c_n}) = \sum_{i=1}^n c_i \cdot d(i).
\] 
Since any function $d \colon [n] \to \N$ defines such a grading, given $d$ we refer to the $d$-degree, $d$-homogeneity, etc. as the corresponding concepts for the grading given by $d$ as above.

The following theorem collects some known facts about Hilbert--Poincar\'e series in relation to (ordinary) Gröbner bases. Recall (Def.~\ref{defn:graded-asl}) that a term order $<$ on a $\Z$-graded polynomial ring $R[X_1,\dotsc,X_n]$ is called a \emph{graded term order} if $\mathrm{deg}(m) < \mathrm{deg}(m') \Rightarrow m < m'$, where here the grading is given by the function $\deg \colon \W^n \to \Z$. 

\begin{theorem} \label{thm:hilbert-series-facts}
\begin{enumerate}[label=(\Roman*),ref=(\Roman*)]
\item (cf. \cite[\S 9.3, Prop.~4]{CLO}) Let $<$ be graded term order on monomials in $R[X_1, \ldots, X_n]$, relative to the grading given by $\mathrm{deg}_d$. Then for any $d$-homogeneous ideal $I \subseteq R[X_1, \ldots, X_n]$, \[H_{R[X_1, \ldots, X_n] / I} = H_{R[X_1, \ldots, X_n] / LI(I)}.\] If $G$ is a Gröbner basis for $I$ relative to $<$, then \[H_{R[X_1, \ldots, X_n] / I} = H_{R[X_1, \ldots, X_n] / \ideal{\{LT(g) \suchthat g \in G\}}}.\]

\item (See App.~\ref{sec:app:hilbert-series-algorithm} below, cf. \cite[\S 10.2, Thm.~6]{CLO}.) There is an algorithm which, given a monomial ideal $I \subseteq R[X_1, \ldots, X_n]$, computes the Hilbert--Poincar\'e series of $\bigslant{R[X_1, \ldots, X_n]}{I}$ as a rational function. 

\item \label{thmhilberseriesfacts-pole} (\cite[Thm.~5.5]{smoke}; see also \cite[Prop~5.3.2]{benson}, \cite[p.~58]{stanley}.) The Krull dimension of $R[X_1,\dotsc,X_n]/I$ is equal to the order of the pole at $t=1$ of $H_{R[X_1,\dotsc,X_n] / I}$. 
\end{enumerate}
\end{theorem}

\subsection{Moving to pseudo-ASLs}
\label{sec:hilbert-asl}
Our strategy to compute the Hilbert--Poincaré series of $\bigslant{A}{I}$ is:
\begin{enumerate}
\item Compute a pseudo-ASL Gröbner basis $G$ for $I$. By Theorem~\ref{thm:hilbert-asl-alt}, $H_{A/I} = H_{A_{lt} / \lism^{lt}(I)}$.

\item Use Prop.~\ref{prop:hilbertInPolyRing} to lift $\langle \pi_{lt}(g) : g \in G \rangle$ to a binomial ideal in $R[H]$ and compute there.

\item In particular, if we use $A_{lt} = A_{disc}$, then the straightening relations for $A_{disc}$ are monomials, so when we lift $\langle \pi_{disc}(g) : g \in G \rangle$ to $R[H]$ we get a \emph{monomial} ideal, for which no further Gröbner basis computation is necessary, as at that point one can directly compute the Hilbert--Poincaré series and Krull dimension (as in, e.g., App.~\ref{sec:app:hilbert-series-algorithm} below, cf. \cite[\S 10.2, Thm.~6]{CLO}).
\end{enumerate}
In the remainder of this section we develop the results that support this approach, namely Theorem~\ref{thm:hilbert-asl-alt} and Proposition~\ref{prop:hilbertInPolyRing}.

Recall the definition of graded pseudo-ASL and graded pseudo-ASL term order (Def.~\ref{defn:graded-asl}). An $\N$-graded pseudo-ASL can be equivalently defined as: there is a function $d \colon H \to \N$, using which we define the $d$-degree of monomials $m \in \W^H$ as above, i.e. ($\deg_d(m) := \sum_{h \in H} m(h) d(h)$), and such that the straightening relations in $A$ are $d$-homogeneous. Observe that the homogeneity of the straightening relations is a necessary and sufficient condition for the $d$-degree of standard monomials to be well-defined.

Henceforth, we will refer to the grading induced by $d$ as the $d$-grading on $A$. Since $A$ is finitely generated by $H$, we have an $R$-algebra isomorphism $\phi: A \rightarrow \bigslant{R[H]}{\mathcal{S}}$, which takes \[H \ni h \mapsto x_h \in R[H],\] naturally extended, where the ideal $\mathcal{S} \subseteq R[H]$ is the defining ideal of $A$ (and includes the straightening relations). As in the case of Gröbner bases, one can compute the Hilbert--Poincaré series in $A$ by passing to a polynomial ring, as in the next proposition; the goal of the remainder of this section is to develop theory to compute the Hilbert--Poincaré series of $A$ ``natively to $A$'', that is, without passing to the polynomial ring.

\begin{proposition} \label{prop:hilbertInPolyRing}
Let $d \colon H \to \N$ be a function such that the ASL $A = R[H]/\mathcal{S}$ is $\N$-graded via $d$. Let $\psi: R[H] \twoheadrightarrow \bigslant{R[H]}{\mathcal{S}}$ be the defining homomorphism of $A$, and for a $d$-homogeneous ideal $I \subseteq A$ let $I' := \psi^{-1}(I)$. Then \[H_{A / I} = H_{R[H] / I'}.\]
\end{proposition}

\begin{proof}
Again, since we have stipulated that $d$ is such that the straightening relations in $A$ are homogeneous, if we give $x_h \in R[H]$ a degree of $d(h)$, $\mathcal{S}$ is a $d$-homogeneous ideal in $R[H]$, and thus $\bigslant{R[H]}{\mathcal{S}}$ inherits a $d$-grading. Since both $h$ and $x_h$ have the same degrees in $A$ and $\bigslant{R[H]}{\mathcal{S}}$ respectively, $\phi$ is in fact an isomorphism of graded $R$-algebras. By the third ring isomorphism theorem, the ideal $I' \subseteq R[H]$ is such that $I' := \mathcal{S} + I$ (where here we view elements of $I$ as linear combinations of standard monomials in $R[H]$), thus if $I$ is $d$-homogeneous in $A$, then $I'$ is $d$-homogeneous in $R[H]$. Thus once again, under the map which takes \[\bigslant{A}{I} \ni h + I \mapsto x_h + I' \in \bigslant{R[H]}{I'},\] $\bigslant{A}{I}$ and $\bigslant{R[H]}{I'}$ are isomorphic as $d$-graded $R$-algebras.
\end{proof}

\begin{remark}
While we restrict our presentation to homogeneous ideals, it is possible to consider affine Hilbert--Poincar\'e series in order to extend our presentation to arbitrary ideals, as in, e.\,g., \cite[Section 5.6]{kreuzer-robbiano-2}.
\end{remark}

\begin{remark}
In general, since every finitely generated algebra $A$ is isomorphic as an algebra to the quotient of a polynomial ring, say $A \cong \bigslant{R[X_1,\dotsc,X_n]}{I}$, one could define the Hilbert--Poincar\'e function of $A$ as the (affine) Hilbert--Poincar\'e series of $\bigslant{R[X_1,\dotsc,X_n]}{I}$. However, in this case, the Hilbert--Poincar\'e series itself is \emph{not} an invariant of $A$, nor even is the Hilbert--Poincaré polynomial: they both can depend on the generators chosen. See \cite[Remark~11.3(b)]{kemper} for an example. However the Krull dimension of $A$ is an invariant of $A$, and thus is independent of the chosen presentation of $A$ \cite[Theorem~5.6.36]{kreuzer-robbiano-2}, so can be computed using the affine Hilbert--Poincaré series of any presentation of $A$ as a quotient of a polynomial ring.
\end{remark}

Now let $A_{lt}$ be an algebra of leading terms for $(A,\preceq)$. Since $A_{lt}$ is finitely generated by $H$, we have an $R$-algebra isomorphism $\phi_{lt}: A_{lt} \rightarrow \bigslant{R[H]}{\mathcal{S}_{lt}}$ for some ideal $\mathcal{S}_{lt} \subseteq R[H]$. $\mathcal{S}_{lt}$ is a binomial ideal for all $A_{lt}$, and specifically, $\mathcal{S}_{disc}$ is a monomial ideal. 

\begin{theorem}[store=thmhilbertaslalt] 
\label{thm:hilbert-asl-alt}
Given a pseudo-ASL $A$ on $H$, $(A,\preceq,A_{lt})$, and $d\colon H \to \N$ such that $\deg_d$ makes $A$ into a graded pseudo-ASL and $\preceq$ a graded pseudo-ASL term order, let $I \subseteq A$ be a $d$-homogeneous ideal. Then defining $\deg_d$ on $A_{lt}$ such that it agrees with $\deg_d$ on $H$ makes $A_{lt}$ into a graded pseudo-ASL such that $\lismlt(I) \subseteq A_{lt}$ is $d$-homogeneous, and we have
\[
H_{A / I} = H_{A_{lt} / \lismlt(I)}.
\]
\end{theorem}

Here, if $\pi_{lt} \colon A \to A_{lt}$ is the $R$-module isomorphism between $A$ and $A_{lt}$ that is the identity map on standard monomials in $A$, then we are extending $\deg_d$ to $A_{lt}$ by defining 
\[
\deg_d(\pi_{lt}(h)) := d(h).
\] 

\begin{proof}
To see that $A_{lt}$ is $\N$-graded by $\deg_d$, it is necessary and sufficient to show that the straightening relations for $A_{lt}$ are $d$-homogeneous. But since the straightening relations in $A_{lt}$ are truncations of the straightening relations in $A$, and the latter are $d$-homogeneous, the same must be true of the straightening relations in $A_{lt}$. 

Thus under the $d$-grading of $A_{lt}$ defined above, $\pi_{lt}: A \rightarrow A_{lt}$ is in fact a graded $R$-module isomorphism, so we have $H_A = H_{A_{lt}}$.

To complete the proof, we will show that for each $s \in \N$ we have 
\[
\dim_R (I \cap A_s) = \dim_R (\lismlt(I) \cap (A_{lt})_s).
\]
This proof is similar to the ``direct'' computational proof for ordinary Gröbner bases (that is, not using Gröbner degenerations). Namely, let $f_1,\dotsc,f_m$ be an $R$-linear basis for $I \cap A_s$, where $m = \dim_R (I \cap A_s)$. Let $\{m_1,\dotsc,m_N\}$ be a list of the standard monomials of $d$-degree exactly $s$, which is, in particular, a basis for $A_s$ (and also, under $\pi_{lt}$, $(A_{lt})_s$), and ordered such that $i < j$ iff $m_i \prec m_j$. Let $M$ be the $m \times \dim A_s$ matrix over $R$ whose $(i,j)$ entry is the coefficient of $m_j$ in $f_i$. Doing row operations on this matrix corresponds to taking $R$-linear combinations of the $f_i$, and if the row operations are invertible, then the polynomials corresponding to the rows still span $I \cap A_s$. Putting $M$ into row echelon form, we find a new basis for $I \cap A_s$ in which all pairs of basis elements have distinct leading monomials. Those leading monomials are then a basis for $\lismlt(I) \cap (A_{lt})_s$, which thus also has dimension $m = \dim (I \cap A_s)$. 
\end{proof}

\subsection{An example: rank-1 matrices}
\label{sec:krull-rank-1}

Fix any $n$ and $m$. In this section, we shall use the approach of the previous section to (re)derive the well known fact that the dimension of the affine variety of rank $1$ $n \times m$ matrices taken as elements of $\C^{n \times m}$ is equal to $n + m - 1$. Obviously this is not a new result, and we don't claim that the combinatorial calculations are easier than known (geometric or algebraic) proofs either. However, our analysis uses Theorem \ref{thm:t-minors-bdgb} which gives a $bd$-\gb basis of the ideal of $2$-minors of $X$, i.e. $I_2$, relative to $A_{disc}$, to calculate the Krull dimension of $\bigslant{\C[X]}{I_2}$. 
The main point in showing this calculation is that getting a $bd$-\gb basis of $I_2$ relative to $A^{bd}_{disc}$ was relatively easier, and that this notion is useful to calculate dimension. This suggests that in cases where obtaining a $bd$-\gb basis is easier, one could subsequently obtain dimension in a relatively easier fashion as well. Another point of this calculation is that we show that $A^{bd}_{disc}$ (as opposed to just $A^{bd}_{gen}$) can be used in an important way.

Consider the ASL $A^{bd}$ as defined in Theorem \ref{thm:bd-hodge-algebra}, where $X$ is now specifically a generic $n \times m$ matrix. The dimension of the variety of rank $1$ $n \times m$ matrices is, by definition, equal to the Krull dimension of $\bigslant{A^{bd}}{I_2}$, where $I_2$ is the ideal generated by the size $2$ minors of $X$. From Theorem \ref{thm:t-minors-bdgb}, we have that the set of minors of size $\ge 2$ forms a $bd$-\gb basis of $I_2$ relative to $A^{bd}_{disc}$. We shall calculate the Krull dimension of $\bigslant{A^{bd}}{I_2}$ by 
\begin{enumerate}
\item giving $A^{bd}$ a grading by $d: H \rightarrow \N$, which takes a minor $h$ and sets $d(h)$ equal to the size of the minor, and
\item calculating $H_{A^{bd} / I_2}$. Specifically, we will use Theorem \ref{thm:hilbert-asl-alt} and the strategy laid out at the beginning of Section~\ref{sec:hilbert-asl} to calculate $H_{A^{bd} / I_2}$ by calculating $H_{A^{bd}_{disc} / LI_{bd}^{disc}(I_2)}$, whose calculation reduces (without any further Gröbner basis calculation) to calculating the Krull dimension of a polynomial ring modulo a monomial ideal.

\item Calculating the latter Krull dimension using the standard algorithm.
\end{enumerate}

Define 
\[
\mathcal{R} := \C\left[\left\{X_{\bidet{S}{T}}\right\}_{S \subseteq [n], T \subseteq [m], |S| = |T|}\right].
\]
Also, let
\[
\phi_{disc}^{bd}: A_{disc}^{bd} \rightarrow \bigslant{\mathcal{R}}{\mathcal{S}_{disc}^{bd}}
\]
be the $\C$-algebra isomorphism that maps $h \mapsto x_h$, where $h \in H$ is a minor. Giving each $x_h \in \mathcal{R}$ a degree of $d(h)$, $\phi_{disc}^{bd}$ is in fact a graded $\C$-algebra isomorphism. The monomial ideal $\mathcal{S}^{bd}_{disc} \subseteq \mathcal{R}$ has a structure that is easy to describe:
\[\mathcal{S}^{bd}_{disc} = \ideal{\left\{x_{\bidet{S_1}{T_1}}x_{\bidet{S_2}{T_2}} \suchthat \bidet{S_1}{T_1}\bidet{S_2}{T_2} \text{ is a non-standard bideterminant in }A^{bd}\right\}}.\]

By Theorem \ref{thm:hilbert-asl-alt} and Proposition~\ref{prop:hilbertInPolyRing},
\[
H_{A^{bd} / I_2} = H_{A^{bd}_{disc} / LI_{bd}^{disc}(I_2)} = H_{\mathcal{R} / \left(\mathcal{S}_{disc} + \phi_{disc}\left(LI_{bd}^{disc}(I_2)\right)\right)}.
\] 
Conveniently, we have 
\[
\mathcal{S}_{disc} + \phi_{disc}\left(LI_{bd}^{disc}(I_2)\right) = \mathcal{I}_{n, m} + \ideal{\left\{x_{\bidet{S}{T}} \suchthat \text{ $|S| \ge 2$}\right\}},
\] 
where for $b \in [n]$ and $c \in [m]$, 
\begin{equation} 
\label{eqn:I-rank-1-adisc} 
\mathcal{I}_{b, c} := \ideal{\left\{x_{i_1,j_1}x_{i_2,j_2} \suchthat 1 \le i_1 < i_2 \le b, 1 \le j_2 < j_1 \le c\right\}},
\end{equation} 
where here and henceforth we will just write $x_{i, j}$ for $x_{\bidet{i}{j}}$. In other words, $\mathcal{I}_{b, c} \subseteq \mathcal{S}^{bd}_{disc} \subseteq \mathcal{R}$ clearly denotes the ideal generated by all two variable products, where the subscripts of the variables correspond to $1$-minors of the top-left $b \times c$ submatrix of $X$, such that the product of these $1$-minors is a non-standard bideterminant. For example, $\mathcal{I}_{2,3} = \ideal{x_{12}x_{21}, x_{13}x_{21}, x_{13}x_{22}}$. 

The key fact we are using here and below is that, in this particular case, because our bd-Gröbner basis consists of all minors of size $\geq 2$, after modding out by those, the only straightening rules that are left are those for variables corresponding to (products of pairs of) $1 \times 1$ minors. Using $\C[X] := \C\left[\left\{x_{\bidet{i}{j}}\right\}_{i \in [n], j \in [m]}\right]$ for brevity, we see that this is encapsulated in the isomorphism:
\[
\bigslant{\mathcal{R}}{\left(\mathcal{S}_{disc} + \phi_{disc}\left(LI_{bd}^{disc}(I_2)\right)\right)} \cong_{\text{graded $\C$-alg}} \bigslant{\C[X]}{\mathcal{I}_{n, m}},
\]
from which it follows that we just need to calculate the Hilbert--Poincar\'e series of $\bigslant{\C[X]}{\mathcal{I}_{n, m}}$. To have a grading consistent with the others, note that $\C[X]$ is given the standard grading (every variable $x_{i,j}$ has degree 1, as the variables generating $\C[X]$ correspond to $1 \times 1$ minors).	

To apply the inductive algorithm, we will show the following lemma:

\begin{lemma} \label{lem:dimension-bd}
For any fixed $b, c$, suppose that $A \subseteq \{x_{i, j} \suchthat i \in [n], j \in [m]\}$ is such that for all $x \in A$ and for all $y \in \mathcal{I}_{b, c}$, $x \nmid y$. Then we have
\begin{equation} \label{eqn:general-hp-series-big}
g(A, \mathcal{I}_{b, c}) := H_{\C[X] / (\ideal{A} + \mathcal{I}_{b, c})} = \frac{\sum_{i=0}^{\min(b, c) - 1} \binom{b-1}{i}\binom{c-1}{i} t^i}{(1 - t)^{nm - (b-1)(c-1) - |A|}}.
\end{equation}
\end{lemma}

We note that the preceding equation is already in reduced form: since the coefficients in the numerator are all strictly positive, $1$ is not a root of the numerator, so $(1-t)$ does not divide the numerator.

First let us see that this gives the dimension:

\begin{corollary}
For all $1 \leq n \leq m$ we have
\begin{equation} \label{eqn:hilbert-series-m-n}
g(\emptyset, \mathcal{I}_{n, m}) = \frac{\sum_{i=0}^{n - 1} \binom{n-1}{i}\binom{m-1}{i}t^i}{(1 - t)^{n + m - 1}}.
\end{equation}
In particular, the (affine) variety of rank $1$ $n \times m$ matrices has dimension $n+m-1$.
\end{corollary}

\begin{proof}
The result on $g(\emptyset, \mathcal{I}_{n,m})$ follows by plugging in $A = \emptyset$ to Lemma~\ref{lem:dimension-bd}. Since the expression above is in reduced form, the power of $(1-t)$ in the denominator gives the Krull dimension (Thm.~\ref{thm:hilbert-series-facts}(III) or \cite[Definition 5.4.1 + Theorem 5.6.36]{kreuzer-robbiano-2}).
\end{proof}

We now come to the proof of the lemma.

\begin{proof}[Proof of Lemma~\ref{lem:dimension-bd}]
We will prove Equation \eqref{eqn:general-hp-series-big} by induction on $b + c$. For the base case, we will prove Equation \eqref{eqn:general-hp-series-big} when $\min(b, c) = 1$, or equivalently, $\mathcal{I}_{b, c} = 0$. Without loss of generality, suppose $b = 1$. For any $A \subseteq \{x_{1, j}\}_{j \in [c]}$, we get
\[
g(A, \ideal{0}) = H_{\C[X] / \ideal{A}} = \frac{1}{(1-t)^{nm - |A|}},
\]
where the final equality is by Line \ref{alg:hps:line:degree1-trivial} and Line \ref{alg:hps:line:degree1} of Algorithm \ref{alg:hilbert-clo}. This proves the base case.

For the rest of this section, we shall use the $\sqcup$ symbol to denote a union of sets that we claim are disjoint.

We now have $\min(b, c) = 2$, and by the induction hypothesis, we are guaranteed that Equation \eqref{eqn:general-hp-series-big} holds for all smaller values of $b + c$. To derive $g(A, \mathcal{I}_{b, c})$, we shall run Algorithm \ref{alg:hilbert-clo} on the monomial ideal $\ideal{A} + \mathcal{I}_{b, c}$. In Line \ref{alg:hps:line:select-branching-variable}, we will select $x_{1, c}$ (which occurs in the monomial $x_{1,c} x_{2,c-1} \in \mathcal{I}_{b,c}$; here we are using the fact that $\min(b,c) \geq 2$). The monomial ideal in the first recursive call in Line \ref{alg:hps:line-recursion} is 
\[
I^{(1)}_1 := \left(\ideal{A} + \mathcal{I}_{b, c} + \ideal{x_{1, c}}\right).
\] In Claim \ref{claim:colon-ideal-recursive-call}, we will show that 
\[
I^{(1)}_1 = \left(\ideal{A \sqcup \{x_{1, c}\}} + \mathcal{I}_{b, c}\right).
\] The monomial ideal in the second recursive call in Line \ref{alg:hps:line-recursion} is 
\[
I^{(1)}_2 := \left(\ideal{A} + \mathcal{I}_{b, c}\right) \suchthat \ideal{x_{1, c}}.
\] In Claim \ref{claim:colon-ideal-recursive-call}, we will show that 
\[
I^{(1)}_2 = \ideal{A \sqcup \{x_{r, s}\}_{\substack{1 < r \le b \\ 1 \le s \le c-1}}},
\] and that $I^{(1)}_2$ can be handled by the induction hypothesis. In the recursive call with $I_{1}^{(1)}$, we will now choose $x_{2, c}$ in Line \ref{alg:hps:line:select-branching-variable}. After doing so, the monomial ideals in calls in Line \ref{alg:hps:line-recursion} are 
\[
I^{(2)}_1 := I_1^{(1)} + \ideal{x_{2, c}} \quad \text{and} \qquad I^{(2)}_2 := I_1^{(1)} \suchthat \ideal{x_{2, c}},
\] satisfying 
\[
H_{\C[X] / I_1^{(1)}} = H_{\C[X] / I_1^{(2)}} + t\cdot H_{\C[X] / I_2^{(2)}}.
\] Again, Claim \ref{claim:colon-ideal-recursive-call} will show that 
\[
I^{(2)}_1 = \left(\ideal{A \sqcup \{x_{1, c}, x_{2, c}\}} + \mathcal{I}_{b, c}\right) \quad \text{and} \quad I^{(2)}_2 = \ideal{A \sqcup \{x_{1, c}\} \sqcup \{x_{r, s}\}_{\substack{2 < r \le b \\ 1 \le s \le c-1}}} + \mathcal{I}_{2, c-1},
\] and that $I^{(2)}_2$ can be handled by the induction hypothesis.

This pattern will continue. Every time we choose $x_{j, c}$, for $j \in [b-1]$, in Line \ref{alg:hps:line:select-branching-variable}, Claim \ref{claim:colon-ideal-recursive-call} will show that $I_2^{(j)}$ can be handled by the induction hypothesis. Finally, for $j = b-1$, Claim \ref{claim:colon-ideal-recursive-call} will also show that $I_1^{(b-1)}$ can be handled by the induction hypothesis as well, thus closing the recursion.

\begin{claim}[store=claimcolonidealrecursivecall]
\label{claim:colon-ideal-recursive-call}
Let $I_{1}^{(0)} = \ideal{A} + \mathcal{I}_{b, c}$. For $j \in [b-1]$, define 
\begin{equation} \label{eqn:recursive-definitions-both-i1-i2}
I_{1}^{(j)} = I_{1}^{(j-1)} + \ideal{x_{j, c}} \qquad \text{and} \qquad I_{2}^{(j)} = I_{1}^{(j-1)} \suchthat \ideal{x_{j, c}}. \end{equation}
By definition 
\begin{equation}
\label{eqn:recursive-step-onI1}
H_{\C[X] / I_1^{(j)}} = H_{\C[X] / I_1^{(j+1)}} + t\cdot H_{\C[X] /  I_2^{(j+1)}}.
\end{equation}
For all $j \in [b-1]$, we have
\begin{equation} \label{eqn:i-1-of-recursive}
I_{1}^{(j)} = \ideal{A \sqcup \{x_{i, c}\}_{i \in [j]}} + \mathcal{I}_{b, c},
\end{equation}
and 
\begin{equation} \label{eqn:i-2-of-recursive}
I_2^{(j)} = \ideal{A \sqcup \{x_{i, c}\}_{i \in [j-1]} \sqcup \{x_{r, s}\}_{\substack{j < r \le b \\ 1 \le s \le c-1}}} + \mathcal{I}_{j, c-1}.
\end{equation}
Consequently, 
\begin{equation} \label{eqn:colon-induction-hypothesis} 
H_{\C\left[X\right] / I_{2}^{(j)}} = g(A \sqcup \{x_{i, c}\}_{i \in [j-1]} \sqcup \{x_{r, s}\}_{\substack{j < r \le b \\ 1 \le s \le c-1}}, \mathcal{I}_{j, c-1}),
\end{equation} 
and
\begin{equation} \label{eqn:terminal-i1}
H_{\C\left[X\right] / I_{1}^{(b-1)}} = g(A \sqcup \{x_{i, c}\}_{i \in [b-1]}, \mathcal{I}_{b, c-1}).
\end{equation}
\end{claim}

We defer the proof of Claim \ref{claim:colon-ideal-recursive-call} to Appendix \ref{sec:app:dimension}.

We deduce
\begin{align}
&g(A, \mathcal{I}_{b, c}) = H_{\C[X], (\ideal{A} + \mathcal{I}_{b, c})} \nonumber \\
&\,= H_{\C[X], I_{1}^{(1)}} + t\cdot H_{\C[X], I_{2}^{(1)}} \eqcommentnear{Line \ref{alg:hps:line-recursion} of Algorithm \ref{alg:hilbert-clo}} \nonumber \\
&\,= H_{\C[X], I_{1}^{(1)}} + t\cdot g(A \sqcup \{x_{r, s}\}_{\substack{1 < r \le b \\ 1 \le s \le c-1}}, \mathcal{I}_{1, c-1}) \eqcommentnear{by Equation \eqref{eqn:colon-induction-hypothesis}} \nonumber \\
&\,= H_{\C[X], I_1^{(2)}} + t\cdot H_{\C[X], I_2^{(2)}} + t\cdot g(A \sqcup \{x_{r, s}\}_{\substack{1 < r \le b \\ 1 \le s \le c-1}}, \mathcal{I}_{1, c-1}) \eqcommentnear{by Equation \eqref{eqn:recursive-step-onI1}} \nonumber \\
&\,= H_{\C[X], I_1^{(2)}} + t\cdot g(A \sqcup \{x_{1, c}\} \sqcup \{x_{r, s}\}_{\substack{2 < r \le b \\ 1 \le s \le c-1}}, \mathcal{I}_{2, c-1}) \nonumber \\
&\qquad \quad + t\cdot g(A \sqcup \{x_{r, s}\}_{\substack{1 < r \le b \\ 1 \le s \le c-1}}, \mathcal{I}_{1, c-1}) \eqcommentnear{by Equation \eqref{eqn:colon-induction-hypothesis}} \nonumber \\
&\,= H_{\C[X], I_1^{(2)}} \nonumber \\
&\qquad \quad + t\cdot \left(\sum_{j=1}^{2} g(A \sqcup \{x_{i, c}\}_{i \in [j-1]} \sqcup \{x_{r, s}\}_{\substack{j < r \le b \\ 1 \le s \le c-1}}, \mathcal{I}_{j, c-1})\right) \nonumber \\
&\,= H_{\C[X], I_1^{(b-1)}} \nonumber \\
&\qquad \quad + t\cdot \left(\sum_{j=1}^{b-1} g(A \sqcup \{x_{i, c}\}_{i \in [j-1]} \sqcup \{x_{r, s}\}_{\substack{j < r \le b \\ 1 \le s \le c-1}}, \mathcal{I}_{j, c-1})\right) \eqcommentnear{telescoping with \eqref{eqn:recursive-step-onI1} and \eqref{eqn:colon-induction-hypothesis}} \nonumber \\
&\,= g(A \sqcup \{x_{i, c}\}_{i \in [b-1]}, \mathcal{I}_{b, c-1}) \nonumber \\
&\qquad \quad + t\cdot \left(\sum_{j=1}^{b-1} g(A \sqcup \{x_{i, c}\}_{i \in [j-1]} \sqcup \{x_{r, s}\}_{\substack{j < r \le b \\ 1 \le s \le c-1}}, \mathcal{I}_{j, c-1})\right) \eqcommentnear{by Equation \eqref{eqn:terminal-i1}} \nonumber \\
&\,= \frac{\sum_{i=0}^{\min(b, c-1) - 1} \binom{b-1}{i}\binom{c-2}{i} t^i}{(1 - t)^{nm - (b-1)(c-2) - |A| - (b-1)}} \nonumber \\
&\qquad \quad + t\cdot \left(\sum_{j=1}^{b-1} \frac{\sum_{i=0}^{\min(j, c-1) - 1} \binom{j-1}{i}\binom{c-2}{i} t^i}{(1 - t)^{nm - (j-1)(c-2) - |A|-(j-1)-(c-1)(b-j)}}\right) \eqcommentnear{induction hypothesis} \nonumber \\
&\,= \frac{\sum_{i=0}^{\min(b, c-1) - 1} \binom{b-1}{i}\binom{c-2}{i} t^i + \sum_{j=1}^{b-1} \sum_{i=0}^{\min(j, c-1) - 1} \binom{j-1}{i}\binom{c-2}{i} t^{i+1}}{(1 - t)^{nm - (b-1)(c-1) - |A|}}. \label{eqn:induction-step-final}
\end{align}

The denominator of Equation \eqref{eqn:induction-step-final} matches the denominator of the RHS of Equation \eqref{eqn:general-hp-series-big}, thus we only need to verify that the numerators match. 
\begin{claim}[store=claimnumeratorsmatch] 
\label{claim:numerators-match} 
The numerators of Equation \eqref{eqn:induction-step-final} and the RHS of Equation \eqref{eqn:general-hp-series-big} match.
\end{claim}

Claim \ref{claim:numerators-match}, whose proof is deferred to Appendix \ref{sec:app:dimension}, completes the proof.
\end{proof}

%% file: conclusion.tex

\section{Future directions and open questions}
While we have laid much of the groundwork for this theory, much also remains to be done. We highlight here a few avenues we find particularly interesting and think should be approachable.

\subsection{Algorithms and complexity}
As one of the motivations of our theory is potential computational advantages---both for theoretical, ``by-hand'' computations, and algorithms implemented in software---it would be natural, interesting, and useful to implement our ideas in a standard computer algebra system. This would be especially interesting if implemented in a way that allowed one to incorporate more efficient straightening algorithms for the underlying pseudo-ASL in a modular fashion.

\begin{open}
What is the complexity of computing pseudo-ASL Gröbner bases? Can they always be computed uses an exponential amount of space?
\end{open}

Because ordinary \gb bases are a special case of pseudo-ASL \gb bases, computing pseudo-ASL Gröbner bases in general is at least as hard as computing ordinary Gröbner bases. In particular, it is $\mathsf{EXPSPACE}$-hard \cite{MayrMeyer}. A positive answer to the above question would show that this complexity bound gives a tight characterization. Nonetheless, it is also possible that for particular pseudo-ASLs, computing pseudo-ASL Gröbner bases might be easier, or strictly harder; are there nontrivial examples of this phenomenon?

In the case of the ASL structure given by the bideterminant basis, our work highlights the following question:

\begin{open}
What is the complexity of straightening bideterminants? We conjecture that it is $\mathsf{\# P}$-hard.
\end{open}

Here the input is a single, not necessarily standard, bitableau $\bitab{R}{C}$, and one must compute the set of standard bideterminants occurring in the expansion of $\bidet{R}{C}$ in the basis of standard bideterminants, as well as their coefficients. One can ask more fine-grained, related questions. For example, consider the problem: given a (not necessarily standard) bitableaux $\bitab{R}{C}$ and a standard bitableau $\bitab{R'}{C'}$, compute the coefficient of $\bidet{R'}{C'}$ in the expansion of $\bidet{R}{C}$ in the basis of standard bideterminants. Is this problem in $\mathsf{GapP}$? Is computing the absolute value of the coefficient in $\mathsf{\# P}$ and the sign in $\mathsf{FP}$? What is the complexity of deciding whether $\bidet{R'}{C'}$ occurs in $\bidet{R}{C}$ with nonzero coefficient? Is it in $\mathsf{PSPACE}, \mathsf{PH}, \mathsf{NP}, \mathsf{P}$, or $\mathsf{C_{=} P}$? Is it $\mathsf{NP}$-hard?

Our conjecture of $\mathsf{\# P}$-hardness is based partly on intuition from computing many examples and seeing the complicated structure, and partly on $\mathsf{\# P}$-hardness results for other related problems on Young tableaux (e.\,g., \cite{narayanan, BIKronecker, FIplethysm, ARYSchubert}). Although there have been practical improvements in the complexity of straightening bideterminants \cite{desarmenien1980algorithm, ShaoLi}, to our knowledge their formal complexity has not been analyzed. 

A question of a similar flavor is the complexity of computing the set of LCMs of two standard bitableaux relative to $A_{gen}$; we have currently left our algorithms \getadsh and \getadsbd above as naive exhaustive searches over combinatorial spaces that are easily seen to have exponential size. Is this necessary, that is, can a pair of standard bideterminants of an $n \times n$ matrix and degree $n^{O(1)}$ in general have $\geq \exp^{n^{\Omega(1)}}$ distinct LCMs? If so, can those LCMs at least be enumerated by an algorithm with $n^{O(1)}$ delay? 

\subsection{Theory and geometry}
In the theory of ordinary Gröbner bases, the leading ideal is a flat deformation of the original ideal, which implies that the two share many invariants, and reduces the computation of such invariants to the case of monomial ideals. We believe, especially in light of our results on Hilbert series and dimension (Section~\ref{sec:dimension}), that a similar geometric picture should be true for pseudo-ASL Gröbner bases. In particular, we conjecture the following:

\begin{conjecture}[Flatness for pseudo-ASL Gröbner bases] \label{conj:flatdeform}
\begin{enumerate}
\item Given $(A,\preceq)$, $A_{gen}$ is a flat deformation of $A$.

\item Given $(A, \preceq, A_{lt})$, $A_{lt}$ is a flat deformation of $A_{gen}$ (through algebras of leading terms), and $A_{disc}$ is a flat deformation of $A_{lt}$ (through algebras of leading terms).

\item Given $(A,\preceq)$, if two algebras of leading terms are each flat deformations of the other, then the natural bijection between their standard monomials is in fact an isomorphism of algebras.

\item Given $(A,\preceq)$ and two algebras of leading terms $A_{lt}, A'_{lt}$, let $Z$ (resp., $Z'$) be the set of those pairs of monomials in $\W^H \times \W^H$ whose products become $0$ in $A_{lt}$ (resp., $A'_{lt}$). If $Z \subseteq Z'$, then $A'_{lt}$ is a flat deformation of $A_{lt}$.

\item Given $(A,\preceq, A_{lt})$ and an ideal $I \subseteq A$, $\bigslant{A_{lt}}{\lismlt(I)}$ is a flat deformation of $\bigslant{A}{I}$.
\end{enumerate}
\end{conjecture}

As the proof of flat deformations for Gröbner bases that we are aware of uses the fact that every term order is a lexicographic product of weighted-degree orders \cite{robbianoTermOrder}, a satisfactory resolution of the following question may be useful for proving the preceding conjecture:

\begin{open} \label{q:orders}
Characterize pseudo-ASL term orders. In particular, is there a characterization in terms of weights similar to Robbiano's characterization \cite{robbianoTermOrder} in the case of term orders on polynomial rings?
\end{open}

A short proof shows that every pseudo-ASL term order on $A_{gen}$ is isomorphic to an pseudo-ASL term order on $A$ (via Proposition~\ref{prop:orders}), and therefore it suffices to characterize pseudo-ASL term orders on $A_{gen}$, which may be easier since the standard monomials form a multiplicative basis (up to scalars) in $A_{gen}$.

Next, we note that Section~\ref{sec:ordinary} was an instance in which there were two pseudo-ASL structures, each with pseudo-ASL term orders, on the same (or isomorphic) rings, that were in some sense compatible with one another (as in Lemma~\ref{lem:ordinary}). It would be interesting to better understand the relationship between pseudo-ASL Gröbner bases for a single ideal relative to multiple different pseudo-ASL structures. We note that on a polynomial ring $A$ with $n$ variables, any factorization $n=ab$ gives a bideterminant pseudo-ASL structure viewing it as the coordinate ring of $a \times b$ matrices. In fact, even just permuting the entries of a matrix of fixed size can give a different bideterminant pseudo-ASL on the same polynomial ring.

Lastly, in terms of further theoretical developments, we ask:

\begin{open}
Is there a pseudo-ASL analogue of Diana Taylor's complex for monomial ideals \cite{taylor} (see \cite[Ex.~17.11]{eisenbud} for a more widely-accessible description), but now for standard monomial ideals in a monomial pseudo-ASL?
\end{open}

\subsection{Bipermanents}

Despite the superficial similarity between the permanent and determinant polynomials, our understanding of determinantal varieties and ideals is vastly more developed than our understanding of their permanent-based counterparts (see \cite{gesmundoperm,BORALEVI2025110079} for some recent work on permanental ideals). We saw how easy it was to obtain bd-\gb bases of determinantal ideals in Theorem \ref{thm:t-minors-bdgb}. It is possible that a theory of ``bp-\gb'' bases could aid in understanding permanental ideals. A starting point towards such a theory is that there are also straightening laws for bipermanents \cite{clausen,allen}, though we do not know if they actually give an ASL structure, or even a pseudo-ASL structure. 

\begin{open}
Does the coordinate ring of $n \times m$ matrices admit a pseudo-ASL structure in which the standard monomials are standard bipermanents, and which admits a pseudo-ASL term order?
\end{open}

A positive resolution of the preceding question would thereby give a theory of ``bp-\gb bases.'' One might then hope to resolve the question of the codimension of the singular locus of the permanent hypersurface (highlighted in \cite[Open Question 6.3.3.7]{landsberggct}; we believe it originates implicitly in \cite{vzg87}) along the following lines: find a bp-\gb basis of the ideal of $n-1$ subpermaments (the ideal corresponding to the singular locus of the permanent hypersurface), and apply the results of Section~\ref{sec:dimension} (specifically Theorem \ref{thm:hilbert-asl-alt} and Proposition~\ref{prop:hilbertInPolyRing}) to compute its dimension.

\subsection{Extensions to other settings}
\label{sec:conclusion-weyl}

It would be interesting to extend our theory from pseudo-ASL Gröbner basis for ideals to bases of subalgebras, so-called SAGBI bases, but now in pseudo-ASLs.

Finally, as with ordinary Gröbner bases, we expect that the theory here should be extendible to pseudo-ASLs in certain mildly noncommutative settings. In fact, this work first arose out of the authors' attempt to calculate the algebraic de Rham cohomology of certain determinantal varieties using Weyl algebras; the authors are thus working to extend the theory developed in this paper to (pseudo)-ASL structures on Weyl algebras. Denoting the Weyl algebra on $n$ variables as over a ring $R$ as $W_n(R)$, since 
\[
W_n(R) \cong_{R\text{-mod}} R[X_1, \ldots, X_n] \otimes_{R} R[\partial X_1, \ldots, \partial X_n],
\]
we immediately get an $R$-linear ``bideterminant'' basis, consisting of standard bideterminats in the $x_{ij}$ times standard bideterminants in the $\partial_{ij}$. A theory of bd-\gb bases in Weyl algebras could help in studying annihilating D-ideals, which are crucial towards computing algebraic de Rham cohomology.

%% file: appendix.tex

\section{Deferred proofs from Sections \ref{sec:theory} and \ref{sec:sygyzies}}
\label{sec:app:theory}

\getkeytheorem{lemmonomial}

\begin{proof}[Proof of Lemma~\ref{lem:monomial}]
The equivalence of the two conditions is a direct consequence of \cite[Cor.~1.6(b)]{ES}, by Observation~\ref{lem:standard-all-sigma}.

To see the ``furthermore:'' suppose $G = \{m_1, \dotsc, m_k\}$ is a generating set of $I$ consisting only of standard monomials, and $m$ is an arbitrary standard monomial. If $m_i | m$ then $m \in I$. Conversely, if $m \in I$, we can write $m = \sum_{i=1}^k f_i m_i$ for some $f_i \in A'$. Write each $f_i$ as a sum of standard monomials, say $f_i = \sum_{j=1}^{k_i} r_{ij} m_{ij}$ with $r_{ij} \in R$. Then we have \[m = \sum_{i=1}^k \sum_{j=1}^{k_i} r_{ij} m_{ij} m_i.\] Because the product of standard monomials is zero or a standard term, the same is true of each $m_{ij} m_i$. Since the standard monomials form an $R$-basis of $A'$, some nonzero multiple of $m$ must occur in the set $\{m_{ij} m_i \suchthat i \in [k], j \in [k_i]\}$. Thus $m$ is divisible by some $m_i$.
\end{proof}

\getkeytheorem{thmmacaulay}

\begin{proof}[Proof of Theorem~\ref{thm:macaulay}]
In this proof we will use the additional notation 
\[
\lism(I) := \{ \ltsm(f) \suchthat f \in I \}  \subseteq A.
\] 
(Note that $\lism(I) = \pi_{lt}^{-1}(\lismlt(I))$; it is not an ideal, but it is an $R$-subspace.)

First we show that $\bigslant{A}{I} = Span_R(B)$. Suppose $f$ is a nonzero element of $\bigslant{A}{I}$. Then there is a nonzero element $\hat{f} \in A$ such that $\hat{f} \equiv f \pmod{I}$. If $\hat{f}$ contains any standard terms that are in $\lism(I)$, then we may add elements of $I$ to $\hat{f}$ to get another $\hat{f}'$ such that $\hat{f}' \equiv f \pmod I$ and no standard term occurring in $\hat{f}'$ is in $\lism(I)$. But the elements of $B$ are precisely the standard monomials not occuring in $\lism(I)$, and thus $\hat{f}'$ can be written as an $R$-linear combination of elements of $B$, and the same is true of $f$.

Next we show linear independence. Suppose that there is some nontrivial $R$-linear dependence among the elements of $B$, say $\sum r_i m_i \equiv 0 \pmod{I}$ where $r_i$ are nonzero elements of $R$, and $m_i \in B$ for all $i$. Then $\sum r_i m_i \in I$. Without loss of generality suppose we have indexed these monomials so that $m_1 \succ m_2 \succ \dotsb \succ m_k$. Then $m_1 = \lmsm\left(\sum r_i m_i\right)$, and therefore $m_1 \in \lism(I)$. But by assumption we have $m_1 \in B$, contradicting the fact that $m_1 \in \lism(I)$. Thus the elements of $B$ are $R$-linearly independent modulo $I$.
\end{proof}

\getkeytheorem{claimlmsmlessthankappa}

\begin{proof}[Proof of Observation~\ref{claim:lmsm-less-than-kappa}]
Clearly $\lmsm(f_1 - \ltsm(f_1)) \prec \lmsm(f_1)$. Since $f_1$ and $f_2$ are compatible, multiplying both sides here by $\lmsm(f_2)$ results in nonzero elements on both sides, and thus, by \ref{def:termorder-modules:mult}
\[
\lmsm(\lmsm(f_1 - \ltsm(f_1))\lmsm(f_2)) \prec \lmsm(\lmsm(f_1)\lmsm(f_2)).
\]
Next, since the relevant products are nonzero by compatibility, by the module analogue of Equation \eqref{obs:mult-eqn2} of Proposition \ref{prop:mult}, this implies that 
\[
\lmsm((f_1 - \ltsm(f_1))f_2) \prec \lmsm(f_1f_2),
\]
as required.
\end{proof}

\getkeytheorem{thmreduced}

\begin{proof}[Proof of Theorem~\ref{thm:reduced}]
Let $G$ be an ASL Gröbner basis for $I$ with respect to $(A,\preceq, A_{lt})$. 

(Existence and algorithm) First we discard certain elements of $G$ to make it irredundant, and then we make it reduced.

For the algorithmic result, we will use the fact that division of standard terms in $A_{lt}$ is computable (Proposition~\ref{prop:div-alg}). 

To make $G$ irredundant, if $g_j \ltdivides g_i$ for some $j \neq i$, then remove $g_i$ from $G$ (if $g_i$ and and $g_j$ have the same leading monomial, we may remove either one of them). We claim that $G' := G \backslash \{g_i\}$ is still an ASL Gröbner basis for the same ideal. Let $f \in \langle G \rangle$. Then by Proposition~\ref{prop:asl-grobner-divisibility-characterization}, there is some $g \in G$ such that $g \ltdivides f$. If $g \neq g_i$, then $g$ is in $G'$, and thus there is some $g \in G'$ such that $\ltsm(g) \ltdivides \ltsm(f)$. If $g = g_i$, then we have $\ltsm(g_j) \ltdivides \ltsm(g_i) \ltdivides \ltsm(f)$, and thus $\ltsm(g_j) \ltdivides \ltsm(f)$. 

To see that $\langle G' \rangle = \langle G \rangle$, suppose not. Let $f \in \langle G \rangle$ be such that $\ltsm(f)$ is minimal according to $\preceq$ among the leading terms of elements that are in $\langle G \rangle$ but not in $\langle G' \rangle$. By what we have said above, there is some $g \in G'$ such that $g \ltdivides f$. Let $t$ be a standard term such that $\ltsm(gt) = \ltsm(f)$. Then $f-gt$ is in $\langle G \rangle$, and $f-gt$ cannot be in $\langle G' \rangle$ (for otherwise $f$ itself would be). But $\ltsm(f-gt) \prec \ltsm(f)$, contradicting the assumed minimality of $\ltsm(f)$. Thus $G$ and $G'$ generate the same ideal. Finally, by what we have said above, for all $f \in I$, there is some $g \in G'$ such that $\ltsm(g) \ltdivides \ltsm(f)$, and thus by Proposition~\ref{prop:asl-grobner-divisibility-characterization}, $G'$ is an ASL Gröbner basis for $\langle G \rangle$.

Repeat the above process, removing elements of $G$, until the resulting set is an irredundant ASL Gröbner basis.

Now let $G$ be an irredundant ASL Gröbner basis. We will show how to make it reduced.

Rescale each element of $G$ to have leading coefficient $1$.

Let $B_i$ be the set of standard monomials $m$ appearing in $g_i$ such that some other $g_j \ltdivides m$, and let $B = \bigcup_i B_i$. We proceed by downward induction on $m = \max B$. If $B$ is empty, then we are done. Otherwise, let $m = \max B$ (maximum according to $\preceq$), and let 
\[
\mathcal{I}_m = \{i \in \{1,\dotsc,k\} : m \in B_i\}.
\]
Say $g_j \ltdivides m$. Then for each $i \in \mathcal{I}_m$, let $t_i$ be a standard term such that $m$ does not appear in $g_i - t_i g_j$. Replace $g_i$ by $g_i - t_i g_j$. (Note that there is no circularity here, as $j$ cannot be in $\mathcal{I}_m$ because we had previously made $G$ irredundant.) After doing this for all $i \in \mathcal{I}_m$, we strictly reduce $\max B$, and by well-orderedness of ASL term orders (Proposition~\ref{prop:well-order}), this can happen at most finitely many times before the resulting $G$ is reduced. This completes the algorithm and proof of existence.

(Uniqueness) As $\lismlt(I)$ is a standard monomial ideal in $A_{lt}$, by Lemma~\ref{lem:unique-min}, it has a unique minimum generating set of standard monomials, say $M$. Let $M = \{m_1, \dotsc, m_k \}$ where 
\[
m_1 \preceq m_2 \preceq \dotsb \preceq m_k.
\] 

Now, suppose $G,G'$ are two reduced ASL Gröbner bases for $I$ with respect to $(\preceq, A_{lt})$. We proceed by induction on $i$. Let $0 \leq i < k$. Suppose by induction that we have found elements $g_1, \dotsc, g_i \in G$ such that $\lmsm(g_j) = m_j$ for all $1 \leq j \leq i$ and all $g_j$ are also in $G'$; when $i=0$ this set is empty and the ``such that'' is vacuously true. Then there are elements $g_{i+1} \in G$ and $g'_{i+1} \in G'$ such that 
\[
\lmsm(g_{i+1}) = \lmsm(g'_{i+1}) = m_{i+1}.
\]
If $g_{i+1} \neq g'_{i+1}$, then $g_{i+1} - g'_{i+1}$ is nonzero and has leading term $m$ strictly smaller than $\ltsm(g_{i+1})$. Note that $m$ must occur among the standard monomials in either $g_{i+1}$ or $g'_{i+1}$. Without loss of generality, by swapping the roles of $G$ and $G'$ if needed, suppose that $m$ is in $g_{i+1}$. Since $G$ is an ASL Gröbner basis, there must be some $g \in G$ such that $g \ltdivides m$ (by Proposition~\ref{prop:asl-grobner-divisibility-characterization}). But since $m$ is in $g_{i+1}$, this contradicts the reducedness of $G$. Thus we have $g_{i+1} = g'_{i+1}$, as desired. By induction, we thus conclude that $G = G'$.
\end{proof}

\getkeytheorem{lemuniversal}

\begin{proof}[Proof of Lemma~\ref{lem:universal}]
Let \[\mathcal{L} = \{ (\lismlt)_{\preceq}(I) \suchthat \preceq \text{ is an ASL term order on $A$} \},\] where here each element of $\mathcal{L}$ is an ideal in $\Adisc$, and suppose for the sake of contradiction that $\mathcal{L}$ is infinite. Let $f_1$ be a nonzero polynomial in $I$. Then every leading ideal of $I$, regardless of term order, must contain at least one of the monomials of $f_1$. Since $\mathcal{L}$ is infinite, there must be at least one monomial, say $m_1$, that occurs in $f_1$, and such that $m_1$ is contained in infinitely many ideals in $\mathcal{L}$. 

Now suppose by induction that we have found $k \geq 1$ monomials $m_1, \dotsc, m_k$ such that there are infinitely many elements of $\mathcal{L}$ containing $\{m_1, \dotsc, m_k\}$; call the resulting set of ideals $\mathcal{L}_k$. Let $\preceq$ be a term order such that the leading ideal relative to $\preceq$ occurs in $\mathcal{L}_k$, that is, contains $m_1, \dotsc, m_k$. Since $\mathcal{L}_k$ is infinite, in particular we have $|\mathcal{L}_k| \geq 2$, and thus there is some pseudo-ASL term order $\preceq$ such that the leading ideal of $I$ with respect to $(\preceq,\Adisc)$ is in $\mathcal{L}_k$ and is not equal to $\langle m_1, \dotsc, m_k \rangle$ (just because the latter, if it is in $\mathcal{L}_k$ at all, cannot be the only ideal in $\mathcal{L}_k$). Now, since $(\lismlt)_{\preceq}(I) \supsetneq \langle m_1, \dotsc, m_k \rangle$, there is some element $f_{k+1} \in I$ whose $\preceq$-leading monomial is not in $\langle m_1, \dotsc, m_k \rangle$. Among the finitely many standard monomials in $f_{k+1}$, every leading ideal must contain at least one of them. Since $\mathcal{L}_k$ is infinite, there is some monomial $m_{k+1}$ occurring in $f_{k+1}$ such that infinitely many of the ideals in $\mathcal{L}_k$ contain $m_{k+1}$. 

By induction, this gives us a strictly increasing chain of (standard monomial) ideals in $\Adisc$:
\[
\langle m_1 \rangle \subsetneq \langle m_1, m_2 \rangle \subsetneq \langle m_1, m_2, m_3 \rangle \subsetneq \dotsb.
\]
But this contradicts the fact that $\Adisc$ is Noetherian.
\end{proof}

\section{Deferred proofs of claims from Section \ref{sec:bideterminants}}
\label{sec:app:bideterminants}

\getkeytheorem{claimlexicographicunion}

\begin{proof}[Proof of Claim \ref{claim:lexicographic-union}]
Let $\mathcal{D}(A) = (a_1, \ldots, a_s)$ and $\mathcal{D}(B) = (b_1, \ldots, b_t)$. By appending a $0$ to the end of both lists, we can assume, without loss of generality,\footnote{The artifice of appending a $0$ helps us avoid treating the case of $\mathcal{D}(A)$ being a prefix of $\mathcal{D}(B)$ seperately.} that $\mathcal{D}(A) <_{lex} \mathcal{D}(B)$ because there is an index $1 \le i \le \min(s, t)$ such that $a_i < b_i$ and $a_j = b_j$ for all $j < i$. 

The claim is easy to verify when $\gamma \ge b_i$ because this implies that $\gamma > a_i$, which in turn means that in $\mathcal{D}(A \cup \{\gamma\}) = (a_1, \ldots, \gamma, \ldots, a_i, \ldots)$ and $\mathcal{D}(B \cup \{\gamma\}) = (b_1, \ldots, \gamma, \ldots, b_i, \ldots)$, implying $\mathcal{D}(A \cup \{\gamma\}) <_{lex} \mathcal{D}(B \cup \{\gamma\})$. The claim is also easy to verify when $\gamma \le a_i$, because this implies that $\gamma < b_i$ which in turn means that in $\mathcal{D}(A \cup \{\gamma\}) = (a_1, \ldots, a_i, \ldots, \gamma, \ldots)$ and $\mathcal{D}(B \cup \{\gamma\}) = (b_1, \ldots, b_i, \ldots, \gamma, \ldots)$, implying $\mathcal{D}(A \cup \{\gamma\}) <_{lex} \mathcal{D}(B \cup \{\gamma\})$ even in this case.

\begin{table}[H]
\centering
\begin{tabular}{|ccccccc|}
\multicolumn{7}{c}{$\gamma \ge b_i > a_i$}\\
\hline
$\mathcal{D}(A \cup \{\gamma\}):$ & $a_1$ & $\ldots$ & $\gamma$ & $\ldots$ & $a_i$ & $\ldots$ \\
& $\veq$ & $\veq$ & $\veq$ & $\veq$ & $\vlt$ & \\
$\mathcal{D}(B \cup \{\gamma\}):$ & $b_1$ & $\ldots$ & $\gamma$ & $\ldots$ & $b_i$ & $\ldots$ \\
\hline
\end{tabular}
\hspace{1pt}
\begin{tabular}{|cccccc|}
\multicolumn{6}{c}{$\gamma \le a_i < b_i$}\\
\hline
$\mathcal{D}(A \cup \{\gamma\}):$ & $a_1$ & $\ldots$ & $a_i$ & $\ldots$ & $\gamma$\\
& $\veq$ & $\veq$ & $\vlt$ & & \\
$\mathcal{D}(B \cup \{\gamma\}):$ & $b_1$ & $\ldots$ & $b_i$ & $\ldots$ & $\gamma$\\
\hline
\end{tabular}
\end{table}

The only case left to be considered is when $a_i < \gamma < b_i$.  Notice that it is not possible that $\gamma > a_{i-1}$, because if so $\gamma > a_{i-1} = b_{i-1} \ge b_i$, which contradicts our assumption. Thus gamma must necessarily be inserted in between $a_{i-1}$ and $a_i$, which implies that $\mathcal{D}(A \cup \{\gamma\}) = (a_1, \ldots, a_{i-1}, \gamma, a_i, \ldots)$ and $\mathcal{D}(B \cup \{\gamma\}) = (b_1, \ldots, b_{i-1}, b_i, \ldots, \gamma, \ldots)$. Since $\gamma < b_i$, we have that $\mathcal{D}(A \cup \{\gamma\}) <_{lex} \mathcal{D}(B \cup \{\gamma\})$ even in this final case, completing the proof of \ref{point:descending}.

\begin{table}[H]
\centering
\begin{tabular}{|cccccccc|}
\multicolumn{8}{c}{$a_i < \gamma < b_i$}\\
\hline
$\mathcal{D}(A \cup \{\gamma\}):$ & $a_1$ & $\ldots$ & $a_{i-1}$ & $\gamma$ & $a_i$ & $\ldots$ & \\
& $\veq$ & $\veq$ & $\veq$ & $\vlt$ & & & \\
$\mathcal{D}(B \cup \{\gamma\}):$ & $b_1$ & $\ldots$ & $b_{i-1}$ & $b_i$ & $\ldots$ & $\gamma$ & $\ldots$ \\
\hline
\end{tabular}
\end{table}

The proof of \ref{point:descending} gives a proof of \ref{point:ascending} as well.
\end{proof}

\section{Straightening algorithm of \cite{desarmenien1982invariant}}
\label{sec:app:bidet-straightening}

This section gives algorithmic details about Theorem \ref{thm:straightening}-\ref{point:straightening}. Lemma \ref{lem:shuffle} below from \cite{desarmenien1982invariant} helps rewrite a bideterminant $\bidet{R}{C}$ of length $2$, i.e. a product of two minors, as a sum of two kinds of terms. The first kind contains bideterminants of length $2$, and of the same shape as that of $\bidet{R}{C}$. The second kind contains bideterminants of length $3$, with a shape strictly smaller than that of $\bidet{R}{C}$. All bideterminants have the same content; the only difference is that the elements in the tableaux are \emph{shuffled} around.

\begin{lemma}[\cite{desarmenien1982invariant}]
\label{lem:shuffle}
Suppose $i_1 < \ldots < i_p < l_1 < \ldots < l_q$, and $j_1 < \ldots, j_s$, $t_1 < \ldots < t_n$. Also, suppose $s < p + q$, and $n < p + q$. Then
\begin{align}&\left(\left.\parbox{0.5cm}{\centering \tableau{ i_1 \\ \vdots \\ i_s}}\,\right|\parbox{0.5cm}{\centering \tableau{j_1 \\ \vdots \\ j_s}}\right) \left(\left.\parbox{0.5cm}{\centering \tableau{ l_1 \\ \vdots \\ l_n}}\,\right|\parbox{0.5cm}{\centering \tableau{t_1 \\ \vdots \\ t_n}}\right)\nonumber \\
=\;&- \sum_{\substack{\sigma \in \perms{(i_1, \ldots, i_p, l_1, \ldots, l_q)} \setminus \{\mathrm{id}\} \\ \sigma(i_1) < \ldots < \sigma(i_p) \\ \sigma(l_1) < \ldots < \sigma(l_q)}} \sgn{\sigma} \left(\left.\parbox{0.5cm}{\centering \tableau[m]{ \sigma(i_1) \\ \vdots \\ \sigma(i_p) \\ i_{p+1} \\ \vdots \\ i_s}}\,\right|\parbox{0.5cm}{\centering \tableau[m]{j_1 \\ \vdots \\ \vdots \\ \vdots \\ \vdots \\ j_s}}\right) \left(\left.\parbox{0.5cm}{\centering \tableau[m]{ \sigma(l_1) \\ \vdots \\ \sigma(l_q) \\ l_{q+1} \\ \vdots \\ l_n}}\,\right|\parbox{0.5cm}{\centering \tableau[m]{t_1 \\ \vdots \\ \vdots \\ \vdots \\ \vdots \\ t_n}}\right) \nonumber \\
&+ \sum_{\substack{\tau \in \perms{(j_1, \ldots, j_s)}\\ \tau(j_1) < \ldots < \tau(j_p) \\ \tau(j_{p+1}) < \ldots < \tau(j_s) \\ \mu \in \perms{(t_1, \ldots, t_n)} \\ \mu(t_1) < \ldots < \mu(t_q) \\ \mu(t_{q+1}) < \ldots < \mu(t_n)}} \sgn{\tau} \cdot \sgn{\mu} \left(\left.\parbox{0.5cm}{\centering \tableau[m]{ i_1 \\ \vdots \\ i_{p} \\ l_1 \\ \vdots \\ l_q}}\,\right|\parbox{0.5cm}{\centering \tableau[m]{\tau(j_1) \\ \vdots \\ \tau(j_p) \\ \mu(t_1) \\ \vdots \\ \mu(t_q)}}\right) \left(\left.\parbox{0.5cm}{\centering \tableau[l]{ i_{p+1} \\ \vdots \\ i_s}}\,\right|\parbox{0.5cm}{\centering \tableau[l]{\tau(j_{p+1}) \\ \vdots \\ \tau(j_s)}}\right) \left(\left.\parbox{0.5cm}{\centering \tableau[l]{ l_{q+1} \\ \vdots \\ l_n}}\,\right|\parbox{0.5cm}{\centering \tableau[l]{\mu(t_{q+1}) \\ \vdots \\ \mu(t_n)}}\right)\nonumber
\end{align}

\end{lemma}

\begin{proof}[Proof sketch of Theorem \ref{thm:straightening}-\ref{point:straightening}]
Straightening of any arbitrary bideterminant is done by recursively applying Lemma \ref{lem:shuffle}. If a bideterminant is non-standard, then it is sure to contain at least two minors that are incomparable according to the partial order on minors (c.f. Definition \ref{defn:po-minors}). In symbols, suppose that without loss of generality, $s \ge t$, $i_1 < \ldots i_s$, $j_1 < \ldots < j_s$, $l_1 < \ldots < l_n$, and $t_1 < \ldots < t_n$, and that the two \emph{clashing} minors are 
\[\bigbidet{\tableau{i_1 \\ \vdots \\ i_s}}{\tableau{j_1 \\ \vdots \\ j_s}} \qquad \text{and} \qquad \bigbidet{\tableau{l_1 \\ \vdots \\ l_n}}{\tableau{t_1 \\ \vdots \\ t_n}}.\] 

This means that, again without loss of generality, we have an index $r \le n$ such that $i_{\alpha} \le l_{\alpha}$ for all $\alpha < r$, and $i_r > l_r$. This in turn means we have a sequence $l_1 < \ldots < l_r < i_r < i_{r+1} < \ldots < i_s$, which is a strictly increasing sequence of length $s+1$. Thus the hypotheses of Lemma \ref{lem:shuffle} are satisfied, and we can rewrite the above product of two minors. In \cite[Theorem~3.2]{grinberginvariant} it is shown that consistently repeating this procedure will eventually terminate in finding an expression that only contains standard bideterminants.
\end{proof}

\section{Supplement to Section \ref{sec:dimension}}

\subsection{Algorithm to compute Hilbert--Poincar\'e series in the polynomial ring}
\label{sec:app:hilbert-series-algorithm}

Algorithm \ref{alg:hilbert-clo} computes the Hilbert--Poincar\'e series for a monomial ideal in the polynomial ring $R[X_1, \ldots, X_n]$. Algorithm \ref{alg:hilbert-clo} works for any $\N$-grading reposed on $R[X_1, \ldots, X_n]$. The presentation is adapted from \cite[Chapter 10 - \S 2, Theorem 6]{CLO} which is tailored to the standard grading.

\begin{algorithm}[H]
\caption{Algorithm to compute $H_{R[X_1, \ldots, X_n], I}$ when $I$ is a monomial ideal, where $\deg(X_i) = d_i$ for all $i$}
\label{alg:hilbert-clo}
\begin{algorithmic}[1]
\Require{A monomial ideal $I \subseteq R[X_1, \ldots, X_n]$}
\medskip
\Procedure{HPS}{$I$}
\State $T \gets$ minimal set of generators of $I$
\If{$T = \emptyset$}
\State \Return $\prod_{i=1}^{n}\frac{1}{(1-t^{d_i})}$ \label{alg:hps:line:degree1-trivial}
\ElsIf{$T = \{1\}$}
\State \Return $0$
\ElsIf{$T$ consists of variables $x_{i_1}, \dotsc, x_{i_p}$}
\State \Return $\prod_{i \notin \{i_1, \dotsc, i_p\}} \frac{1}{(1-t^{d_i})}$ \label{alg:hps:line:degree1}
\Else
\State Select $x_i$ appearing in a monomial in $T$ that is more than just a lone variable \label{alg:hps:line:select-branching-variable}
\State Return $\textsc{HPS}(I + \ideal{x_i}) + t^{d_i} \cdot\textsc{HPS}(I \suchthat x_i)$ \label{alg:hps:line-recursion}
\EndIf
\EndProcedure
\end{algorithmic}
\end{algorithm}

\subsection{Deferred Proofs}
\label{sec:app:dimension}

\getkeytheorem{claimcolonidealrecursivecall}

\begin{proof}[Proof of Claim \ref{claim:colon-ideal-recursive-call}]
Equation \eqref{eqn:recursive-step-onI1} follows from Line \ref{alg:hps:line-recursion} of Algorithm \ref{alg:hilbert-clo}.

Equation \eqref{eqn:i-1-of-recursive} is immediate from definition \eqref{eqn:recursive-definitions-both-i1-i2}, and the fact that since no $x \in A$ divides any $y \in \mathcal{I}_{b, c}$, $A \cap \{x_{i, c}\}_{i \in [j]} = \emptyset$.

By definition \[I_{2}^{(j)} = I_{1}^{(j-1)} \suchthat \ideal{x_{j, c}} = \left(\ideal{A \sqcup \{x_{i, c}\}_{i \in [j-1]}} + \mathcal{I}_{b, c}\right) \suchthat \ideal{x_{j, c}}.\] For any fixed $j$, $x_{j, c}$ is incompatible with all of $\{x_{r,s}\}_{\substack{j < r \le b \\ 1 \le s \le c-1}}$, thus all of them are present as degree $1$ variables in $I_{2}^{(j)}$, which in turn eliminates all incompatibilities involving $\{x_{r,s}\}_{\substack{j < r \le b \\ 1 \le s \le c-1}}$. In other words, $I_2^{(j)}$ will not contain any incompatibilities from row $j+1$ or higher. Also, since the degree $1$ variables $\{x_{i, c}\}_{i \in [j-1]}$ are already present, $I_2^{(j)}$ will not contain any incompatibilities involving $\{x_{i, c}\}_{i \in [j-1]}$. Finally, $x_{j, c}$ cannot be present in the list of incompatibilities because (a) no incompatibilities from row $j+1$ or higher are present, and (b) nothing in row $1$ to row $j$ can be incompatible with $x_{j, c}$. Thus, we are left with all the incompatibilities in the top $j \times (c-1)$ submatrix of $X$, i.e. $\mathcal{I}_{j, c-1}$. This completes the proof of Equation \eqref{eqn:i-2-of-recursive}, i.e. that \[
I_2^{(j)} = \ideal{A \sqcup \{x_{i, c}\}_{i \in [j-1]} \sqcup \{x_{r, s}\}_{\substack{j < r \le b \\ 1 \le s \le c-1}}} + \mathcal{I}_{j, c-1}.\]

Equation \eqref{eqn:colon-induction-hypothesis} follows just by Equation \eqref{eqn:i-2-of-recursive} and by definition of $g$, i.e. \eqref{eqn:general-hp-series-big}.

Finally, note that (a) $I_{1}^{(b-1)}$ contains $\{x_{i, c}\}_{i \in [b-1]}$ are degree $1$ variables, which in turn means all incompatibilities involving $\{x_{i, c}\}_{i \in [b-1]}$ are no longer present in $I_{1}^{(b-1)}$, and (b) $x_{b, c}$ is compatible with all elements of $X$. This in turn means that $I_{1}^{(b-1)}$ contains $\mathcal{I}_{b, c-1}$, thus proving Equation \eqref{eqn:terminal-i1}.
\end{proof}

\getkeytheorem{claimnumeratorsmatch}
\begin{proof}[Proof of Claim \ref{claim:numerators-match}]
Without loss of generality, suppose that $b \le c$. The numerator of Equation \eqref{eqn:induction-step-final}, say $P_1 \in \Z[t]$, and the numerator of the RHS of Equation \eqref{eqn:general-hp-series-big}, say $P_2 \in \Z[t]$, are both polynomials of degree exactly $b-1$. We will match the coefficients of $P_1$ and $P_2$. The coefficient of $t^0$ in $P_1$ is $\binom{b-1}{0}\binom{c - 2}{0} = 1$, which matches the coefficient of $t^0$ in $P_2$. 

Next we shall match the coefficients of $t^{b-1}$ in $P_1$ and $P_2$. If $b = c$, then the coefficient of $t^{b-1}$ in $P_1$ is equal to $\binom{b-2}{b-2}\binom{b-2}{b-2} = 1$ which matches the coefficient of $t^{b-1}$ in $P_2$. Otherwise it means that $b < c$. Then the coefficient of $t^{b-1}$ in $P_1$ is equal to $\binom{c-2}{b-1} + \binom{b-2}{b-2}\binom{c-2}{b-2} = \binom{c-1}{b-1}$, which again matches the coefficient of $t^{b-1}$ in $P_2$.

Finally, for any $b-2 \ge d \ge 1$, the coefficient of $t^d$ in $P_1$ is
\begin{align}
\binom{b-1}{d}\binom{c-2}{d} &+ \sum_{j=d}^{b-1} \binom{j-1}{d-1}\binom{c-2}{d-1} \nonumber \\
&= \binom{b-1}{d}\binom{c-2}{d} + \binom{c-2}{d-1}\left(\binom{d-1}{d-1} + \binom{d}{d-1} + \ldots + \binom{b-2}{d-1}\right) \nonumber \\
&= \binom{b-1}{d}\binom{c-2}{d} + \binom{c-2}{d-1}\left(\binom{d}{d} + \binom{d}{d-1} + \ldots + \binom{b-2}{d-1}\right) \nonumber \\
&= \binom{b-1}{d}\binom{c-2}{d} + \binom{c-2}{d-1}\binom{b-1}{d} \nonumber \\
&= \binom{b-1}{d}\binom{c-1}{d} \nonumber,
\end{align}
which matches the coefficient of $t^d$ in $P_2$.
\end{proof}